\documentclass[onefignum,onetabnum]{siamonline220329}

\ifpdf
\hypersetup{
  pdftitle={Error Analysis of Triangular Optimal Transport Maps for Filtering},
  pdfauthor={Doe}
}
\fi

\newcommand{\editstart}{\color{black}}
\newcommand{\editend}{\normalcolor}

\newcommand{\alphamin}{\alpha_{\text{min}}}
\newcommand{\betamax}{\beta_{\text{max}}}
\newcommand{\U}{\mathcal{U}}
\newcommand{\F}{\mathcal{F}}
\newcommand{\Y}{\mathcal{Y}}
\newcommand{\PP}{\mathcal{P}}
\newcommand{\T}{\mathcal{T}}
\renewcommand{\Re}{\mathbb R}
\newcommand{\cS}{\mathcal{S}}

\renewcommand{\d}{{\,\rm{d}}}

\newcommand{\Expect}{\mathbb{E}}
\newcommand{\CPI}{C_{\rm PI}}
\newcommand{\CVX}{{\rm CVX}}

\newcommand{\V}{\mathscr{V}}
\newcommand{\W}{\mathscr{W}}

\newcommand{\Cfs}{C_{\rm stab.}}
\newcommand{\Cf}{C_{\rm filter}}

\newcommand{\bias}{{\text{bias}}}
\newcommand{\wh}[1]{\widehat{#1}}

\newcommand{\distance}{\mathrm{AM\text{-}W}_2}

\usepackage{subcaption}
\usepackage{lipsum}
\usepackage{amsfonts}
\usepackage{graphicx}
\usepackage{epstopdf}
\usepackage{algorithmic}
\usepackage{pict2e}
\usepackage{amsmath}
\usepackage{amssymb}
\usepackage{mathrsfs}
\usepackage{mathtools}
\usepackage{overpic}
\usepackage{algorithm}
\usepackage[scr=boondox]{mathalfa}
\usepackage{xr}
\usepackage[normalem]{ulem}
\usepackage{changepage} %
\usepackage{bbm} %
\usepackage{tikz}
\usetikzlibrary{arrows.meta, calc, decorations.pathreplacing, fit}

\ifpdf
  \DeclareGraphicsExtensions{.eps,.pdf,.png,.jpg}
\else
  \DeclareGraphicsExtensions{.eps}
\fi
\DeclareRobustCommand{\lowerrighttriangle}{%
  \begingroup
  \setlength{\unitlength}{1ex}%
  \begin{picture}(1,1)
  \polyline(1,0)(0,1)(0,0)(1,0)(.5,0)
  \end{picture}%
  \endgroup
}

\DeclareMathOperator*{\argmin}{arg\,min}

\usepackage{enumitem}
\setlist[enumerate]{leftmargin=.5in}
\setlist[itemize]{leftmargin=.5in}

\newsiamremark{remark}{Remark}
\newsiamremark{hypothesis}{Hypothesis}
\newsiamremark{assumption}{Assumption}
\crefname{assumption}{Assumption}{Assumptions}
\newsiamthm{problem}{Problem}
\newsiamremark{sproof}{Proof Sketch}
\crefname{hypothesis}{Hypothesis}{Hypotheses}
\crefname{theorem}{Theorem}{Theorems}
\newsiamthm{claim}{Claim}
\headers{Error Analysis of Triangular Optimal Transport Maps for Filtering}{M. Al-Jarrah, B. Hosseini, N. Jin,  M. Martino, A. Taghvaei}

\title{Error Analysis of Triangular Optimal Transport Maps for Filtering
\thanks{Submitted to the editors DATE.
\funding{This work was supported by the National Science Foundation (NSF) awards EPCN-2318977, EPCN-2347358, DMS-2208535, and DMS-2337678.}}}

\author{Mohammad Al-Jarrah\thanks{Department of Aeronautics \& Astronautics, University of Washington, Seattle; (\email{mohd9485@uw.edu}, \email{amirtag@uw.edu}).}
\and Bamdad Hosseini\thanks{Department of Applied Mathematics,  University of Washington, Seattle; (\email{bamdadh@uw.edu}, \email{njin2@uw.edu}, \email{michem29@uw.edu}).}
\and Niyizhen Jin\footnotemark[3] \and Michele Martino\footnotemark[3] \and Amirhossein Taghvaei\footnotemark[2]}

\usepackage{amsopn}

\renewcommand{\appendixname}{Supplementary Materials}
\crefname{appendix}{SM}{SM}

\newcommand*\Laplace{\mathop{}\!\mathbin\bigtriangleup}

\begin{document}

\maketitle
\begin{abstract}
We present a systematic analysis of estimation errors for a class of optimal transport based algorithms for filtering and data assimilation. Along the way, we extend previous error analyses of Brenier maps to the case of conditional Brenier maps that arise in the context of simulation based inference. We then apply these results in a filtering scenario to analyze the optimal transport filtering algorithm of \cite{al-jarrah2024nonlinear}. An extension of that algorithm along with numerical benchmarks on various non-Gaussian and high-dimensional examples are provided to demonstrate its effectiveness and practical potential. 
\end{abstract}

\begin{keywords}
Nonlinear filtering, 
Optimal transport, 
Data assimilation, 
Bayesian inference, 
Simulation based inference.
\end{keywords}

\begin{MSCcodes}
60G35, %
49Q22, %
65M32, %
62F15. %
\end{MSCcodes}

\section{Introduction}\label{sec:intro}
This paper outlines a systematic analysis of a class of optimal transport (OT)
algorithms for filtering of nonlinear systems. Our analysis 
provides quantitative estimation rates for such algorithms combining 
statistical errors due to finite samples as well as approximation errors 
due to parameterization of transport maps. 
Numerical experiments further highlight the applicability of our OT methods 
and investigate the manifestation of our theory in practical benchmarks.

Nonlinear filtering is the problem of inferring the state of a 
dynamical system from partial and noisy observations. This task is 
achieved via a probabilistic formulation leading to 
a sequence of conditional distributions for the state of the dynamical 
system obtained by successive applications of Bayes' rule.  
Early engineering applications of filtering include target tracking~\cite{bar2004estimation}, satellite orbit determination \cite{manarvi2023application}; navigation \cite{levy1997kalman}; and spaceflight, namely the Apollo missions~\cite{hoag1969apollo}. Soon after, filtering 
found broad applications in other areas such as 
economic forecasting~\cite{brigo1998some,javaheri2003filtering};
geoscience~\cite{van2010geosciences,reich2015probabilistic};
neuroscience~\cite{truccolo2005point};
psychology~\cite{kording2007dynamics,drugowitsch2014optimal};
and  machine learning~\cite[Ch. 13]{bishop2006pattern};
see~\cite{kutschireiter2020hitchhiker} for more applications of filtering 
and historical remarks.

In this article, we focus on the discrete-time 
filtering problem where a hidden Markov process $\{U_{t}\}_{t=0}^{\infty}$ represents the state of a dynamical system, and a random process $\{Y_{t}\}_{t=1}^{\infty}$ represents the observations. We assume that the state and observation processes follow 
\begin{equation}\label{eq:generic-filtering-model}
\left\{
\begin{aligned}
    &U_{t} \sim a(\cdot \mid U_{t-1}), \quad U_{0} \sim \pi^{0},
    \\
    &Y_{t} \sim h(\cdot \mid U_{t}),
\end{aligned}
\right.
\end{equation}
where $\pi_{0}$ is the initial distribution of the state, 
$a(u' \mid u)$ is the transition probability kernel from state $u$ to the state $u'$, and $h(y \mid u)$ is the transition kernel or the likelihood distribution of the observation $y$ given the state $u$ \footnote{Here we assumed that $a$ and $h$ are time-invariant for simplicity 
but our algorithms and results directly apply to the time-varying case where 
these kernels depend on $t$.}. 
The goal of filtering is to infer the conditional distribution $\pi^t$ of the hidden state $U_{t}$ from a sequence of observations 
$\mathscr{Y}_{t} = \{Y_{1},Y_{2},...,Y_{t}\}$, i.e., the approximation of the sequence 
of distributions $\pi^{t}(\cdot) \coloneqq \mathbb{P}(U_{t} \in \cdot \mid \mathscr{Y}_{t})$, for $t=1,2,...$, also known as \textit{the posterior distribution}. The sequence $\pi^t$ admits a recursive relationship thanks to the model \eqref{eq:generic-filtering-model}: Define the operators,
\begin{equation}\label{filtering-recursive-update-rule-0}
\begin{aligned}
    &(\textnormal{Forecast step}) \quad \eta_\U^{t} = \mathcal{A}[\pi^{t-1}] \coloneqq \int_{\U} a(\cdot \mid u) \, \d \pi^{t-1}(u),
    \\
    &(\textnormal{Analysis step}\footnotemark) \quad  \pi^{t} = \mathcal{B}_{y} [\eta_\U^{t}] \coloneqq \frac{h(y \mid \cdot)\eta_\U^{t}(\cdot)}{\int_{\U}h(y \mid u)\d \eta_\U^{t}(u)}.
\end{aligned}
\end{equation}
\footnotetext{
\editstart 
The formula for $\mathcal{B}_y$ is well-defined for
$\nu^t_\mathcal{Y}$-a.e. $y\in\mathcal{Y}$, i.e. whenever
$\int_\mathcal{U} h(y|u)\,d\eta^t_\mathcal{U}(u)>0$.
Otherwise we make sense of the Bayesian update via 
the disintegration of measures.
\editend
}
Then  the posteriors follow the sequential update law~\cite{cappe2009inference}
\begin{equation}\label{filtering-recursive-update-rule}
    \pi^t = (\mathcal{B}_{Y_t} \circ \mathcal{A})[ \pi^{t-1}]. 
\end{equation}
The operator $\mathcal{A}$ in the forecast step is known as the {\it propagation operator}
while the operator $\mathcal{B}_y$ in the analysis step is known as the {\it conditioning or 
Bayesian update operator}. Given the previous posterior 
$\pi^{t-1}$, the propagation operator $\mathcal{A}$ updates the state according to the dynamics of the model resulting in
the intermediate distribution $\eta_\U^{t} = \mathbb{P}(U_{t} \in \cdot \mid \mathscr{Y}_{t-1})$. Then the conditioning operation $\mathcal{B}_{Y_t}$ applies Bayes' rule with $\eta_\U^t$
as the prior distribution to obtain the next 
posterior distribution $\pi^t$ in the sequence.

While analytically simple, the numerical evaluation of 
\eqref{filtering-recursive-update-rule}, in particular the  
conditioning operator $\mathcal{B}_y$, is challenging~\cite{kantas2009overview,doucet09,sarkka2023bayesian,chen2003bayesian}. 
\editstart
While the forecast step can be carried out  at the level of the particles representing $\pi^{t-1}$, simply by simulating the dynamics forward, the implementation of the analysis step, for a general nonlinear/non-Gaussian model, admits no such straightforward particle-level implementation. 
\editend
To address this shortcoming, following the recent works,  
\cite{al-jarrah2023optimal,al-jarrah2024nonlinear,taghvaei2022variationalformulation,spantini2022coupling}, we consider approximating $\mathcal{B}_y$ using a transport map. 
The key idea is to assume the model
\begin{equation}\label{basic-filter}
        T_t(y,\cdot) \# \eta_\U^{t} = \pi^{t} = \mathcal{B}_y [\eta_\U^{t}],
\end{equation}
where $\#$ denotes the push-forward operator \footnote{$T\# \mu (A) = \mu( T^{-1}(A))$ for a Borel measure $\mu$, and Borel map $T$, 
and for any Borel set $A$.}. Then $T_t(y, \cdot)$ denotes a family of transport maps, parameterized 
 by time $t$ and the observations $y$, 
 that push $\eta_\U^{t}$ to $\pi^{t}$, thereby coinciding with 
 the Bayesian update operator $\mathcal{B}_y$. The algorithms 
 that we study directly approximate $T_t$ with a sequence of  maps $\wh{T}_t$
 that can be computed in practice from an ensemble of particles 
 and parameterized using polynomials, neural nets, kernel methods, or other 
 practical models. To this end, we consider the approximate sequence of posteriors
 \begin{equation}\label{intro:empirical-transport-bayesian-update}
   \wh \pi^t := \wh{T}_t(y, \cdot) \# \wh{\eta}_\U^t 
\quad \text{where} \quad \wh{\eta}_\U^t := \mathcal{A} [\wh{\pi}^{t-1}], \quad \text{and} \quad \wh{\pi}^0 = \pi^0.   
 \end{equation}
We are particularly interested in the setting where $\wh{T}$ are 
conditional OT maps as reviewed in \Cref{sec:conditional-OT-review};
 see also 
\cite{hosseini2025conditional, chemseddine2025conditional}.
Then our goal  is to understand the error of $\wh \pi^t$, compared with $\pi^t$, thereby putting the 
algorithms of \cite{al-jarrah2023optimal,al-jarrah2024nonlinear,taghvaei2022variationalformulation} on  firm theoretical ground.

\subsection{Summary of contributions} 
The article makes three key contributions towards characterizing 
and understanding the error of \textit{quadratic} conditional OT methods 
for filtering:
\begin{enumerate}[label=(\roman*)]
    \item \emph{\textbf{Error analysis for conditional Brenier maps}}. 
    Our first main contribution is the quantification 
    of statistical estimation errors 
    for conditional OT maps as outlined in \Cref{Sec:Error analysis for conditional OT}. Given a reference measure $\eta_\U \in \PP(\Re^n)$ and 
    a joint target measure $\nu \in \PP( \Re^m \times \Re^n)$, we consider the 
    family of conditional Brenier maps that satisfy 
    $\nabla_u \phi^\dagger( y, \cdot) \# \eta_\U = \nu( \cdot \mid y)$  for $y \in \Re^m$. The functions $\phi^\dagger(y, \cdot)$ are the Brenier potentials, parameterized by $y$, associated with the quadratic optimal transportation 
    of $\eta_\U$ to $\nu(\cdot \mid y)$ as measures on $\Re^n$. We then consider an empirical estimation 
    $\wh\phi$ of $\phi^\dagger$ by solving the dual conditional OT problem 
    using finitely many samples $\{ (y_i, u_i) \}_{i=1}^N \overset{\text{i.i.d.}}{\sim}\nu$
    and by restricting $\widehat\phi$ to a restricted approximation class. 
    With this setup, we establish error bounds of the following form:
    \begin{align*}
    \mathbb{E}^{\text{train}} \int_{\Y} & \|\nabla_u \wh{\phi}(y,\cdot) - \nabla_u \phi^\dagger(y,\cdot)\|_{L_{\eta_\U}^2}^2\d \nu_\Y(y) \\
    & \lesssim \begin{cases}
    N^{-1/2} + \text{approx. bias}\, \quad &\textnormal{(Slow rate, \Cref{thm: slow rate})}\, ,\\
    \frac{\log(N)}{N} {\editstart + \text{ approx. bias}}\, \quad &\textnormal{(Fast rate, \Cref{thm: fast rate})}\, ,
    \end{cases}
    \end{align*}
    where the outer expectation is with respect to the empirical samples used to train/estimate $\wh\phi$ and $\lesssim$ hides independent constants. The slow rate (\Cref{thm: slow rate}) holds under very general assumptions 
    \editstart
    while the faster rate (\Cref{thm: fast rate}) requires stronger regularity 
    assumptions.
    \editend
    We note that the above error bound is of independent 
    interest regarding conditional OT and more broadly, 
    simulation based inference \cite{cranmer2020frontier}.  
    
    \item \emph{\textbf{Error analysis for the OT filter (OTF) algorithm}}. 
    Building upon contribution (i), we analyze the estimation 
    error of an idealized OTF algorithm in \Cref{Sec:Error analysis for OTF}. We consider a sequence of empirical Brenier potentials 
    satisfying $\nabla_u \wh \phi_t(y, \cdot) \# \wh \eta^t_\U 
    = \wh \pi^t$ in the same notation as \eqref{basic-filter}.
    Then we quantify the error of $\wh \pi^t$ by considering 
     appropriate divergences $D: \PP(\Re^n) \times \PP(\Re^n) \to \Re_{\ge 0}$ (e.g. the Wasserstein-1 metric), and giving a bound of the form 
    \begin{align}\label{eq: filt. error}
\mathbb{E}^{\textnormal{train}}_{1:t} \mathbb{E}_{\mathscr{Y}_t} D( \wh \pi^t, 
    \pi^t) \lesssim 
    \begin{cases}
    N^{-1/4} + \text{approx. bias}\, \quad &\textnormal{(Slow rate)}\, ,\\
    \left(\frac{\log(N)}{N} \right)^{1/2}{\editstart + \text{ approx. bias}}\, \quad &\textnormal{(Fast rate)}\,.
    \end{cases}
    \end{align}
    Here $N$ denotes a set of i.i.d. empirical samples that 
    are used to train the subsequent conditional OT maps 
    at every time step that are then used to approximate 
    the Bayesian update. 
    This result is stated in \Cref{thm: multiplicative+additive bound exact filtering error} and is a consequence of a preliminary result 
    in \Cref{thm: approx filter error bound} that 
    proves \eqref{eq: filt. error} with a modified 
    distribution for the data $\mathscr{Y}_t$.
    As before,  the slow rate holds under very general assumptions while the fast rate requires stringent conditions 
    on the filtering distributions at 
    every time step.
    
    \item \emph{\textbf{A practical algorithm and numerical experiments}}. Finally, we present an overview and a new variant of the OTF algorithm of \cite{al-jarrah2023optimal, al-jarrah2024nonlinear} in \Cref{Sec:Numerical experiments}.
    We apply our algorithms on a number of benchmarks 
    including: the Lorenz 63 model, as a highly chaotic and 
    multi-modal problem, and Lorenz 96, as a high-dimensional example, 
    among other benchmarks. Our experiments demonstrate the ability 
    of OTF to capture highly non-Gaussian posteriors surpassing 
    the performance of common algorithms such as the ensemble Kalman filter (EnKF)~\cite{evensen1994sequential}. 
    
\end{enumerate}

\subsection{Literature survey}\label{Literature_survey} 

 Classical nonlinear filtering algorithms, such as the Kalman filter  and its nonlinear extensions~\cite{kalman1960new,kalman-bucy,bar2004estimation}, and sequential importance resampling (SIR) particle filters (PF)~\cite{gordon1993novel,arulampalam2002tutorial,doucet09}
 are widely used in practical problems and commercial applications. However, they 
 are subject to fundamental limitations that prohibit their application to modern high-dimensional problems with strong nonlinear effects: Kalman filters are sensitive to initial conditions and fail to represent multi-modal distributions~\cite{ristic2003beyond,budhiraja2007survey} while SIR 
 suffers from the  particle degeneracy phenomenon that becomes severe in high-dimensional problems, an issue known as the curse of dimensionality~\cite{bickel2008sharp,rebeschini2015can,beskos2014error,bengtsson08}.  
 
 These limitations motivated the development of alternative coupling or transport-based methods in recent years with the aim of overcoming the curse of dimensionality~\cite{daum10,crisan10,reich11,yang2011mean,el2012bayesian,reich2013nonparametric,de2015stochastic,yang2016,marzouk2016introduction,mesa2019distributed,reich2019data,pathiraja2021mckean,taghvaei2021optimal,calvello2022ensemble} see also~\cite{spantini2022coupling} and \cite{taghvaei2023survey} for a recent survey of these topics. 
 The coupling method proposed in~\cite{spantini2022coupling} is the closest method to the OT method
 studied in this paper as well as the previous works~\cite{al-jarrah2023optimal,al-jarrah2024nonlinear,taghvaei2022variationalformulation,grange2023computational}.
  Both approaches are similar since they are considered likelihood-free and amenable to neural network parameterizations.
 A key distinction lies in the structure of the transport map: while our method seeks to compute the OT
 map from the prior to the posterior by solving a min-max problem, \cite{spantini2022coupling}  
 constructs the Knothe–Rosenblatt rearrangement using a maximum likelihood estimator.
 
We would also note that an active area of research around conditional simulation and its 
application for the solution of inverse problems has been developed recently; see~\cite{wang2023efficient,ray2022efficacy,ray2023solution,padmanabha2021solving,dasgupta2025conditional}. By observing the connection between Bayesian inference and the Bayesian update step in 
filtering we realize the opportunity for extending the aforementioned works to the 
filtering problem, revealing a diverse set of potential algorithms.

The error analysis of coupling/OT-based nonlinear filtering algorithms is challenging, as the algorithm involves an interacting particle system. Results are available for EnKF  in the linear Gaussian setting~\cite{gland2009,mandel2011convergence,tong2016nonlinear,stuart2014stability,mandel2015,delmoral2016stability}, with extensions to limited nonlinear setups~\cite{delmoral2017stability,bishop2018stability,jana2016stability}. 
In this paper, we avoid the complications due to the interacting particle  system \editstart by generating independent particles at each time step using 
empirical OT maps \editend 
(see~\Cref{subsec: OTF particle algorithms} for details). 

Our analysis of conditional OT map estimation is based on recent works on statistical
estimation errors of OT maps~\cite{hutter2019minimax, divol2025optimal, chewi2024statistical}. These works rely on the analysis of the semidual formulation of the OT problem via the combination of a stability result for OT maps~\cite[Prop. 10]{hutter2019minimax} and methods from empirical processes and statistical learning theory such as symmetrization for generalization  bounds~\cite{koltchinskii2000rademacher, wainwright2019high,vershynin2018high}, chaining~\cite{dudley1969speed,dudley2010universal, vanderVaartWellner2023}, and one-shot localization techniques~\cite{van1987new,van2002m}. We adapt and extend these ideas to the conditional OT problem~\cite{carlier2016vector, hosseini2025conditional}. We note that the statistical analysis 
of transport problems, beyond OT maps, has been a popular topic in recent 
years \cite{irons2022triangular,marzouk2024distribution,wang2022minimax}. However,
to our knowledge our analysis is the first of such results for conditional 
OT maps and the first extension to the filtering problem.

\subsection{Notation} 
For any set $\mathcal{X} \subset \Re^d$ we write 
$\mathscr{B}(\mathcal{X})$ to denote the $\sigma$-algebra of 
the Borel subsets of $\mathcal{X}$ and in turn $\mathcal{P}(\mathcal{X})$
to denote the space of Borel probability measures supported on $\mathcal{X}$. For $\mu\in \mathcal{P}(\mathcal{X})$
\editstart
, $\mathcal{Z}\subset \mathbb{R}^d$ 
\editend
and $1 \leq k < \infty$, define the $\mu$-weighted Lebesgue space to be \(
L_\mu^k(\mathcal{X}; \mathcal{Z}) := \Bigl\{ f : \mathcal{X} \to \mathcal{Z} \;\Big|\; 
\|f\|_{L_\mu^k(\mathcal{X};\mathcal{Z})} < \infty \Bigr\}
\), where
\(
\|f\|_{L_\mu^k(\mathcal{X};\mathcal{Z})} := 
\left( \int_{\mathcal{X}} \|f(x)\|_{\mathcal{Z}}^k \, d\mu(x) \right)^{1/k}
\). 
For $k = \infty$, 
\(
\|f\|_{L_\mu^\infty(\mathcal{X},\mathcal{Z})} := \operatorname*{ess\,sup}_{x \in \mathcal{X}} \|f(x)\|_{\mathcal{Z}}
\) 
\editstart
 and for simplicity we directly set $\|f\|_{L^\infty}\coloneqq \sup_{x\in \mathcal{X}}\|f(x)\|_{\mathcal{Z}}$. 
\editend 
\editstart
Moreover, for a prescribed weight function $w:\mathcal{X}\to \mathbb{R}$, we also define the $(\mu,w)$-weighted Lebesgue space  \(
L_\mu^k(\mathcal{X}; \mathcal{Z},w) := \Bigl\{ f : \mathcal{X} \to \mathcal{Z} \;\Big|\; 
\|f\|_{L_\mu^k(\mathcal{X};\mathcal{Z},w)} < \infty \Bigr\}
\), where
\(
\|f\|_{L_\mu^k(\mathcal{X};\mathcal{Z},w)} := 
\left( \int_{\mathcal{X}} \|f(x)\|_{\mathcal{Z}}^k \, w(x)d\mu(x) \right)^{1/k}
\). In this paper, we primarily use weights that grow polynomially and hence define the compact notation $ \langle x\rangle \coloneqq 1+\|x\|_\mathcal{X}$. 
\editend
Next, we define $C^k(\mathcal{X})$ to be the space of $k$-differentiable real-valued functions defined on $\mathcal{X}$. 
Similarly, we introduce the weighted Sobolev space
\(
H_\mu^k(\mathcal{X}) := \Bigl\{ f \in L^2_\mu(\mathcal{X}) \;\Big|\; 
\|f\|_{H_{\mu}^k(\mathcal{X})} < \infty \Bigr\}
\),
where
\(
\|f\|_{H_\mu^k(\mathcal{X})}^2 := \sum_{|\alpha|\leq k} 
\bigl\| \partial^\alpha f \bigr\|_{L_\mu^2(\mathcal{X})}^2
\) and $\alpha$ is a multi-index. 
For $f\in L_\mu^1(\mathcal{X})$, when convenient, we use equivalent notation $\mu(f)\coloneqq \int_{\mathcal{X}}f(x)\d \mu(x)  = \mathbb{E}_{x \sim \mu} f(x) = \mathbb{E}_\mu f$ for the 
expectation of $f$ with respect  to $\mu$ and 
similarly  $\textnormal{Var}_\mu(f)\coloneqq \left\|f-\mu(f)\right\|_{L_\mu^2(\mathcal{X})}^2$ for its variance.
\editstart
Finally, when the context is clear, 
we omit the domain and codomain spaces $\mathcal{X}$ and $\mathcal{Z}$ and simply denote $L_\mu^k$ the $\mu$-weighted Lebesgue space, $L_\mu^k(w)$ the $(\mu,w)$-weighted Lebesgue space, and $H_\mu^k$ the subsequent Sobolev spaces.
\editend

\editstart
We say that $\mu$ satisfies the $\langle\cdot\rangle^{-q}$-strengthened Poincar\'e inequality if there exists a constant $C_{\textnormal{PI}}^{\mu,q} \in [0,+\infty)$ such that $\textnormal{Var}_{\mu}(f) \leq 
    C_{\textnormal{PI}}^{\mu,q} \|\nabla f\|_{L_\mu^2(\mathcal{X};\Re^d, \langle\cdot\rangle^{-q})}^2$ for all $f \in H^1_{\mu}(\mathcal{X})$. 
    For, $q=0$, we simply say that $\mu$ satisfies the Poincar\'e inequality with constant $\CPI^\mu$.
\editend
If a measure $\mu$ is absolutely continuous with respect to the Lebesgue measure, we abuse the notation by adopting the same letter $\mu$ to denote 
its associated Lebesgue density. 

Given two spaces $\Y, \U$ we let 
their product space be $\Y \times \U$ and 
for any $\mu \in \PP(\Y \times \U)$, we define the $\Y$ marginal of $\mu$ as
$\mu_\Y(A) := \mu(A \times \U)$ for any Borel set $A\subset\Y$. We denote similarly the $\U$ marginal as $\mu_\U$. Moreover, we write $\mu( \cdot \mid y)$ to 
denote the regular conditional measures~\cite[Ch.~10]{bogachev2007measure} of $\mu$
conditioned on $y$, and similarly $\mu(\cdot \mid u)$ 
for the $u$ conditionals. Later in the paper, we use the notation $\phi[\mu]$ to denote  an operator $\phi$ that maps a measure $\mu$ to a certain target space.
\editstart
Finally, given numbers $a, b \in \Re$ we write $a \vee b$ to denote their maximum 
and $a \wedge b$ to denote the minimum.
\editend

\subsection{Outline of the article}
The rest of the article is organized as follows:
\Cref{Sec:Error analysis for conditional OT} contains 
our analysis towards quantifying the estimation 
error of conditional OT maps using finite training data 
and parameterizations, containing the full statement of our main 
theoretical results constituting contribution (i). This section 
contains a description of our proof techniques and postpones the 
detailed proofs to the Supplementary Materials (SM). 
\Cref{Sec:Error analysis for OTF} extends our 
analysis of conditional OT maps to the case of an 
idealized OTF algorithm that constitutes our contribution (ii). Once 
again this section is focused on the statement and explanation of 
our main theorems and postpones the detailed proofs to the SM. 
\Cref{Sec:Numerical experiments} contains our implementation 
details and practical aspects of the OTF algorithm, along with 
numerical results, while \Cref{Sec:Conclusions}
contains our concluding remarks.

\section{Error analysis for conditional OT}\label{Sec:Error analysis for conditional OT}

In this section, we present quantitative error bounds for approximating conditional OT maps using quadratic costs. While these results form the foundation of our error analysis of OTF, they are of great independent interest in the context of conditional OT, inverse problems, and simulation based inference. 
\Cref{sec:conditional-OT-review} reviews 
relevant preliminary results followed by problem statements and setup in \Cref{sec:conditionalOT:problemstatement}.
\Cref{sec:conditionalOT-theory-main-results} outlines our main results including the slow and 
fast rates with technical proofs deferred to \Cref{app:proofs:conditionalOT}. 
Finally, \Cref{sec:conditionalOT:applications} gives an example 
application.

\subsection{Review of conditional OT}
\label{sec:conditional-OT-review}

The material in this section summarizes the results of 
\cite{hosseini2025conditional, chemseddine2025conditional, baptista2024conditional}.
Henceforth,
we consider open sets
$\Y \subseteq \Re^m$, $\U \subseteq \Re^n$
and their tensor product
$\Y \times \U \subseteq \Re^d$ where $d = m +n$.
We further consider a target probability 
measure $\nu \in \PP( \Y \times \U)$
and a reference probability measure 
$\eta \in \PP(\Y \times \U)$  which is assumed to have 
the form $\eta = \nu_\Y \otimes \eta_\U$ where $\eta_\U \in \PP(\U)$ is 
arbitrary\footnote{We recall that the independence coupling structure is just to facilitate our exposition and that this assumption can be dropped in the derivation of our results.}. Then the goal of conditional transport is to find 
a map $T: \Y \times \U \to \U$ so that $T(y, \cdot) \# \eta_\U 
= \nu(\cdot \mid y)$ for all $y \in \Y$.

In general, one can construct the map $T: \Y \times \U \to \U$ by taking any 
family of maps $T_y: \U \to \U$, viewed as a family parameterized
by $y$, that satisfy 
$T_y \# \eta_\U = \nu( \cdot \mid y)$, and simply ``glue" the maps 
together to obtain the desired conditioning map $T$. However, such a construction 
is not practical in situations where the conditionals $\nu(\cdot \mid y)$ 
are unknown or we do not have access to empirical samples from  them. 
A solution to this problem can be obtained by working with the class of triangular transport maps.
\begin{definition}[Triangular transport maps]
\label{def: triangular map}
    A map $T_{\lowerrighttriangle}:\Y \times \U \to \Y \times \U$ is called a (block) triangular transport map if there exists a map $T: \Y \times \U \to \U$ such that $T_{\lowerrighttriangle}(y, u) = (y, T(y, u))$ for all $(y,u)\in \mathcal Y \times \mathcal U$
    \footnote{Throughout the paper we mostly work with the ``bottom components" $T$ of a triangular map and we will use the $T_{\lowerrighttriangle}$ notation to denote the fully triangular map obtained from  $T$ by mapping the $y$ component using the identity map.}.
\end{definition}
Triangular maps have a natural connection to conditioning as stated in 
the following lemma. Note that the lemma holds even when the $u$ and $y$ coordinates of $\eta$ are not independent, although we primarily 
consider that setting in the rest of the paper.
\begin{lemma}[{\cite[Thm.~2.4]{baptista2024conditional}}]
\label{lem: triangular-maps-can-condition}
    Let $T_{\lowerrighttriangle}$ be a triangular map of the form \Cref{def: triangular map}
    and suppose $\eta$ is any measure such that $\eta_\Y = \nu_\Y$. 
    If $T_{\lowerrighttriangle} \# \eta = \nu$ then  $T(y, \cdot) \# \eta_\U(\cdot\mid y) = \nu(\cdot \mid y)$ for $\nu_\Y$ a.e. $y \in \Y$.
\end{lemma}
Triangular maps are not unique  and various constructions are possible. Here, we consider their OT characterization  based on a quadratic conditional Monge problem  \cite{hosseini2025conditional, chemseddine2025conditional, baptista2025conditional}:
    \begin{equation}\label{prob:Conditional-Monge}
        \begin{aligned}
            &\inf_{T}  \int_{\Y \times \U} \frac{1}{2} \| u - T(y, u) \|_\U^2  \: \d\eta(y,u) \quad  \textnormal{subject to (s.t.)} \quad  T_{\lowerrighttriangle} \#\eta = \nu.
        \end{aligned}
    \end{equation}
 \editstart   
 Here we identify the minimizer $T$ of the problem above with the bottom component of its associated block-triangular map $T_{\lowerrighttriangle}$ of \Cref{def: triangular map}, so that the optimization problem~\eqref{prob:Conditional-Monge} is implicitly understood as ranging over the space of block-triangular transport maps.
 \editend
 \Cref{lem: triangular-maps-can-condition} together with the constraint 
imply $T(y, \cdot) \# \eta_\U = \nu(\cdot \mid y)$ for the solution of this problem as desired. 
The conditional Monge problem admits a Kantorovich relaxation:  
    \begin{equation}
    \label{prob: conditional kantorovich}
    \left\{
        \begin{aligned}
        &\inf_{\pi \in \Pi(\eta, \nu)} K(\pi), \quad K(\pi) \coloneqq 
        \int_{(\Y \times \U) \times (\Y \times \U)} c_{\chi}(z,v,y,u) \d \pi(z,v,y,u),\\
        & \textnormal{where} \: \: c_{\chi}(z,v,y, u) = \frac{1}{2}\| v- u\|_\U^2 + \chi(z, y), \:\:
        \chi(z,y) \coloneqq
        \begin{cases}
            0 \;\;\; y=z,\\
            +\infty \;\; \textnormal{otherwise}.
        \end{cases}
        \end{aligned}
    \right.
    \end{equation}
where $\Pi(\eta, \nu)$ denotes the space of couplings between $\eta, \nu$
\cite{villani2009optimal}. Mirroring the classic OT theory, the conditional Kantorovich problem
admits strong duality, which 
forms the foundation of our exposition.
Define the functionals
\begin{align}
\label{eq: marginal dual}
    &\cS(\phi \mid y) = \int \phi(y , v) \d \eta_\U(v) 
    + \int \phi^{*}(y, u) \d \nu( u \mid y),
    \\
\label{eq: dual}
   &\cS(\phi) = \int \mathcal{S}(\phi \mid y) \d\nu_{\Y}(y),
\end{align}
where 
$\phi^{*}(y,u) = \sup_{v}\{\langle v,u \rangle_\U - \phi(y,v)\}$ and $\textnormal{CVX}_\U$ denotes the set of functions $\phi : \Y \times \U \to \Re$ 
that are convex in the $\U$ coordinate.
Then the following strong duality holds:
\begin{proposition}[\cite{hosseini2025conditional} Prop. 3.6]\label{prop: duality}
Assume $C_{\eta,\nu} := \frac{1}{2} \left(\int \|v\|_\U^2\d\eta(y,v) + \int \|u\|_\U^2 \d\nu(y,u)\right)<\infty.$ 
Then 
        $\inf_{\pi\in\Pi(\eta,\nu)}K(\pi) = C_{\eta,\nu} - \inf_{\phi \in \textnormal{CVX}_\U} \mathcal{S}(\phi).$
\end{proposition}
\editstart
\begin{proof}[Proof sketch]
We briefly outline the proof sketch since it is a straightforward 
consequence of \cite[Prop.~3.6]{hosseini2025conditional}: 
First, expand $\frac{1}{2}\|v-u\|_\U^2 = \frac{1}{2}\|v\|_\U^2 
+ \frac{1}{2}\|u\|_\U^2 - \langle v,u\rangle_\U$ 
and plug into $K(\pi)$ and separate $C_{\eta,\nu}$ from 
the terms that depend on $\pi$. Since $\chi(z,y) = +\infty$ 
for $y \neq z$, the dual constraint reduces to $\phi(y,u) - \psi(y,v) \leq \frac{1}{2}\|v-u\|_\U^2$, so the optimal $\psi$ is given by $\psi(y,v) = \sup_u\{\phi(y,u) - \frac{1}{2}\|v-u\|_\U^2\} = \phi^*(y,v)$. 
Restricting without loss of generality to $\phi\in\mathrm{CVX}_{\U}$ and substituting into the dual objective of \cite[Prop.~3.6]{hosseini2025conditional} then yields the result.
\end{proof}
\editend
Finally, we turn our attention to the existence and uniqueness of conditional Brenier maps. We will assume the following throughout the remainder of the paper. 
\begin{assumption}
\label{assump: lebesgue measure + finite moment} 
It holds that:
\begin{enumerate}
    \item     $\eta_{\U}$ admits a Lebesgue density  and has  convex support.\label{assump: reference lebesgue measure}
    \item  For each $y \in \mathcal{Y},$ 
    the conditionals $\nu(\cdot \mid y)$ admit Lebesgue densities.\label{assump: target lebesgue measure}
    \item  $\eta,$ $\nu$ have finite second $\U$-moments, i.e., 
    $\int \|v\|_\U^2 \d\eta(y,v) < \infty, \; \int \|u\|_\U^2 \d\nu(y,u) < \infty.$    
\end{enumerate}
\end{assumption}

\begin{proposition}
[{\cite[Prop.~3.8]{hosseini2025conditional}, \cite[Thm. 2.3]{carlier2016vector}}]
\label{prop: solvability of monge + conditional brenier}
Suppose \Cref{assump: lebesgue measure + finite moment}
holds. Then \eqref{prob:Conditional-Monge}
has a unique solution
$ T^\dagger(y, u)= \nabla_u \phi^\dagger(y, u)$\footnote{Here $\nabla_u$ indicates the gradient with respect to the $u$-coordinate.}
for almost every $u \in \U$ and $y \in \Y$, and
the potential $\phi^\dagger$ is the unique solution (up to  constant shifts) to the dual 
problem:
\begin{equation}\label{eq:phi-dagger-def}
    \phi^\dagger = \argmin_{\phi \in \CVX_\U} \; \cS(\phi).
\end{equation}
\end{proposition}

\subsection{Problem setup and motivation}
\label{sec:conditionalOT:problemstatement}
\Cref{prop: solvability of monge + conditional brenier} 
provides a natural avenue for estimating the conditional map $T^\dagger$ by 
first computing the dual potential $\phi^\dagger$, following approaches similar to those developed for standard OT problems \cite{baptista2025conditional,peyre2019computational}.
Here we consider the setting where  $\phi^\dagger$ is estimated from a set of empirical 
samples  $(y_i,u_i)_{i=1}^N \sim \nu$. 
Naturally, obtaining a ``good" approximation to 
$\phi^\dagger$ allows us to compute an 
approximation of $T^\dagger$, which can be used in downstream conditioning 
tasks such as the filtering problem outlined in \Cref{sec:intro}.
We further assume that the reference $\eta_\U$ is known and can be 
simulated easily, enabling access to a set of reference 
samples $(y_i, v_i)_{i=1}^N \sim  \eta$. To this end, we define 
the empirical estimator 
\begin{equation}\label{eq:empirical-min}
  \wh{\phi} := \argmin_{\phi \in \F} \; \wh{\cS}(\phi), 
\qquad     \wh{\cS}(\phi) := \frac{1}{N} \sum_{i=1}^{N} \phi(y_{i}, v_{i}) + \phi^{*}(y_{i}, u_{i}),
\end{equation}
where $\F \subset \CVX_\U$ denotes a model class (e.g., polynomials or
neural nets) over which the empirical optimization problem is solved. 
We note that the assumption that elements of $\F$ are convex in the 
$\U$ coordinate can be restrictive in practical implementations
(although it can be done with input convex neural nets for example  \cite{pmlr-v70-amos17b}). For the sake of our analysis, however, we assume this convexity constraint holds.

Our goal in the rest of this section is to obtain quantitative 
upper bounds on  $\| \nabla_u \wh \phi - \nabla_u \phi^\dagger \|^2_{L^2_\eta}$. This measure of the error is natural since 
it directly controls the quality of 
the approximate conditional measures 
$\wh \nu(\cdot \mid y) := \nabla_u \wh \phi(y, \cdot) \# \eta_\U$. 
To see this, consider a 
divergence $D: \PP(\U) \times \PP(\U) \to \Re_{\ge 0}$. Following 
\cite{baptista2023approximation}, we say $D$ is $\eta_\U$-stable if 
\begin{equation}\label{eq: metric D stability}
    D( T_1 \# \eta_\U, T_2 \# \eta_\U) \le C \| T_1 - T_2 \|_{L^2_{\eta_\U}},
    \quad \forall T_1, T_2 \in L^2_{\eta_\U},
\end{equation}
where $C(\eta_\U, D) > 0$ is a constant; a list of some common  stable divergences appears in~\cref{sec:divergences}.
It is verified in \cite{baptista2023approximation} that Wasserstein 
distances as well as maximum mean discrepancies are stable
under very general conditions. Then, for any $\eta_\U$-stable 
divergence, a straightforward application of the Jensen's inequality gives 
\begin{equation}\label{eq: conditional metric}
    D(\widehat{\nu}\|\nu):=\int_\Y D( \wh \nu(\cdot \mid y), 
    \nu(\cdot \mid y) ) \d \nu_\Y(y)
    \le C \| \nabla_u \wh \phi - \nabla_u\phi^\dagger \|_{L^2_\eta}.
\end{equation}
Here, $D(\widehat{\nu}\|\nu)$ quantifies the \emph{mean conditional estimation error} of $\widehat{\nu}$ versus $\nu$. 

\subsection{Statement of main results 
and overview of proofs}
\label{sec:conditionalOT-theory-main-results}

We present error bounds with two convergence rates: one with a slower rate of order $O(N^{-1/2})$ and another that can be as fast as $O(N^{-1})$. The slow rate proof is simpler and has fewer assumptions while the fast rate 
is technical and has stringent assumptions. 
\editstart
Our analysis largely relies on the techniques developed in~\cite{chewi2024statistical} for the statistical estimation of standard 
OT maps following the seminal paper~\cite{hutter2019minimax}
and its generalization in~\cite{divol2025optimal}.
We borrow two main simplifying assumptions from ~\cite{chewi2024statistical}
that allow us to present simpler and more interpretable proofs but 
we expect that our main results can be extended beyond these assumptions:
\begin{enumerate}
    \item We mainly consider a hypothesis class $\F$ that does not scale with the ambient dimension and for which the ``curse of dimensionality" can be avoided (e.g. parametric classes and sufficiently smooth constrained classes). Throughout the analysis, we give several comments on how 
    our analysis can be extended beyond this assumption using more 
     refined chaining bounds adopted in~\cite{divol2025optimal}. 

     \item We assume that $\F$ is a convex set of candidate potentials $\phi$ 
     such that $\phi(y,\cdot)$ are themselves strongly convex for any 
     value of $y$. However, a more sophisticated analysis akin to ~\cite{divol2025optimal} can allow us to drop the 
     assumption that $\F$ is convex. We comment more on this in the following sections.
\end{enumerate}
\editend

\subsubsection{Slow rate}\label{sec: slow rate} 
We begin by outlining our technical assumptions on the class $\F$ 
\editstart 
which are shared by both our slow and fast rate results.
\editend
\begin{assumption}\label{assump: slow rate} 
 It holds that:
\begin{enumerate} 
    \item \label{Uniform boundedness} 
    \editstart
    There exist $q_\Y,q_\U>0$ such that $\sup_{\phi\in\F}\|\phi\|_{L^\infty(w)}\leq R<\infty$ for the  weight function $w(y,u)\coloneq \langle y\rangle^{-q_\Y} \langle u\rangle^{-q_\U}$
    and $R\geq 1$
    \footnote{Assume $R \ge 1$ is not crucial but 
    it simplifies our statements later. Same applies to $K\geq \alphamin/4$.}.

    \item \label{Lipshitzness in y for class} There exists $\widetilde{q}>0$ such that $\sup_{\phi\in\F}\|\nabla_{yu}\phi\|\leq L_\F\langle y\rangle^{\widetilde{q}}$ on $\Y\times\U$ and $\nu_\Y(\langle\cdot \rangle^{\widetilde{q}+1})<\infty$.

    \item \label{moment condition for unbounded chaining} The mixed moment conditions $\eta(w^{-2})<\infty$ and $\nu(w_*^{-2})<\infty$ 
    hold where we defined the conjugate weight function $w_*(y,u) := \langle y\rangle^{-(q_\Y + (1+\widetilde{q})q_\U)}\langle u\rangle^{-q_\U}$.
    \editend
    
    \item \label{regularity}For all functions  $\phi \in \mathcal{F}$, $\phi$ is lower-semicontinuous, and $\phi(y,\cdot)$ is $\alpha(y)$-strongly convex and $\beta(y)$-smooth for $\nu_\Y$ a.e. $y \in \Y$ \footnote{Recall that a function $\phi: \U \to \Re$
    is $\alpha$-strongly convex if $\phi( t u + (1 - t) u') < 
    t \phi(u) + (1 - t) \phi(u') - \frac{\alpha}{2} \| u - u' \|^2_\U$ for all $u, u' \in \U$
    and $\beta$-smooth if $\phi(u') \le \phi(u) + \nabla \phi(u)(u' - u) 
    + \frac{\beta}{2} \| u - u' \|_\U^2$.}. 
    Moreover, there exist constants  $0 < \alpha_{min}, \beta_{max} < + \infty$ such that   $\sup_{y\in \Y}\beta(y)\leq \betamax$ and 
    $\inf_{y\in\Y}\alpha(y)\geq \alphamin$ for all $\phi \in \F$.

    \item \label{Entropy condition} The complexity of $\mathcal{F}$ is controlled by the log-covering number \cite[Def.~5.1]{wainwright2019high} bound
    \[\log \mathcal{N}\left(\delta,\mathcal{F}, \|\cdot\|_{L^\infty{\color{black}(w)}}\right)\leq C_{\F}\delta^{-\gamma}\log(1+\delta^{-1}),\]
    for all sufficiently small $\delta > 0$ and with  fixed constants $\gamma\in[0,1)$ and $C_\F \ge 1$. 
    \item \label{anchor point condition} 
    \editstart
    There exists a $K\geq \alphamin/4$ such that $\sup_{\phi\in\F}\|\nabla_{u}\phi(\cdot,0)\|_{L^1_{\nu_\Y}}\leq K<\infty$.
    \editend
\end{enumerate}
\end{assumption}
We note that most of these assumptions are common in the literature, however, some of them can be relaxed to further generalize our results. 
Assumptions 
\editstart
1, 3, and 5
\editend
are standard in the empirical process theory literature \cite[Ch.~8]{vershynin2018high}, 
\editstart
where they are used to control the Dudley integrals arising from the chaining technique~\cite{talagrand1996majorizing,dudley1969speed,dudley2010universal}, a central tool for deriving rates in empirical process theory. More specifically, throughout the proofs presented, such integrals have the form
\begin{align}
&\frac{1}{\sqrt{N}}\int_{0}^{R} \sqrt{\log \mathcal{N}(\delta, \mathcal{F}, \|\cdot\|_{L^\infty(w)})} \d\delta\, \ \textnormal{(in slow rate)}\, ,\label{eq: Dudley 1 for slow rate}\\
&\frac{1}{\sqrt{N}}\int_{0}^{r} \sqrt{\log \mathcal{N}(\delta, \mathcal{F}, \|\cdot\|_{L_\eta^2})} \d\delta + \frac{1}{N}\int_{0}^{R} \log \mathcal{N}(\delta, \mathcal{F}, \|\cdot\|_{L^\infty(w)}) \d\delta\, \ \textnormal{(in fast rate)}\, ,\label{eq: Dudley 2 for fast rate}
\end{align}
with $\sup_{\mathcal{F}}\|\phi\|_{L^\infty(w)}\leq R$ and $\sup_{\mathcal{F}}\|\phi\|_{L_\eta^2}\leq r$.
\editend
Assumption 5 can be verified for many common approximation classes \cite{park2023towards,yang2020functionapproximationreinforcementlearning}, 
\editstart 
including parametric classes ($\gamma = 0$) and sufficiently $s$-smooth RKHS and Sobolev balls ($\gamma = d/s$). Inspecting the upper bound on the integrals just presented provided by condition 3,~\eqref{eq: Dudley 2 for fast rate} converges precisely for $\gamma\in [0,1)$ while~\eqref{eq: Dudley 1 for slow rate} (and just the first integral of~\eqref{eq: Dudley 2 for fast rate}) stays finite for $\gamma\in [0,2)$. To handle cases where $\gamma$ is too big to keep the Dudley integrals finite, the standard procedure is to adopt more refined bounds truncating their lower ends away from $0$ by a properly calibrated $\delta(\gamma,N, C_\F)>0$.
\editend

\editstart
Assumption 4 allows us to state our bounds in a convenient form, but it can be replaced with uniform integrability conditions on $\beta(y)$ and $\alpha(y)$. Moreover, the notions of smoothness and strong convexity could be generalized allowing polynomial (rather than constant) envelopes bounding the potentials' Hessians without altering the core of the provided proofs. Condition 2 ensures sufficient control over the variability of the hypothesis class in the conditioning variable $y$. Finally,
with an adaptation to the conditional case of~\cite[Prop.8, Lem.21]{divol2025optimal} detailed in~\Cref{sec: verifying anchor point cond wolog}, condition 6 (which we state to simplify our proofs) can be dropped when assuming that $\eta_\U$ and $\nu$ are sub-exponential (a condition that automatically holds by~\Cref{assump: fast rate} for
example). 
We also highlight that, for uniformly bounded potentials defined on compact domains (i.e., $q_\Y = q_\U = 0$), conditions 2,3, and 6 are not needed and the final rate does not inherit any dependence on $L_\F,K$, or moments, see~\Cref{proof: thm: slow rate}.
\editend

With the above assumptions in place, we can state our first main theorem, which establishes a ``slow" 
convergence rate in terms of the sample size $N$. We refer to this as a ``slow" rate
since it states that the $L^2_\eta$-error between $\nabla_u \wh{\phi}$ and $\nabla_u \phi^\dagger$
is $\mathcal{O}(N^{-1/4})$, which is slower than the expected Monte-Carlo rate $\mathcal{O}(N^{-1/2})$.

\begin{theorem}
\label{thm: slow rate}
    Under~\cref{assump: slow rate}, the empirical estimator $\wh{\phi}$ satisfies
    \begin{equation}\label{disp-slow-rate}
    \mathbb{E}^{\textnormal{train}} \|\nabla_u \wh{\phi} - \nabla_u \phi^\dagger\|_{L_\eta^2}^2 \lesssim 
    \editstart
    \frac{\betamax}{\alphamin}\left[ 
    \inf_{\phi \in \mathcal{F}} \|\nabla_{u}\phi - \nabla_{u}\phi^{\dagger}\|_{L_{\eta}^2}^2 + 
    \sqrt{ \frac{\kappa \widetilde{K}^2 R^2 C_{\F}}{N}} \right],
    \editend
    \end{equation}
    where  
    \editstart
    $\kappa = \eta(w^{-2})+\nu(w_*^{-2})$, $\widetilde{K} = \left(K+ \left(1+ \nu_\Y(\langle \cdot\rangle^{\widetilde{q}+1})\right)L_\F\right)^{q_\U}$,
    \editend
    and the hidden constant is  independent of the other problem parameters. The expectation on the left hand side is taken with respect 
    to the empirical data.
\end{theorem}

The proof of this theorem relies on the following stability result: for any $\phi \in \mathcal{F}$, 
\begin{align}\label{eq: stab res}
 \frac{1}{2\betamax}\|\nabla_u\phi^{\dagger}-\nabla_u\phi\|_{L_{\eta}^2}^2\leq \cS(\phi) - \cS(\phi^\dagger) \leq
      \frac{1}{2\alphamin}\|\nabla_{u} \phi^{\dagger} - \nabla_{u} \phi\|_{L_{\eta}^2}^2,
\end{align}
which is stated and proved in~\cref{cor: map stability}.
We then proceed by obtaining a bound for 
$\mathbb{E} \cS (\widehat{\phi}) - \cS(\phi^\dagger)$, which decomposes into a 
bias term, concerning the approximation error $\inf_{\phi \in \F} 
\| \nabla_u \phi - \nabla_u \phi^\dagger \|_{L^2_\eta}$ which also appears in \eqref{disp-slow-rate},
as well as a variance term,  concerning the empirical process $\sup_{\phi \in \F} | \wh{\cS}(\phi) - \cS(\phi)|$.
The variance term is bounded using classical chaining results from empirical process 
theory. 
\editstart
In doing so, as briefly mentioned before, Item 5 of~\Cref{assump: slow rate} ensures that the obtained rate is $\gamma$ independent and therefore any dimension dependence carried through it is suppressed. The same argument goes through even in the case where $\gamma\in [1,2),$ at the expense of carrying an additional constant proportional to $1/(2-\gamma)$ in front of the rate (see the calculation of~\Cref{lem: prop const slow rate}). The case $\gamma>2$ can be handled by truncating the Dudley integral of~\eqref{eq: Dudley 1 for slow rate} yielding an even slower rate of $(\frac{C_\F}{N})^{1/\gamma}$ which suffers the curse of dimensionality for $\gamma$ scaling with respect to the ambient dimension (e.g. $\gamma = d/2$ for smooth classes). Furthermore, the constant $C_\F$ might depend on $d$ (or even on $N$ as discussed in the later sections) and its dependence enters the rate multiplicatively. 
For further discussion on the dimension dependencies appearing in our rates see ~\Cref{tab:cond-ot-rates} and the peripheral 
discussions.
Finally, we note that the proof strategy of~\Cref{thm: slow rate}  closely follows the approach of \cite[Thm. 3.7]{chewi2024statistical} 
\editend
for standard OT maps, with the key distinction that in our setting convex conjugacy applies only to the $u$-coordinate. The complete details are given in \Cref{proof: thm: slow rate}.

\begin{remark}\label{rem: comment on assumption 1}
    We highlight that the assumptions of \Cref{thm: slow rate} primarily concern the 
    approximation class $\F$ and do not explicitly involve the 
    \editstart
    potential $\phi^\dagger$ or the
    reference and target measures 
    $\eta, \nu$ besides mild moment conditions.
    \editend
    This flexibility enables the application of \Cref{thm: slow rate} 
    to many practical problems, including our filtering problems  in \Cref{sec:example-filtering}.
    \editstart
    Moreover, we highlight that 
    the strong convexity condition ~\Cref{assump: slow rate}(4) 
    is not needed for bounding the variance term in \Cref{thm: slow rate} and is only needed for bounding the objective bias 
    $\inf_{\phi\in\F}\cS(\phi) - \cS(\phi^\dagger)$ with the 
    map bias $\frac{1}{\alphamin} 
    \inf_{\phi \in \mathcal{F}} \|\nabla_{u}\phi - \nabla_{u}\phi^{\dagger}\|_{L_{\eta}^2}^2$. 
    \editend

\end{remark}

\subsubsection{Fast rate}\label{sec: fast rate} 
In this section, we  improve the statistical estimation rate in \eqref{disp-slow-rate}
from $\mathcal{O}(N^{-1/2})$ to $\mathcal{O}(N^{- \frac{2}{2 + \gamma}})$ by extending the proof technique of \cite[Thm. 3.15]{chewi2024statistical} to the conditional setting. 
To do so, we impose more stringent assumptions, in particular requiring 
$\eta$ and $\nu$ to satisfy certain Poincar{\'e} inequalities and $\F$ to satisfy stronger regularity conditions. 

\editstart
\begin{assumption}
\label{assump: fast rate}
It holds that:
\begin{enumerate}

    \item 
    Condition 2 in~\Cref{assump: slow rate} is extended to $\phi^\dag$, i.e., $   \|\nabla_{y u} \phi^\dag\| \leq L_\F\langle y\rangle^{\widetilde{q}}$ everywhere in $\Y\times\U$.

   \item 
   $\eta_\U$, 
   and $\nu(\cdot \mid y)$, for $\nu_\Y$ a.e. $y\in \Y$, satisfy the Poincar\'e inequality with constants $\CPI^{\eta_\U}$, 
   and $\sup_{y\in\Y}C_{\textnormal{PI}}^{\nu(\cdot \mid y)}\leq C_{\textnormal{PI}}^{\nu(\cdot \mid \Y)}$, respectively. 
   Moreover, the marginal $\nu_\Y$ satisfies the $\langle\cdot\rangle^{-\widetilde{q}}$-strengthened Poincaré inequality with constant $\CPI^{\nu_\Y,\widetilde{q}}$.

    \item \label{convexity of class}  $\mathcal{F}$ is convex, that is, if $\phi,\phi'\in\mathcal{F}$, then $\lambda\phi+(1-\lambda)\phi'\in\mathcal{F}$ for any $\lambda\in[0,1]$.
   
\end{enumerate}
\end{assumption}
\editend
We view conditions 1 and 2 as additional regularity conditions 
on the reference and target measures $\eta, \nu$ as well as $\phi^\dag$. 
\editstart
Condition 2 automatically strengthens condition 3 in~\Cref{assump: slow rate} since any distribution satisfying the Poincaré inequality is subexponential \cite{bobkov1997poincare} and therefore has finite moments of any order. In contrast to the unconditional setting, which requires the Poincar\'e inequality only for the source and target measures, condition 2 requires the Poincar\'e inequality on the source $\eta_\mathcal{U}$ and the target slices $\nu(\cdot \mid y)$ together with strengthened Poincaré inequality on the marginal $\nu_\Y$.  We remark that analogously to the unconditional setting, in the case where $\phi^{\dagger}(y, \cdot)$ is $\betamax$-smooth for $\nu_\Y$ a.e. $y\in \Y$, the Poincar\'e inequality condition on $\eta_\U$ directly implies the one on the target conditionals $\nu(\cdot\mid y)$ via $C_{\textnormal{PI}}^{\nu(\cdot\mid y)}\leq \betamax^2 C_{\textnormal{PI}}^{\eta_\U}$, see \Cref{lem: PI for conditionals}. 
Condition 3 simplifies the proofs and can be relaxed; see our 
comments below after \Cref{thm: fast rate}.
\editend
We can now state our next main result:

\begin{theorem}
\label{thm: fast rate}
\editstart
Suppose \Cref{assump: slow rate,assump: fast rate} hold. Then 
\begin{align*}
\mathbb{E}^{\textnormal{train}}\|\nabla_u \wh\phi - \nabla_u \phi^\dagger\|_{L_{\eta}^{2}}^2
& \lesssim  \frac{\betamax}{\alphamin} \inf_{\phi \in \F} \| \nabla_u \phi - \nabla_u \phi^\dagger \|_{L^2_\eta}^2+ C_{\textnormal{est.}}\left(\left(\frac{C_\F\log(N)}{N}\right)^{\frac{2}{2+\gamma}} \vee \frac{C_\F}{N}\right)\, ,
\end{align*}
\editend
where $C_\textnormal{est.} = C_\textnormal{est.}(C_{\textnormal{PI}}^{\eta_\U},
\editstart
C_{\textnormal{PI}}^{\nu_\Y,\widetilde{q}}
\editend
, C_{\textnormal{PI}}^{\nu(\cdot \mid \Y)},\alphamin, \betamax,L_\F, \editstart
R, \widetilde{K}
\editend
,\gamma) >0\,$ 
depends on the 
problem parameters \footnote{See \Cref{rem:C_sigma_claculation}  for details regarding 
the form of this constant and how it can be characterized.}\editstart,
$\widetilde{K} = \left(K+ \left(1+\nu_\Y(\langle \cdot\rangle^{\widetilde{q}+1})\right)L_\F\right)^{q_\U}$.
\editend
\end{theorem}

\editstart
See \Cref{proof: thm: fast rate} for the detailed proof of this theorem.
We note that, in the underparameterized regime $C_\F\ll N$, the variance term of the bound is simply dominated by $\left(\frac{C_\F\log(N)}{N}\right)^{\frac{2}{2+\gamma}}$ for the considered setting $\gamma\in[0,1)$.
Similar to the unconditional analyses \cite{chewi2024statistical, divol2025optimal},
\editend
the proof of the theorem uses localization techniques from the empirical process theory, in particular the one-shot method attributed to van de Geer \cite{van1987new,van2002m}. Using classical arguments, we obtain both expectation and 
high probability bounds on the uniform excess risk's empirical gap $\sup_{\phi \in \mathcal{F}_{\epsilon}} |(\cS (\phi) - {\cS}(\phi^{\dagger}))-(\wh\cS (\phi) - \wh{\cS}(\phi^{\dagger}))|$,
where $\F_\epsilon$ is a localized subset of $\F$ centered at the 
true solution $\phi^\dagger$. 
\editstart
The main theoretical departure from the unconditional case is encapsulated by ~\Cref{thm: PI extension consequence for COT,thm: centered semidual empirical estimation rate}, which provide the necessary conditions to apply the ``generic chaining"~\cite{talagrand1996majorizing} estimates at the core of the localization argument.
Doing so allows us to control
the above empirical
error by a term of order $\psi(\epsilon)N^{-1/2}$ together with another
term of order $N^{-1}$ arising from \eqref{eq: Dudley 2 for fast rate}.
\editend
Here, $\psi(\epsilon)$ is an increasing function of $\epsilon$ and
the fast rate can be obtained 
by carefully controlling $\psi(\epsilon)$ to obtain the fastest possible rate, 
which happens to be $\left(\frac{\log (N)}{N}\right)^{\frac{2}{2+\gamma}}$. 
\editstart
For $\gamma\in [0,1)$, we attain rates, modulo logarithmic terms, in the regime $[N^{-1}, N^{-2/3})$ that do not suffer the curse of dimensionality in the exponent of $N$. However, any potential dimension dependence in $C_\F$ and the 
Poincar\'e constants $\sqrt{C_{\textnormal{PI}}^{\eta_\U}},\sqrt{C_{\textnormal{PI}}^{\nu_\Y,\widetilde{q}}}, \sqrt{C_{\textnormal{PI}}^{\nu(\cdot \mid \Y)}}$ (absorbed in $C_\textnormal{est.}$) enters multiplicatively and scales at most linearly with respect to these constants; see \Cref{rem:C_sigma_claculation}.

Moreover, in the case of large parametric classes ($\gamma = 0$), such as neural networks and polynomial classes like the one considered in~\Cref{sec:conditionalOT:applications}, $C_\F$ generally carries also an $N$-dependence in order for the class $\F$ to achieve a certain approximation error to balance with the variance term of the derived rate and ultimately yield global rates. In this case, under additional assumptions beyond the simplified setting we consider (akin to ~\cite[condition C2 and Prop.3]{divol2025optimal}) $C_\F$ can be sharpened to be $C_\F^{1-2/d}$ resulting in the derivation of optimal global rates; see~\Cref{rem: where to adapt proof,rem:C_sigma_claculation}.

In the more general case of $\gamma\geq1$, the Dudley integrals of $\eqref{eq: Dudley 2 for fast rate}$ have to be handled via truncation  and the above localization argument  has to be adapted accordingly. This is indeed the method used in~\cite[Prop. A.2]{divol2025optimal} in the standard OT case. We underline that the results at the core of the conditional extension which we present when $\gamma\in[0,1)$ are disjoint from the chaining arguments and therefore can be 
easily adapted to the more technical chaining calculations of~\cite{divol2025optimal}. For this reason, we expect that the conditional rates can also be attained for $\gamma\geq1$ by incorporating the calculations of~\cite{divol2025optimal}; see~\Cref{rem: where to adapt proof,rem: conditional extension core}. Then we expect the rate to be $(\frac{C_\F}{N})^{2/(2+\gamma)} \vee (\frac{C_\F}{N})^{1/\gamma}$ which, in the underparameterized regime $C_\F\ll N$ for $\gamma>2$, matches the 
slow rate; see also ~\Cref{tab:cond-ot-rates}. 
\editend

Furthermore, we highlight that the convexity requirement on $\F$ in~\Cref{assump: fast rate}(3) is a convenient simplification. However, as done in~\cite{divol2025optimal}, this condition can be completely dropped with a few modifications to the proof by restricting the localized subset $\F_\epsilon$. 
The strong convexity of the candidate potentials in~\Cref{assump: slow rate}(4) can be relaxed too. Indeed, the successful control of the excess risk fluctuations over $\F_\epsilon$ can be achieved without strong convexity when $\eta$ and $\nu$ are compactly supported with the same limitation on the bias representation pointed out in~\Cref{rem: comment on assumption 1}. For measures supported on unbounded domains, only the strong convexity of a subset of $\F$ that approximates $\phi^\dag$ is needed.

The remaining conditions 
in 
\editstart
\Cref{assump: fast rate}
\editend
mainly appear through technical arguments in the proof. In 
particular, Poincar\'e inequalities are used widely to derive bounds on $\phi - \phi^\dagger$ from bounds on $\nabla_u \phi - \nabla_u \phi^\dagger$.
\editstart
Understanding whether the Poincar\'e conditions are necessary 
 is  an open question
in both the conditional and standard OT settings.
\editend

\subsection{An example application}\label{sec:conditionalOT:applications}
We now present an example application of our theorems
concerning the transport of 
a log-concave reference $\eta_\U$ to a target with log-concave conditionals, using maps parameterized in the Legendre polynomial basis.

Let $\Y = [-1, 1]^m$ and $\U = [-1, 1]^n$ so that
$\Y \times \U = [-1, 1]^d$ along with the measures  $\eta, \nu \in \PP(\Y\times \U)$ 
so that  $\nu = \exp( - \W( y, u) ) \d y \d u$, and 
$\eta = \nu_\Y \otimes \eta_\U$ with $\eta_\U = \exp( - \V(u)) \d u$ 
    with the  potentials $\V, \W$ satisfying:
\begin{enumerate}[label=(\roman*)]
    \item $\| \V \|_{C^2(\U)} < + \infty,$ and 
    $\| \W \|_{C^2(\Y \times \U)} < + \infty$.
    \item  $ \vartheta_1  I \preceq \nabla^2 \V( u ) \preceq \vartheta_2 I$ for all $u\in\U$ with constants $\vartheta_1,
    \vartheta_2>0$. 
    \item  $\theta_1 I\preceq \nabla^2 \W( y, u) \preceq \theta_2 I $.
    \footnote{Here $\nabla^2\W = \begin{bmatrix}
\nabla_y^2\W & \nabla_{yu}^2\W \\
\nabla_{uy}^2\W & \nabla_u^2\W
\end{bmatrix}$. $\nabla_u^2$ denotes the Hessian in the $u$
variable for a fixed $y$. We define similarly $\nabla_y^2$. Finally, $\nabla^2_{uy}=(\nabla^2_{yu})^\top$ corresponds to the mixed derivatives.}
    for all $(y,u) \in \Y\times \U$ with constants $\theta_1,\theta_2> 0$.
\end{enumerate}
Now choose
$\alphamin \le \sqrt{\frac{\vartheta_1}{\theta_2}}$ 
and $\betamax \ge \sqrt{\frac{\vartheta_2}{\theta_1}}$ and consider 
the approximation class 
\begin{equation*}
\begin{aligned}
    \F = \F(M)  = 
    \Bigg\{ & \phi: \Y \times \U \to \Re \;  \Bigg| \; 
    \phi(y, u) = \sum_{ 1 \le |\pmb \alpha|_\infty \le M } c_{\pmb\alpha} 
    p_{\pmb \alpha} (y, u),  \\
    & \quad | \phi(y, u) | \le R \quad  \text{and} \quad  \alphamin I \preceq \nabla_u^2 \phi(y, u) \preceq \betamax I, \quad \forall (y,u) \in \Y \times \U 
    \Bigg\},
\end{aligned}
\end{equation*}
where $\pmb \alpha = (\alpha_1, \dots, \alpha_d)$ is a multi-index 
of non-negative integers with $| \pmb \alpha |_\infty := \max_j |\alpha_j|$
and $p_{\pmb\alpha}$
are the $d$-dimensional Legendre polynomials of degree 
$|\pmb\alpha|$, normalized in $L^2(\Y \times \U)$; see \cite[Ch.~1]{adcock2022sparse}. Note that 
condition $| \pmb \alpha | \ge 1$ indicates that we are discarding the zero order 
Legendre polynomial, i.e., constant shifts. 

First, we check that \Cref{thm: slow rate} is applicable in this example. 
Observe that 
\editstart
the uniformly bounded class $\F$ 
defined on a compact domain
readily satisfies \Cref{assump: slow rate}(1--4,6) with $q_\Y = q_\U = \widetilde{q} = 0$; in fact, in this case \Cref{assump: slow rate}(2,3,6) are not really needed towards applying~\Cref{thm: slow rate}.
Since $\F$ is finite dimensional and we assumed that $\phi \in \F$ are uniformly bounded, by \cite[5.6]{wainwright2019high} we also obtain \Cref{assump: slow rate}(5) with $\gamma = 0$ and $C_\F \propto (M +1)^d$. Hence \Cref{thm: slow rate} is 
applicable. 
\editend

Let us now control the approximation bias 
$\inf_{\phi\in\F} \| \nabla_u \phi - \nabla_u \phi^\dagger \|_{L^2_\eta}$. To do so 
we need to show certain regularity properties of $\phi^\dagger$. Since Brenier potentials 
are unique up to constant shifts, we simply assume $\int_{\Y \times \U} \phi^\dagger(y, u) dy du = 0$.
Further applying \Cref{thm: caffarelli contraction thrm plus convexity} for the transport 
from $\eta_\U$ to each conditional $\nu(\cdot \mid y)$ (thanks to (ii) and (iii) and the observation that  any principal sub-matrix of a positive semi-definite matrix is positive semi-definite) 
yields
\begin{equation}\label{phi-dagger-upper-lower-bounds}
    \sqrt{\frac{\vartheta_1}{\theta_2}}  I
    \preceq \nabla^2_u \phi^\dagger(y, \cdot) \preceq
    \sqrt{\frac{\vartheta_2}{\theta_1}} I, \qquad \forall y \in \Y.
\end{equation}
Using the Poincar\'e inequality we infer that $ \sup_y \| \phi^\dagger(y, \cdot) \|_{H^2(\U)} < + \infty$. 
Observe that this is not yet enough for us to obtain a rate for the 
approximation error, since $\phi^\dagger$ may have jump discontinuities in $y$. In this regard, we recall~\cite[Cor.~1.2]{gonzalez2024linearization} where the regularity of the mapping $y\mapsto\nabla_{u}\phi^{\dagger}(y,u)$ was investigated through the linearized Monge-Amp\'ere equation. By the application of this result on a compact domain, condition (i) is sufficient\footnote{Indeed, it was shown that if $\mathscr{V}\in C^{1,\gamma_\mathscr{V}}$ (denoting 
the class of $\gamma_\mathscr{V}$-Holder differentiable functions), $\mathscr{W}(y,\cdot)\in C
^{1,\gamma_\mathscr{W}}$, for $\gamma_\mathscr{V}<\gamma_\mathscr{W}$, and the mapping $y\mapsto \mathscr{W}(y,\cdot)$ is $C^1$ for $\nu_\Y$-a.e. $y\in\Y$, it follows that $\nabla_{u}\phi^\dagger(\cdot,u)\in C^1$ and consequently $\nabla_{yu}\phi^{\dagger}(\cdot,\cdot)\in C(\Y\times\U)$. In this way,~\cite{gonzalez2024linearization} improves the regularity result of~\cite{hosseini2025conditional} which yielded at most $1/2$-H\"older regularity in $y$ on compact domains under a weaker continuity assumption on the $\nu$ conditionals.} to ensure
\Cref{assump: fast rate}(1) for the mixed derivatives of $\phi^{\dagger}$ and the  Lipschitz condition
\begin{align}\label{eq: Lipshitzness in y for compact domain}
\| \nabla_u \phi^\dagger(y, \cdot) - \nabla_u \phi^\dagger(y', \cdot) \|_{L^2_{\eta_\U}}\le L_\Y \| y - y' \|_\Y\, .
\end{align}
for some $L_\Y>0$. In other words, $\phi^\dagger$ belongs to the mixed regularity space $C^{1}( \Y ; H^1(\U))$,
i.e., Lipschitz functions taking values in $H^1(\U)$.

Since the projection of $\phi^\dagger$ onto Legendre polynomials of degree $M$ converges in $C^2$ thanks to Caffarelli's regularity theorem \cite{caffarelli1992regularity} (see also \cref{prop:BT-regularity}), 
we infer that for sufficiently large $M$, it is guaranteed that the projection is strongly convex and smooth, and thus belongs to $\mathcal{F}$. Hence, the standard approximation rates for 
Legendre polynomials can be applied in that regime. 
To this end, combining the approximation rates of Legendre polynomials 
for $H^1$-functions \cite{canuto1982approximation} as well as Lipschitz continuous 
functions \cite{devore1993constructive}, yields the bound,  
\begin{equation}\label{eq:legendre-F-rate}
    \inf_{\phi \in \F} \| \nabla_u \phi - \nabla_u \phi^\dagger \|_{L^2_\eta} 
    \le C M^{-1} \| \phi^\dagger \|_{C^{1}(\Y; H^1(\U))},
\end{equation}
provided $M$ is large enough.
Let us now verify the conditions of \Cref{thm: fast rate}. 
Thanks to conditions (ii) and (iii)  we readily have the Poincar\'e inequalities of~\Cref{assump: fast rate}(2) after a direct application of the Brascamp-Lieb inequality; see \cref{thm: BL inequality} and also~\cite{brascamp1976extensions}. Since $\nu_\Y=\int_\U \nu(\cdot, \d u)$, by a corollary of the Brascamp-Lieb inequality, \eqref{eq: BL hessian bound}, we get
\begin{align*}
\W_\Y\coloneqq -\log(\nu_\Y)\, ,\qquad \frac{\int_\U \left(\nabla_y^2\W-\nabla_{yu}^2\W(\nabla_u^2\W)^{-1}\nabla_{uy}^2\W\right)(y,u) \d \nu(y,u)}{\int_{\U} \d \nu(y,u)}\preceq \nabla^2 \W_\Y(y)\, .
\end{align*}
By the Schur decomposition of $\nabla^2\W$~\cite[App. 5.5]{boyd2004convex},
this then amounts to the log-strong concavity of $\nu_\Y$ and the Poincar\'e condition of~\Cref{assump: fast rate}
\editstart
(2) (with $\widetilde q = 0$)
for $\nu_\Y$. Condition (3) is 
\editend
satisfied automatically 
for our polynomial model class. 
\editstart
 Thus, applying~\Cref{thm: fast rate}, we have that, ignoring logarithmic factors and $M$-independent constants,
\begin{align}\label{eq: global rate polynomials}
\mathbb{E}^{\textnormal{train}}\|\nabla_u \wh\phi - \nabla_u \phi^\dagger\|_{L_{\eta}^{2}}\lesssim M^{-1} + C_\F N^{-\frac{1}{2}}\approx M^{-1}+M^{d}N^{-\frac{1}{2}}\, .
\end{align}
Optimizing over $M$ as a function of $N$ yields the global rate $\mathbb{E}^{\textnormal{train}}\|\nabla_u \wh\phi - \nabla_u \phi^\dagger\|_{L_{\eta}^{2}}\lesssim N^{-\frac{1}{2(d+1)}}$. This rate is suboptimal and can be sharpened to $\mathbb{E}^{\textnormal{train}}\|\nabla_u \wh\phi - \nabla_u \phi^\dagger\|_{L_{\eta}^{2}}\lesssim N^{-\frac{1}{2(d-1)}}$ under additional assumptions, by tightening $C_\mathcal{F}$ to $C_\mathcal{F}^{1-\frac{2}{d}}$ in~\eqref{eq: global rate polynomials}, as explained in~\Cref{rem: where to adapt proof,rem:C_sigma_claculation}. Furthermore, when $\phi^\dag\in C^s(\mathcal{Y}; H^s(\mathcal{U}))$, the curse of dimensionality in the above rates can be mitigated: optimizing over $M$ under this additional regularity yields the improved global rate $\mathbb{E}^{\textnormal{train}}\|\nabla_u \wh\phi - \nabla_u \phi^\dagger\|_{L_{\eta}^{2}}\lesssim N^{-\frac{s}{2(d+s-2)}}$. 
Thus if $s$ is large the rate approaches $N^{-1/2}$.
Verifying this additional 
regularity can be challenging for the OT maps, however, 
it can be verified for Knothe-Rosenblatt rearrangements 
under further assumptions; see \cite[Rem.~2.19]{santambrogio2015optimal}.
\editend

\begin{remark}\label{rem: missing condition to prove}
\editstart
 The Poincar\'e condition of~\Cref{assump: fast rate}(2) with 
 a dimension-independent constant can be restrictive in practice, 
 as it effectively limits us to log-concave measures. However,  allowing the Poincar\'e constant to depend on the dimension
 enables access to broad classes of measures. Indeed, this class includes measures supported on bounded and connected Lipschitz domains~\cite[Sec. 5.8.1]{evans1997partial} and further cases on unbounded domains; see for example~\cite[Sec. 6.7.5]{maz2013sobolev}. For these measures, the conclusion of~\Cref{thm: fast rate} still holds with a proportionality constant scaling with dimension which may be 
 tolerable. 
\editend
\end{remark}

\editstart
\begin{remark}\label{main-rem: SIR-dim-dependence}
    It is enlightening to consider how the curse of dimensionality 
    manifests for OT maps compared to more classic methods such as 
    importance sampling, which underlies the SIR particle filter. 
    Here the posterior $\nu(\cdot \mid y)$ is approximated 
    with an empirical measure $\widehat \nu(\cdot \mid y) 
    = \sum_{i=1}^N \upsilon_i \delta_{v_i} $ where the weights 
    $\upsilon_i$ are estimated using a likelihood-ratio. In the 
    limit $N \to \infty$, due to law of large numbers, $\widehat \nu(\cdot \mid y)$ converges to $\nu(\cdot \mid y)$ and it is 
    possible to obtain an asymptotic convergence rate ~\cite[Thm. 9.1.8]{cappe2009inference}. Notably, assuming $\nu(\cdot \mid y)$ 
    has an independent product structure, one can show \cite{rebeschini2015can}:  
    \begin{equation}\label{eq:SIR-curse-display}
    \sup_{\psi \in \Psi}\,
    \sqrt{\,\mathbb{E}\!\left[\,\bigl|\nu(\cdot|y)(\psi) - \widehat{\nu}(\cdot|y)(\psi)\bigr|^2\,\right]}\approx O(\frac{C^d}{\sqrt{N}}),
    \end{equation}
where $\Psi$ is the set of bounded test functions $\psi: \U \to \Re$.
Hence, the number of particles needs to scale exponentially with 
$d$ to obtain a small error.
Our OT bounds deviate from this case in two important 
directions: 
\begin{enumerate}
    \item The curse of dimensionality in importance sampling is 
    due to the estimation of the weights $\upsilon_i$ which becomes exponentially more difficult as the state and observation dimension increase. The curse of dimensionality that the OT approach exhibits is due to the trade-off between the bias and variance terms in the 
    error bounds that are effectively tied to the regularity of 
    the ground truth maps and the approximation classes, e.g., there 
    is no curse if the maps are infinitely smooth. 
    
    \item Additionally, our OT bounds are stronger than \eqref{eq:SIR-curse-display} as they can be used to bound various stable 
    divergences between $\nu(\cdot \mid y)$ and $\widehat \nu(\cdot \mid y)$ such as the Wasserstein distance. In this case, \eqref{eq:SIR-curse-display} can be viewed as a bound for a weaker norm as the 
    supremum appears outside of the expectation compared to, 
    for example, Wasserstein-1 distances.
    
\end{enumerate}
\end{remark}
\editend

\section{Error analysis for OTF}\label{Sec:Error analysis for OTF}
In this section, we extend the error analysis of \Cref{Sec:Error analysis for conditional OT} to the filtering problem outlined in \Cref{sec:intro}, focusing on the setting where the maps 
$\wh T_t$ in \cref{intro:empirical-transport-bayesian-update} are empirical conditional Brenier maps. Our goal is to control the error between the predicted distribution $\widehat\pi^t$ and the true posterior $\pi^t$. 
We outline our setup in \Cref{sec:OTF-setup} followed by a summary 
of our main results in \Cref{sec:OTF-theory-main-results}. 
An example application is given in \Cref{sec:example-filtering} 
followed by a discussion of idealized algorithms in \Cref{subsec: OTF particle algorithms}.

\subsection{Problem setup}\label{sec:OTF-setup}
We continue with the notation and terminology introduced in \Cref{sec:intro},
where the sequence of posteriors $\{ \pi^t \}_{t \ge 0}$ follows a recursive update defined by $\pi^t = (\mathcal{B}_{Y_t} \circ \mathcal{A}) [\pi^{t-1}]$,
with $\mathcal{A}$ representing the state/dynamic update and 
$\mathcal{B}_{Y_t}$ the Bayesian update for the data $Y_t$. 
In this light, we introduce the notation
\begin{equation}\label{eq:filter-operator}
\T_{y}:= \mathcal{B}_{y} \circ \mathcal{A}, \qquad 
\T_{Y_{t,\tau}} = \T_{Y_t} \circ \dots \circ \T_{Y_{\tau +1}}, \qquad t > \tau,
\end{equation}
with the convention that $\T_{Y_{t,t}}$ is the identity map. 
With this notation, we have $\pi^t = \T_{Y_t} [\pi^{t-1}]$, and, by extension, $\pi^t = \T_{Y_{t, \tau}} [\pi^\tau]$. We emphasize that the maps $\T_{Y_t}$  and $\T_{Y_{t, \tau}}$
are defined in terms of the observed process $\mathscr{Y}_t$ and are hence random objects. 

A numerical filtering algorithm can then be viewed as a method to approximate $\T_y$ with the maps $\widehat\T_y$. Similarly, $\widehat\T_{Y_{t, \tau}}$ defines the approximate sequence of posteriors $\widehat\pi^t = \widehat\T_{Y_t}[ \widehat\pi^{t-1}]$. To analyze the error of such approximations, it is necessary for the true sequence of posteriors to satisfy stability properties 
that prevent errors from accumulating over time. To this end, 
we introduce the following {\it filter stability} assumption.
In what follows, $D: \mathbb{P}(\U) \times \mathbb{P}(\U) \to \Re_\ge 0$ 
denotes a statistical divergence in the parlance of \cref{eq: metric D stability}.

\begin{assumption}[Filter stability]
    \label{def: uniformly geometrically stable filter}
   The true filter $\T_{Y_{t, \tau}}$ is uniformly geometrically stable with respect to the 
    divergence $D$ in the sense that there exist constants $\lambda \in (0,1)$ and $\Cfs>0$ such that, for any pair of measures $\pi_1$, $\pi_2$ and $t>\tau\geq 0$, it holds that     \begin{equation}\label{eq:filter-stability} 
    D(\T_{Y_{t,\tau}}[\pi_1], \T_{Y_{t,\tau}}[\pi_2]) \leq \Cfs
     (1-\lambda)^{t-\tau+1}D(\pi_1,\pi_2), \quad \text{uniformly with respect to } 
     Y_{t,\tau}. \footnote{We note that some of our 
     main theoretical results later can also be shown with 
     {
     \cref{eq:filter-stability} holding in expectation} rather than uniformly.}
    \end{equation}
\end{assumption}

 This stability assumption is familiar in the error analysis of SIR filters~\cite{del2001stability}. It is valid if the
 transition kernel $a$ is bounded below and above by  positive constants, i.e., there exists a constant $c$ such that the density 
 of $a$  satisfies
\editstart
 \begin{align}\label{eq: minorization condition}
 \sqrt{c} \leq a(u'\mid u) \leq 1/\sqrt{c}, \qquad \forall u,u'\in \U,
 \end{align}
 then~\cref{eq:filter-stability} 
 holds with $\Cfs = c^{-1}$~\cite[Thm.~3.7]{rebeschini2014nonlinear}, where the stability is established with respect to the total variation
distance.
\editend
 This condition leads to the so-called  minorization condition that ensures geometric 
 ergodicity of Markov processes~\cite{meyn2012markov}, which is shown to be inherited by
 the filter. We note that this condition is  relatively strong and can only be verified for a restricted class of systems, e.g.,
 when $\U$ is compact. A complete characterization of  systems with uniformly geometrically stable filters is a 
 challenging  open problem in the field.
 More insight is available for the weaker notion of asymptotic stability of the filter, which holds when the system is  
 ``detectable'' in a sense that is suitable for nonlinear stochastic dynamical systems~\cite{van2010nonlinear,chigansky2009intrinsic,van2009observability,kim2022duality,kim2024stability}; see also~\cite{crisan2011oxford,kim2024stability} for a survey of filter stability results. 

\subsection{Statement of main results and overview of proofs}\label{sec:OTF-theory-main-results}

The OT representation of conditional distributions, as established in~\cref{prop: solvability of monge + conditional brenier}, allows us to replace the Bayesian update with the push-forward of conditional Brenier maps. In particular, the filtering  operator
$\T_y$, for any probability measure $\pi$, can be written as
\begin{align}\label{eq: OTF exact operator}
\text{(exact filter operator)} \qquad \T_y[\pi] \coloneqq \nabla_{u} \phi^{\dagger}[\pi](y,\cdot) \# \eta_{\U}[\pi] ,  \quad \textnormal{where} \quad \eta_{\U}[\pi] \coloneqq \mathcal{A}\pi\, .
\end{align}
Here, $\phi^{\dagger}[\pi]$ denotes the operator acting on $\pi$ that yields the minimizer of the dual problem~\cref{eq:phi-dagger-def}, where the marginal distributions $\eta$ and $\nu$ are defined in terms of $\pi$ according to 
\begin{align}\label{eq: COT measures notation} 
\eta[\pi]\coloneqq \nu_\Y[\pi]\otimes \eta_{\U}[\pi]\, ,\qquad \nu[\pi](y, u) \coloneqq h( y \mid u)\eta_\U[\pi](u).
\end{align}
As the notation implies, the probability measures and the conditional Brenier potential depend 
\editstart
implicitly 
\editend
on the input measure $\pi$. This formulation 
motivates the following approximation:
\begin{align}\label{eq: OTF approximation operator}
\text{(approximate filter operator)}\qquad \widehat \T_y[\pi]\coloneqq  \nabla_{u} \widehat \phi[\pi](y,\cdot) \# \eta_{\U}[\pi]. 
\end{align}
The difference is that, instead of $\phi^\dagger[\pi]$, we use $\widehat \phi[\pi]$, which is the minimizer of the empirical dual problem~\cref{eq:empirical-min} constructed from independent samples of the marginal distributions
defined in \cref{eq: COT measures notation}. With ~\cref{eq: OTF exact operator} and \cref{eq: OTF approximation operator} at hand, the exact and approximate posteriors evolve as
\begin{align*}
\pi^t &= \T_{Y_t} [\pi^{t-1}] = \nabla_u \phi^\dagger[\pi^{t-1}]\#\eta_\U[\pi^{t-1}], \quad
\widehat \pi^t &= \widehat \T_{Y_t} [\widehat \pi^{t-1}] = \nabla_u \wh{\phi}[\wh{\pi}^{t-1}]\#\eta_\U[\wh{\pi}^{t-1}]\, ,\quad \widehat \pi^0 = \pi^0\,.
\end{align*}
To simplify notation, let us write
\begin{align}\label{eq: different phis}
{\phi_t^\ddagger}\coloneqq \phi^\dagger[\pi^{t-1}]\, ,\qquad
\phi_t^\dagger \coloneqq \phi^\dagger[\widehat \pi^{t-1}]\, ,\qquad 
\widehat \phi_t \coloneqq \widehat \phi[\widehat \pi^{t-1}]\,,
\end{align}
as well as 
\begin{align*}
\eta^t\coloneqq \eta[ \pi^{t-1}]\, ,\qquad \nu^t \coloneqq \nu[ \pi^{t-1}]\, 
 ,\qquad \widehat \eta^t\coloneqq \eta[ \widehat \pi^{t-1}]\, ,\qquad \widehat \nu^t \coloneqq \nu[ \widehat \pi^{t-1}]\,.
\end{align*}
We highlight that ${\phi_t^\ddagger}$ corresponds to the exact filter  applied to the exact posterior $\pi^{t-1}$, while $\phi_t^\dagger$ corresponds to the exact filter  applied to the approximated posterior $\wh{\pi}^{t-1}$. \Cref{fig:diag1} shows the various distributions 
and transport maps involved in our analysis.

\editstart
\begin{figure}[t]
\footnotesize
\centering
\begin{tikzpicture}[
  arr/.style={->, >=stealth, thick},
]
  \node[anchor=west] (A12) at (5, 2)   {$\widehat \pi^t = \widehat {\T}_{Y_t}[\widehat\pi^{t-1}]$};
  \node (A21) at (0, 1)   {$\widehat \pi^{t-1}$};
  \node[anchor=west] (A22) at (5, 1)   {$\T_{Y_t}[\widehat\pi^{t-1}]$};
  \node (A31) at (0, 0)   {$\pi^{t-1}$};
  \node[anchor=west] (A32) at (5, 0)   {$\pi^t = \T_{Y_t}[\pi^{t-1}]$};

  \draw[arr] (0.5,0) -- (5,0) node[midway, above] {$\nabla_u \phi^{\ddagger}_t$};
  \draw[arr] (0.5,1) -- (5,1) node[midway, above] {$\nabla_u \phi^{\dagger}_t$};
  \draw[arr] (0,1.3) |- (5,2)
    node[pos=0.78, above] {$\nabla_u \widehat \phi_t$};
    
  \draw[decorate, decoration={brace, amplitude=6pt}]
    (A31.west) -- (A21.west)
    node[midway, left=3pt] {$D(\widehat{\pi}^{t-1}, \pi^{t-1})$};
  \draw[decorate, decoration={brace, mirror, amplitude=6pt}]
    (A32.east) -- (A12.east)
    node[midway, right=3pt] {$D(\widehat{\pi}^t, \pi^t)$};

  \node[draw, fit=(A12)(A22), inner xsep=-2pt, inner ysep=-2pt] {};
\end{tikzpicture}
\caption{Diagram illustrating the different potentials in~\eqref{eq: different phis} and the quantities we  bound. Arrows depict 
transformations between pertinent measures via OT maps.
The conditional OT  error analysis  from~\Cref{sec:conditionalOT-theory-main-results} is used to bound the distance between the 
objects within the box.}
\label{fig:diag1}
\end{figure}
\editend

With this setup, our goal is to bound the approximation error of the posteriors:
\begin{align}\label{eq: exact filter error}
\text{{(exact mean filtering error)}}\qquad  \mathbb{E}^{\textnormal{train}}_{1:t} \mathbb{E}_{\mathscr{Y}_t} D( \wh \pi^t, 
    \pi^t)\, ,
\end{align}
where the outer expectation is with respect to the empirical samples used to train/compute the potentials $\{\widehat\phi_\tau\}_{\tau=1}^t$, and the inner expectation is with respect to the observed data
\editstart
$\mathscr{Y}_t = \{Y_1,Y_2,\dots,Y_t\}$, with $Y_{\tau}\sim\nu_{Y}^{\tau}$, the observation marginal of the true joint $\nu^{\tau}(y,u)=h(y\mid u)\mathcal{A}\pi^{\tau-1}$. 
\editend
However, as will be clear in \Cref{lem: div between true and estimated posterior}, controlling this error is challenging. We therefore consider the approximate error:
\begin{align}\label{eq: approximate filter error}
  (\text{approximate mean filtering error})\quad  
 \Expect_{1:t}^{\textnormal{train}} \mathbb{E}_{ \widehat{\mathscr{Y}}_{t}}  D( \wh \pi^t, \pi^t)\, ,
 \end{align}
where the inner expectation is now taken with respect to the observations $\widehat{\mathscr{Y}}_{t} = \{\widehat Y_{1},\widehat Y_{2},...,\widehat Y_{t}\}$, where $\widehat Y_\tau\sim \widehat \nu_\Y^\tau$ for all $\tau\leq t$, and $\widehat \nu^\tau = h( y \mid u) \mathcal{A} \widehat\pi^{\tau-1}$.
We will bound both of these quantities 
\editstart
which effectively allows us to control the notion of mean conditional estimation error introduced in~\eqref{eq: conditional metric} extended to the filtering setup. We stress that $\wh \pi^t$ is not an empirical measure and corresponds to the underlying distribution from which the particles of the implemented  OTF algorithm are sampled. At the same time, $\wh \pi^t$ is not the approximation obtained in the mean field limit (analogous to~\eqref{def:tilde-phi}) and depends on the sample size $N$ of the particle ensemble propagated in the algorithm in order 
to set up problem~\eqref{eq:empirical-min} at each timestep. 
\editend

Our first goal is to relate the  exact and approximate mean filtering errors to the difference between the empirical estimate of the potential function $\wh{\phi}_t$ and the exact potential $\phi^\dagger_t$. For this, we make the following assumption on the divergence $D$, extending the notation of  \cref{eq: metric D stability}.

\begin{assumption}\label{assump: divergence}
    The divergence $D$ is uniformly stable over $\mathbb{P}^2(\U)$, the space of probability measures with bounded second moments, in the sense that
    \begin{equation*}
        D( T_1 \# \mu, T_2 \# \mu) \le C_D \| T_1 - T_2 \|_{L^2_\mu}, \qquad 
        \forall \; T_1, T_2 \in L^2_\mu 
         \text{ and }
        \mu \in \mathbb{P}^2(\U),
    \end{equation*}
    for a uniform constant $C_D >0$.
    Moreover, $D$ satisfies the triangle inequality:
\begin{align}\label{eq: triangle ineq.}
D(\mu_1,\mu_3)\leq D(\mu_1,\mu_2)+D(\mu_2,\mu_3)\ ,\qquad \forall \mu_1,\mu_2,\mu_3\in\mathbb{P}(\U)\,.
\end{align} 
\end{assumption}
In \Cref{sec:divergences}, we present 
several commonly used divergences that satisfy \cref{assump: divergence}, including the Wasserstein-2 metric, as well as the maximum mean discrepancies with appropriate
kernels. We are now ready to present a lemma that allows us to control 
the mean filtering errors in terms of the estimation error of the Brenier potentials.

\begin{lemma}
\label{lem: div between true and estimated posterior}
 Suppose  \Cref{def: uniformly geometrically stable filter,assump: divergence} holds. Then, the exact and approximate mean filtering errors satisfy the bounds
\begin{align}\label{eq:first-bound-exact-filter-error}
     \Expect_{1:t}^{\textnormal{train}} \mathbb{E}_{ {\mathscr{Y}}_{t}}D( \wh \pi^t, 
    \pi^t)\, &\leq C_D \Cfs\sum_{\tau=1}^{t} (1-\lambda)^{t-\tau} \Expect_{1:\tau}^{\textnormal{train}} \mathbb{E}_{{\mathscr{Y}}_{\tau-1}} \|\nabla_u \widehat \phi_\tau - \nabla_u \phi^{\dagger}_{\tau}\|_{L_{\widetilde \eta^\tau}^2},\\\label{eq:first-bound-approx-filter-error}
        \Expect_{1:t}^{\textnormal{train}} \mathbb{E}_{ \wh{\mathscr{Y}}_{t}}D( \wh \pi^t, 
    \pi^t)\, &\leq C_D \Cfs\sum_{\tau=1}^{t} (1-\lambda)^{t-\tau} \Expect_{1:\tau}^\textnormal{train} \mathbb{E}_{\wh{\mathscr{Y}}_{\tau-1}} \|\nabla_u \widehat \phi_\tau - \nabla_u \phi^{\dagger}_{\tau}\|_{L_{\wh{\eta}^\tau}^2}, 
\end{align}
where $\widetilde \eta^\tau \coloneqq \wh{\eta}^\tau_\U \otimes \nu^\tau_\Y$, and we 
recall $\widehat \eta^\tau = \widehat \eta^\tau_\U \otimes \widehat \nu^\tau_\Y$. 
\end{lemma}
The proof of~\cref{lem: div between true and estimated posterior} follows the approach of~\cite[Prop. 2]{al-jarrah2023optimal}. 
\editstart
As illustrated by~\Cref{fig:diag2}, 
\editend
we decompose the filtering errors via the repeated use of the triangle inequality coupled with the uniform geometric stability of the filter in~\cref{eq:filter-stability} and the stability of $D$ in~\eqref{eq: metric D stability}. The details  are supplied in~\Cref{subsec: proof triangle inequality filtering decomposition}.

\editstart
\begin{figure}[h]
\footnotesize
\centering
\begin{tikzpicture}[
  arr/.style={->, >=stealth, thick},
]

  \node (N11) at (0, 0)  {$\wh\pi^0 =\pi^0$};
  \node (N31) at (2, 1)  {$\wh\pi^1$};
  \node (N41) at (4.5, 2)  {$\wh\pi^2$};
  \node (N51) at (5.7, 3.3)  {$\hdots$};
  \node (N52) at (7.7, 3.3)  {$\wh\pi^{t-1}$};
  \node[anchor=west] (N61) at (9.3, 4.3) {$\wh\pi^t = \wh\T_{Y_t}[\wh\pi^{t-1}]$};

  \node (N12) at (2, 0)  {$\pi^1$};
  \node[anchor=west] (N13) at (9.3, 0) {$\pi^t = \T_{Y_{1,t}}[\pi^1]$};
  \node (N32) at (4.5, 1)  {$\T_{Y_{2}}[\wh \pi^1]$};
  \node[anchor=west] (N33) at (9.3, 1.45) {$\T_{Y_{2,t}}[\wh \pi^2]$};
  \node[anchor=west] (N21) at (9.3, 0.45) {$\T_{Y_{1,t}}[\wh \pi^1]$};
  \node[anchor=west] (N53) at (9.3, 3.3) {$\T_{Y_{t}}[\wh \pi^{t-1}]$};
  
  \draw[arr] (0,0.2) |- (1.6,1) node[pos=0.75, above]  {$\wh\T_{Y_1}$};
  \draw[arr] (2,1.5) |- (3.7,2) node[pos=0.80, above]  {$\wh\T_{Y_2}$};
  \draw[arr] (4.5,2.5) |- (5.4,3.3) node[pos=0.75, above]  {$\wh\T_{Y_3}$};
  \draw[arr] (6.1,3.3) -- (7.2,3.3) node[midway, above]   {$\wh\T_{Y_{t-1}}$};
  \draw[arr] (7.7,3.6) |- (9.3,4.3) node[pos=0.80, above] {$\wh\T_{Y_t}$};
  \draw[arr] (0.7,0) -- (1.6,0) node[midway, above] {$\T_{Y_1}$};
  \draw[arr] (2.4,1) -- (3.7,1) node[midway, above] {$\T_{Y_2}$};
  \draw[arr] (8.2,3.3) -- (9.3,3.3) node[midway, above] {$\T_{Y_t}$};
  \draw[arr] (2.4,0) -- (9.3,0) node[pos=0.35, sloped, above, inner sep=1pt] {$\T_{Y_{1,t}}$};
  \draw[arr] (5.3,1) -- (9.3,1) node[midway, sloped, above, inner sep=1pt] {$\T_{Y_{2,t}}$};
  \draw[arr] (2.5,0.9) to[out=-13, in=180] node[pos=0.6, above, inner sep=1pt] {$\T_{Y_{1,t}}$} (9.3,0.45);
  \draw[arr] (5.3,2) to[out=-15, in=180] node[pos=0.6, above] {$\T_{Y_{2,t}}$} (9.3,1.45);

\node at (10, 2.5) {$\vdots$};
\node[anchor=west] (N99) at (9.3, 1) {$\T_{Y_2,t}[\T_{Y_1,}[\wh \pi^1]]$};

\node[draw, fit=(N12)(N31), inner xsep=-2pt, inner ysep=-1pt] {};
\node[draw, fit=(N32)(N41), inner xsep=-2pt, inner ysep=-1pt] {};
\node[draw, fit=(N53)(N61), inner xsep=-2pt, inner ysep=-1pt] {};

\draw[decorate, decoration={brace, mirror, amplitude=6pt}]
    (N32.east) -- ($(N32.east)+(0,1)$)
    node[midway, right=3pt] {$D_2$};
\draw[decorate, decoration={brace, mirror, amplitude=6pt}]
    (N12.east) -- ($(N31.east)$)
    node[midway, right=3pt] {$D_1$};
\draw[decorate, decoration={brace, mirror, amplitude=6pt}]
    (N13.east) -- ($(N13.east)+(0,0.45)$)
    node[midway, right=3pt] {$\frac{\Cfs}
     {(1-\lambda)^{-t}}D_1$};
\draw[decorate, decoration={brace, mirror, amplitude=6pt}]
($(N99.east)$) -- ($(N99.east)+( 0,0.45)$)
node[midway, right=3pt] {$\frac{\Cfs}
     {(1-\lambda)^{-(t-1)}}D_2$};
\draw[decorate, decoration={brace, mirror, amplitude=6pt}]
($(N61.east)-(0,1)$) -- ($(N61.east)$)
node[midway, right=3pt] {$
     D_t$};

\end{tikzpicture}
\caption{\color{black}Diagram illustrating the error decomposition of~\Cref{lem: div between true and estimated posterior}. 
Quantities enclosed in boxes, denoted as $D_\tau$, are controlled by the stability inequality~\eqref{eq: metric D stability} and then 
bounded by the conditional OT  error analysis  of~\Cref{sec:conditionalOT-theory-main-results}. The magnitude of these errors decays exponentially in time by the uniform geometric filter stability of~\Cref{def: uniformly geometrically stable filter}.}
\label{fig:diag2}
\end{figure}
\editend

\Cref{lem: div between true and estimated posterior} enables us to apply \Cref{thm: slow rate,thm: fast rate} to bound the  mean filtering errors. However, this is not directly possible for the exact error due to the discrepancy in the probability distribution that appears in the $L^2$-norms on the 
right hand side. In particular, the application of \Cref{thm: slow rate,thm: fast rate}  concludes a bound for $\|\nabla_u \widehat \phi_\tau - \nabla_u \phi^{\dagger}_{\tau}\|_{L_{\wh{ \eta}^\tau}^2}$ while the bound \cref{eq:first-bound-exact-filter-error} is in terms of $\|\nabla_u \widehat \phi_\tau - \nabla_u \phi^{\dagger}_{\tau}\|_{L_{\widetilde \eta^\tau}^2}$, i.e., we have a mismatch in the $\Y$-marginals. On the other hand, 
the correct $L^2$-norm appears on the right hand side of \cref{eq:first-bound-approx-filter-error}, motivating our choice to work with the approximate mean filtering error.

\subsubsection{Approximate mean filtering error}
We now present our first main theorem that provides quantitative 
control of the approximate mean filtering error.

\begin{theorem}
\label{thm: approx filter error bound}
 Suppose \Cref{def: uniformly geometrically stable filter,assump: divergence} hold. Then, 
 for any $t > 0$ it holds that
   \begin{align}\label{eq: approx filter error bound}
\mathbb{E}&_{1:t}^{\textnormal{train}}\mathbb{E}_{\wh{\mathscr{Y}}_{t}}D(\pi^t,\wh{\pi}^t)\leq 
\Cf  e_t,
    \end{align}
    where 
    $\Cf = \frac{C_D \Cfs}{\lambda}$ and  $e_t$ takes different forms 
    depending on further assumptions:
    \begin{enumerate}
        \item (Slow rate) If \Cref{assump: slow rate} holds, then
    \editstart
        \begin{align*}
            e_t \lesssim\frac{\betamax}{\alphamin}\left[ 
    \max_{1\leq\tau\leq t}\inf_{\phi \in \mathcal{F}} \|\nabla_{u}\phi - \nabla_{u}\phi^{\dagger}_\tau\|_{L_{\wh \eta^\tau}^2}^2 + 
    \sqrt{ \frac{\kappa\widetilde{K}^2 R^2 C_{\F}}{N}} \right].
        \end{align*}
    \editend
    \item (Fast rate) If \Cref{assump: slow rate,assump: fast rate} hold for $\widehat\eta^\tau$ and $\widehat \nu^\tau$ for all  $\tau\leq t$, then
        \editstart
            \begin{align*}
            e_t \lesssim  \max_{1 \le \tau \le t} 
            \left[ \frac{\betamax}{\alphamin}  \inf_{\phi \in \F} \| \nabla_u \phi - \nabla_u \phi^\dagger_\tau \|_{L^2_{\wh \eta^\tau}}^2+ C_{\textnormal{est.}}(\tau) \left(\left(\frac{C_\F\log(N)}{N}\right)^{\frac{2}{2+\gamma}} \vee \frac{C_\F}{N}\right)\, \right] ,
            \end{align*}
        \editend
        where $C_{\textnormal{est.}}(\tau) >0$ is the constant appearing 
        in \cref{thm: fast rate} that may change in time due to variations in the Poincar\'e constants of $\widehat\eta_\U^\tau$, $\wh \nu^\tau(\cdot\mid y)$, and $\widehat\nu_\Y^\tau$.

    \end{enumerate}
\end{theorem}

The proof of the theorem follows by the application of \Cref{thm: slow rate,thm: fast rate} to the right hand side of \cref{eq:first-bound-approx-filter-error}; see~\Cref{subsec: proof of approximate error bound} for full details. Note that both fast and slow rate results require time-uniform control over the approximation error $\inf_{\phi \in \F}\|\nabla_{u}\phi - \nabla_{u}\phi^{\dagger}_{\tau}\|_{L_{\widehat\eta^\tau}^2}$ 
while the fast rate requires uniform control over the time-varying constants $C_\textnormal{est.}(\tau)$ as well.
As pointed out in~\Cref{rem: comment on assumption 1}, the slow rate is derived under assumptions pertaining to the hypothesis class with minimal requirements on the underlying filtering system. In contrast, the fast rate requires us to verify additional assumptions 
with constants appearing in the fast-rate estimate that may 
implicitly depend on the state dimension through the Poincar\'e 
constants in~\Cref{assump: fast rate} which in turn may vary in time
which can add significant challenges in practice.
We investigate these questions further for an example application 
in~\Cref{sec:example-filtering}.

\editstart
\begin{remark}
The bias term in \Cref{thm: approx filter error bound} is measured relative to
$\phi_\tau^\dagger$, which is the exact conditioning potential for the
approximate distribution generated by the algorithm at time $\tau$. The learned map $\widehat\phi_\tau$ is trained using
samples from the approximate distribution $\widehat\eta^\tau$, not  the truth. Thus, the bias term measures how well the class $\mathcal F$ can approximate the exact map that the algorithm would apply if the conditional OT problem were solved exactly at the current approximate filtering state. However, it does not measure the approximation error relative to the exact conditioning map of the true filter.
\end{remark}
\editend

\subsubsection{Exact mean filtering error}\label{subsec:exact-mean-filtering-error}
\editstart
We now explore two possible approaches towards bounding the true mean filtering error. Both approaches relate the exact filtering error bound in~\eqref{eq:first-bound-exact-filter-error} and the approximate filtering error bound in~\eqref{eq:first-bound-approx-filter-error}, whose control has already been established by~\Cref{thm: approx filter error bound}. 
\begin{enumerate}
\item The first route relies on relating the summands in~\eqref{eq:first-bound-exact-filter-error} and~\eqref{eq:first-bound-approx-filter-error} via a multiplicative bound when the uniform control of the density ratio between $\wh{\nu}^\tau_\Y$ and ${\nu}^\tau_\Y$ is achievable. This is indeed the case in the setting where the minorization condition of~\eqref{eq: minorization condition} holds. This 
often restricts us to bounded domains but leads to more convenient bounds with otherwise relaxed conditions. 

\item Whenever the minorization condition is not valid, the second approach considers an additive decomposition of the summands in~\eqref{eq:first-bound-exact-filter-error}. This is done via an application of Gronwall's lemma under the stability conditions of~\Cref{assump: exact filter error extra assump} enabling the control of the discrepancy between the two processes $\mathscr{Y}_{t}$ and $\wh{\mathscr{Y}}_{t}$.
This result requires restrictive regularity conditions that are often difficult to verify but 
remains applicable on unbounded domains.

\end{enumerate}
\begin{assumption} 
\label{assump: exact filter error extra assump}
It holds that:
    \begin{enumerate}
        \item \label{y-Lipshitzness} 
                There exists $L_\Y >0$ such that for all $u \in \U$,
                \begin{align}\label{eq:Lipshitzness in y}
                \begin{split}
                \| \nabla_u \phi^{\dagger}_{\tau}(y,u) - \nabla_u \phi^\dagger_\tau(y',u)\|_\U &\le L_\Y \| y- y'\|_\Y, \\
                \| \nabla_u \widehat\phi_{\tau}(y,u) - \nabla_u \widehat\phi_\tau(y',u) \|_\U            &\le L_\Y \| y- y'\|_\Y.
                \end{split}
                \end{align} 
        \item \label{weird stability condition for exact bound}
        For any pair of probability measures $\mu_1, \mu_2 \in \mathbb{P}(\U)$ 
        define the measures 
        \begin{equation*}
            \rho_1(y) = \int h( y \mid u) \mathcal{A}\mu_1( \d u), \qquad 
            \rho_2(y) = \int h( y \mid u) \mathcal{A}\mu_2( \d u). 
        \end{equation*}
        Then there exists a coupling $\Gamma \in \Pi(\rho_1, \rho_2)$  and a 
        constant $C_D'(h, \mathcal{A})>0$ so that  
        \begin{equation} \label{eq:coupling-assump}  
        \mathbb{E}_{(Y_1,Y_2)\sim \Gamma} \|Y_1 - Y_2 \|_{\Y}
        \leq C_D' D(\mu_1,\mu_2).
        \end{equation}
\end{enumerate}

\end{assumption}
The first Lipschitz condition of~\Cref{assump: exact filter error extra assump} is similar to~\Cref{assump: slow rate}(2) and~\Cref{assump: fast rate}(1). The second condition 
is valid under regularity assumptions on the dynamic operator $\mathcal{A}$ and 
the observation kernel $h$ and for appropriate choices of $D$.
In \Cref{subsec: calculation for Y stability} we verify this condition
for the Wasserstein-1 metric and a practical choice of $\mathcal{A}$ and  $h$. 

\Cref{thm: multiplicative+additive bound exact filtering error} below provides the first statistically complete error analysis of OTF, improving upon~\cite{al-jarrah2023optimal} which provided a
preliminary error decomposition of \eqref{eq:first-bound-exact-filter-error} and a bound in terms of an 
optimization gap but left the statistical error analysis open. The proof of 
this result along with additional discussions are summarized in~\Cref{subsec: proof of exact error bound 2} and~\Cref{subsec: proof of exact error bound}.

\begin{theorem}\label{thm: multiplicative+additive bound exact filtering error}
Suppose \Cref{def: uniformly geometrically stable filter,assump: divergence} hold
and let $e_t$ be the same quantity in \Cref{thm: approx filter error bound}.
Then:
\begin{enumerate}
\item \label{multiplicative} (With minorization) If the minorization condition~\eqref{eq: minorization condition} holds, for any $t > 0$,
   \begin{align}\label{eq: multiplicative exact filter error bound}
\mathbb{E}&_{1:t}^{\textnormal{train}}\mathbb{E}_{\wh{\mathscr{Y}}_{t}}D(\pi^t,\wh{\pi}^t)\leq 
\Cf  e_t,
    \end{align}
where 
    $\Cf = \frac{C_D \Cfs ^2}{\lambda}$.

\item \label{additive} (With Lipschitz regularity) 
If \Cref{assump: exact filter error extra assump} holds, for any $t>0$,
         \begin{align}\label{eq: main-filter-error-bound}
\mathbb{E}&_{1:t}^{\textnormal{train}}\mathbb{E}_{\mathscr{Y}_{t}}
D(\pi^t,\wh{\pi}^t)\leq 
\, \Cf' \max(1,\varrho^t) e_t,
    \end{align}
    where 
    $\Cf'\coloneqq \frac{C_D \Cfs}{\lambda}\left(1+ \frac{2L_{\Y}C_D'C_D \Cfs}{|\lambda-2L_{\Y}C_D'C_D \Cfs|}\right)$ and $\varrho \coloneqq 2L_{\Y}C_D'C_D \Cfs +1-\lambda.$
\end{enumerate}
\end{theorem}

\begin{remark}\label{rem: general remark on mult+add bounds}
We notice that \Cref{assump: exact filter error extra assump}(1), necessary for statement 2, may be difficult to verify on non-compact domains. One may also wonder if the Lipschitz condition \eqref{eq:Lipshitzness in y} can be relaxed to a H\"older condition
for an exponent $\zeta\in[0,1)$, since 
this condition may be verified easily \cite{hosseini2025conditional}. However, as shown in \Cref{lem: exact filtering error} and discussed in~\Cref{rem: zeta=1 regime} our asymptotic bounds remain controlled 
only for $\zeta=1$ and diverge for $\zeta<1$. Moreover, To ensure that~\cref{eq: main-filter-error-bound} remains uniform in time, it is necessary to assume $\varrho \leq 1$ implying
\begin{align}\label{eq: summability in time condition}
2 L_\Y C_D' C_D \Cfs \leq \lambda.
\end{align}
This is a restrictive condition that may fail for practical systems and we conjecture that it is an artifact of our proof technique. On the other hand, statement 1, while 
only applicable under the minorization condition of~\eqref{eq: minorization condition}, offers a viable alterantive approach that bypasses the additional regularity conditions of~\Cref{assump: exact filter error extra assump} and ~\Cref{eq: summability in time condition} but it is often verifiable in the compact support case; see also ~\Cref{rem: comments on mult bound extension} for further discussions.  
\end{remark}

\editend

\subsection{An example application}
\label{sec:example-filtering}
We now verify some of our assumptions in the context of a practical filtering model. We start by considering a class of filtering systems for which we can guarantee that, at each time step, conditions (i) and (ii) in the static example of ~\Cref{sec:conditionalOT:applications} can be preserved as time evolves. This ensures that the distributions $\widehat \eta_\U^t$ and $\widehat \nu^t$ remain strongly log-concave and log-smooth.
To this end, we characterize the filtering system~\cref{eq:generic-filtering-model} by update equations
\begin{align}
\label{eq: Filtering eqs}
U_t \sim \exp(-\mathscr{a}(\cdot|U_{t-1})), \quad U_0\sim\pi_0, \quad  Y_t\sim \exp(-\mathscr{h}(\cdot|U_t))\, ,
\end{align}
under the following set of assumptions:
 \begin{assumption}
 \label{assump:filtering system}
\begin{enumerate}[label=(\roman*)]
\item  
\(\pi_0(u)=\exp(-\mathscr{W}_{\pi_0}(u))\) and there are
positive and bounded constants $\sigma_{\min}(\mathscr{W}_{\pi_0})$,
$\sigma_{\max}(\mathscr{W}_{\pi_0})$ so that
$\sigma_{\min}(\mathscr{W}_{\pi_0}) I \preceq \nabla^2 \mathscr{W}_{\pi_0}(u) \preceq \sigma_{\max}(\mathscr{W}_{\pi_0}) I,$
\item For positive and bounded constants $\theta_{\min}$ and $\theta_{\max}$, it holds that
$\theta_{\min} I \preceq \nabla^2 \mathscr{h}(y\mid u) \preceq \theta_{\max} I.$
\item There are positive and bounded constants $\sigma_{\min}(\cdot)$,
$\sigma_{\max}(\cdot)$ so that for all $u,u' \in \U$,
\begin{align*}
&\sigma_{\min}(a_{u}) I \preceq \nabla_u^2 \mathscr{a}(u\mid u')\preceq \sigma_{\max}(a_{u}) I\, ,
\\
&\sigma_{\min}(a_{u'}) I \preceq \nabla_{u'}^2 \mathscr{a}(u\mid u')\preceq \sigma_{\max}(a_{u'}) I\, ,
\\
&\sigma_{\min}(a_{uu'}) I \preceq \nabla_{uu'}^2 \mathscr{a}(u\mid u')\preceq \sigma_{\max}(a_{uu'}) I\,.
\end{align*}
\end{enumerate}
 \end{assumption}
Moreover, we restrict our hypothesis class to be quadratic in the $u$-coordinate, namely
\begin{align}\label{eq: quadratic forms class}
\begin{split}
\F_{\text{quad}}:=\bigr\{(y,u)\mapsto \frac{1}{2}u^\top Q(y)u+u^{\top}b(y)&\big| b(y)\in\mathbb{R}^d\, , Q(y)\in\mathbb{R}^{d\times d}\, ,\\
&0\prec \sigma_{\min}(Q) I\prec Q(y)\prec \sigma_{\max}(Q) I\, ,\forall y\in\Y\bigr\}\, .
\end{split}
\end{align}
To simplify the notation, we also introduce the functions
\begin{align}\label{eq: update functions}
\begin{split}
\mathfrak{m}(x) \coloneqq \left(\sigma_{\min}(a_{u})-\frac{\sigma_{\max}^2( a_{uu'})}{\sigma_{\min}( a_{u'})+x}\right)\, , \quad 
\mathfrak{M}(x) \coloneqq \left(\sigma_{\max}(a_{u})-\frac{\sigma_{\min}^2( a_{uu'})}{\sigma_{\max}( a_{u'})+x}\right)\, .
\end{split}
\end{align}
We  then obtain the  following proposition which states that the distributions 
$\widehat\eta_\U^t$ and $\widehat\nu^t$ are log-concave and smooth.
\begin{proposition}\label{prop: log concavity smoothness result}
Consider the filtering system of~\eqref{eq: Filtering eqs}, take $\F=\F_\textnormal{quad}$ as in~\eqref{eq: quadratic forms class},
and suppose \Cref{assump:filtering system} holds.
Then, denoting
\begin{align}\label{eq: Hessian defined for measures}
\widehat \eta_\U^t(u)\propto\exp(- \mathscr{W}_{\widehat \eta_\U^t}(u))\, ,\qquad \widehat \nu^t(y,u)\propto\exp(- \mathscr{W}_{\widehat \nu^t}(y,u))\, ,
\end{align}
it holds that
\begin{align}\label{eq: Hessian bounds for measures of interest}
\begin{split}
\gamma_t I &\preceq \nabla_u^2 \mathscr{W}_{\widehat \eta_\U^t}(u)\preceq \Gamma_t I\, ,\\
\left(\theta_{\min}+\gamma_t\right) I &\preceq \nabla^2 \mathscr{W}_{\widehat \nu^t}(y,u)\preceq \left(\theta_{\max}+\Gamma_t\right) I \, ,
\end{split}
\end{align}
where, starting from $\gamma_0 \coloneqq \sigma_{\min}(\mathscr{W}_{\pi_0})$ and $ \Gamma_0 \coloneqq \sigma_{\max}(\mathscr{W}_{\pi_0})$, we define recursively
\begin{align}\label{eq: recursive relations gamma}
    \gamma_t\coloneqq \mathfrak{m}\left(\gamma_{t-1}/\sigma_{\max}^2(Q)\right)\, ,\qquad 
    \Gamma_t\coloneqq \mathfrak{M}\left(\Gamma_{t-1}/\sigma_{\min}^2(Q)\right)\, .
\end{align}
Consequently, if we further impose that
\begin{align}\label{eq: conditions for fixed point}
\begin{split}
&\sigma_{\max}^2(a_{uu'})< \sigma_{\min}(a_{u})\sigma_{\min}(a_{u'})\,,
\\
&\sigma_{\max}^2(a_{uu'})< \sigma_{\min}^2(a_{u'})\sigma_{\max}^2(Q)\, ,\\
&\sigma_{\min}^2(a_{uu'})< \sigma_{\max}^2(a_{u'})\sigma_{\min}^2(Q)\, ,
\end{split}
\end{align}
the update functions $\mathfrak{m}(\cdot/\sigma_{max}^2(Q))$ and $\mathfrak{M}(\cdot/\sigma_{min}^2(Q))$ have positive fixed points and the Hessians of the potentials $F_{\widehat \eta_\U^t},F_{\widehat \nu^t}$  remain uniformly  bounded in time 
\footnote{see \eqref{eq: bounds uniform in time} for the explicit bound.}.
\end{proposition}
The key  idea of the proof is the fact that affine-in-$u$ maps $ \widehat T(Y_t,\cdot)=\nabla_u\widehat\phi(Y_t,\cdot)$ preserve strong log-convexity and smoothness under pushforwards. Proceeding akin to~\cite[Lem. 5.1]{pathiraja2021mckean}, Brascamp–Lieb and Cram\'er–Rao inequalities then guarantee that the filtering updates  preserve lower and upper Hessian bounds. The recursive relations then follow; see~\Cref{proof: thm: general conditions on FEs}. 

The result of~\cref{prop: log concavity smoothness result}, analogously to what was shown in~\Cref{sec:conditionalOT:applications}, implies two main consequences: First, by the Brascamp-Lieb inequality (\cref{thm: BL inequality}), the measures $\wh \eta_\U^t$, $\wh \nu^t(\cdot\mid y)$ (uniformly over $y$), and $\nu_{\Y}^{t}$ satisfy the Poincar\'e inequality with constants
$\CPI^{\wh \eta_\U^t}\leq \frac{1}{\gamma_t}\,$, $\CPI^{\wh \nu^t(\cdot\mid y)}\leq \frac{1}{\gamma_t+\theta_{\min}}\,$ and $\CPI^{\wh\nu_{\Y}^{t}}\leq  \frac{1}{\gamma_t+\theta_{\min}}.$
Therefore satisfying~\Cref{assump: fast rate}
\editstart
(2) (with $\widetilde{q} = 0$)
\editend
at any time $t$. We stress that only the strong log-concavity parts of~\eqref{eq: Hessian bounds for measures of interest} are needed to satisfy this assumption towards the application of fast rates. Second, the maps $\phi^{\dagger}_t(y,\cdot)$ are both smooth and strongly convex uniformly over $y$ with
$\sqrt{\frac{\gamma_t}{\Gamma_t+\theta_{\max}}} I\preceq\nabla_u^2\phi^{\dagger}_t(y,\cdot)\preceq\sqrt{\frac{\Gamma_t}{\gamma_t+\theta_{\min}}} I\, ,$
as a consequence of the Caffarelli contraction theorem (\cref{thm: caffarelli contraction thrm} and \cref{thm: caffarelli contraction thrm plus convexity}). This regularity can then be leveraged in order to derive approximation rates 
with an argument analogous to the one mentioned in~\Cref{sec:conditionalOT:applications} for the Legendre polynomials.

\begin{remark}
As mentioned above,~\cref{prop: log concavity smoothness result} heavily relies on fixing $\F=\mathcal{F}_{\textnormal{quad}}$ as in~\eqref{eq: quadratic forms class}. This becomes significantly more difficult to show when considering more complex hypothesis classes. Extending our results in this 
direction is definitely an appealing area of future work.
\editstart
We also emphasize that the slow-rate estimate in~\Cref{thm: approx filter error bound} should primarily be viewed as a general consistency result requiring minimal assumptions on the filtering system and hypothesis class. In many practical settings, such as compact domains with smooth potentials, the assumptions for the fast rate are expected to hold. However, in filtering problems the associated constants may deteriorate depending on the dimension and the dynamics of the system, making uniform-in-time verification challenging. Another concrete
application are discretized ODE/PDE problems where $\nabla^2 \mathscr{a}(u\mid u')$ is a perturbation of the identity 
matrix. We anticipate that under mild regularity assumptions and with small time steps, 
such dynamics would satisfy \Cref{prop: log concavity smoothness result}.
\editend
\end{remark}

Our second result establishes explicit bounds on the Poincar\'e constants of $\widehat \eta_\U^t$ and $\widehat \nu^t$, therefore satisfying two of the three inequalities in~\Cref{assump: fast rate}
\editstart
(2) (with $\widetilde{q} = 0$)
\editend
without any further assumption on the hypothesis class besides $\betamax$-smoothness
\editstart
by~\Cref{assump: slow rate}(4). 
\editend
This can be achieved by characterizing the filtering system in~\cref{eq:generic-filtering-model} by the update equations  
\begin{align}\label{eq: simple filtering system}
\begin{split}
U_t = a(U_{t-1})+V_t\, ,\qquad U_0\sim\pi_0\, ,
\quad \text{and} \quad 
Y_t = h(U_t)+W_t\, ,
\end{split}
\end{align}
which, via convolution, injects independent noise that satisfies a  Poincar\'e inequality to Lipschitz deterministic updates by $a$ and $h$, denoted with a slight abuse of notation. Moreover, assuming the existence of a Lipschitz inverse $h^{-1}$, one can establish the Poincar\'e inequality also for the conditionals $\widehat \nu(\cdot\mid y)$ uniformly in $y$, 
\editstart
ultimately satisfying~\Cref{assump: fast rate}(2).
\editend
\begin{proposition}\label{prop: Poincare filtering}
Consider the filtering system characterized by~\eqref{eq: simple filtering system}
where $V_t,W_t$ are i.i.d. noise  variables satisfying the Poincar\'e inequality with constants $C_{\textnormal{PI}}^{V}$, $C_{\textnormal{PI}}^{W}$, and $a$ and $h$ are deterministic maps. Then, the next two results follow.
\begin{enumerate}[label=(\roman*)]
\item If the hypothesis class $\mathcal F$ is $\betamax$-smooth
\editstart
(as in~\Cref{assump: slow rate}(4)),
\editend
$a$ and $h$ are $L_{a}$ and $L_{h}$ -Lipschitz respectively, and $\pi_0$ satisfies the Poincar\'e inequality with constant $C_{\textnormal{PI}}^{\pi_0}$, then $\widehat \eta_\U^t$ and $\widehat \nu_\Y^t$ satisfy the Poincar\'e inequality with constants $C_{\textnormal{PI}}^{\widehat \nu_\Y^t}\leq L_{h}C_{\textnormal{PI}}^{\widehat \eta_\U^t}+C_{\textnormal{PI}}^{\pi_W}$ and $C_{\textnormal{PI}}^{\widehat \eta_\U^t}\leq L_{a}^t\betamax^{2(t-1)}C_{\textnormal{PI}}^{\pi_0}+C_{\textnormal{PI}}^{\pi_V}\sum_{k=0}^{t-1}(L_{a}\betamax^2)^k.$
\item Moreover, if $h$ has an $L_{h}'$-Lipschitz inverse $h^{-1}$, then the conditionals $\widehat \nu^t(\cdot\mid y)$ satisfy the Poincar\'e inequality uniformly over $y$ with constants
$C_{\textnormal{PI}}^{\widehat \nu^t(\cdot\mid y)}\leq (L_{h}')^2C_{\textnormal{PI}}^{\pi_W}$.
\end{enumerate}
\end{proposition}
    The proof is based on the repeated application of the following two basic results about the Poincar\'e inequality:
    (1) By~\cite[Thm. 3]{courtade2020bounds}, given two probability measures $\mu$ and $\mu'$ satisfying the Poincar\'e inequality with constants $C_{\textnormal{PI}}^\mu$ and $C_{\textnormal{PI}}^{\mu'}$, their convolution $\mu*\mu'$ also satisfies the Poincar\'e inequality with sub-additive constant $C_{\textnormal{PI}}^{\mu*\mu'}\leq C_{\textnormal{PI}}^\mu+C_{\textnormal{PI}}^{\mu'}$\footnote{\cite{courtade2020bounds} actually shows a sharper upper bound with an additional term subtracted on the left-hand side of the inequality but we will work with 
    the simpler result.};
    (2) By~\Cref{lem: PI for conditionals}, if $\mu$ satisfies the Poincar\'e inequality with constant $C_{\textnormal{PI}}^{\mu}$, and $T$ is a $\beta$-Lipschitz map, then $T\# \mu$ satisfies the Poincar\'e inequality with  $C_{\textnormal{PI}}^{T\#\mu}\leq \beta^2C_{\textnormal{PI}}^{\mu}$.

\begin{remark}
\label{rem: poincare of conditionals w/o smoothness}
In regard of~\Cref{prop: Poincare filtering}, we underline that while result (i) follows under achievable conditions, result (ii) relies on the strong assumption of the invertibility of the likelihood kernel $h$. This is not satisfied in several applications, most notably when the observations $Y_t$ provide 
partial information about  $U_t$.
\end{remark}

\subsection{Towards algorithms}\label{subsec: OTF particle algorithms}
We briefly discuss how our earlier analysis can be connected to 
an idealized  algorithm.
In order to produce the samples that are required to learn $\wh{\phi}_t$ in an algorithm, we assume the transition kernels in \eqref{eq:generic-filtering-model} are known and can be 
simulated. 
\begin{align}\label{eq: simulater model}
U_t=a(U_{t-1}, V_t)\,, \qquad Y_t=h(U_{t}, W_t)\, ,
\end{align}
where $\{V_t\}$ and $\{W_t\}$ are i.i.d sequences with known and easy to sample from distributions, such as Gaussians. Then, starting from two sets of independent samples  $\{x_i\}_{i=1}^N$ and $\{\widetilde x_i\}_{i=1}^N$ from the initial distribution $\pi_0$, the training samples are generated according to   
\begin{align}\label{eq: simulation based algorithm}
\begin{split}
v_i^t&=\left[ a(\cdot,V_t^i)\circ \widehat T_{t-1}(Y_{t-1},\cdot)\circ...\circ \wh{T}_1(Y_{1},\cdot)\circ  a(\cdot, V_1^i)\right](x_i)\, ,\\ 
u_i^t&=\left[ a(\cdot,\widetilde V_t^i)\circ \widehat T_{t-1}(Y_{t-1},\cdot)\circ...\circ \wh{T}_1(Y_{1},\cdot)\circ  a(\cdot, \widetilde V_1^i)\right](\widetilde x_i)\, ,\\ 
y_i^t&= h(u_i^t,W_t^i)\, .\end{split}
\end{align}
with $\wh{T}_\tau = \nabla_u \wh{\phi}_\tau$, and $\{V_\tau^i\}_{i,\tau=1}^{N,t}$, $\{\widetilde V_\tau^i\}_{i,\tau=1}^{N,t}$ and $\{W_t^i\}_{i=1}^N$ being independent copies of $V_t$ and $W_t$. Note that, in order to ensure independence across all times, the samples that are used to train $\wh{\phi}_\tau$, for $\tau <t$, should be independent of the samples used to train $\wh{\phi}_t$. Therefore, at each time $t$, the samples 
 $\{x_i\}_{i=1}^N$, $\{\widetilde x_i\}_{i=1}^N$,   $\{V_\tau^i\}_{i,\tau=1}^{N,t}$, $\{\widetilde V_\tau^i\}_{i,\tau=1}^{N,t}$ and $\{W_t^i\}_{i=1}^N$ that are used in~\eqref{eq: simulation based algorithm} should be regenerated. Moreover, this ideal procedure requires the  storage of all of the previous  maps $\{\widehat T_\tau\}_{\tau= 1}^t$.

The ideal procedure is useful in simplifying the error analysis. However, computationally, it becomes intensive as the time $t$ grows. A more practical approach is to simulate an interacting particle system according to 
\begin{align}\label{eq: interacting particles algorithm}
\begin{split}
v_i^t& =  a(\cdot,V_t^i)\circ \widehat T_{t-1}(Y_{t-1},v_i^{t-1}),\quad 
u_i^t=  a(\cdot,\widetilde V_t^i)\circ \widehat T_{t-1}(Y_{t-1},u_i^{t-1}), \quad 
y_i^t= h(u_i^t,W_t)\\ 
v^0_i & = x_i, \quad  u^0_i = x_i
\end{split}
\end{align}
Unlike~\eqref{eq: simulation based algorithm}, the  particle system does not require storage of the maps and resampling. 
The efficiency improvement comes at the expense of correlated
particles, making our  analysis not applicable. The error analysis for the interacting particle 
system is more challenging and may be approached using a propagation of chaos analysis \cite{sznitman2006topics}.

\section{Numerical experiments}\label{Sec:Numerical experiments}
In this section, we perform several experiments and benchmarks to compare the performance of the OTF algorithm alongside other filtering algorithms like EnKF~\cite{evensen2006} and SIR~\cite{doucet09}. 
The Python code for reproducing the numerical results is available online\footnote{\url{https://github.com/Mohd9485/OTF_SIAM}}.
We begin with a brief summarize  the OTF approach.

    The optimization problem \eqref{eq:phi-dagger-def} involves the convex conjugate $\phi^*$, which is numerically challenging to approximate for a general class of functions. This issue is resolved  using the identity $\phi(y,v) = \max_{\varphi  \in \textit{CVX}_\mathcal{U}}v^\top \nabla_u \varphi(y,v) - \phi^*(y,\nabla_u \varphi(y,v) 
    )$, derived from Fenchel-Young inequality \cite{Fenchel1949OnCC}, and representing the convex conjugate $\phi^*$ with a  function $\psi \in \textit{CVX}_\mathcal{U}$, leading to a min-max problem of the form
    \begin{equation}\label{eq:min-max_convex}
    \min_{\psi \in \textit{CVX}_\mathcal{U}}\, \max_{\varphi \in \textit{CVX}_\mathcal{U}} \int \big[v^\top\nabla_u \varphi(y,v)-\psi(y,\nabla_u \varphi(y,v))\big] \d\eta(y,v) + \int \psi(y,u) \d\nu(y,u).
    \end{equation} 
    The map $\nabla_u \varphi(y,v)$, resulting from the solution to this min-max problem, serves as the desired conditional OT map $T^\dagger(y,v)$~\cite{MakTagOhLee19,taghvaei2022variationalformulation}.

    As proposed in~\cite[sec.(V)]{al-jarrah2023optimal}, it is numerically useful to relax 
    the constraint that $T(y, \cdot)$ is of the gradient form and replace  $\nabla_u \varphi(y,u)$ with the maximization over unconstrained maps $T(y,u)$.  The resulting optimization problem becomes 
    \begin{equation}\label{eq:min-max_T}
    \min_{\psi \in \textit{CVX}_\mathcal{U}}\, \max_{T \in \mathcal{M}(\eta)} \int \big[v^\top T(y,v)-\psi(y,T(y,v))\big] \d\eta(y,v) + \int \psi(y,u) \d\nu(y,u)\, ,
    \end{equation}
    where $\mathcal{M}(\eta)$ are the set of $\eta$-measurable maps.

The formulation of~\eqref{eq:min-max_T} requires parameterizing $\psi$ as a convex function. To enforce this, previous approaches~\cite{MakTagOhLee19, taghvaei2022variationalformulation, al-jarrah2023optimal} have resorted to the use of Input Convex Neural Networks (ICNNs)~\cite{pmlr-v70-amos17b} that are challenging to train~\cite{bunne2022supervised, korotin2021neural}. Alternatively,~\cite[App. A.3]{al-jarrah2024nonlinear} derives a formulation which, replacing $\psi$ with the representation $\frac{1}{2}\|u\|_\U^2-\psi(y,u)$, turns the min-max problem to the max-min problem:
\begin{equation}\label{eq:ccv-objective}
    \max_{\psi \in \textit{c-CCV}_\U}\, \min_{T \in \mathcal M(\eta)}
 \int \big[c(T(y,v), v)-\psi(y,T(y,v))\big] \d\eta(y,v) + \int \psi(y,u) \d\nu(y,u),
 \end{equation}
 with the cost
 $c(u,u')\coloneqq\frac{1}{2}\|u-u'\|_\U^2$, and the class of $c$-concave maps in the $u$-coordinate $\textit{c-CCV}_\U$.
 We recall that the $c$-concavity of $\psi(y, \cdot)$ is equivalent to the convexity of $\frac{1}{2}\|\cdot\|_\U^2-\psi(y,\cdot)$, a condition that is 
 easier to enforce in practice.

In the filtering setup of~\eqref{eq: interacting particles algorithm}, the objective of \eqref{eq:ccv-objective} can then be approximated empirically using samples as
\begin{equation}\label{eq:empirical_loss_time} 
J_{(N)}^{t}(\psi,T) \coloneqq \frac{1}{N} \sum_{i=1}^N 
\Big[\tfrac{1}{2}\|T(y_i^t,v_i^t)- v_i^t\|^2 - \psi(y_i^t,T(y_i^t,v_i^t)) + \psi(y_i^t,u_i^t)\Big].
\end{equation}

\editstart
Throughout all numerical experiments, we draw \((y_i^t,u_i^t)\) according to~\eqref{eq: interacting particles algorithm} and form \(v_i^t\) by shuffling the \(\{u_i^t\}\), rather than generating \(v_i^t\).
\editend
We denote by $\mathcal F$ and $\mathcal T$ the parametric classes representing the potential $\psi$ and the map $T$, respectively. Then we solve
\begin{equation}\label{eq:empirical-optimization} 
\max_{\psi\in \mathcal F}\,\min_{T \in \mathcal T}\, J_{(N)}^t(\psi,T)
\editstart
+\lambda_T R_T^t+\lambda_\psi R_\psi^t.
\editend
\end{equation}
We enforce the $c$-concavity of the potentials $\psi(y,\cdot)$ and the monotonicity of the transport maps $T(y,\cdot)$ (recall Brenier's theorem) through the addition of two regularization terms:
\begin{align}\label{eq: regularization terms}
    \editstart
    R_T^t
    \editend
    &= \frac{1}{N(N-1)}\sum_{i=1}^N\sum_{j=1}^N g_T\left(\langle T(y_{i}^t,v_{i}^t) - T(y_{i}^t,v_{j}^t),v_{j}^t-v_{i}^t  \rangle\right)\\
    \editstart 
    R^t_\psi
    \editend
    &= \frac{1}{N}\sum_{i=1}^N g_\psi\left( |\Laplace_u \psi(y_{i}^t,u_{i}^t) |^2 \right )
\end{align}
where $g_T,g_\psi$ denote smooth monotone functions and $\Laplace_u$ is the Laplace operator.

In our numerical experiments, both $\psi$ and $T$ are parameterized as ResNet-type neural networks and 
the parameters are updated via a gradient ascent--descent procedure using Adam. To reduce the computational overhead, the number of iterations per time step is gradually decreased, since 
consecutive OT maps often vary slightly between successive time steps.
\editstart
For all of our experiments, except for those in \Cref{sec: static bimodal example}, we used 
regularization parameters $\lambda_T = \lambda_\Psi = 0.1$ and tuned other hyperparameters 
using the \texttt{SMAC3} package \cite{lindauer2022smac3}.
\editend

\subsection{Utilizing an EnKF-based reference measure}\label{sec: EnKF integration}  
\editstart
Below we propose a further extenstion of OTF that uses an additional EnKF layer to improve the approximation 
power of OT maps in nearly Gaussian settings. Consider the $t$-th step of the filter, where we aim to find the conditional OT map $T_t$ from  $\eta^t(y,v)= \nu^t_\Y(y)\eta^t_\U(v)$ to  $\nu^t(y,u)=h(y\mid u)\eta^t_\U(u)$. We change the source distribution $\eta^t$ with the distribution resulting from an EnKF approximation  defined according to
$\eta^t_\text{EnKF}(y,w) = \nu^t_\Y(y) \eta^t_{\text{EnKF}}(w\mid y)$.
Here $\eta^t_{\text{EnKF}}(w\mid y)$ is the conditional  kernel associated with the EnKF stochastic map
$w= \overline v + K^t(y - \overline y)$ where we assume 
$(\overline y,\overline v) \sim h(\overline y\mid \overline v)\eta^t_\U(\overline v)\,,$ and $K^t := \text{Cov}(\overline v,\overline y)\text{Cov}(\overline y)^{-1}\,.$  

Solving the optimization problem~\eqref{eq:phi-dagger-def} with $\eta^t_\text{EnKF}$ as the source and $\nu^t$ as the target distribution, concludes the map $T_t$ such that $T_t(y,\cdot)\# \eta^t_\text{EnKF}(\cdot \mid y) = \nu^t(\cdot\mid y)$ a.e. $y\in \Y$. Our experiments below show that, as expected, this modified OTF algorithm performs well whenever 
$\nu^t$ is nearly Gaussian since we expect $T_t$ to be close to the identity as the EnKF update would have 
accounted for most of the transport. 
We equip our OTF algorithm with the EnKF-based reference measure through minor adjustments to the optimization problem formulation~\eqref{eq:empirical-optimization}.
In particular, we replace samples $(y^t_i,v^t_i)\sim \eta^t$ in the objective function~\eqref{eq:empirical_loss_time}, with samples $(y^t_i,w^t_i) \sim\eta^t_\text{EnKF}$. These new samples are obtained according to $w^t_i = v^t_i + \wh{K}^t(y^t_i - \overline y^t_i)$, where $(y^t_i,v^t_i) \sim \eta^t$, and $\overline y^t_i \sim h(\cdot \mid v^t_i)$. The matrix $\widehat{K}^t:=\widehat{\text{Cov}}(\overline v,\overline y)\widehat{\text{Cov}}(\overline y)^{-1}$ is the 
empirical approximation to $K^t$. Finally, we replace the objective~\eqref{eq:empirical_loss_time} with 
\begin{equation}\label{eq:OT_EnKF_layer_loss}
    J_{(N)}^t\bigl(\psi,\widetilde{T}\bigr)
    \;:=\; \frac{1}{N} \sum_{i=1}^N  \Bigl[\,\psi\bigl(y_i^t,u_i^t\bigr)\;+\;\tfrac12\bigl\|\widetilde{T}\bigl(y_i^t,w_i^t\bigr)\bigr\|^2 
    \;-\; \psi\Bigl(y_i^t,\;\widetilde{T}\bigl(y_i^t,w_i^t\bigr)\;+\;w_i^t\Bigr)\Bigr],
\end{equation}
since the transport map is now parameterized as a perturbation of the identity $T = \text{Id} + \widetilde{T}$.
We highlight that if the user decides to include the regularization $R^t_T$, then this quantitiy should be 
applied to the whole map $T$ and not just $\widetilde{T}$. Finally, since $\widetilde{T}$ is a perturbation of 
the identity we often initialize our training/gradient descent algorithms at $\widetilde{T} = 0$. 
The corresponding implementation details are summarized in \Cref{alg:otf}. 
\editend

\subsection{A Bimodal static example}\label{sec: static bimodal example}
In this experiment, we investigate the impact of the regularization terms introduced in~\eqref{eq: regularization terms}. 
\editstart
We consider a static example without time stepping to understand the efficacy of OTF for one-step conditioning.
We begin with a one-dimensional example where we compare the resulting OT potentials $\overline{\psi},\overline{T}$ under the model
\begin{align} \label{example:squared}
    Y=\frac{1}{2}U\odot U + \sigma W, \qquad W \sim N(0,I_n),
\end{align}
 with the observation $Y=1$.
Here $\odot$ denotes the element-wise product and $\sigma^2=10^{-2}$.
\editend

The OT map $T$ satisfying $T(y,\cdot)\#\eta_\U=\nu(\cdot\mid y)$ admits a closed-form solution in the one-dimensional setting via the composition of cumulative distribution functions (CDFs) \cite[Ch.~2]{santambrogio2015optimal} which 
gives the closed form expression $\psi(y, u)= \int_{-\infty}^u (u'-T^{-1}(y,u'))du'$ with $T(y,u) = F^{-1}_{\nu(\cdot\mid y)} \circ F_{\eta_\U}(u)$ and $F_{\eta_\U}, F_{\nu(\cdot\mid Y =1)}$ denoting the CDFs of $\eta_\U$
and $\nu(\cdot \mid Y =1)$ respectively.

The numerical results for this experiment are presented in~\Cref{fig:bimodal}, employing the exponential linear unit (ELU) activation function for both  $g_T$ and $g_\psi$ with a hyperparameter  $\alpha=0.01$, and using a fixed number of particles  $N=1000$. 
The figure depicts results for both the unregularized and regularized objectives, where $\lambda_T=\lambda_\psi =\lambda=0,10^{-2},10^{-1}, 1$.  The left panel illustrates the kernel density estimator of the transported particles in comparison with the exact posterior distribution $\nu(\cdot\mid Y=1)$. The middle and right panels display the exact and approximate values of $\psi$ and $T$, respectively. The results indicate that as the regularization parameter $\lambda$ increases, the function $\frac{1}{2}\|\cdot\|_\U^2-\psi(Y=1,\cdot)$ becomes more convex, and the transport map $T$ 
approaches monotonicity 
at the cost of less accurate conditioning.

\editstart
Next we consider \eqref{example:squared} in higher dimensions where we expect the posterior to have 
exponentially many modes. Note that in this case we no longer have access to analytic forms of the 
true map but we can obtain an accurate posterior using the SIR algorithm with a very large number of particles. 
\editend
A quantitative inspection of the accuracy of various filters is presented in~\Cref{fig:bimodal_vs_dim_time_particles}, where results are shown as a function of the dimension $n$ and the number of particles $N$ over $10$ independent simulations. The left panel presents the 
\editstart
Average Marginal Wasserstein-2 distance, denoted $\distance$,  defined as the average of the one-dimensional $W_2$ distances 
between the empirical particle distributions for each algorithm and the corresponding true posterior marginal as a function of $n$ and for a fixed number of particles $N=5000$. We chose $\distance$ since computing the standard $W_2$ distance suffers 
from numerical errors in high dimensions that pollute our results.
The second panel 
reports the corresponding computational time as a function of dimension. These results demonstrate that both regularized and unregularized OTF approaches yield a superior approximation of the true multi-modal posterior, albeit at an increased computational cost; see~\cite{al-jarrah2024data,al-jarrah2025fast} for recent ideas on how to reduce the 
computational burden of OTF.
Finally, the right two panels illustrate the $\distance$ distance as a function of the number of particles $N$ for fixed dimensions $n=2,10,$ respectively. The figure  suggests that while OTF and EnKF scale better with dimension compared to SIR method, OTF is able to achieve  more accurate approximations. Although it still deteriorates in 
higher dimensions. This failure mode is depicted in
\Cref{sm:fig:marginal_densities} where we show the marginal densities produced by OTF method at $\dim = 20$ and $\dim = 30$, illustrating how the algorithm begins to miss 
certain posterior modes in higher dimensions.
\editend

\begin{figure}[htp]
    \centering
    \begin{overpic}[width=0.62\linewidth,trim={35 0 0 0},clip]{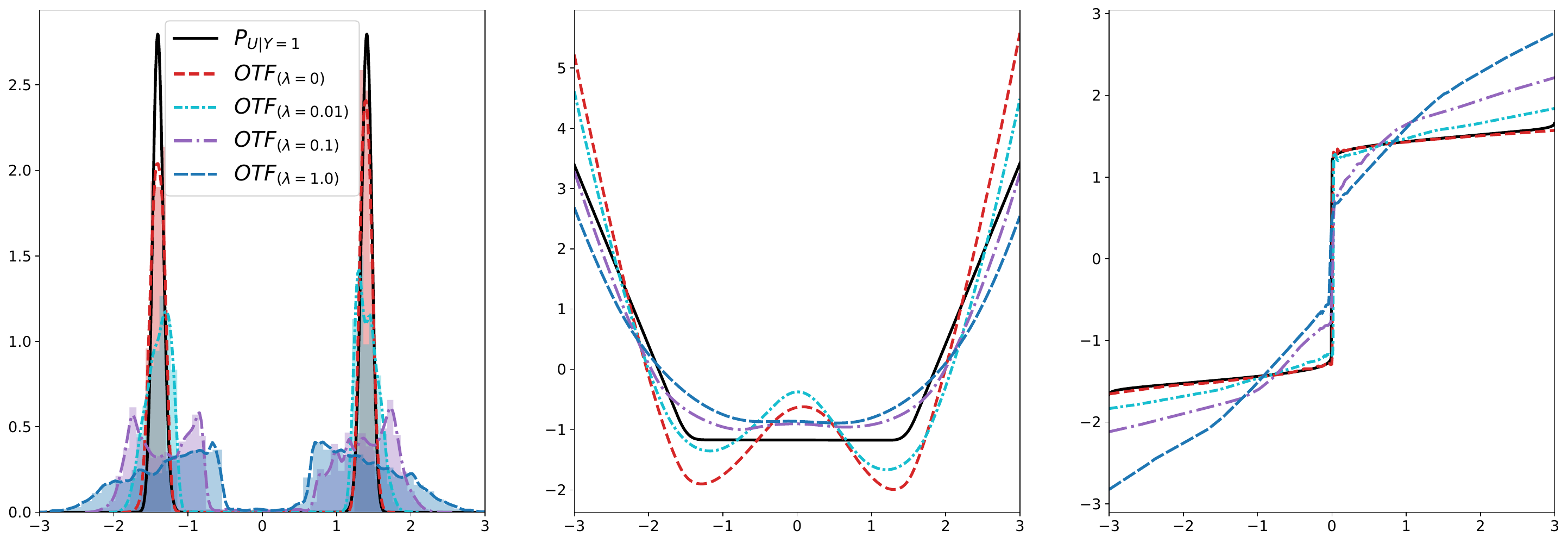}
    \put(-1.5,17.5){\makebox(0,0)[c]{\rotatebox{90}{\scriptsize Density}}}
    \put(14.6,-0.5){\makebox(0,0)[c]{\scriptsize U}}

    \put(31.5,17.5){\makebox(0,0)[c]{\rotatebox{90}{\scriptsize $\frac{1}{2}\|U\|^2-\psi(Y=1,U)$}}}
    \put(49.6,-0.5){\makebox(0,0)[c]{\scriptsize U}}

    \put(66.5,17.5){\makebox(0,0)[c]{\rotatebox{90}{\scriptsize $T(Y=1,U)$}}}
    \put(84.55,-0.5){\makebox(0,0)[c]{\scriptsize U}}
  \end{overpic}
    \caption{The left figure shows the kernel density estimate function of the transported particles in comparison with the exact $\nu(\cdot\mid Y=1)$. The middle and right figures show the exact and approximate values of $\psi$ and $T$, respectively.}
    \label{fig:bimodal}
\end{figure}

 \begin{figure}[htp]
      \begin{subfigure}[b]{0.23\textwidth}
         \centering
         \begin{overpic}[width=.9\linewidth]{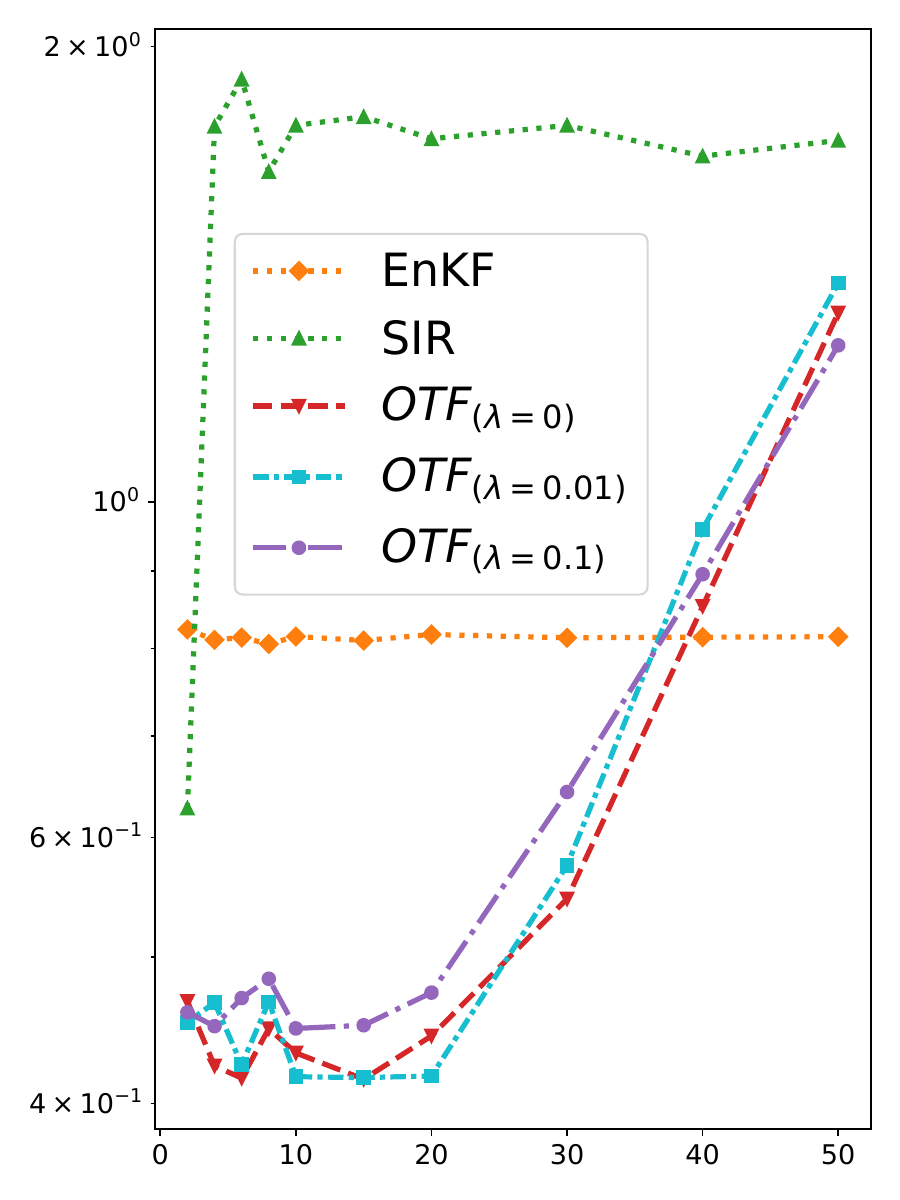}
            \put(43,-1){\makebox(0,0)[c]{\footnotesize dim}}
            \put(-2,50){\makebox(0,0)[c]{\rotatebox{90}{\footnotesize $\distance$}}}
          \end{overpic}
     \end{subfigure}
     \hfill
     \begin{subfigure}[b]{0.23\textwidth}
         \centering
         \begin{overpic}[width=.9\linewidth]{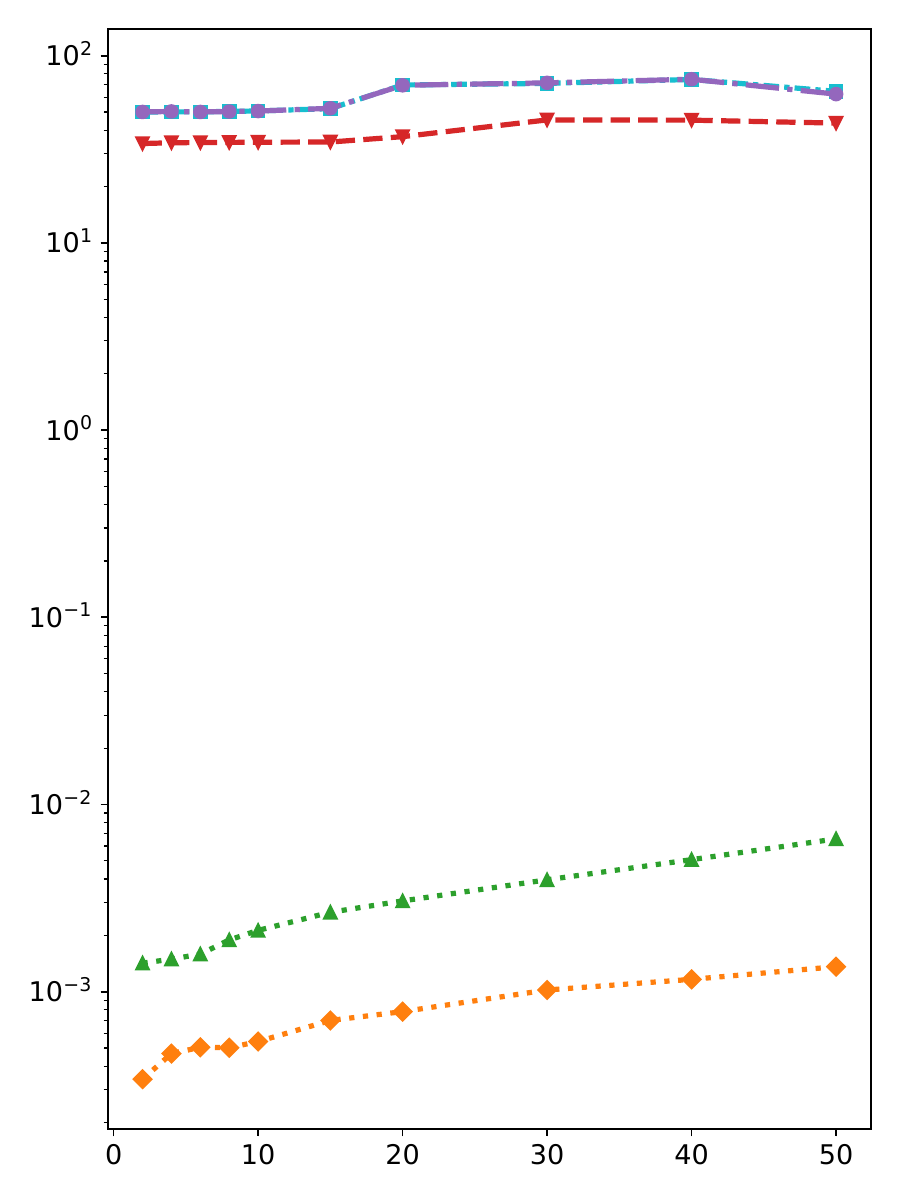}
            \put(43,-1){\makebox(0,0)[c]{\footnotesize dim}}
            \put(-2,50){\makebox(0,0)[c]{\rotatebox{90}{\footnotesize computational time}}}
          \end{overpic}
     \end{subfigure}
     \hfill
     \begin{subfigure}[b]{0.23\textwidth}
         \centering
         \begin{overpic}[width=.9\linewidth]{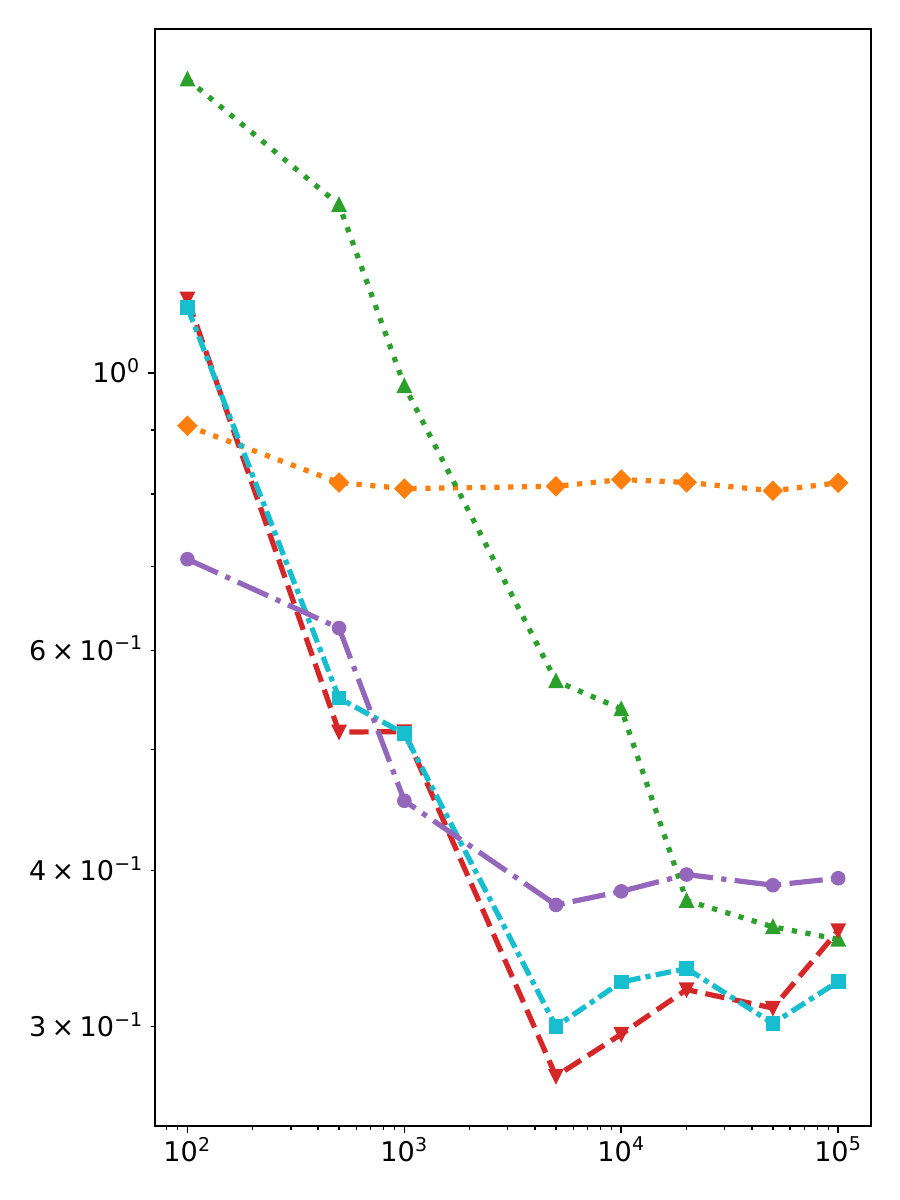}
           \put(43,-1){\makebox(0,0)[c]{\footnotesize \# of particles}}
           \put(-2,50){\makebox(0,0)[c]{\rotatebox{90}{\footnotesize $\distance$}}}
           \put(45,80){\makebox(0,0)[c]{\footnotesize dim = 2}}
          \end{overpic}
     \end{subfigure}
     \hfill
     \begin{subfigure}[b]{0.23\textwidth}
         \centering
         \begin{overpic}[width=.9\linewidth]{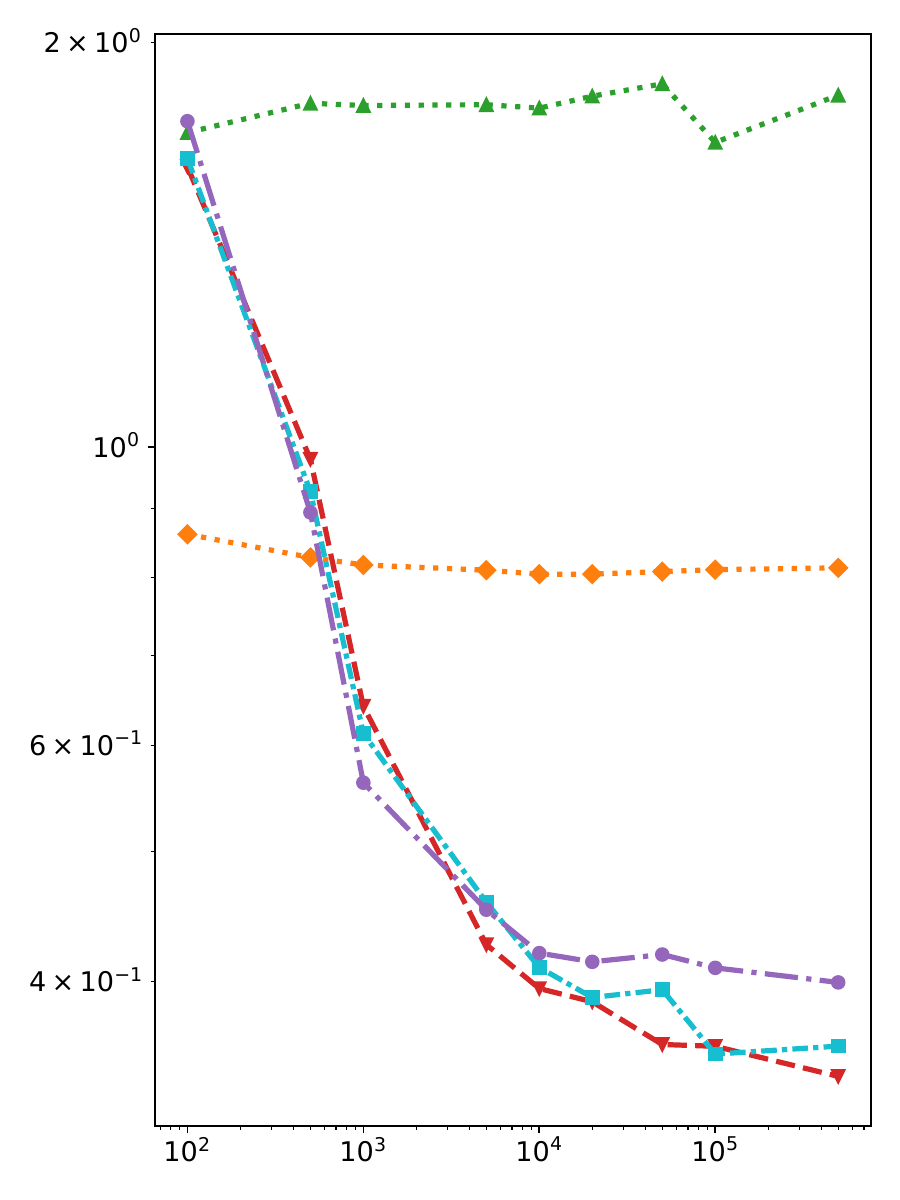}
            \put(43,-1){\makebox(0,0)[c]{\footnotesize \# of particles}}
            \put(-2,50){\makebox(0,0)[c]{\rotatebox{90}{\footnotesize $\distance$}}}
            \put(45,80){\makebox(0,0)[c]{\footnotesize dim = 10}}
          \end{overpic}
     \end{subfigure}

    \caption{The left figure shows the $\distance$ distance as a function of dimension $n$ for a fixed number of particles $N=5000$. The second figure shows the corresponding computational time as a function of dimension. The right two figures show the $\distance$ distance as a function of the number of particles $N$ for a fixed dimension $n=2,10$, respectively.  }
    \label{fig:bimodal_vs_dim_time_particles}
\end{figure}

\subsection{Lorenz 63}\label{sec: L63}

In this numerical experiment, we compare the performance of the same filters as in \Cref{sec: static bimodal example} on the three-dimensional Lorenz 63 model. The state $U_t \in \mathbb{R}^3$ and the observation $Y_t \in \mathbb{R}$ are defined such that $Y_t$ provides noisy measurements of the third component of $U_t$. This leads to a bimodal distribution in the first two state components due to symmetries in the Lorenz model. 

To generate the true state trajectories we used forward Euler discretizations of the Lorenz 63 system 
(see \Cref{sm:Lorenz-models} for details and parameter values)
with a time step of $\Delta t = 0.01$, starting from the initial condition 
$U_0 \sim \mathcal{N}\bigl(5 \cdot \mathbf{1}_3, \mathbf{I}_3\bigr),$
and evolve it without any process noise. 
\editstart
The observation model is  $Y_t = U_t(3) + W_t,$ 
with $W_t \sim \mathcal{N}\bigl(0, 10\bigr)\footnotemark$.
\editend
\footnotetext{Here $U_t(k)$ denotes the $k$-th state/coordinate of the vector $U_t$.}
In all experiments, we used $N = 1000$ particles. The particle forecast update incorporates the process noise of variance $\sigma^2=0.1$, i.e., 
\[
    U_{t+1} = a_{\textnormal{L63}}(U_t) +\sigma V_t, 
    \quad V_t \sim \mathcal{N}\bigl(0, \mathbf{I}_3\bigr), \quad U_0 \sim \mathcal{N}\bigl(0, 10^2 \cdot \mathbf{I}_3\bigr)\, ,
\]
where $a_{\textnormal{L63}}(\cdot)$ denotes the discrete dynamics of the Lorenz 63 equations.

\editstart
\Cref{fig:L63_w2_states} depicts the resulting particle trajectory distributions for various filters comparing them 
to a "true distribution" computed using SIR with $10^5$ particles.
Most notably, the OTF methods successfully capture both modes in the first two unobserved states, while EnKF and SIR fail.
For a more quantitative assessment, we computed the $\distance$ distance between the true distribution and 
the filtered results averaged over 10 independent simulations. The right panel of~\Cref{fig:L63_w2_states} presents these $\distance$ distances, indicating that the OTF methods yield a more accurate approximation. Moreover, among the various OTF methods, the unregularized version shows the largest improvement in capturing the underlying bimodality compared to alternative strategies. 

Recall that our fast rate analyses from \Cref{Sec:Error analysis for OTF} largely relied on controlled Poincar\'e
constants of the filtering distributions which we showed can be controlled under stringent conditions such as 
uniform strong log-concavity in time and for quadratic maps; recall \Cref{sec:example-filtering}. Those conditions 
are clearly violated in this experiment and so it is interesting to empirically verify whether the 
Poincar\'e constants remain controlled for our Lorenz 63 experiment. Towards this goal 
we report two diagnostic quantities for the Poincar\'e constants $C_{\textnormal{PI}}^\mu$ of a measure $\mu$
(we will take $\mu$ to be the prior and posterior at each time step) computed from $n$ i.i.d.\ samples: (i) a lower bound 
$\textnormal{LB}(\mu) \coloneqq \lambda_{\max}(\widehat\Sigma)$, where $\widehat\Sigma$ is the empirical covariance matrix of the samples; and (ii) a kernel-based estimator $\textnormal{RKHS}(\mu)$ that maximizes the same Rayleigh quotient over a regularized RKHS ball and is a consistent estimator of $C_{\textnormal{PI}}^\mu$ itself as $n\to\infty$ \cite{pillaud2020poincare}. We applied the RKHS estimator with a standard RBF kernel at every time step, averaged over the same $10$ independent simulations as \Cref{fig:L63_w2_states}.

The results are shown in \Cref{fig:poincare_l63} where we observed that both the $\textnormal{LB}$ and $\textnormal{RKHS}$
estimates remained bounded in time but exhibit oscillations of order 10 with large constants when the 
distributions become clearly bimodal and smaller values when the modes converge. We consider these results as 
suggestive of a time-uniform Poincar\'e constant for the Lorenz 63 system which implies that our analysis may be 
applicable here. We leave a rigorous, non-asymptotic justification of this empirical observation to future work.

\editend

\begin{figure}[htp]
    \centering
     \begin{subfigure}[b]{0.23\textwidth}
         \centering
         \begin{overpic}[width=.9\linewidth]{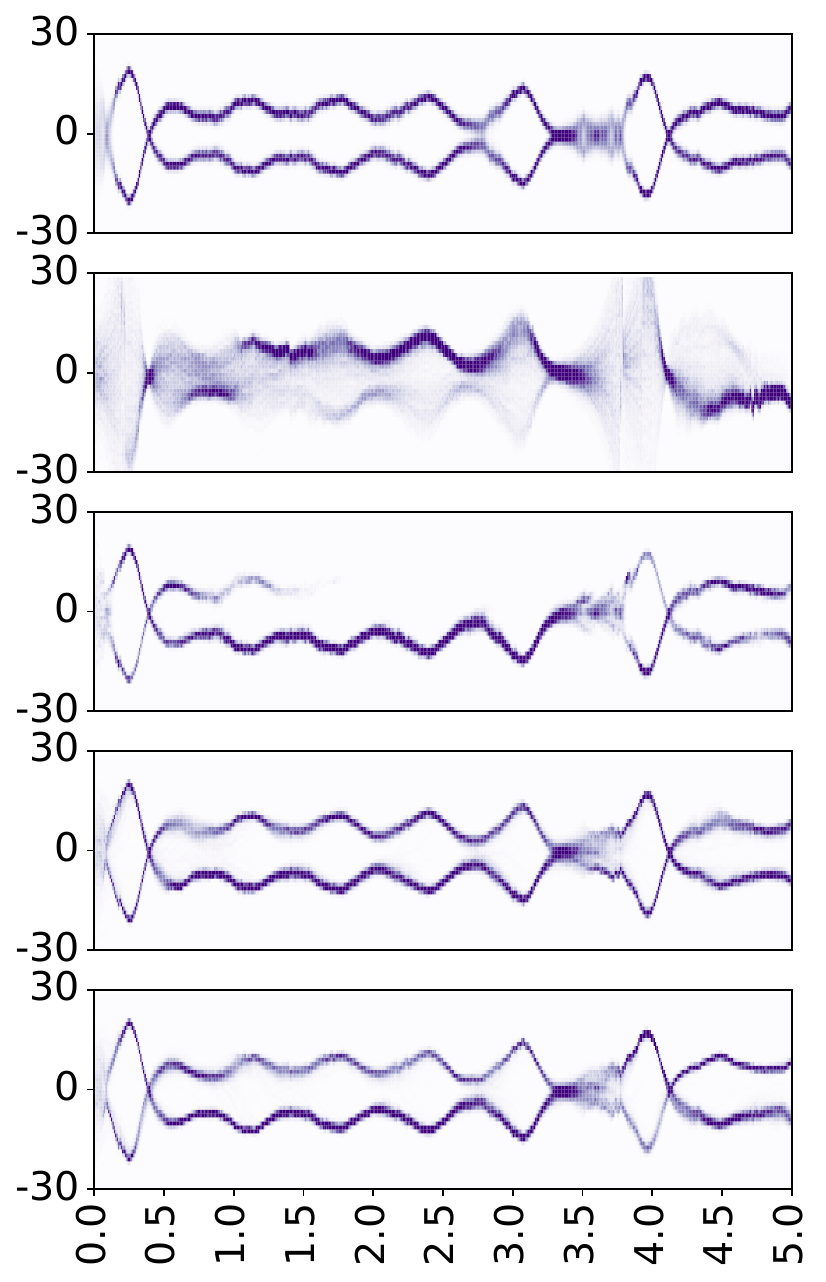}
            \put(-1,15){\makebox(0,0)[c]{\rotatebox{90}{\footnotesize OTF}}}
            \put(-1,33){\makebox(0,0)[c]{\rotatebox{90}{\footnotesize OTF}}}
            \put(-1,53){\makebox(0,0)[c]{\rotatebox{90}{\footnotesize SIR}}}
            \put(-1,72){\makebox(0,0)[c]{\rotatebox{90}{\footnotesize EnKF}}}
            \put(-1,90){\makebox(0,0)[c]{\rotatebox{90}{\footnotesize True}}}
            \put(35,-1){\makebox(0,0)[c]{\footnotesize time}}

            \put(38,21){\makebox(0,0)[c]{\footnotesize$\lambda=0.1$}}
            \put(38,39){\makebox(0,0)[c]{\footnotesize$\lambda=0$}}
          \end{overpic}
         \caption{$U_t(1)$}
         \label{fig:L63_x1}
     \end{subfigure}
     \hfill
     \begin{subfigure}[b]{0.23\textwidth}
         \centering
         \begin{overpic}[width=.9\linewidth]{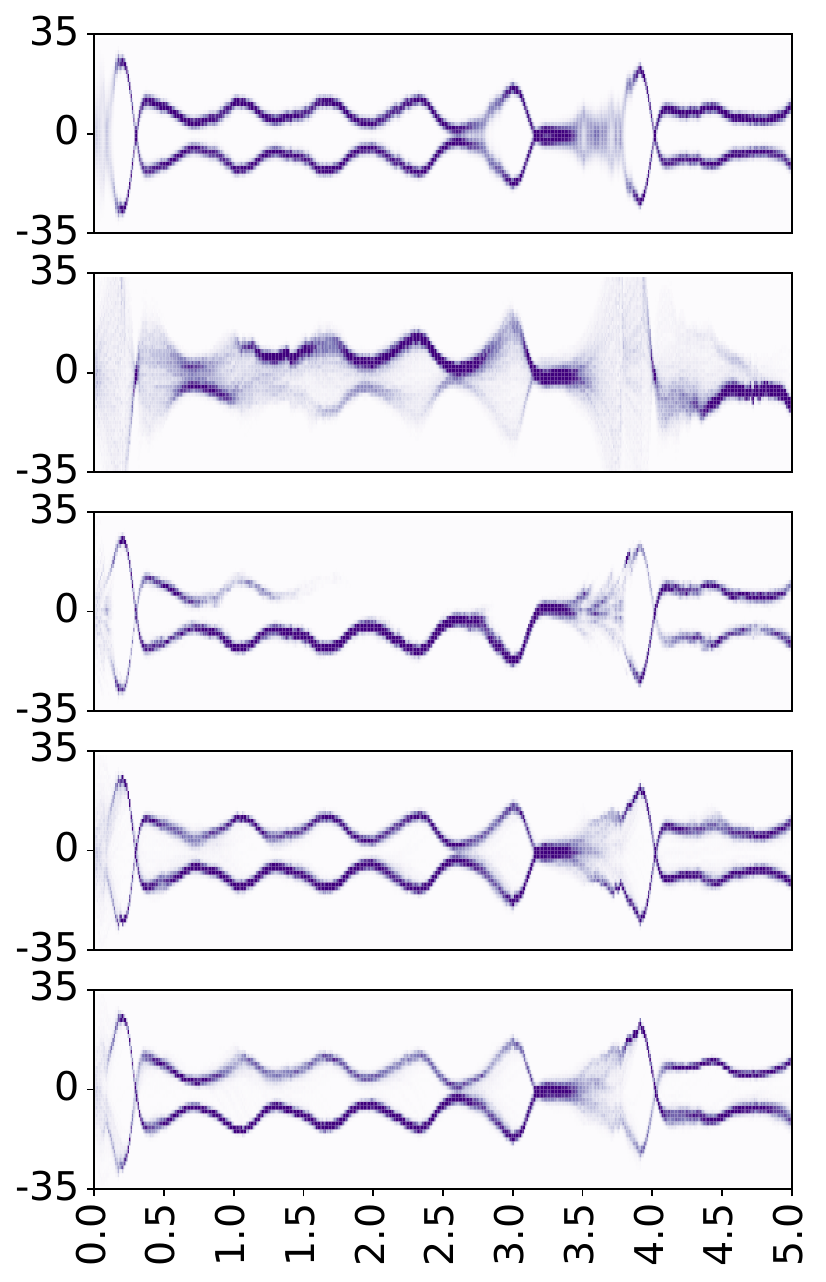}
            \put(35,-1){\makebox(0,0)[c]{\footnotesize time}}
          \end{overpic}
         \caption{$U_t(2)$}
         \label{fig:L63_x2}
     \end{subfigure}
     \hfill
     \begin{subfigure}[b]{0.23\textwidth}
         \centering
         \begin{overpic}[width=.9\linewidth]{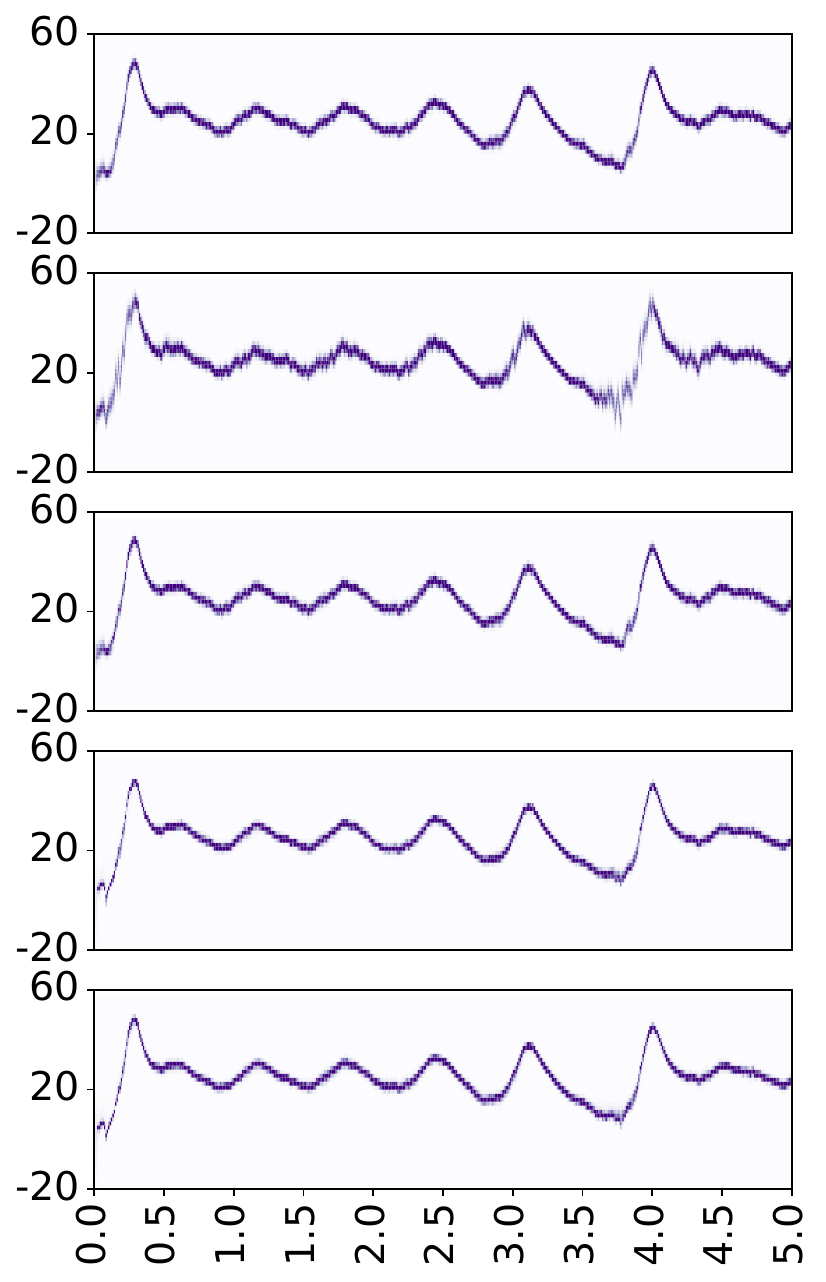}
            \put(35,-1){\makebox(0,0)[c]{\footnotesize time}}
          \end{overpic}
         \caption{$U_t(3)$}
         \label{fig:L63_x3}
     \end{subfigure}
     \hfill
     \begin{subfigure}[b]{0.23\textwidth}
         \centering
         \begin{overpic}[width=.9\linewidth]{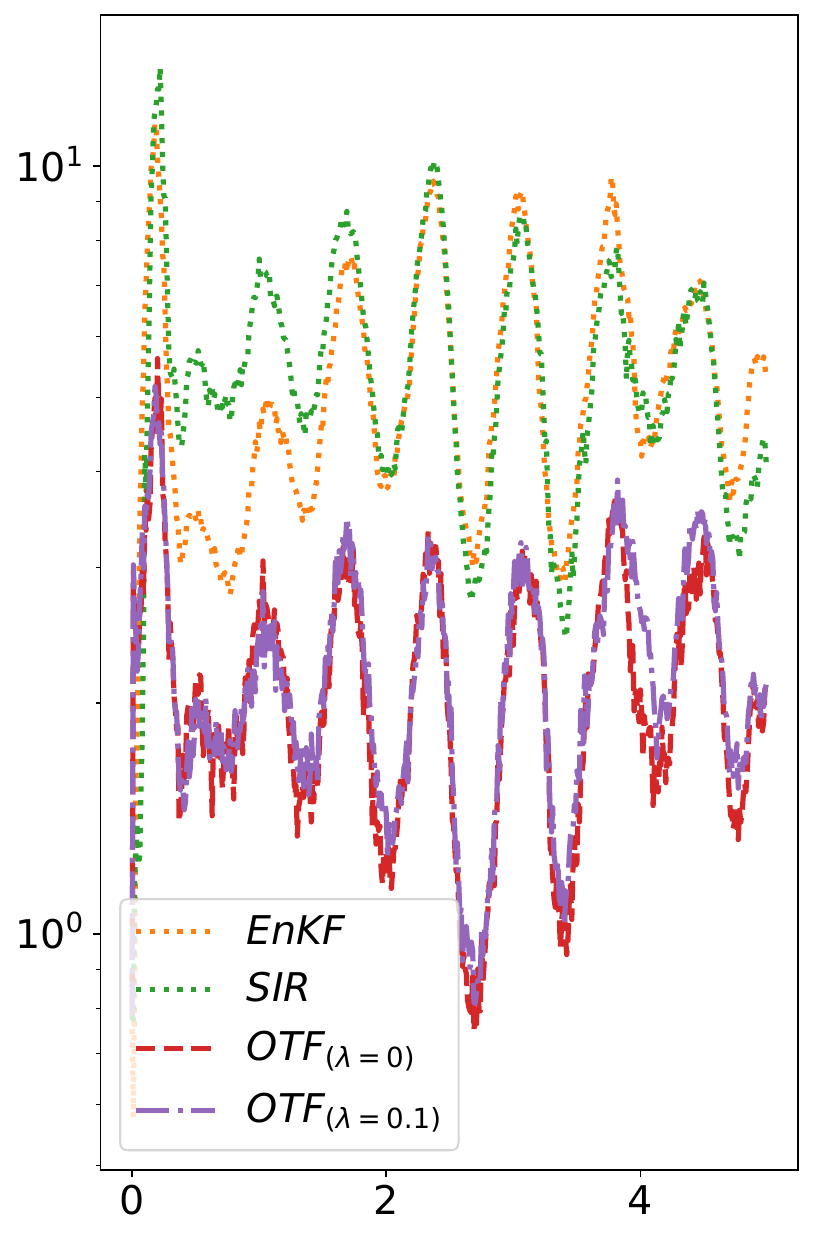}
    \put(-2,50){\makebox(0,0)[c]{\rotatebox{90}{\footnotesize $\distance$}}}
    \put(35,-1){\makebox(0,0)[c]{\footnotesize time}}
  \end{overpic}
         \caption{$\distance$ distance.}
         \label{fig:W_2}
     \end{subfigure}
    \caption{ Numerical results for the Lorenz 63 example. The left three panels illustrate the true particle trajectory distribution and the corresponding distributions produced by each method. The right panel presents $\distance$ distances between each method and the true distribution over $10$ independent simulations.}
    \label{fig:L63_w2_states}
\end{figure}

\begin{figure}[htp]
    \centering
    \begin{overpic}[width=0.35\linewidth,trim=0 0 0 0,clip]{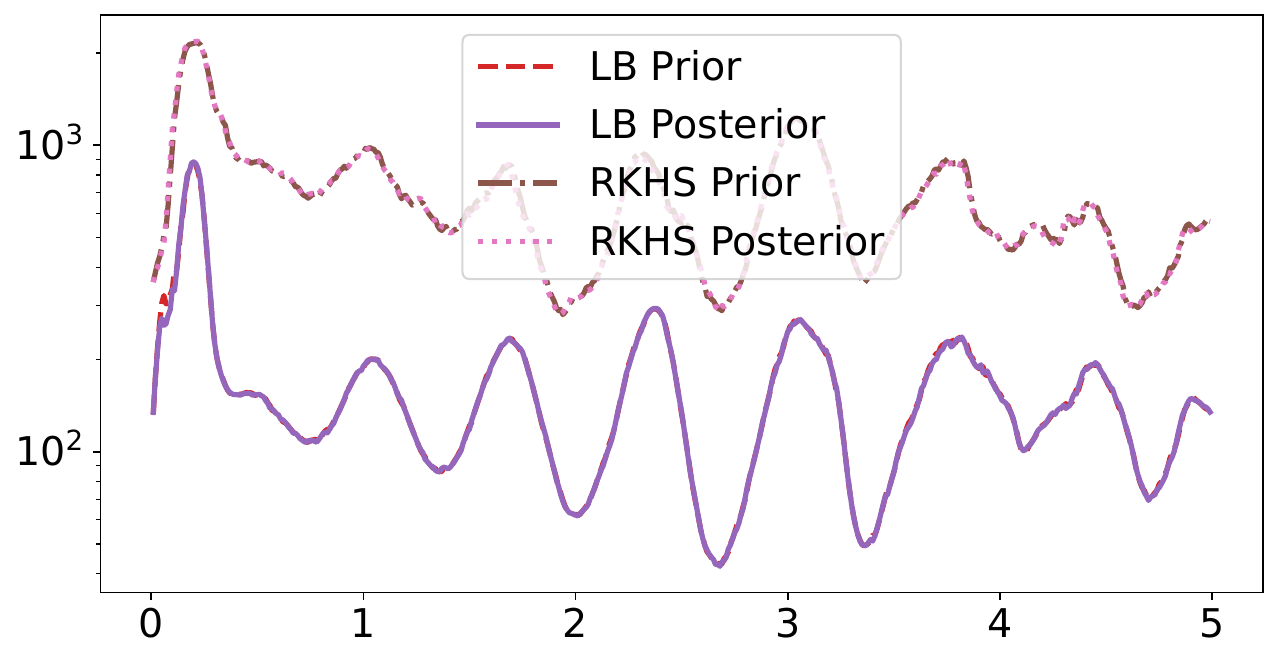}
    \put(-5,27){\makebox(0,0)[c]{\rotatebox{90}{\scriptsize Emp. PI const.}}}
    \put(50,-2){\makebox(0,0)[c]{\scriptsize time}}
    \end{overpic}
    \caption{Empirical Poincar\'e diagnostics for the prior and posteriors  for the Lorenz 63 benchmark averaged over $10$ independent simulations. Solid lines show the lower bound $\textnormal{LB}$ while dashed lines 
    show the RKHS estimator.}
    \label{fig:poincare_l63}
\end{figure}

\subsection{Lorenz 96}\label{sec:L96}
\editstart
For our final set of experiments we considered the $9$-dimensional Lorenz 96 model; see \Cref{sm:Lorenz-models}.
Here the state $U_t \in \Re^{10}$ evolves according the dynamics 

\[
    U_{t+1} = a_{\text{L96}}(U_t) + \sigma V_t, \quad U_0 \;\sim\; \mathcal{N}\bigl(10 \cdot \mathbf{1}_9,\; 10 \cdot \mathbf{I}_9\bigr),
\]
where $a_{\text{L96}}$ is the  Runge-Kutta (RK4) discretization of the Lorenz 96 equations with the time step of $0.01$ 
and $V_t \sim \mathcal{N}\bigl(\mathbf{0}, 0.01 \cdot \mathbf{I}_9\bigr)$ is a normal process noise. The observation $Y_t \in \mathbb{R}^{3}$ consists of noisy measurements of a subset of the state coordinates, i.e., $Y_t 
= (U_t(1), U_t(4), U_t(7)) + \mathcal{N}(\mathbf{0}, 0.1 \cdot \mathbf{I}_3)$. We used $N= 1000$ particles 
for all filtering algorithms and compared our results to SIR with a large number of particles as the ground truth.

\Cref{fig:L96_mse_states} summarizes our findings. 
The left three panels depict the empirical trajectoris of one observed component $(U_t(1))$ and two unobserved components $(U_t(2), U_t(3))$. Since the filtering distributions are 
unimodal the EnKF algorithm demonstrates good performance as expected. At the same time the EnKF-augmented-OTF 
method appears to track the performance of EnKF while the standard OTF without the EnKF layer has worse performance, 
indicating that the additional EnKF layer has meaningful contributions.
This is also evident in the right panel of~\Cref{fig:L96_mse_states} where we report mean-squared error (MSE) of state estimation, averaged over $10$ independent simulations. 

\editend

\begin{figure}[htp]
    \centering
     \begin{subfigure}[b]{0.23\textwidth}
         \centering
         \begin{overpic}[width=.9\linewidth]
         {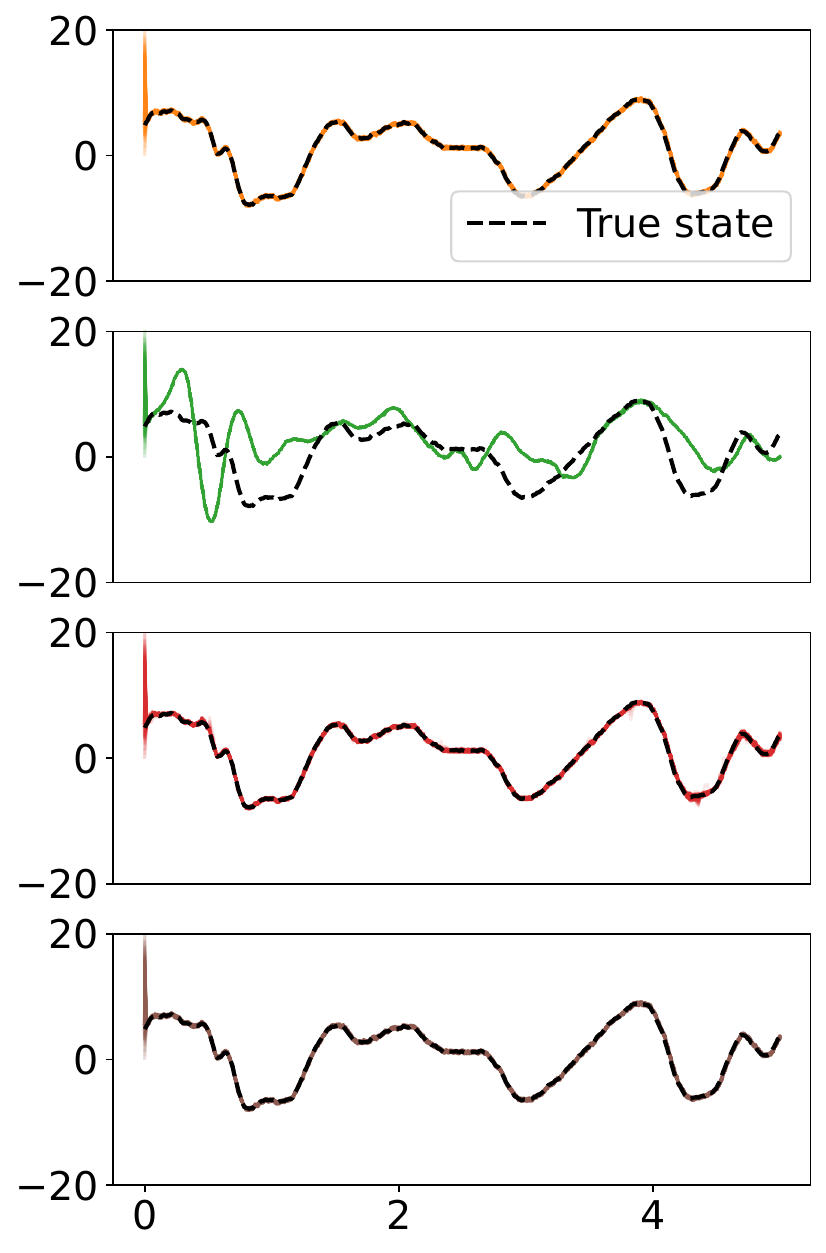}
        \put(-1,15){\makebox(0,0)[c]{\rotatebox{90}{\footnotesize OTF}}}
        
        \put(-1,40){\makebox(0,0)[c]{\rotatebox{90}{\footnotesize OTF}}}
        \put(-1,65){\makebox(0,0)[c]{\rotatebox{90}{\footnotesize SIR}}}
        \put(-1,90){\makebox(0,0)[c]{\rotatebox{90}{\footnotesize EnKF}}}
        \put(40,-1){\makebox(0,0)[c]{\footnotesize time}}

            \put(35,22){\makebox(0,0)[c]{\tiny with EnKF layer}}
            \put(35,46){\makebox(0,0)[c]{\tiny without EnKF layer}}
          \end{overpic}
         \caption{$U_t(1)$}
         \label{fig:L96_x1}
     \end{subfigure}
     \hfill
     \begin{subfigure}[b]{0.23\textwidth}
         \centering
         \begin{overpic}[width=.9\linewidth]
         {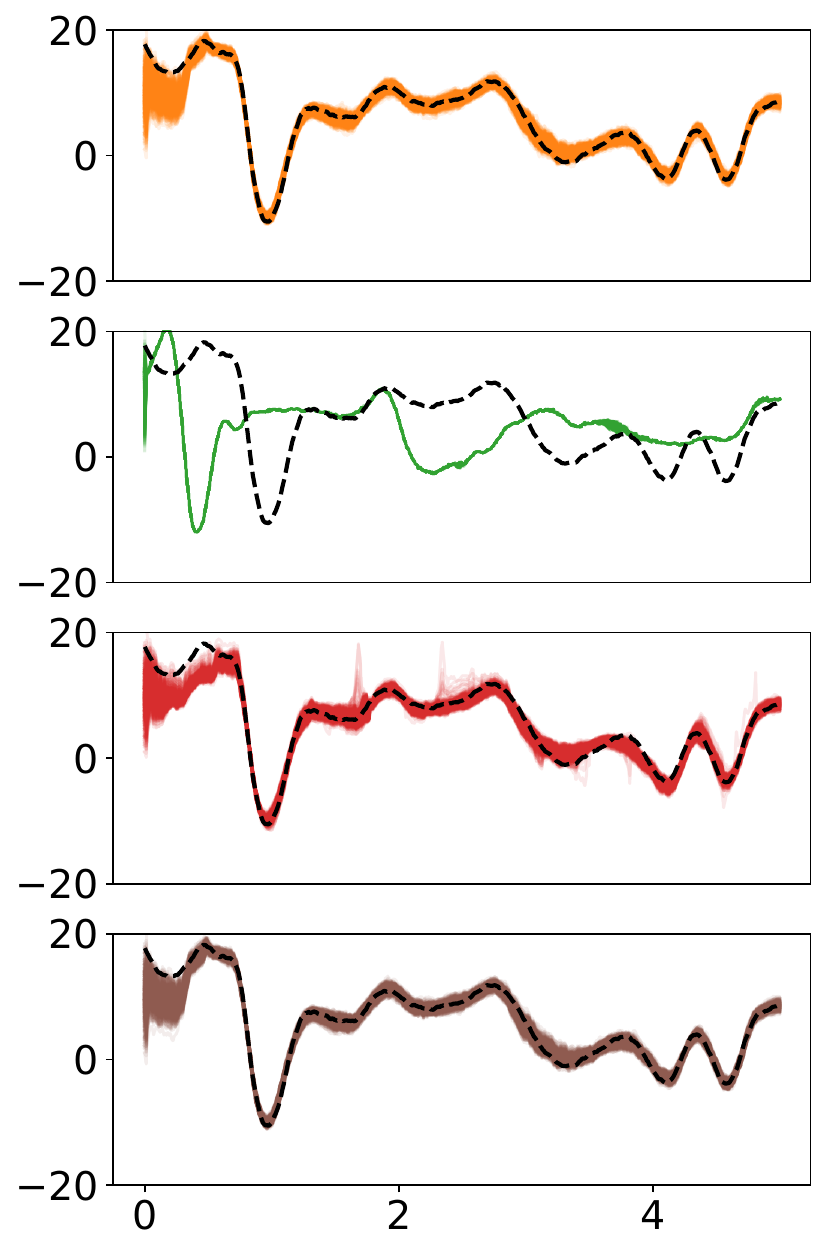}
            \put(40,-1){\makebox(0,0)[c]{\footnotesize time}}
          \end{overpic}
         \caption{$U_t(2)$}
         \label{fig:L96_x2}
     \end{subfigure}
     \hfill
     \begin{subfigure}[b]{0.23\textwidth}
         \centering
         \begin{overpic}[width=.9\linewidth]
         {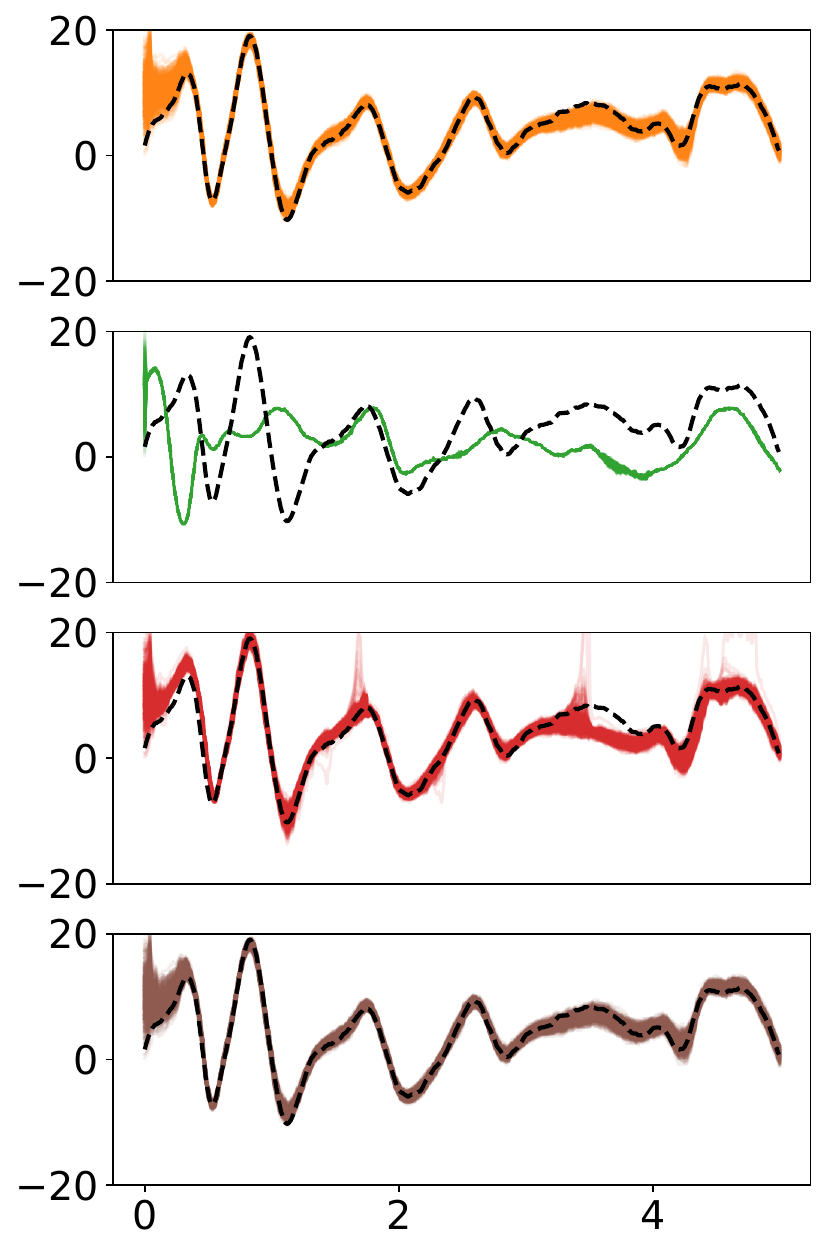}
            \put(40,-1){\makebox(0,0)[c]{\footnotesize time}}
          \end{overpic}
         \caption{$U_t(3)$}
         \label{fig:L96_x3}
     \end{subfigure}
     \hfill
     \begin{subfigure}[b]{0.23\textwidth}
         \centering
         \begin{overpic}[width=.9\linewidth]{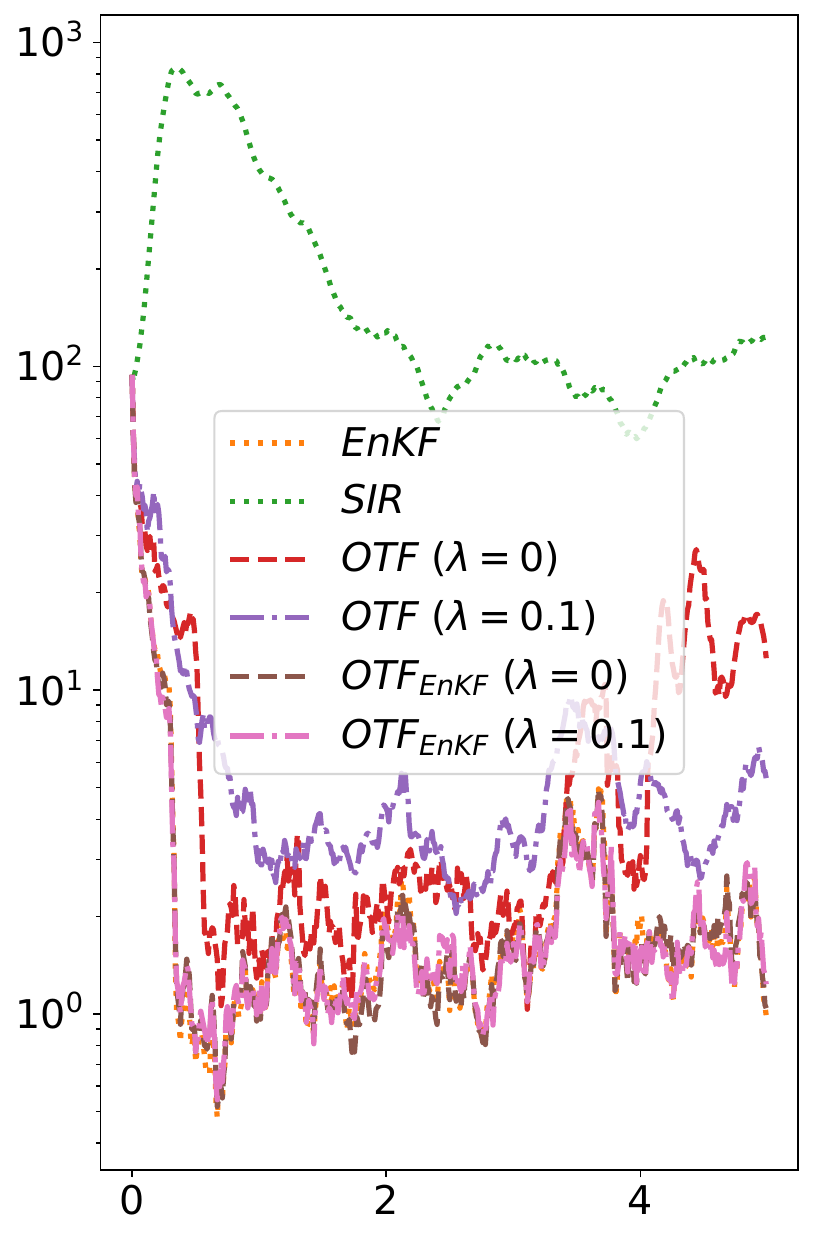}
    \put(-2,50){\makebox(0,0)[c]{\rotatebox{90}{\footnotesize MSE}}}
    \put(40,-1){\makebox(0,0)[c]{\footnotesize time}}
  \end{overpic}
         \caption{MSE}
         \label{fig:L96_mse}
     \end{subfigure}
     
    \caption{Numerical results for the Lorenz 96 example. The left three panels depict the trajectory of one observed component $(U_t(1))$ and two unobserved components $(U_t(2), U_t(3))$ of the true state, along with the corresponding particle trajectories for each filtering method. The right panel displays the MSE of state estimation, averaged over 10 independent simulations. }
    \label{fig:L96_mse_states}
\end{figure}

\section{Conclusion}\label{Sec:Conclusions}
We outlined a quantitative analysis of conditional OT maps for conditional simulation, as well as their extension to filtering and data assimilation via the OTF algorithm. In our theoretical analysis of conditional OT we outlined a ``slow rate" result with a simple proof technique and minimal assumptions on the reference and target measures, followed by a stronger ``fast rate" result that uses stronger assumptions regarding smoothness and Poincar\'e inequalities of the reference and target conditionals. We then extended these results to OTF algorithms that compute conditional OT maps from particles to approximately perform the Bayesian updates. We showed that under appropriate assumptions, our bounds for conditional OT maps can be extended to a bound on the filtering errors of OTF. Finally, we presented some extensions to the OTF methodology, including an additional EnKF layer that helps with nearly Gaussian problems. 

While our theoretical analysis yields a first quantitative bound for conditional OT problems and the OTF algorithm, it leads us to a number of interesting open questions and future directions of research. Most notably, verifying some of our assumptions is challenging in practice and in particular, the Lipschitzness of the Brenier maps $\nabla_u \phi(y, \cdot)$ with respect to $y$ is challenging to establish and is currently open when $\Y$ is not compact. It is also interesting to try to relax some of our assumptions regarding Poincar\'e constants and smoothness of Brenier potentials in order to extend our fast rates to more practical problems. Another interesting direction of research is to incorporate uncertainties regarding the dynamic update $\mathcal{A}$. While we always assumed that this operator is exact, in most practical problems, we are unsure of the state dynamics and may even want to estimate it from empirical data.

\section*{Acknowledgments}
The authors are grateful to anonymous reviewers whose detailed comments and suggestions enabled significant
improvement of an earlier version of this article.

\bibliographystyle{siamplain}
\bibliography{references}

\appendix
\section{Examples of stable divergences}~\label{sec:divergences}

Below we list a few commonly used divergences $D$ which are stable in the sense of 
\eqref{eq: metric D stability} and for which~\Cref{lem: div between true and estimated posterior} is applicable. 
In the following $\U$ denotes a generic metric space.
\begin{itemize}
    \item Maximum mean discrepancy defined by a Lipschitz kernel $K:\U\times\U\mapsto\mathbb{R}$ with corresponding reproducing kernel Hilbert space (RKHS) $\mathcal{K}$ defined as
    \[\textnormal{MMD}_{K}(\mu,\mu'):=\left\|\int_\U K (u,\cdot)\d\mu(u)- \int_\U K (u,\cdot)\d\mu'(u)\right\|_{\mathcal{K}}\, .\]

    \item Wasserstein distance $\textnormal{W}_{p}$, for $p \in [1,2]$, defined as
    \[\textnormal{W}_{p}(\mu,\mu'):=\inf_{\pi\in\Pi(\mu,\mu')}\int_{\U\times\U}\|u-u'\|_\U^p\d \pi(u,u')\, .\]

    \item Dual bounded-Lipschitz  defined as
    \[D_{BL}(\mu,\mu'):=\sup_{g\in\text{Lip}(\U)} \sqrt{\mathbb{E}\left|\int g(u) \d\mu(u) - \int g(u') \d\mu(u')\right|^{2}}\, .\] 
\end{itemize}

\section{Proofs for \Cref{Sec:Error analysis for conditional OT}}
\label{app:proofs:conditionalOT}

In this section we collect the auxiliary theoretical results and the technical proofs of 
our main results from \Cref{Sec:Error analysis for conditional OT}. 
In \Cref{app:conditionalOT:preliminary-results} we give a series of 
preliminary lemmata concerning the dual of the conditional Kantorovich problem 
followed by the proof of \Cref{thm: slow rate} (slow rate) 
in \Cref{proof: thm: slow rate,proof: thm: slow rate unbounded extension} and the proof of \Cref{thm: fast rate} 
(fast rate) in \Cref{proof: thm: fast rate,proof: thm: fast rate unbounded extension}.

\subsection{Preliminary estimates}\label{app:conditionalOT:preliminary-results}
We present a fundamental lemma that serves as the basis for our subsequent error analysis. This result builds upon key insights from the prior works \cite{chewi2024statistical,divol2025optimal,hutter2019minimax}, which we generalize to the setting of conditional OT maps.  In order to state the result, define the pointwise and aggregated excess risks, respectively, as 
\begin{equation}\label{eq: centered SD}
    \cS^c(\phi, y):= \cS(\phi, y) - \cS(\phi^{\dagger}, y),\quad \text{and}\quad  \cS^c(\phi) \coloneqq \cS(\phi) - \cS(\phi^{\dag}),
\end{equation}
for any function $\phi$, where $\cS$ is defined according to~\eqref{eq: dual} and $\phi^\dagger$ is the unique minimizer of $\cS(\phi)$, whose existence is ensured in \Cref{prop: solvability of monge + conditional brenier}, i.e., the 
conditionally optimal Brenier potential function. Furthermore, recall that 
for $\nu_\Y$ a.e. $y$, $\nabla_{u} \phi^{\dagger}(y, \cdot) \# \eta_{\mathcal{U}} = \nu(\cdot\mid y)$.

\begin{lemma}[Map stability]
\label{cor: map stability}
 Suppose $\F$ satisfies \Cref{assump: slow rate}. Then for every $\phi \in \F$
 it holds that
     \[\frac{1}{2\betamax}\|\nabla_u\phi^{\dagger}-\nabla_u\phi\|_{L_{\eta}^2}^2\leq \cS^c(\phi) \leq
      \frac{1}{2\alphamin}\|\nabla_{u} \phi^{\dagger} - \nabla_{u} \phi\|_{L_{\eta}^2}^2.
      \]
\end{lemma}
Note that the above result does not require any smoothness or 
convexity assumptions on $\phi^\dagger$ but only on the 
approximation class $\F$.

    \begin{proof}
        This result is an extension of \cite[Prop. 2.1]{divol2025optimal}. The idea is to leverage the duality between convexity and smoothness. We begin with the upper bound.
    Recall that $\phi(y,\cdot)$ is $\alpha(y)$-strongly convex in $u$ iff $\phi^*(y,\cdot)$ is $\frac{1}{\alpha(y)}$-smooth in $u$. Hence, for any $z,w \in \mathbb{R}^d$,
\[\phi^*(y,w) \leq \phi^*(y,z) + \langle \nabla_{u} \phi^*(y,z) , w-z \rangle + \frac{1}{2\alpha(y)} \|z-w\|^2.\]
Set $z=\nabla_{u} \phi(y,u)$ and $w=\nabla_{u} \phi^{\dagger}(y,u)$. Using the fact that $\nabla_{u}\phi^*(y,\nabla_{u}\phi(y,u)) = u$, and $\langle u,\nabla_{u}\phi(y,u) \rangle = \phi(y,u) + \phi^*(y,\nabla_{u}\phi(y,u))$, we obtain:
\begin{align*}
    \phi^*(y,\nabla_{u}\phi^{\dagger}(y,u)) + \phi(y,u) \leq \langle u, \nabla_{u} \phi^{\dagger}(y,u) \rangle + \frac{1}{2\alpha(y)} \|\nabla_{u} \phi^{\dagger}(y,u) - \nabla_{u} \phi(y,u)\|^2.
\end{align*}
Integrating over $u \sim \eta_{\U}$, noting that $\nabla_{u}\phi^{\dagger}\#\eta_{\U} = \nu(\cdot\mid u)$, and recalling \eqref{eq: marginal dual} and \eqref{eq: dual}, we obtain the bound
\[\cS(\phi,y) \leq \int \langle u, \nabla_{u} \phi^{\dagger}(y,u) \rangle + \frac{1}{2\alpha(y)} \|\nabla_{u} \phi^{\dagger}(y,u) - \nabla_{u} \phi(y,u)\|^2 \, \d\eta_{\mathcal{U}}(u).\]
Now since $\mathcal{S}(\phi^{\dagger},y) = \int \langle u, \nabla_{u}\phi^{\dagger}(y,u)\rangle \d\eta_{\U}(u)$, the excess risk satisfies
\begin{align*}
    \cS^c(\phi,y) = \cS(\phi,y) - \cS(\phi^{\dagger},y) \leq \frac{1}{2\alpha(y)} \|\nabla_{u} \phi^{\dagger}(y,\cdot) - \nabla_{u} \phi(y,\cdot)\|^2_{L_{\eta_{\mathcal{U}}}^2}.
    \label{eq: map stability convexity bound}
\end{align*}
Integrating over $y \sim \nu_{\Y}$, and applying the inequality $\frac{1}{2\alpha(y)} \leq \frac{1}{2\alphamin}$ concludes the upper-bound. 
The proof for the lower-bound follows identically, using the  dual statement: $\phi(y,\cdot)$ is $\beta(y)$-smooth iff $\phi^{*}(y,\cdot)$ is $\frac{1}{\beta(y)}$-strongly convex. 
{}
\end{proof}
Next we obtain an upper bound on the error of the conjugate Brenier maps ,i.e., 
the inverse conditional transport maps, in terms of the error of the 
forward conditional maps.
\begin{corollary}\label{cor:conjugate-norm-bound}
    Suppose $\F$ satisfies \Cref{assump: slow rate}. Then for every $\phi \in \F$
 it holds that
    \begin{equation}
            \frac{1}{\betamax}\|\nabla_{u}\phi^{\dagger*} - \nabla_{u}\phi^{*}\|_{L_{\nu}^2}^2 \leq \frac{1}{\alphamin}\|\nabla_{u}\phi^\dagger - \nabla_{u}\phi\|_{L_{\eta}^2}^2.
    \end{equation}
\end{corollary}
\begin{proof}
    Define 
\[\cS^*(\phi) := \int \phi^{*}(y,u) d\eta(y,u) + \int \phi(y,u) d\nu(y,u).\] 
Since $\phi$ is the convex conjugate of $\phi^{*}$, we observe that
\begin{align*}
    \cS^*(\phi^{*}) = \int \phi(y,u) d\eta(y,u) + \int \phi^{*}(y,u) d\nu(y,u) = \mathcal{S}(\phi),
\end{align*}
and the minimizer of $\cS^*$ is given by $\phi^{\dagger *}$.
By \cref{cor: map stability}, we have the inequality
\begin{align}
    \frac{1}{2\beta_{max}}\|\nabla_{u}\phi^* - \nabla_{u}\phi^{\dagger*}\|_{L_{\nu}^2}^2 &\leq \cS^*(\phi^{*}) - \cS^*(\phi^{\dagger*}) \label{eq: conjugate gradient norm bound}
    \\
    \nonumber
    &= \mathcal{S}(\phi) - \mathcal{S}(\phi^{\dagger}) \leq \frac{1}{2\alpha_{min}}\|\nabla_{u}\phi - \nabla_{u}\phi^{\dagger}\|_{L_{\eta}^2}^2,
\end{align}
completing the proof.
\end{proof}

\editstart
Next, we delve into the proof for the slow rate. To simplify the exposition of the result, we first prove it in the simplified setting where the potentials are bounded (e.g., compact domain case) in~\Cref{proof: thm: slow rate}. Afterwards, we show the extension to unbounded potentials in~\Cref{proof: thm: slow rate unbounded extension}.
\editend

\editstart
\subsection{Proof of \Cref{thm: slow rate} (Slow rate) for bounded potentials}
\label{proof: thm: slow rate}
Assuming uniformly bounded potentials, we work under~\Cref{assump: slow rate} setting $q_\Y=q_\U=0$ and where conditions 2,3, and 6 are not needed.
\editend
Before presenting the proof we collect two technical results: 
an upper bound on the covering number of the function class $\mathcal{F}$ and its conjugate class $\mathcal{F}^{*}$ (defined below), and a classic generalization bound for bounded function classes from empirical process theory. 

\begin{lemma}
\label{lem: Conjugate class complexity}
Define $\mathcal{F}^{*}:=\{\phi^{*}\mid\phi\in\mathcal{F}\}$. 
Then for any $\delta>0$,
\[\mathcal{N}(\delta, \mathcal{F}^*, \|\cdot\|_{L^\infty(\Y \times \U)})\leq \mathcal{N}(\delta, \mathcal{F}, \|\cdot\|_{L^\infty(\Y \times \U)}).\]
\end{lemma}

\begin{proof}
For any $\phi,\psi\in \mathcal{F}$ we have that
    \begin{align*}
    |\phi^*(y,u)-\psi^*(y,u)|&=|\sup_{v}\{\langle u,v \rangle-\phi(y,v)\}-\sup_{v}\{\langle u,v \rangle-\psi(y,v)\}|\\
    &\leq\sup_{v}|\phi(y,v)-\psi(y,v)|\leq \sup_{y,v}|\phi(y,v)-\psi(y,v)|=\|\phi-\psi\|_{L^{\infty}(\Y \times \U)}.
    \end{align*}
Therefore, if $\{\phi_1,...,\phi_N\}$ forms a $\delta$-cover of $\mathcal{F}$, then $\{\phi_1^*,...,\phi_N^*\}$ forms a $\delta$-cover of $\mathcal{F}^*$
which yields the desired result.
\end{proof}
Next, we recall a classical generalization bound for bounded function classes.
\begin{proposition}
\label{thm: generalization for slow rate}
Let $\mathcal{F}$ be a class of real-valued functions on a vector space $\mathcal{X}$ 
such that $\|\phi\|_{L^{\infty}(\mathcal{X})} \leq R$ for all $\phi \in \mathcal{F}$ and 
let $\{x_i\}_{i=1}^N \sim \mu$ for some $\mu \in \PP(\mathcal{X})$. 
Then,
\[\mathbb{E}\left[\underset{\phi\in\mathcal{F}}{\sup} \frac{1}{N} \sum_{i=1}^{N}\phi(x_i) - 
\mu(\phi) \right] \le  \frac{C}{\sqrt{N}} \int_{0}^{R} \sqrt{\log \mathcal{N}(\delta, \mathcal{F}, \|\cdot\|_{L^\infty(\mathcal{X})})} \d\delta,\]
where $C >0$ is a universal constant and the expectation is with respect to the empirical samples.
\end{proposition}
\begin{proof}
The proof is based on  two fundamental results in statistical learning theory. First, using the symmetrization lemma~\cite[Lem. 26.2]{shalev2014understanding}, we obtain the inequality
\begin{align*}
\mathbb{E}\left[\underset{\phi\in\mathcal{F}}{\sup} \frac{1}{N} \sum_{i=1}^{N} \phi(x_i) - \mu(\phi) \right]\leq \frac{2}{N}\mathbb{E}\left[\sup_{\phi\in\F}\sum_{i=1}^N \epsilon_i\phi(x_i)\right]\, ,
\end{align*}
where $\epsilon_i \in\{\pm 1\}^N$ are a Rademacher  random variables and the quantity on the right-hand side is known as the Rademacher complexity of $\F$~\cite{koltchinskii2000rademacher}. Next, observe that the process 
$Z_\phi\coloneqq \frac{1}{\sqrt{N}}\sum_{i=1}^N\epsilon_i\phi(x_i)$ has sub-Gaussian increments $Z_{\phi}-Z_{\phi'}$ 
which allows us to use the chaining technique to bound the 
Rademacher complexity  in terms of
 Dudley's integral (see \cite[Thm 5.22, Ex. 5.24]{wainwright2019high} or \cite[Thm. 8.1.3, Rem.8.1.5, Thm. 8.2.3]{vershynin2018high} for more details):
\begin{align}
\frac{2}{N}\mathbb{E}\left[\sup_{\phi\in\F}\sum_{i=1}^N \epsilon_i\phi(x_i)\right]&\leq\frac{48}{\sqrt{N}} \int_{0}^{2R} \sqrt{\log \mathcal{N}(\delta, \mathcal{F}, \|\cdot\|_{2,\mu^N}\footnotemark)} \d\delta \label{point where chaining is modified for unbounded case}\\
&\leq \frac{48}{\sqrt{N}} \int_{0}^{2R} \sqrt{\log \mathcal{N}(\delta, \mathcal{F}, 2\|\cdot\|_{L^\infty(\mathcal{X})})} \d\delta\nonumber\\
& = \frac{48}{\sqrt{N}} \int_{0}^{2R} \sqrt{\log \mathcal{N}(\delta/2, \mathcal{F}, \|\cdot\|_{L^\infty(\mathcal{X})})} \d\delta \nonumber\\
& = \frac{96}{\sqrt{N}} \int_{0}^{R} \sqrt{\log \mathcal{N}(\delta, \mathcal{F}, \|\cdot\|_{L^\infty(\mathcal{X})})} \d\delta\, ,\nonumber
\end{align}
\footnotetext{Here we define the induced empirical metric as $\|f-g\|_{2,\mu^N}\coloneqq \left(\sum_{i=1}^N|\phi(x_i)-\phi'(x_i)|^2\right)^{1/2}$ for any $\phi,\phi'\in\F$.}
where the second step follows by triangle inequality, the third one by the rescaling property of the covering number, and the last one by change of variable.
\end{proof}
We are now in position to present the main proof for this section.
\begin{proof}[Proof of~\Cref{thm: slow rate}]
We begin by deriving a bias-variance decomposition for the excess risk $\cS^c(\wh{\phi})$. 
For simplicity define
\begin{equation}\label{def:tilde-phi}
    \widetilde{\phi} := \argmin_{\phi \in \mathcal{F}} \| \nabla_u \phi - \nabla_u \phi^\dagger \|_{L^2_\eta},
\end{equation}
and suppose this minimizer is defined (up to constant shifts). If the minimizer does not 
exist we can repeat the rest of our proof with $\widetilde{\phi}$ replaced 
with a minimizing sequence.
Then we can decompose the excess risk of $\wh{\phi}$ as 
\begin{align*}
    \cS^c(\wh{\phi}) &= [\cS(\wh{\phi}) - \cS(\widetilde{\phi})] + [\cS(\widetilde{\phi})- \cS(\phi^\dagger)]
    \\
    &= \left[(\cS(\wh{\phi}) - \wh{\cS}(\wh{\phi})) + (\wh{\cS}(\wh{\phi}) - \wh{\cS}(\widetilde{\phi})) + (\wh{\cS}(\widetilde{\phi}) - \cS(\widetilde{\phi}
    ))\right] + \left[\mathcal{S}(\widetilde{\phi})- \mathcal{S}(\phi^\dagger)\right].
\end{align*}
Since $\wh{\phi} = \arg\min_{\phi \in \mathcal{F}} \wh{\cS}(\phi)$ is optimal, it follows that $\wh{\cS}(\wh{\phi}) - \wh{\cS}(\widetilde{\phi}) \leq 0$. Thus, 
\[
\cS^c(\wh{\phi})  \leq \cS^c(\widetilde{\phi}) + 2 \sup_{\phi \in \mathcal{F}} 
|\cS(\phi) - \wh{\cS}(\phi)|.
\]
We view this bound as a bias-variance decomposition. The first term is the 
excess risk of $\widetilde{\phi}$ encoding the bias in approximating 
$\phi^\dagger$ due to the choice of the class $\mathcal{F}$. The second term
encodes the stochastic errors or the variance part of the error bound due to randomness 
in the empirical approximation of $\cS$ with $\wh{\cS}$.
Upon application of \Cref{cor: map stability} we can bound  the bias term 
\begin{align*}
    \cS^c(\widetilde{\phi}) \leq \frac{1}{2\alphamin}| |\nabla_{u}\widetilde{\phi} - \nabla_{u}\phi^{\dagger}\|_{L_{\eta}^2}^2. 
\end{align*}
The expectation of the variance term can be further decomposed as
\begin{align}\label{eq: decomposition emp process slow rate}
  \Expect\left[ \sup_{\phi \in \mathcal{F}} |\cS (\phi) - \wh{\cS}(\phi)|\right]  \leq 
   \mathbb{E} \Bigl[\sup_{\phi \in \mathcal{F}} |(\eta - \eta^N)(\phi)|\Bigr] + 
   \mathbb{E} \Bigl[\sup_{\phi \in \mathcal{F}}| (\nu - \nu^N)(\phi^*)| \Bigr],
\end{align}
where we used the shorthand notation $\eta^N$ and $\nu^N$ for the empirical 
measures associated to  $\eta$ and $\nu$ respectively.
By \Cref{assump: slow rate} and \Cref{lem: Conjugate class complexity}, there exists $R>0$ such that $\|\phi\|_{L^{\infty}(\Omega)}\leq R$ for all $\phi \in \mathcal{F}$ and the same bound holds for all $\phi^{*}\in\mathcal{F}^{*}$. Applying \Cref{thm: generalization for slow rate} with $\tau =0$, we obtain:
\[\mathbb{E} \Bigl[\sup_{\phi \in \mathcal{F}} |(\eta - \eta^N)(\phi)| \Bigr] \le C \frac{1}{\sqrt{N}} \int_0^R \sqrt{\log \mathcal{N}(\delta, \mathcal{F}, \|\cdot\|_{L^\infty})} \, \d \delta,\]
and
\begin{align*}
\mathbb{E} \Bigl[\sup_{\phi \in \mathcal{F}} |(\nu - \nu^N)(\phi^*)| \Bigr] &\le C \frac{1}{\sqrt{N}} \int_0^R \sqrt{\log \mathcal{N}(\delta, \mathcal{F}^*, \|\cdot\|_{L^\infty})} \, \d \delta\\
&\leq C\frac{1}{\sqrt{N}} \int_0^R \sqrt{\log \mathcal{N}(\delta, \mathcal{F}, \|\cdot\|_{L^\infty})} \, \d \delta.
\end{align*}
By the covering number assumption in \Cref{assump: slow rate},
\[\log \mathcal{N}(\delta,\mathcal{F},\|\cdot\|_{L^\infty}) \leq C_{\F}\delta^{-\gamma} \log(1 + \delta^{-1}),\]
and hence,
\[\int_{0}^{R}\sqrt{\log \mathcal{N}(\delta,\mathcal{F}, \|\cdot\|_{L^\infty})}\d\delta \le  \sqrt{C_{\F}} \int_{0}^{R}\delta^{-\gamma/2}\sqrt{\log(1+\delta^{-1})}\d \delta.\]
Combining all terms, we conclude:
\begin{equation}\label{disp-1}
\cS^c(\wh{\phi}) \le C \left[  \frac{1}{\alphamin} \| \nabla_u \widetilde{\phi} - \nabla_u \phi^\dagger \|^2_{L^2_\eta} 
+ \sqrt{\frac{C_\F}{N}} 
\int_0^R \delta^{- \gamma/2} \sqrt{ \log (1 + \delta^{-1})} 
\d \delta \right].
\end{equation}
Using \Cref{cor: map stability} once again along with 
\Cref{lem: prop const slow rate} below (to bound the 
integral on the right hand side) we arrive at the 
desired bound:
\begin{align*}
    \mathbb{E} \|\nabla_u \wh{\phi} - \nabla_u \phi^\dagger\|_{L_\eta^2}^2 
    &\le C \cdot \betamax \left[ \frac{1}{\alphamin} \|\nabla_{u}\widetilde{\phi} - \nabla_{u}\phi^{\dagger}\|_{L_{\eta}^2}^2 +  \sqrt{\frac{ R^2 C_{\F}}{N}} 
    \right]\, ,
\end{align*}
where $C > 0$ is a universal constant. 
\end{proof}
Let us now present a calculation that gives a bound on 
the integral on the right hand side of \eqref{disp-1}
in terms of the parameter $R$.
\begin{lemma}\label{lem: prop const slow rate}
For all $R\ge 1$ and $\gamma \in [0,1)$ it 
holds that
\begin{align*}
\int_0^R \delta^{-\gamma/2} \sqrt{\log (1 + \delta^{-1})} \d \delta \leq  4R.
\end{align*}
\end{lemma}
\begin{proof}
Decompose the integral into two  intervals $[0,1]$ and $(1,R]$:
\begin{align}
\int_0^{R}\delta^{-\gamma/2}&\sqrt{\log(1+\delta^{-1})}\d \delta= \int_0^{1}\delta^{-\gamma/2}\sqrt{\log(1+\delta^{-1})}\d \delta+\int_1^{R}\delta^{-\gamma/2}\sqrt{\log(1+\delta^{-1})}\d \delta \nonumber\\
& \leq \int_0^{1}\delta^{-\gamma/2}\sqrt{\log(2/\delta)}\d \delta+\sqrt{\log(2)}\int_1^{R}\delta^{-\gamma/2}\d \delta \nonumber \\
&\overset{s=\log(2/\delta)}=2^{1-\gamma/2}\int_{\log(2/R)}^\infty e^{-(1-\gamma/2)s}s^{1/2}\d s+\frac{2\sqrt{\log(2)}}{2-\gamma}\left(R^{1-\gamma/2}-1\right) \nonumber \\
&\leq 2^{1-\gamma/2}\int_{0}^\infty e^{-(1-\gamma/2)s}s^{1/2}\d s+\frac{2\sqrt{\log(2)}}{2-\gamma}\left(R^{1-\gamma/2}-1\right) \nonumber \\
&=2^{1-\gamma/2}\frac{\Gamma(3/2)}{2(1-\gamma/2)^{3/2}}+\frac{2\sqrt{\log(2)}}{2-\gamma}\left(R^{1-\gamma/2}-1\right) \nonumber \\
&=2^{-(1+\gamma/2)}\frac{\sqrt{\pi}}{(1-\gamma/2)^{3/2}}+\frac{2\sqrt{\log(2)}}{2-\gamma}\left(R^{1-\gamma/2}-1\right)\, .\label{eq: bound C_S}
\end{align}
Simplifying the right hand side using the fact that $\gamma \in [0,1)$ and $1\leq R$: 
\begin{align*}
  2^{-(1+\gamma/2)}\frac{\sqrt{\pi}}{(1-\gamma/2)^{3/2}}+\frac{2\sqrt{\log(2)}}{2-\gamma}\left(R^{1-\gamma/2}-1\right)&\leq   2^{-1}\frac{\sqrt{\pi}}{(1-1/2)^{3/2}}+\frac{2\sqrt{\log(2)}}{2-1}\left(R-1\right)\\&\leq \sqrt{2\pi}R+2\sqrt{\log(2)}R \leq 4 R.
\end{align*}
\end{proof}

\editstart
\subsection{
Extending the proof of \Cref{thm: slow rate} (Slow rate) to unbounded potentials}
\label{proof: thm: slow rate unbounded extension}
We now extend our proof to the case of unbounded potentials under~\Cref{assump: slow rate} in its full generality.
We begin with a set of preiminary results.

\begin{lemma}\label{lem: gradients boundedness}
Under conditions 2 and 6 in~\Cref{assump: slow rate}, for any $y\in\Y$,
\begin{align*}
\|\nabla_u\phi(y,0)\|\leq\widetilde{K}\langle y\rangle^{1+\widetilde{q}}\, ,\qquad \widetilde{K} = K+ \left(1+ \nu_\Y(\langle \cdot\rangle^{\widetilde{q}+1})\right)L_\F\, .
\end{align*}
\end{lemma}
\begin{proof}
To compactify notation, let $g(y)\coloneqq \nabla_u\phi(y,0)$. Next, note that by the fundamental 
theorem of calculus
\begin{align*}
g(y)-g(0) = \int_{0}^{1} \nabla_{y} g(ty)y\d t\, \forall y \in \Y.
\end{align*}
Then,
\begin{align*}
\|g(y)-g(0)\|\leq \|y\| L_\F \int_{0}^{1}\langle t y\rangle^{\widetilde{q}}\d t\leq L_\F\langle y\rangle^{\widetilde{q}+1}\, .
\end{align*}
By triangle inequality,
\begin{align*}
\|g(0)\|\leq \|g(y)\|+L_\F\langle y\rangle^{\widetilde{q}+1}\, .
\end{align*}
Consequently, marginalizing over $\nu_\Y$,
\begin{align}\label{eq: adaptable inequality}
\|g(0)\|\leq \int \left(\|g(y)\|+L_\F\langle y\rangle^{\widetilde{q}+1}\right) \d \nu_\Y(y)\leq K+ \nu_\Y(\langle \cdot\rangle^{\widetilde{q}+1})L_\F .
\end{align}
Finally, 
\begin{align*}
\|g(y)\| \leq \|g(0)\| + \|y\|L_\F\langle y\rangle^{\widetilde{q}}\leq \widetilde{K}\langle y\rangle^{\widetilde{q}+1}\, ,
\end{align*}
concluding the lemma.
{}
\end{proof}

Next, we show that the metric entropy of the conjugate class $\F^*$ with respect to the conjugate weighted norm $L^\infty(w_*)$ can be controlled by the corresponding metric entropy of the original class $\F$ with respect to $L^\infty(w)$  exploiting $\alphamin$-strong convexity and~\Cref{lem: gradients boundedness}. In doing so, we adapt the argument of
~\cite[Lem. A.8]{divol2025optimal-sm} to the conditional setting.
Recall that in \Cref{assump: slow rate} we defined the weights 
\begin{equation*}
  w(y,u)\coloneq \langle y\rangle^{-q_\Y} \langle u\rangle^{-q_\U} \quad \text{and} \quad 
  w_*(y, u) := \langle y\rangle^{-(q_\Y + (1+\widetilde{q})q_\U)}\langle u\rangle^{-q_\U}.
\end{equation*}
\begin{lemma}\label{lem: conjugate to non conjugate unbounded}
Under conditions 2, 4 and 6 of~\Cref{assump: slow rate}, there exists a constant $C^*(\alphamin, K,L_\F, q_\U, \widetilde{q})>0$ such that 
\begin{align}\label{eq: conjugate weighted to weighted}
\log\mathcal{N}\left(\delta, \F^*, \|\cdot\|_{L^\infty(w_*)}\right)\leq \log\mathcal{N}\left(C^{*-1}\delta, \F, \|\cdot\|_{L^\infty(w)}\right)\, .
\end{align}
\end{lemma}
\begin{proof}
Fix any two elements $\phi,\psi\in\F$. By applying~\cite[Lem. D.2]{divol2025optimal-sm} (setting $a=0$ in that lemma) followed by~\Cref{lem: gradients boundedness} for any fixed $y\in\Y$, $\nabla_u\phi(y,\cdot)$ is a bijection and 
\begin{align*}
\|\nabla_u\phi(y,\cdot)^{-1}(v)\|
\leq \frac{4}{\alphamin}\left(\|v\|+\|\nabla_u\phi(y,0)\|\right)
\leq \frac{4}{\alphamin}\left(\|v\|+ \widetilde{K}\langle y\rangle^{1+\widetilde{q}}\right)\, .
\end{align*}
The same holds for $\psi$. Then, letting $u_\phi(v)\coloneqq \nabla_u\phi(y,\cdot)^{-1}(v)$ and $u_\psi(v)\coloneqq \nabla_u\psi(y,\cdot)^{-1}(v)$,
\begin{equation}\label{disp-2}
\langle u_\phi(v)\rangle\leq C \langle y\rangle^{1+\widetilde{q}} \langle v\rangle\, ,\qquad \langle u_\psi(v)\rangle\leq C \langle y\rangle^{1+\widetilde{q}} \langle v\rangle\, ,
\end{equation}
for $C = \frac{4}{\alphamin}\widetilde{K}$.
Moreover, by definition of convex conjugate,
\begin{align*}
\phi^*(y,v) &= \langle u_\phi(v),v\rangle-\phi(y,u_\phi(v))\geq \langle u_\psi(v),v\rangle-\phi(y,u_\psi(v))\, ,\\
\psi^*(y,v) &= \langle u_\psi(v),v\rangle-\psi(y,u_\psi(v))\geq \langle u_\phi(v),v\rangle-\psi(y,u_\phi(v))\, .
\end{align*}
Thus,
\begin{align*}
\phi^*(y,v)-\psi^*(y,v)&\leq \left[\langle u_\phi(v),v\rangle-\phi(y,u_\phi(v))\right] - \left[\langle u_\phi(v),v\rangle-\psi(y,u_\phi(v))\right] \\
&= \psi(y,u_\phi(v))-\phi(y,u_\phi(v))\, ,\\
\psi^*(y,v)-\phi^*(y,v) &\leq \phi(y,u_\psi(v))-\psi(y,u_\psi(v))\, .
\end{align*}
Combining the last two inequalities,
\begin{align}
\left|\phi^*(y,v)-\psi^*(y,v)\right|\leq \left|\phi(y,u_\phi(v))-\psi(y,u_\phi(v))\right| + \left|\phi(y,u_\psi(v))-\psi(y,u_\psi(v))\right|\, .
\end{align}
Then, to reintroduce the weight functions, we can bound 
\begin{align*}
\left| \phi^*(y,v)-\psi^*(y,v)\right| &\leq \|\phi- \psi\|_{L^{\infty}(w)}\langle y\rangle^{q_\Y}
\left( \langle u_\phi(v)\rangle^{q_\U} + \langle u_\psi(v)\rangle^{q_\U} \right)\\
&\leq  2C^{q_\U} \langle y\rangle^{q_\Y+q_\U(1+\widetilde{q})}\langle v\rangle^{q_\U} \|\phi-\psi\|_{L^\infty(w)}\, .
\end{align*}
where we used \eqref{disp-2} to obtain the second inequality.
\editstart
Multiplying both sides by $\langle y\rangle^{-(q_\Y+q_\U(1+\widetilde{q}))}\langle v\rangle^{-q_\U}$  and taking the essential supremum over $(y,v)$ yields
\begin{align*}
\|\phi^*-\psi^*\|_{L^\infty(w_*)}\leq C^*\|\phi-\psi\|_{L^\infty(w)}\ ,\qquad C^* \coloneqq 2C^{q_\U}\, .
\end{align*}
Thus, any $C^{*-}\delta$-net of $\F$ in $\|\cdot\|_{L^\infty(w)}$ is also a $\delta$-net of $\F^*$ in $\|\cdot\|_{L^\infty(w_*)}$, yielding~\eqref{eq: conjugate weighted to weighted}.
{}
\end{proof}

We are now ready to establish our extension of the proof in \Cref{proof: thm: slow rate}
to the unbounded potential setting. Since most of the arguments are repetitive we 
give a detailed list of required modifications:

\begin{itemize}
\item [(1)] \textbf{The general tool: a weighted chaining bound.}
Using the same approach as in~\cite[Prop.~A.2]{divol2025optimal-sm}, the generalization
result of~\Cref{thm: generalization for slow rate} can be modified to replace the metric
entropy with respect to the uniform norm $L^\infty$ with one defined in terms of an
anisotropically weighted uniform norm $L^\infty(\langle\cdot\rangle^{-q'_\Y}\langle\cdot\rangle^{-q'_\U})$,
for any $q'_\Y,q'_\U>0$, at the expense of a $(q'_\Y,q'_\U)$-dependent multiplicative
constant $C(q'_\Y,q'_\U)$ in place of the universal $C$. Concretely, this is the following weighted analogue of~\Cref{thm: generalization for slow rate}. 
\begin{proposition}
\label{thm: generalization for slow rate weighted} 
Let $\mathcal{F}$ be a class of real-valued functions on $\Y\times\U$, let $q'_\Y,q'_\U>0$ and set $w'(y,u):=\langle y\rangle^{-q'_\Y}\langle u\rangle^{-q'_\U}$.  
Suppose $\|\phi\|_{L^\infty(w')}\le R'$ for all $\phi\in\F$, and let $\{x_i\}_{i=1}^N\sim\mu$ for some $\mu \in \mathcal{P}(\Y \times \U)$ such that
$C(q'_\Y,q'_\U,\mu):=\|\langle
\cdot\rangle^{q'_\Y}\langle\cdot\rangle^{q'_\U}\|_{L^2_\mu(\Y\times\U)}<\infty$. Then,
\begin{align*}
\mathbb{E}\left[\underset{\phi\in\mathcal{F}}{\sup} \frac{1}{N} \sum_{i=1}^{N}\phi(x_i) -
\mu(\phi) \right] \le  \frac{C(q'_\Y,q'_\U,\mu)}{\sqrt{N}} \int_{0}^{R'} \sqrt{\log \mathcal{N}(\delta, \mathcal{F}, \|\cdot\|_{L^\infty(w')})} \d\delta\, .
\end{align*} 
\end{proposition}
\begin{sproof}
The result follows the proof of~\Cref{thm: generalization for fast rate} closely up to the one
step where the uniform bound on $\F$ enters. 
There, the pointwise bound $\|\phi-\psi\|_{2,\mu^N}\leq 2R\|\phi-\psi\|_{L^\infty(\mathcal{X})}$ is replaced by the Cauchy--Schwarz estimate
\[\frac{1}{N}\sum_{i=1}^N|\phi(x_i)-\psi(x_i)|^2 \leq \|\phi-\psi\|_{L^\infty(w')}^2\cdot\frac{1}{N}\sum_{i=1}^N w'(x_i)^{-2}\, ,\]
whose expectation over $\{x_i\}_{i=1}^N\overset{\mathrm{i.i.d.}}{\sim}\mu$ is exactly
$\|\phi-\psi\|_{L^\infty(w')}^2 C(q'_\Y,q'_\U,\mu)^2$ by definition of $C(q'_\Y,q'_\U,\mu)$.
The
remainder of the argument is unchanged.
\end{sproof}
When $\gamma<2$, this weighted bound can be invoked in the proof for the slow rate in
place of the unweighted one. We now apply it twice, with two different choices of
$q'_\Y$ and $q'_\U$.
\item [(2)] \textbf{Bounding the $\eta$-term in \eqref{eq: decomposition emp process slow rate}:} Applying ~\Cref{thm: generalization for slow rate weighted} 
with $q_\Y' = q_\Y$, $q'_\U = q_\U$, $\mu = \eta$ and $R' = R=\sup_{\phi\in\F}\|\phi\|_{L^\infty(w)}$ gives
\begin{align*}
\mathbb{E}\Bigl[\sup_{\phi \in \mathcal{F}} |(\eta - \eta^N)(\phi)|\Bigr]
\le \frac{C(q_\Y,q_\U,\eta)}{\sqrt{N}} \int_0^R \sqrt{\log \mathcal{N}(\delta, \mathcal{F}, \|\cdot\|_{L^\infty(w)})} \, \d \delta\, ,
\end{align*}
which bounds the first term on the right-hand side of~\eqref{eq: decomposition emp process slow rate}.

\item [(3)] \textbf{Bounding the $\nu$-term in \eqref{eq: decomposition emp process slow rate}:} We bound the empirical process with respect to $\nu$ in terms of the conjugate class $\F^*$ in~\eqref{eq: decomposition emp process slow rate} in two sub-steps. 
\begin{itemize}
    \item [(3a)] First, apply the weighted chaining result with $q'_\Y = q_\Y + q_\U(1+\widetilde{q})$ and $q'_\U = q_\U$, setting $R' = \sup_{\phi^*\in\F^*}\|\phi^*\|_{L^\infty(w_*)}\leq C{^*}\sup_{\phi\in\F}\|\phi\|_{L^\infty(w)} = C^*R$, with the bound holding by the proof of ~\Cref{lem: conjugate to non conjugate unbounded}. This yields
    \begin{adjustwidth}{-\dimexpr\textwidth-\linewidth\relax}{0pt}
    \[
    \begin{aligned}
\mathbb{E}\Bigl[\sup_{\phi \in \mathcal{F}} |(\nu - \nu^N)(\phi)|\Bigr]
\le \frac{C(q_\Y + q_\U(1+\widetilde{q}),q_\U,\nu)}{\sqrt{N}} \int_0^{C*R} \sqrt{\log \mathcal{N}(\delta, \mathcal{F^*}, \|\cdot\|_{L^\infty(w_*)})} \, \d \delta\, ,
\end{aligned}
\]
\end{adjustwidth}
    \item [(3b)] Then use \Cref{lem: conjugate to non conjugate unbounded} in place of~\Cref{lem: Conjugate class complexity} to bound the metric entropy of $\mathcal{F}^{*}$ from (3a) by that of $\mathcal{F}$, following with a linear change of variable,
    \begin{align}\label{eq: change var Dudley int}
    \int_0^R \sqrt{\log \mathcal{N}(\delta, \mathcal{F}, \|\cdot\|_{L^\infty(w)})} \, \d \delta&\leq \int_0^{C^*R} \sqrt{\log\mathcal{N}\left(C^{*-1}\delta, \F, \|\cdot\|_{L^\infty(w)}\right)}\d \delta \nonumber\\
    & = C^{*-1}\int_0^{(C^*)^2 R} \sqrt{\log\mathcal{N}\left(\delta, \F, \|\cdot\|_{L^\infty(w)}\right)}\d \delta \\
    & \leq 4 \frac{(C^*)^2}{C^*}R = 4C^*R\, .\nonumber
    \end{align}
by the same calculation as~\Cref{lem: prop const slow rate}.
\end{itemize}
\item [(4)]\textbf{Combining the two terms.} Putting items (2) and (3) together, the proof of~\Cref{thm: slow rate} in the bounded case extends verbatim to the unbounded case by replacing final $R \mapsto \left(1\vee C^*\right)R = \frac{4}{\alphamin}\widetilde{K}R$ and introducing the proportionality constant $C(q_\Y,q_\U,\eta)+C(q_\Y+(1+\widetilde{q})q_\U, q_\U,\nu) =\sqrt{\eta(w^{-2})}+\sqrt{\nu(w_*^{-2})}\leq 2\kappa$.
\end{itemize}

\editend

\editstart
\subsection{Proof of \cref{thm: fast rate} (Fast rate)  for bounded potentials}
\label{proof: thm: fast rate}
Mirroring the slow rate proofs we begin by presenting a detailed argument 
in the case where the potentials are assumed to be uniformly bounded, i.e., setting $q_\Y = q_\U = 0$ and discarding conditions 2,3, and 6 in~\Cref{thm: slow rate}. 
Once the technical steps are laid out we present the unbounded extention in~\Cref{proof: thm: fast rate unbounded extension}.
\editend
Our proof technique for \Cref{thm: fast rate} is as an extension of 
the approach in \cite[Prop.~11]{hutter2019minimax} and 
\cite[Prop.~14]{chewi2024statistical} to the conditional case.
Since the proof is long and technical we begin with a discussion of the overall strategy 
and summary of the key tools from empirical process theory that will be used. 
Next we present the main proof which relies on two technical 
propositions whose proofs are our novel theoretical contribution. To keep the proof 
focused on the main result we postpone the proofs of those propositions 
to the end of this section.

\paragraph{Strategy of the proof and some preliminaries}
We begin with an error bound on the excess risk 
$\sup_{\phi \in \F} \frac{1}{N} \phi(y_i, u_i) - \Expect \phi(y, u)$  for an i.i.d. set 
of samples using well-established bounds from the theory of 
empirical processes. In particular, we give bounds on the 
expectation of the excess risk as well as a high-probability/tail bound. Our interest in such a bound is clear in light 
of \Cref{cor: map stability} and proof of \Cref{thm: slow rate}. 
However, we wish to obtain an upper bound on $\mathbb{E} \|\nabla_u \wh{\phi} - \nabla_u \phi^\dagger\|_{L_\eta^2}$ that 
is $\mathcal{O}(N^{-\frac{1}{2 + \gamma}})$ rather than $\mathcal{O}(N^{-\frac14})$. 
To do this we need to control the excess risk 
by localizing our analysis in a neighborhood of the true map $\phi^\dagger$. 

The following proposition 
is familiar in the empirical process theory literature:

\begin{proposition}
\label{thm: generalization for fast rate}
Let $\mu$ be a probability measure supported on a vector space $\mathcal{X} \subseteq \mathbb{R}^{d}$, and let $\{x_i\}_{i=1}^{N} \overset{\mathrm{i.i.d}}{\sim}\mu$. Suppose $\mathcal{F}$ is a class of real-valued functions such that $\|\phi\|_{L_{\mu}^2(\mathcal{X})} \leq r$ and $\|\phi\|_{L^{\infty}(\mathcal{X})} \leq R$ for all $\phi \in \mathcal{F}$.  Define
\begin{align}\label{eq: J_n bound quantity}
J_N(\mathcal{F}):= 
\editstart
\frac{1}{\sqrt{N}}\int_{0}^{r} \sqrt{\log \mathcal{N}(\delta,\mathcal{F}, \|\cdot\|_{L^2_\mu(\mathcal{X})})}\d\delta 
\editend
+ \frac{1}{N}\int_{0}^{R}\log \mathcal{N}(\delta,\mathcal{F}, \|\cdot\|_{L^\infty(\mathcal{X})})\d\delta\, .
\end{align}
Then, it holds that:
\begin{enumerate}
\item There exists a universal constant $C_{\textnormal{exp}} > 0$ such that
\begin{align}\label{eq: expectation bound emp. proc.}
\mathbb{E}\left[\underset{\phi\in\mathcal{F}}{\sup} \frac{1}{N} \sum_{i=1}^{N}\phi(x_i) - \mu(\phi) \right]\leq 
C_{\textnormal{exp}}J_N(\mathcal{F})\, .
\end{align}
\item  There exists a universal constant 
$C_{\textnormal{prob}}>0$ such that for any $t\geq 0$,
\begin{equation}\label{eq: probability bound emp. proc.}
\begin{aligned}
& \mathbb{P}\left[ \underset{\phi\in\mathcal{F}}{\sup} \frac{1}{N} \sum_{i=1}^{N}\phi(x_i) 
- \mu(\phi) \geq 
C_{\textnormal{prob}}\left(J_N(\mathcal{F})+r \sqrt{\frac{t}{N}}+R\frac{t}{N}\right)
\right]\leq\exp(-t).
\end{aligned}
\end{equation}
\end{enumerate}
\end{proposition}
\begin{proof}

The  bound in \eqref{eq: expectation bound emp. proc.} follows from 
\cite[Prop. A.2]{divol2025optimal-sm} by taking, in the notation of that result, 
$\epsilon = \widetilde{\epsilon} = \eta = 0$.
We note that this setup is much simpler than \cite{divol2025optimal-sm} which suggests a 
simpler proof may be possible, see for example \cite[Cor.~3.5.7 and Prop.~3.5.15]{gine2021mathematical}.
The high probability bound in \eqref{eq: probability bound emp. proc.} follows 
from the tail bound~\cite[Thm. 2.14.25]{vanderVaartWellner2023}, which provides a way to estimate with high probability how the process $\sup_{\phi\in\mathcal{F}} \frac{1}{N} \sum_{i=1}^{N}\phi(x_i) - \mu(\phi)$ deviates from its mean, which is controlled by $J_N(\F)$ according to \eqref{eq: expectation bound emp. proc.}. In the notation of \cite[Thm. 2.14.25]{vanderVaartWellner2023} set
\begin{align*}
\frac{1}{\sqrt{N}}\mathbb{G}_N\coloneqq \underset{\phi\in\mathcal{F}}{\sup} \frac{1}{N} \sum_{i=1}^{N} \phi(x_i) - \mu(\phi)\quad \text{and} \quad \mu_N\coloneqq \mathbb{E}\mathbb{G}_N\, , 
\end{align*}
to obtain the bound
\begin{align*}
\mathbb{P}\left(\frac{1}{\sqrt{N}}\mathbb{G}_N>C\left(\frac{1}{\sqrt{N}}\mu_N+\frac{s}{\sqrt{N}}\right)\right)\leq \exp\left(-D\min\left\{\frac{s^2}{\sigma^2},\frac{s\sqrt{N}}{R}\right\}\right)\, ,
\end{align*}
for universal constants $C,D$. 
To express the bound in terms of $t$ while satisfying both tail-regime constraints, set
    $\frac{s}{\sqrt{N}}\coloneqq \max\left\{r\sqrt{\frac{t}{DN}},\frac{Rt}{DN}\right\}$ so that
\begin{align*}
    \frac{s}{\sqrt{N}} \leq r \sqrt{\frac{t}{DN}} + \frac{Rt}{DN},\quad \textnormal{and} \; \min\left\{\frac{s^2}{r^2},\frac{s\sqrt{N}}{R}\right\} = \frac{t}{D}.
\end{align*}
This allows us to rewrite the bound as
\begin{align}\label{eq: prob. bound after substitution}
\mathbb{P}\left(\frac{1}{\sqrt{N}}\mathbb{G}_N>C\left(\frac{1}{\sqrt{N}}\mu_N + r \sqrt{\frac{t}{N}}+R\frac{t}{N}\right)\right)\leq \exp\left(-t\right)\, ,
\end{align}
absorbing the constant $\max\{1/D,1/\sqrt{D},1\}$ in $C$.
Ultimately, \eqref{eq: probability bound emp. proc.} follows from \eqref{eq: prob. bound after substitution} in combination with
\begin{align*}
\begin{split}
&\mathbb{P}\left(\frac{1}{\sqrt{N}}\mathbb{G}_N>C\left(J_N(\F)+ r \sqrt{\frac{t}{N}}+R\frac{t}{N}\right)\right) \\
& \hspace{8ex} \leq \mathbb{P}\left(\frac{1}{\sqrt{N}}\mathbb{G}_N>C\left(\frac{1}{\sqrt{N}}\mu_N
+ r\sqrt{\frac{t}{N}}+R\frac{t}{N}\right)\right)\, ,
\end{split}
\end{align*}
which is given immediately by the bound $\frac{\mu_N}{\sqrt{N}}\leq B J_N(\F)$ of \eqref{eq: expectation bound emp. proc.} and absorbing  $B$ into $C$.
\end{proof}
We can further simplify \Cref{thm: generalization for fast rate} by providing 
an upper bound on $J_N(\F)$:
\begin{lemma}\label{lem: bound-JF}
Assume $\mathcal{N}(\delta,\mathcal{F}, \|\cdot\|_{L^\infty})\leq C_{\F} \delta^{-\gamma}\log(1+\delta^{-1})$ and let $R>1$ and $r = Q \epsilon$, for some $\epsilon \in [0,1]$. Then
there exists a universal constant $C >0$ so that 
\begin{equation}
J_N(\F) \le C \left[  \sqrt{\frac{C_\F}{N}} Q \epsilon^{1-\gamma/2} \sqrt{ \log(1+\epsilon^{-1})}  + \frac{1}{N}C_\F R(1-\gamma)^{-2} \right]\, .
\end{equation}
\end{lemma}
\begin{proof}
Using the hypothesis of the theorem 
\editstart
and bounding the $L^2$ metric entropy by the $L^\infty$ one in the first integral of~\eqref{eq: J_n bound quantity},
\editend
we can readily write
\begin{align}\label{eq: J_n decomposition}
J_N(\mathcal{F})\leq \sqrt{\frac{C_\F}{N}} I_1+ \frac{C_\F}{N}I_2\,,
\end{align}
where we introduced the quantities
\[
  I_{1}
  :=
  \int_{0}^{Q \epsilon}
        \delta^{-\gamma/2}\sqrt{\log(1+\delta^{-1})}\,d\delta,
  \qquad
  I_{2}
  :=
  \int_{0}^{R}
        \delta^{-\gamma}\log(1+\delta^{-1})\,d\delta.
\] 
Below we will show that 
\begin{align}
  I_{1}
  &\leq
  10 Q \,
  \epsilon^{1-\gamma/2}
  \sqrt{\log(1+\epsilon^{-1})},
  \label{eq: I1goal}\\
  I_{2}
  &\leq
  2 R(1-\gamma)^{-2}\, ,
  \label{eq: I2goal}
\end{align}
which completes the proof upon substitution in \eqref{eq: J_n decomposition}.
\vspace{0.5ex}
\textit{Bound for $I_{1}$.}
Introduce the rescaling $\delta= q \epsilon $ with $q\in[0,Q]$ to
obtain
\begin{equation}\label{eq:I1def}
  I_{1}
  =
  \epsilon^{1-\gamma/2}
  \int_{0}^{Q}
        q^{-\gamma/2}
        \sqrt{\log\bigl(1+(\epsilon q)^{-1}\bigr)}\,dq.
\end{equation}

Using the elementary inequalities 
$\log( 1 + 1/ab) \le \log( 1 + 1/a) 
+ \log(1 + 1/b)$ 
and 
$\sqrt{a+b}\le\sqrt a+\sqrt b$
we can write

\[
  \sqrt{\log\bigl(1+(\epsilon u)^{-1}\bigr)}
  \le
  \sqrt{\log(1+\epsilon^{-1})}
  +\sqrt{\log(1+u^{-1})}.
\]
Since
\(
 \int_{0}^{Q}q^{-\gamma/2}\,dq
   =\tfrac{2Q^{1-\gamma/2}}{2-\gamma},
\)
substituting into~\eqref{eq:I1def} and recalling that $\sqrt{\log(1+\epsilon^{-1})}\geq\sqrt{\log(2)}$ for $\epsilon<1$, gives
\begin{equation*}
  I_{1}
  \le\left(
  \frac{2}{2-\gamma}\,
  Q^{1-\gamma/2}+\frac{\int_0^{Q}q^{-\gamma/2}\sqrt{\log(1+q^{-1})}\d q}{\sqrt{\log(2)}}\right)
  \epsilon^{1-\gamma/2}
  \sqrt{\log(1+\epsilon^{-1})}.
\end{equation*}
Applying \Cref{lem: prop const slow rate} to further bound the integral inside the 
brackets further simplifies the bound to the desired form
\begin{align*}\label{eq:I1small}
I_1&\leq \left(\frac{2}{2-\gamma}\,Q^{1-\gamma/2}+\frac{4Q}{\sqrt{\log(2)}}\right)\epsilon^{1-\gamma/2}\sqrt{\log(1+\epsilon^{-1})}\\
&\leq 10Q \epsilon^{1-\gamma/2}\sqrt{\log(1+\epsilon^{-1})}\, ,
\end{align*}
where we note that the last inequality holds since  $\epsilon, \gamma \in(0,1)$. 

\vspace{0.5ex}
\textit{Bound for $I_{2}$.}
Decompose
\(
\log(1+\delta^{-1})=\log(\delta^{-1})+\log(1+\delta)
\)
and note that $\log(1+\delta)\le\log(1+R)$. Further noting the identities
\[
  \int_{0}^{R}\!\delta^{-\gamma}\,d\delta
      =\frac{R^{1-\gamma}}{1-\gamma},
  \qquad
  \int_{0}^{R}\!\delta^{-\gamma}\log(\delta^{-1})\,d\delta
      =R^{1-\gamma}\left(\frac{1}{(1-\gamma)^{2}}-\frac{\log(R)}{1-\gamma}\right),
\]
we obtain the bound
\begin{equation*}%
  I_{2}
  \le
  R^{1-\gamma}
  \Bigl(
     \frac{1}{(1-\gamma)^{2}}
     +\frac{\log\left(1+1/R\right)}{1-\gamma}
  \Bigr) \leq 2\frac{R^{1-\gamma}}{(1-\gamma)^2}\leq 2R(1-\gamma)^{-2}\, ,
\end{equation*}
where we  used the fact that $\log(1+1/R)\leq \frac{1}{R} \leq 1 \leq  \frac{1}{1-\gamma}$
in the second inequality.
\end{proof}

\editstart

\begin{remark}\label{rem: where to adapt proof}
We make two points regarding the extension of the above lemma beyond the simplified setting considered here: (i)
 how one can handle $\gamma > 1$; and (ii) how the result can be sharpened for large parametric classes at $\gamma = 0$. 
 \begin{itemize}
     \item [(i)] \textbf{Extending to $\gamma > 1$.} The calculation of~\Cref{lem: bound-JF} 
     (~\Cref{lem: prop const slow rate} respectively for the slow rate), which follows from applying ~\Cref{thm: generalization for fast rate} (~\Cref{thm: generalization for slow rate} respectively for the slow rate), is precisely the part of the proof that must be adapted to handle $\gamma>1$. In that regime, one can instead use chaining bounds with truncated Dudley integrals, as in \cite[Prop. A.2]{divol2025optimal-sm}, to obtain a bound analogous to~\Cref{lem: bound-JF}, and then proceed with the rest of the proof as we present it (with the needed adjustments) to obtain the rates reported in ~\Cref{tab:cond-ot-rates}.
     \item [(ii)] \textbf{Sharpening at $\gamma=0$.} For the case $\gamma = 0$, a more refined calculation -- following ~\cite[Prop. A.9]{divol2025optimal-sm}
     under additional smoothness assumptions and $L_2$-metric entropy bounds akin to   
     in~\cite[Condition~C2]{divol2025optimal} -- sharpens the bound of ~\Cref{lem: bound-JF} from $C_\F$ to $C_\F^{1-2/d}$, where $d\geq 2$ is the ambient dimension of $\Y \times \U$. This improvement matters for deriving optimal global rates, as discussed later in ~\Cref{rem:C_sigma_claculation}.
 \end{itemize}

We stress that both (i) and (ii) rely on chaining calculations established in~\cite{divol2025optimal}; they are independent of and compatible with the results we provide to extend the rates to the conditional setting ( \Cref{thm: PI extension consequence for COT} and \Cref{thm: centered semidual empirical estimation rate}). We therefore expect these possible adaptations to carry over smoothly to the conditional setting as well.
\end{remark}
\editend

The result of \cref{thm: centered semidual empirical estimation rate} below combines \Cref{thm: generalization for fast rate,lem: bound-JF} to bound, for a localized class $\F_\epsilon\subset\F$ around the true potential $\phi^\dagger$,
\begin{equation*}
    \Expect \left[\sup_{\phi \in \F_\epsilon} |\cS^c(\phi)-\wh\cS^c(\phi)| \right] 
    \lesssim \frac{\psi(\epsilon)}{\sqrt{N}} + \frac{1}{N},
\end{equation*}
with a function $\psi(\epsilon)$ that vanishes as $\epsilon \downarrow 0$. 
We then aim to make the two terms on the right hand side comparable at order $N^{-\frac{1}{2 + \gamma}}$ (saving logarithmic factors) by controlling $\psi(\epsilon)$ through $\F_\epsilon$, more precisely defined as
\begin{align}\label{eq: localized class}
\mathcal{F}_{\epsilon}:=\{\phi \in \mathcal{F}\quad|\quad \|\nabla_{u}\phi-\nabla_{u}\phi^\dag\|_{L_{\eta}^2}\leq\epsilon\}\, .
\end{align}
The particular localization technique we use  is attributed to the seminal  papers of 
van de Geer \cite{van1987new,van2002m} and was also employed in \cite{chewi2024statistical,divol2025optimal,hutter2019minimax} 
although the technical aspects of the proof deviate significantly for us, especially in controlling the excess risk's localized empirical gap in \Cref{thm: centered semidual empirical estimation rate} in the conditional OT setting.

\paragraph{Proof of \Cref{thm: fast rate}}
 Let $\widetilde{\phi}$ be the best approximator of $\phi^\dagger$ in $\F$ as in \eqref{def:tilde-phi}. Once again we assume that this element exists and if it does not, then we simply take 
 $\widetilde \phi$ to be an element of a minimizing sequence and pass 
 to the limit in the end. For convenience of notation let us define the bias error
\begin{align*}
    \bias(\widetilde \phi) := \| \nabla_u \widetilde \phi - \nabla_u \phi^\dagger \|_{L_{\eta}^{2}}.
\end{align*}
By the triangle inequality we can then write 
\begin{align*}
    \|\nabla_u \widehat \phi - \nabla_u \phi^\dagger\|_{L_{\eta}^{2}} \leq 
    \|\nabla_u \widehat \phi - \nabla_u \widetilde \phi\|_{L_{\eta}^{2}} +  \|\nabla_u \phi^\dagger - \nabla_u \widetilde \phi\|_{L_{\eta}^{2}}
 \leq \|\nabla_u \widehat \phi - \nabla_u \widetilde \phi\|_{L_{\eta}^{2}} + \bias(\widetilde \phi),
\end{align*}
which resembles a bias-variance decomposition similar to our proof of \Cref{thm: slow rate}.
Our main task is to control the first term, i.e., the variance term.

\noindent
{\it A high probability bound:} Consider the element $\phi_\sigma \in \F$ defined as 
\[
\phi_{\sigma} := (1-\lambda)\widetilde\phi + \lambda\widehat{\phi}, \quad\lambda := \frac{\sigma}{\sigma + \|\nabla_{u} \widehat \phi - \nabla_{u}\widetilde\phi\|_{L_{\eta}^{2}}}.
\] 
Here  $\sigma >0 $ is a parameter to be selected later. 
A direct calculation then shows that 
\begin{align*}
   \| \nabla_u \phi_\sigma - \nabla_u \widetilde \phi \|_{L_{\eta}^{2}}  
   = \lambda \| \nabla_u \widehat \phi - \nabla_u  \widetilde \phi \|_{L_{\eta}^{2}} 
   = \frac{\sigma}{\sigma + \|  \nabla_u  \widetilde \phi - \nabla_u \widehat \phi \|_{L_{\eta}^{2}} } 
   \| \nabla_u \widehat \phi - \nabla_u  \widetilde \phi \|_{L_{\eta}^{2}},
\end{align*}
implying, in turn, that 
\begin{align*}
    \| \nabla_u  \widehat \phi - \nabla_u \widetilde \phi \|_{L_{\eta}^{2}} \leq \sigma
    \quad \text{whenever} \quad 
    \| \nabla_u  \phi_\sigma - \nabla_u \widetilde \phi \|_{L_{\eta}^{2}} \leq \frac{\sigma}{2}.
\end{align*}
This calculation allows us to work with $\phi_\sigma$ rather than $\widehat{\phi}$ which is  
 helpful because both $\widetilde{\phi}$ and $\phi_\sigma$ belong to $\F_\epsilon$ 
whenever $\epsilon = \bias(\widetilde \phi) + \sigma$. This is the essence of 
the localization argument since we can now focus on obtaining a high-probability 
bound on $\| \nabla_u \phi_\sigma  - \nabla_u  \widetilde\phi \|_{L^2_\eta}$ 
while working with $\F_\epsilon$ rather than $\F$. 

Indeed by the triangle inequality and an application of \Cref{cor: map stability} 
we have 
\begin{equation}\label{eq:phi-sigma-bound-interim}
\begin{aligned}
 \| \nabla_u   \phi_\sigma - \nabla_u \widetilde \phi \|_{L_{\eta}^{2}} 
 &\leq \| \nabla_u   \phi_\sigma - \nabla_u \phi^\dagger \|_{L_{\eta}^{2}} + \| \nabla_u \widetilde \phi - \nabla_u \phi^\dagger\|_{L_{\eta}^{2}} \\
 & \leq  \sqrt{2\betamax \mathcal{S}^c(\phi_\sigma)} + \bias(\widetilde \phi).
\end{aligned}
\end{equation}
Towards controlling the first term, we can write 
\begin{align*}
    \mathcal{S}^{c}(\phi_{\sigma}) - \mathcal{S}^{c}(\widetilde\phi) 
    &=  [\mathcal{S}^{c}(\phi_{\sigma}) - \widehat{\mathcal{S}}^{c}(\phi_{\sigma})] + [\widehat{\mathcal{S}}^{c}(\phi_{\sigma}) - \widehat{\mathcal{S}}^{c}(\widetilde\phi)] + [\widehat{\mathcal{S}}^{c}(\widetilde\phi) - \mathcal{S}^{c}(\widetilde\phi)]
    \\
    &\leq 2 \sup_{\phi \in \mathcal{F}_{\epsilon}}|(\mathcal{S}^{c} - \widehat{\mathcal{S}}^{c})(\phi)| + [\widehat{\mathcal{S}}(\phi_{\sigma}) - \widehat{\mathcal{S}}(\widetilde\phi)]
    \\
    &\leq 2\sup_{\phi \in \mathcal{F}_{\epsilon}}|(\mathcal{S}^{c} - \widehat{\mathcal{S}}^{c})(\phi)| + \lambda[\widehat{\mathcal{S}}(\widehat\phi) - \widehat{\mathcal{S}}(\widetilde\phi)]
    \\
    &\leq 2\sup_{\phi \in \mathcal{F}_{\epsilon}}|(\mathcal{S}^{c} - \widehat{\mathcal{S}}^{c})(\phi)|,
\end{align*}
where the second to last inequality follows from the fact that $\widehat{\mathcal{S}}$
is convex with respect to $\phi$ while the last inequality is a consequence of 
the fact that $\widehat{\phi}$ minimizes $\widehat{\mathcal{S}}$.

The following proposition allows us to further bound the last display above. 
The proof is postponed to the end of this section.
\begin{proposition}
\label{thm: centered semidual empirical estimation rate}
    Suppose \Cref{assump: slow rate,assump: fast rate} hold. Then, 
    with probability at least $1 - \exp(t)$ for any $t \ge 0$, it holds that
    \begin{equation}
    \label{eq:empirical process bound}
    \sup_{\phi\in\mathcal{F}_{\epsilon}} |(\mathcal{S}^{c} - \widehat{\mathcal{S}}^{c})(\phi)| \lesssim \theta(N,\epsilon,t),
    \end{equation}
 where 
\begin{equation}\label{eq:theta}
    \theta(N,\epsilon,t):=\frac{\epsilon}{\sqrt{N}}\sqrt{\CPI} \left(\sqrt{C_\F}\epsilon^{-\gamma/2} \sqrt{\log(1+\epsilon^{-1})} +   \sqrt{t} \right)+ \frac{1}{N}R(C_\F(1-\gamma)^{-2} + t)\, ,
\end{equation}
and
\editstart
$\CPI: = \sqrt{\frac{\betamax}{\alphamin}}\left(\sqrt{C_{PI}^{\eta_\U}}+\sqrt{C_{PI}^{\nu(\cdot\mid\Y)}}+2L_\F\sqrt{C_{PI}^{\nu_\Y,\widetilde{q}}}\right)$.
\editend
\end{proposition}
Combining this result with \Cref{cor: map stability} further yields that 
\begin{align}\label{eq:Sc-bound}
    \mathcal{S}^c(\phi_\sigma) 
    \leq \mathcal{S}^c (\widetilde \phi) + \theta(N, \epsilon,t) 
    \leq \frac{1}{2\alphamin}\bias(\widetilde{\phi})^2 + \theta(N, \bias(\widetilde{\phi}) + \sigma,t),
\end{align}
with probability $1 - e^{-t}$.  Combining~\eqref{eq:Sc-bound} and~\eqref{eq:phi-sigma-bound-interim} yields 
\begin{align*}
 \| \nabla_u   \phi_\sigma - \nabla_u \widetilde \phi \|_{L^2_\eta} &\leq  
 \sqrt{2\betamax \theta(N,\bias(\widetilde\phi) + \sigma,t)} + \left(\sqrt{\frac{\betamax}{\alphamin}}+1\right)
 \bias(\widetilde{\phi}),
\end{align*}
and as a result  we infer that 
\begin{align*}
    \| \nabla_u  \widehat \phi - \nabla_u \widetilde \phi \|_{L_{\eta}^{2}} \leq \sigma \, ,
\end{align*}
with probability  $1-e^{-t}$
for all $\sigma$ satisfying
\begin{align}\label{eq:sigma-condition}
  \sqrt{2\betamax \theta(N,\bias(\widetilde{\phi}) + \sigma,t)} + \left(\sqrt{\frac{\betamax}{\alphamin}}+1 \right)
  \bias(\widetilde{\phi}) \leq \frac{\sigma}{2}.
\end{align}
Since $\sigma$ was arbitrary, it remains for us to check whether an appropriate choice 
is possible \footnote{This  is the one-shot localization step of our proof since we 
identify the choice of $\sigma$ (equivalently $\epsilon$) in terms of $N$ in a 
single step to achieve our fast rate later.}

\editstart
{\it Selection of $\sigma$:} 
Next, by the calculation provided in~\Cref{lem: sigma calculation}, the ansatz 
\begin{align}\label{eq: ansatz sigma}
\sigma = C_\textnormal{bias}\textnormal{bias}(\widetilde{\phi}) + C_{\sigma,1}\left(\frac{\log(N)}{N}\right)^{\frac{1}{2+\gamma}}+C_{\sigma,2}\sqrt{\frac{t+1}{N}}\, ,
\end{align}
satisfies the inequality condition of~\eqref{eq:sigma-condition} with
\begin{align}\label{eq: constants identification}
\begin{split}
C_{\textnormal{bias}} &\geq 4\left(\sqrt{\frac{\betamax}{\alphamin}}+1\right)\, ,\qquad C_{\sigma,1}\geq (128\sqrt{6})^{\frac{2}{2+\gamma}}(\betamax^2 \CPI C_\F)^{\frac{1}{2+\gamma}} \, ,\\
C_{\sigma,2} &\geq \left(256\betamax\sqrt{\CPI} \vee8\right)8\frac{\sqrt{R C_\F}}{(1-\gamma)}\, .
\end{split}
\end{align}
With this selection, so far we have established that
\begin{equation*}
    \| \nabla_u \widehat{\phi} - \nabla_u \widetilde{\phi} \|_{L^2_\eta}
    \le C_{\textnormal{bias}} \| \nabla_u \widetilde{\phi} - \nabla_u \phi^\dagger \|_{L^2_\eta} + C_{\sigma,1}\left(\frac{\log(N)}{N}\right)^{\frac{1}{2+\gamma}}+C_{\sigma,2}\sqrt{\frac{t+1}{N}}\, ,
\end{equation*}
with probability $ 1 - e^{-t}$.

\begin{lemma}\label{lem: sigma calculation}
The ansatz in~\eqref{eq: ansatz sigma} satisfies the inequality condition of~\eqref{eq:sigma-condition} with constant dependencies as in~\eqref{eq: constants identification}.
\end{lemma}
\begin{proof}
For simplicity, assume $N\geq e$ so that $\log(N)\geq 1$ in the calculation below. The inequality condition in~\eqref{eq:sigma-condition} can be rewritten as 
\begin{align*}
\sqrt{\mathrm{I}+\mathrm{II}+\mathrm{III}+\mathrm{IV}}+\mathrm{V}\leq \frac{\sigma}{2}\, ,
\end{align*}
introducing in compact notation
\begin{align*}
\mathrm{I}&\coloneqq2\betamax \sqrt{\frac{\CPI C_\F}{N}}\left(\textnormal{bias}(\widetilde \phi)+\sigma\right)^{1-\gamma/2}\sqrt{\log\left(1+\left(\textnormal{bias}(\widetilde{\phi})+\sigma\right)^{-1}\right)}\, ,\\
\mathrm{II}&\coloneqq 2\betamax\sqrt{\frac{\CPI}{N}}\left(\textnormal{bias}(\widetilde \phi)+\sigma\right)\sqrt{t}\, ,\qquad \mathrm{III}\coloneqq \frac{R C_\F}{N(1-\gamma)^2}\, ,\qquad
\mathrm{IV}\coloneqq \frac{t}{N}\, ,\\
\mathrm{V}&\coloneqq \left(\sqrt{\frac{\alphamin}{\betamax}+1}\right)\textnormal{bias}(\widetilde \phi)\, .
\end{align*}
Straightforwardly, the ansatz in~\eqref{eq: ansatz sigma} ensures that $\mathrm{II},\mathrm{III},\mathrm{IV}$ are bounded by $\sigma^2/64$ and that $\mathrm{V}\leq \sigma/4$. To bound $\mathrm{I}$, first note that
\begin{align*}
\left(\textnormal{bias}(\widetilde{\phi})+\sigma\right)^{1-\gamma/2}\leq (2\sigma)^{1-\gamma/2}\leq 2\sigma^{1-\gamma/2}\, .
\end{align*}
Thus, using the inequality $\log(1+y)\leq 1+\left(0 \vee \log(y)\right)$, we can further bound
\begin{align*}
\log\left(1+\left(\textnormal{bias}(\widetilde \phi)+\sigma\right)^{-1}\right)&\leq 1 + \left(0 \vee\log\left((\textnormal{\bias}(\widetilde\phi)+\sigma)^{-1}\right)\right)\\
&\leq 1 + \left(0\vee \log\left(\sigma^{-1}\right)\right) \leq 1+ \log\left(\frac{N^{\frac{1}{2+\gamma}}}{C_\sigma}\right)\\
&\leq \log(N)\left(1+\frac{1}{2+\gamma}\right)\leq \frac{3}{2}\log(N)\, .
\end{align*}
Thus, the ansatz of~\eqref{eq: ansatz sigma} ensures $\mathrm{I}\leq \sigma^2/64$. Combining all terms verifies~\eqref{eq:sigma-condition}. 
{}
\end{proof}

{\it The expectation bound:}
With the high-probability bound at hand we now move on to obtaining an expectation 
bound. 
This part of the proof is essentially a calculation that relates our tail bound to
an expectation bound. 

So far we have shown that, with probability $1 - e^{-t}$, we have the bound 
\begin{align*}
\|\nabla_u \widehat \phi - \nabla_u \phi^\dagger\|_{L_{\eta}^{2}}^2\leq 2\left((C_{\textnormal{bias}}+1)^2\| \nabla_u \widetilde{\phi} - \nabla_u \phi^\dagger \|_{L^2_\eta}^2+C_{\sigma,1}^2\left(\frac{\log(N)}{N}\right)^{\frac{2}{2+\gamma}}+C_{\sigma,2}^2\frac{t+1}{N}\right)\, .
\end{align*}
Then, by Fubini-Tonelli theorem, we can decompose 
\begin{align*}
\mathbb{E}\|\nabla_u \widehat \phi - \nabla_u \phi^\dagger\|_{L_{\eta}^{2}}^2  =\int_{0}^A \mathbb{P}\left(\|\nabla_u \widehat \phi - \nabla_u \phi^\dagger\|_{L_{\eta}^{2}}^2>x\right)\d x + \int_{A}^\infty \mathbb{P}\left(\|\nabla_u \widehat \phi - \nabla_u \phi^\dagger\|_{L_{\eta}^{2}}^2>x\right)\d x\, ,
\end{align*}
where we set
\begin{align*}
A&\coloneqq 2\left((C_{\textnormal{bias}}+1)^2\| \nabla_u \widetilde{\phi} - \nabla_u \phi^\dagger \|_{L^2_\eta}^2+C_{\sigma,1}^2\left(\frac{\log(N)}{N}\right)^{\frac{2}{2+\gamma}}+C_{\sigma,2}^2\frac{1}{N}\right)\, .
\end{align*}
We additionally let $B= \frac{2C_{\sigma,2}^2}{N}$. Next, we can trivially upper bound
\begin{align*}
\int_{0}^A \mathbb{P}\left(\|\nabla_u \widehat \phi - \nabla_u \phi^\dagger\|_{L_{\eta}^{2}}^2>x\right)\d x\leq A\, ,
\end{align*}
while the second term is bounded after a change of variables as follows:
\begin{align*}
\int_{A}^\infty \mathbb{P}\left(\|\nabla_u \widehat \phi - \nabla_u \phi^\dagger\|_{L_{\eta}^{2}}^2>x\right)\d x &= B\int_{0}^\infty \mathbb{P}\left(\|\nabla_u \widehat \phi - \nabla_u \phi^\dagger\|_{L_{\eta}^{2}}^2>A+Bt\right)\d t\\
&\leq B\int_{0}^\infty \exp(-t)\d t = B\, .
\end{align*}
Thus, substituting $A$, $B$ and~\eqref{eq: constants identification} and combining the integral bounds
\begin{align*}
\mathbb{E}^{\textnormal{train}}\|\nabla_u \widehat \phi - \nabla_u \phi^\dagger\|_{L_{\eta}^{2}}^2&\leq 2(C_{\textnormal{bias}}+1)^2\| \nabla_u \widetilde{\phi} - \nabla_u \phi^\dagger \|_{L^2_\eta}^2+2C_{\sigma,1}^2\left(\frac{\log (N)}{N}\right)^{\frac{2}{2+\gamma}}+\frac{4C_{\sigma,2}^2}{N}\\
& \lesssim \frac{\betamax}{\alphamin}\| \nabla_u \widetilde{\phi} - \nabla_u \phi^\dagger \|_{L^2_\eta}^2+ C_{\textnormal{est.}}\left(\left(\frac{C_\F\log(N)}{N}\right)^{\frac{2}{2+\gamma}} \vee \frac{C_\F}{N}\right)\, ,
\end{align*}
for a suitable parameters-independent proportionality constant and $C_{\textnormal{est.}} = C_{\sigma,1}^2 \vee C_{\sigma,2}^2$, which 
concludes our proof. 

\begin{remark}\label{rem:C_sigma_claculation}
Note that the above calculation characterizes 
constant $C_\textnormal{est.} = C_{\sigma,1}^2 \vee C_{\sigma,2}^2$  in terms of the relevant constants pertaining the measures involved in the conditional OT problem setup and the hypothesis class of candidate potentials $\F$. In fact, by~\eqref{eq: constants identification} and the expression for $\CPI$ in~\Cref{thm: centered semidual empirical estimation rate}, it is possible to observe that $C_\textnormal{est.}$ scales at most linearly with $R, \frac{1}{1-\gamma^2}, \betamax\sqrt{\frac{\betamax}{\alphamin}}, \sqrt{C_{PI}^{\eta_\U}}, \sqrt{C_{PI}^{\nu(\cdot\mid\Y)}},\sqrt{C_{PI}^{\nu_\Y,\widetilde{q}}}$ and $L_\F$. 

Moreover, as introduced in~\Cref{rem: where to adapt proof}, we recall the dependence on $C_{\F}$, in the case $\gamma = 0$, can be sharpened to be with respect to $C_\F^{1-\frac{2}{d}}$ where $d$ is the ambient dimension of $\Y\times \U$ under additional smoothness assumptions and $L_2$-metric entropy bounds following the more technical chaining calculations of the proof of~\cite[Prop.~4.2]{divol2025optimal}. Improving the dependence in terms of $C_\F^{1-\frac{2}{d}}$ is important in order to derive optimal global rates by balancing the bias and variance error of the derived rate for large parametric classes that are characterized by a usually polynomial dependence of $C_\F$ with respect to $N$ to achieve a needed approximation power. We provide an example of such a class in~\Cref{sec:conditionalOT:applications}.
\end{remark}
\editend

\paragraph{Proof of \Cref{thm: centered semidual empirical estimation rate}}
We now present the proof of \Cref{thm: centered semidual empirical estimation rate} 
which relies on an auxiliary technical results that allows us to 
control various errors involving potentials $\phi$ and the ground truth $\phi^\dagger$
provided that their gradients are sufficiently close. Naturally, this 
result relies heavily on Poincar\'e inequalities.

\begin{proposition}
\label{thm: PI extension consequence for COT}
Under \Cref{assump: fast rate} the following inequalities hold for all $\phi \in \mathcal{F}$ such that $\| \nabla_u \phi - \nabla_u \phi^\dag\|_{L_{\eta}^2}\leq\epsilon$
for a constant $\epsilon > 0$:
\begin{align}
    &\quad\|\phi-\phi^{\dagger}-\eta_\U(\phi(y,\cdot)-\phi^{\dagger}(y,\cdot))\|_{L_{\eta}^2}\leq \sqrt{C_{PI}^{\eta_\U}}\epsilon \label{eq: PI extention eta_U}\, ,
    \\
    &\quad \|\phi^*-\phi^{\dagger*}-\nu(\cdot|y)(\phi^*(\cdot,u)-\phi^{\dagger*}(\cdot,u))\|_{L_{\nu}^2}\leq 
    \sqrt{\frac{\betamax C_{PI}^{\nu(\cdot\mid \Y)}}{\alphamin}}\epsilon \label{eq: PI extention nu(u|y)}\, ,
    \\
    &\quad \|F(y) - \nu_\Y(F)\|_{L_{\nu_\Y}^2}
    \leq 2L \sqrt{\frac{C_{PI}^{\nu_\Y,\widetilde{q}}\, \betamax}{\alphamin}}\epsilon\, ,
    \label{eq: PI extention eta_U + nu(u|y)}
    \\
    \nonumber
    &\quad \textnormal{where\,} F(y): = \eta_\U(\phi(y,\cdot)-\phi^\dagger(y,\cdot))+\nu(\cdot|y)(\phi^*(y,\cdot)-\phi^{\dagger*}(y,\cdot)). 
\end{align}
\end{proposition}

\begin{proof}
We first prove \cref{eq: PI extention eta_U}. By the Poincar\'e inequality for $\eta_{\U}$,
\begin{align*}
\|\phi-\phi^\dag-\eta_\U(\phi(y,\cdot)-\phi^\dag(y,\cdot))\|_{L_{\eta}^2}^2&=\int\|\phi(y,\cdot)-\phi(y,\cdot)^\dag-\eta_\U(\phi(y,\cdot)-\phi^\dag(y,\cdot))\|_{L_{\eta_\U}^2}^2d\nu_\Y(y)\\
&\leq C_{PI}^{\eta_\U}\int \|\nabla_u\phi(y,\cdot)-\nabla_u\phi^\dag(y,\cdot)\|_{L_{\eta_\U}^2}^2d\nu_\Y(y)\\
&=C_{PI}^{\eta_\U}\|\nabla_u\phi-\nabla_u\phi^\dag\|_{L_{\eta}^2}^2\leq C_{PI}^{\eta_\U}\epsilon^2.
\end{align*}

Next, we prove \cref{eq: PI extention nu(u|y)}. By the Poincar\'e inequality for $\nu(\cdot \mid y)$,
\begin{align*}
    \|\phi^*-\phi^{\dag*}-\nu(\cdot|y)&(\phi^*(\cdot,u)-\phi^{\dag*}(\cdot,u))\|_{L_{\nu}^2}^2 
    \\
    &= \int\int \|\phi^*-\phi^{\dag*}-\nu(\cdot|y)(\phi^*(\cdot,u)-\phi^{\dag*}(\cdot,u))\|^2 \d\nu(u\mid y)\d\nu_{Y}(y) \\
    &\leq \int C_{\textnormal{PI}}^{\nu(\cdot \mid y)} \int \|\nabla_{u} \phi^{*}(y,u) -\nabla_{u} \phi^{\dagger*}(y,u)\|^2 \d\nu(u\mid y)\d\nu_{\Y}(y)
    \\
    &\leq C_{\textnormal{PI}}^{\nu(\cdot \mid \Y)}\|\nabla_{u} \phi^{*}-\nabla_{u} \phi^{\dagger*}\|_{L_{\nu}^2}^2\\
    &\leq C_{\textnormal{PI}}^{\nu(\cdot \mid \Y)} \frac{\betamax}{\alphamin}\|\nabla_{u} \phi-\nabla_{u} \phi^{\dagger}\|_{L_{\nu}^2}^2\leq \frac{C_{\textnormal{PI}}^{\nu(\cdot \mid \Y)}\betamax}{\alphamin}\epsilon^2.
\end{align*}
where we used~\Cref{cor:conjugate-norm-bound} in the last step.
Finally, we prove \cref{eq: PI extention eta_U + nu(u|y)}.
The 
\editstart 
strengthened Poincar\'e 
\editend
inequality for $\nu_{\Y}$ implies
\editstart
\begin{align*}
\|F\|_{L_{\nu_\Y}^2}^2\leq  C_{PI}^{\nu_\Y,\widetilde{q}}\|\nabla_y F\|_{L_{\nu_\Y}^2(\langle\cdot\rangle^{-\widetilde{q}})}^2\, .
\end{align*}
\editend
Therefore, the proof follows by bounding 
\editstart
$\|\nabla_y F\|_{L_{\nu_\Y}^2(\langle\cdot\rangle^{-\widetilde{q}})}^2$.
\editend
By the Leibniz's integral rule, 
\begin{align*}
\nabla_y F  = \text{I}+\text{II}\, ,
\end{align*}
where we introduced
\begin{align*}
&\text{I}:=\int\nabla_y\Bigr( \phi(y,u)-\phi^\dag(y,u)\Bigr) d \eta_{\U}(u),
\\
&\text{II}:=\nabla_y \int \Bigr(\phi^*(y,u)-\phi^{\dag*}(y,u)\Bigr) d \nu(u \mid y).
\end{align*}
Recalling that
$\nabla_u\phi^{\dag}(y,\cdot)\#\eta_{\U}=\nu(\cdot \mid y)$, we can express the second term as
\begin{align*}
\text{II} =  \nabla_y &\int \Bigr(\phi^*(y,\nabla_u\phi^\dag(y,u))-\phi^{\dag*}(y,\nabla_u\phi^\dag(y,u))\Bigr) \d \eta_\U(u) 
 =\text{III}-\text{IV},
\end{align*}
where 
\begin{align*}
&\text{III}:=\int \nabla_y\Bigr[\phi^*\Bigr(y,\nabla_u\phi^\dag(y,u)\Bigr)\Bigr]d\eta_\U(u)\, ,\\
&\text{IV}:=\int\nabla_y \Bigr[\phi^{\dag*}\Bigr(y,\nabla_u\phi^\dag(y,u)\Bigr)\Bigr] d\eta_\U(u)\, .
\end{align*}
The envelope theorem implies $\nabla_y \phi^*(y,w) = -\nabla_y \phi(y,\nabla_u \phi^*(y,w))$, for all $w$. Therefore, 
\begin{align*}
\text{III}&=\int -\nabla_y\phi\Bigr(y, \nabla_u\phi^{*}\Bigr(y,\nabla_u\phi^\dag(y,u)\Bigr)\Bigr)+\nabla_{u}\phi^*(y, \nabla_u\phi^\dag(y,u) )\nabla_{yu}\phi^\dag(y,u)d\eta_\U(u).
\end{align*}
Similarly,
\begin{align*}
\text{IV}&=\int -\nabla_y\phi^\dag\Bigr(y,\nabla_u\phi^{\dag*}\Bigr(y,\nabla_u\phi^\dag(y,u)\Bigr)\Bigr)+\nabla_{u}\phi^{\dag*}(y, \nabla_u\phi^\dag(y,u) )\nabla_{yu}\phi^\dag(y,u)d\eta_\U(u)\\
&=\int -\nabla_y\phi^\dag(y,u)+\nabla_{u}\phi^{\dag*}(y, \nabla_u\phi^\dag(y,u) )\nabla_{yu}\phi^\dag(y,u)d\eta_\U(u).
\end{align*}
where the last step follows from the fact that $\nabla_u\phi^\dag(y,\cdot)^{-1}=\nabla_u\phi^{\dag*}(y,\cdot)$. Collecting the terms, we can express $\nabla_y F(y) =\text{I} + \text{II} = \text{I} + \text{III}- \text{IV} = \text{V} +  \text{VI}$ where
\begin{align*}
&\text{V}:=\int \nabla_y\phi(y,u)-\nabla_y\phi\Bigr(y,\nabla_u\phi^{*}\Bigr(y,\nabla_u\phi^\dag(y,u)\Bigr)\Bigr)d\eta_\U(u),
\\
&\text{VI}:=\int \nabla_{u}\phi^{*}(y, \nabla_u\phi^\dag(y,u) )\nabla_{yu}\phi^\dag(y,u)-\nabla_{u}\phi^{\dag*}(y, \nabla_u\phi^\dag(y,u) )\nabla_{yu}\phi^\dag(y,u)d\eta_\U(u).
\end{align*}
By change of measure $\nabla_u\phi^{\dag*}(y,\cdot)\#\nu(\cdot|y) =\eta_{\U}$, we can rewrite these as
\begin{align*}
&\text{V} = \int \nabla_y\phi(y,\nabla_u\phi^{\dag*}(y,u))-\nabla_y\phi(y,\nabla_u\phi^{*}(y,u))d\nu(u \mid y),
\\
& \text{VI}= \int \nabla_{u}\phi^{*}(y,u)\nabla_{yu}\phi^\dag(y,u)-\nabla_{u}\phi^{\dag,*}(y,u)\nabla_{yu}\phi^\dag(y,u)d\nu(u \mid y).
\end{align*}
By the boundedness assumption on the mixed derivatives of any potential $\phi\in\mathcal{F}$, it follows that 
\editstart
the maps $\nabla_y\phi(y,\cdot)$ are $L_\F\langle y\rangle^{\widetilde{q}}$-Lipschitz.
\editend
Using this fact and applying Jensen's inequality, we obtain
\editstart
\begin{align*}
\|\text{V}\|\leq L_\F\langle y \rangle^{\widetilde{q}} \int\|\nabla_u\phi^{\dag*}(y,u)-\nabla_u\phi^{*}(y,u)\|d\nu(u \mid y).
\end{align*}
\editend
Similarly, applying H\"older's inequality ($p=1$,$q=\infty$) together with the boundedness assumption on the mixed derivatives of $\phi^\dag$, we have
\editstart
\[\|\text{VI}\|\leq L_\F \langle y \rangle^{\widetilde{q}}\int\|\nabla_u\phi^{\dag*}(y,u)-\nabla_u\phi^{*}(y,u)\|d\nu(u \mid y).\]
\editend
Then, by the triangle inequality, we can bound:
\editstart
\begin{align*}
\|\nabla_y F\|_{L_{\nu_\Y}^2(\langle\cdot\rangle^{-\widetilde{q}})}^2&=\int \|\nabla_y F(y)\|^2 \langle y\rangle^{-2\widetilde{q}}d\nu_{\Y}\leq2\int \|\text{V}\|^2+\|\text{VI}\|^2 \langle y\rangle^{-2\widetilde{q}}d\nu_{\Y}(y)\\
&\leq 4L_\F^2 \int\left(\int\|\nabla_u\phi^{\dag*}(y,u)-\nabla_u\phi^{*}(y,u)\|^2 d\nu(u\mid y) \right)\langle y\rangle^{2\widetilde{q}} \langle y\rangle^{-2\widetilde{q}}d\nu_\Y(y) \\&
= 4L_\F^2\|\nabla_u\phi^{\dag*}-\nabla_u\phi^{*}\|_{L_{\nu}^2}^2\leq 4L^2\frac{\betamax}{\alphamin}\|\nabla_u\phi^\dag(y,\cdot)-\nabla_u\phi(y,\cdot)\|_{L_{\eta}^2}^2\, ,
\end{align*}
\editend
where we used{\color{black}~\eqref{cor:conjugate-norm-bound}} in the last step. Finally, the 
\editstart 
strengthened Poincar\`e inequality
\editend
for $\nu_{\Y}$ yields 
\editstart
\begin{align*}
\|F\|_{L_{\nu_\Y}^2}^2\leq C_{PI}^{\nu_\Y}\|\nabla_y F\|_{L_{\nu_\Y}^2(\langle\cdot\rangle^{-\widetilde{q}})}^2\leq 4L_\F^2C_{PI}^{\nu_\Y,\widetilde{q}} \frac{\betamax}{\alphamin} \epsilon^2\, ,
\end{align*}
\editend
which completes the proof. 
{}
\end{proof}

We are now ready to present the proof of \cref{thm: centered semidual empirical estimation rate}.
\begin{proof}[Proof of \Cref{thm: centered semidual empirical estimation rate}]
We begin with the decomposition
\begin{align}
\label{eq: empirical process bound decomposition}
\sup_{\phi\in\mathcal{F}_{\epsilon}} |(\mathcal{S}^{c} - \widehat{\mathcal{S}}^{c})(\phi)| &\leq \sup_{\phi\in\mathcal{F}_{\epsilon}} |(\eta - \eta^N)(\phi-\phi^{\dagger}) + (\nu-\nu^N)(\phi^{*}-\phi^{\dagger *})|
    \\
    \nonumber
    &\leq \text{I}+\text{II}+\text{III}\, ,
\end{align}
where we recall our shorthand notation $\eta^N, \nu^N$ for the empirical 
measures associated to i.i.d. samples from $\eta, \nu$, and 
where we introduced the terms:
\begin{align*}
&\text{I}:= \sup_{\phi \in \mathcal{F}_{\epsilon}} \Bigl\{ |(\eta^{N} - \eta)(\phi - \phi^{\dagger}-\eta_\U(\phi(y,\cdot)-\phi^\dag(y,\cdot))) |\Bigr\}, \\
&\text{II}:= \sup_{\phi \in \mathcal{F}_{\epsilon}}\Bigr\{| (\nu^{N} - \nu)(\phi^{*} - \phi^{\dagger *}-\nu(\cdot|y)(\phi^*(y,\cdot)-\phi^{\dag*}(y,\cdot)))|\Bigr\},\\
&\text{III}:=\sup_{\phi \in \mathcal{F}_{\epsilon}} \Bigr\{| (\nu^{N}_\Y - \nu_\Y)(F(y) - \nu_\Y(F))|\Bigr\},
\end{align*}
and 
\begin{align}\label{eq: F expression}
    F(y)\coloneqq \eta_\U(\phi(y,\cdot)-\phi^\dagger(y,\cdot))+\nu(\cdot|y)(\phi^*(y,\cdot)-\phi^{\dagger*}(y,\cdot))\, .
\end{align}
Note that the above decomposition is reliant on the assumption that 
$\nu_\Y=\eta_\Y$.

In order to bound~\eqref{eq: empirical process bound decomposition}, we apply 
\Cref{thm: generalization for fast rate} to each individual term. 
Starting with (I), we have $\|\phi-\phi^{\dagger}-\eta_\U(\phi(y,\cdot)-\phi^{\dagger}(y,\cdot))\|_{L_{\eta}^2}\leq \sqrt{C_{PI}^{\eta_\U}}\epsilon$ for all $\phi \in \F_\epsilon$ due to~\eqref{eq: PI extention eta_U}. Therefore, the assumption of \Cref{thm: generalization for fast rate} holds with $\sigma=\sqrt{C_{PI}^{\eta_\U}} \epsilon$.  As a result, 
\begin{align*}
    \text{I}
    &\lesssim  J_N(\F) + \sqrt{C_{PI}^{\eta_\U}} \epsilon \sqrt{\frac{t}{N}} + R\frac{t}{N}
    \\
    &\lesssim  \frac{\epsilon}{\sqrt{N}}\sqrt{C_{PI}^{\eta_\U}} \left(\sqrt{C_\F}\epsilon^{-\gamma/2} \sqrt{\log(1+\epsilon^{-1})} +   \sqrt{t} \right)+ \frac{1}{N}R(C_\F(1-\gamma)^{-2} + t).
\end{align*}
with probability larger than $1-\exp(-\frac{t}{3})$, 
where we used the result of~\Cref{lem: bound-JF} to bound $J_N(\F)$ with $C=\sqrt{C_{PI}^{\eta_\U}}$.Here $\lesssim$ hides universal constants.
The bounds for II and III terms follow similarly, with $\sqrt{C_{PI}^{\eta_\U}}$ replaced by $\sqrt{\frac{\betamax C_{PI}^{\nu(\cdot\mid \Y)}}{\alphamin}}$ and $2L\sqrt{\frac{\betamax C_{PI}^{\nu_\Y,\widetilde{q}}}{\alphamin}}$. Combining the three terms, using the union bound on probabilities, and
\begin{align*}
    \sqrt{C_{PI}^{\eta_\U}}  + \sqrt{\frac{\betamax C_{PI}^{\nu(\cdot\mid \Y)}}{\alphamin}} + 2L\sqrt{\frac{\betamax C_{PI}^{\nu_\Y}}{\alphamin}} \leq \sqrt{\frac{\betamax}{\alphamin}}\left(\sqrt{C_{PI}^{\eta_\U}}+\sqrt{C_{PI}^{\nu(\cdot\mid\Y)}}+2L\sqrt{C_{PI}^{\nu_\Y,\widetilde{q}}}\right)=\CPI\, ,
\end{align*}
we conclude the proof. 
{}
\end{proof}

\begin{remark}\label{rem: conditional extension core}
\Cref{thm: PI extension consequence for COT} and \Cref{thm: centered semidual empirical estimation rate} constitute the main theoretical extension to the conditional OT setup from the proof of \cite{chewi2024statistical} for the standard OT 
case. {\color{black} In particular, \Cref{thm: centered semidual empirical estimation rate} provides the right empirical processes decomposing the localized difference of the semidual excess risk and its empirical counterpart when dealing with conditional Brenier maps rather than the Brenier map pushing jointly $\eta$ to $\nu$ as done in the standard OT case. This results in the need to control the empirical processes in~\eqref{eq: PI extention eta_U},~\eqref{eq: PI extention nu(u|y)}, and~\eqref{eq: PI extention eta_U + nu(u|y)} by controlling the second moments of their elements (\Cref{thm: PI extension consequence for COT}). This is a crucial condition to meet to apply the chaining bound of~\Cref{thm: generalization for fast rate} on each of the three individuated empirical processes. These  steps constitute the adaptation of the proof of~\cite{chewi2024statistical} to the conditional setting and they are independent from the chaining bounds and $\sigma$-selection later adopted to complete the derivation of the rates. For this reason, the more technical chaining calculations discussed in~\Cref{rem: where to adapt proof} (and the consequent modification of the $\sigma$-selection argument) adopted in~\cite{divol2025optimal} are expected to be compatible and may provide an avenue to generalize our results.}
\end{remark}
We conclude the section with a well-known lemma (see for example
\cite[Prop. 3.14]{chewi2024statistical}) on 
the preservation of Poincar\'e inequalities under Lipschitz transformations that can be used to control $C_{PI}^{\nu(\cdot\mid\Y)}\leq \betamax C_{PI}^{\eta_\U}$ whenever we further assume that $\phi^{\dagger}(y,\cdot)$ is $\betamax$-smooth for $\nu_\Y$ a.e. $y\in \Y$.

\begin{lemma}\label{lem: PI for conditionals}
Suppose $\mu$ satisfies the Poincar\'e inequality with constant 
$C_{\textnormal{PI}}^\mu>0$ and let $T$ be a $\beta$-Lipschitz map.
Then, the pushforward $T\# \mu$ satisfies the Poincar\'e inequality with constant $C_{\textnormal{PI}}^{T\#\mu}\leq \beta^2 C_{\textnormal{PI}}^\mu$.
\end{lemma}

\begin{proof}
By definition of the pushforward, we have for any smooth $\psi$,
    $\text{Var}_{T\#\mu}(\psi) = \text{Var}_{\mu}(\psi \circ T)$.
By the Poincar\'e inequality for $\mu$ and chain rule we get
$\text{Var}_{\mu}(\psi \circ T)\leq C_{\text{PI}}^{\mu} \int \|\nabla T(u)\nabla\psi(T(u))\|^2\d\mu(u)$.
 By the hypothesis that $T$ is $\beta$-Lipschitz we have that 
 $\|\nabla T(u) \|\leq \beta$ a.e. and so 
\begin{align*}
    \text{Var}_{T\#\mu}(\psi(u))&\leq \beta^2C_{\text{PI}}^{\mu}\int\|  \nabla\psi(T(u))\|^2 \d\mu(u) = \beta^2 C_{\text{PI}}^{\mu} \int\|  \nabla\psi(u)\|^2 \d(T\#\mu)(u).
\end{align*}
\end{proof}

\begin{table}[htp]
\color{black}
\centering
\footnotesize
\renewcommand{\arraystretch}{1.25}
\setlength{\tabcolsep}{4pt}
\begin{tabular}{|p{0.17\textwidth}|p{0.38\textwidth}|p{0.38\textwidth}|}
\hline
 & \textbf{Slow rate} & \textbf{Fast rate} \\
\hline
\textbf{Quantity bounded}
& $\mathbb{E}\|\nabla_u\wh\phi - \nabla_u\phi^\dagger\|^2_{L^2_\eta}$ (squared $L^2_\eta$ gradient error)
& Same \\
\hline
\textbf{Variance-term bound}
& $R\,\betamax\sqrt{C_\F/N}$
& $C_{\textnormal{est.}}\,(C_\F\log N / N)^{2/(2+\gamma)}$ \\
\hline
\textbf{Squared rate in $N$} (for $\gamma\in[0,1)$)
& $N^{-1/2}$
& $N^{-2/(2+\gamma)} \in [N^{-1}, N^{-2/3})$ \\
\hline
\textbf{Where $\gamma$ enters}
& Dudley-integral constant only, bounded uniformly by $4R$ over $\gamma\in[0,1)$ via~\Cref{lem: prop const slow rate}. No blow-up as $\gamma\to 1$.
& (i) Rate exponent $\tfrac{2}{2+\gamma}$; (ii) multiplicatively inside $C_{\textnormal{est.}}$ via the $(1-\gamma)^{-2}$ factor from~\Cref{lem: bound-JF}; the channel through which the un-truncated chaining argument breaks as $\gamma\to 1$. \\
\hline
\textbf{Where $C_\F,\CPI,L_\F$ enter}
& $C_\F, L_\F$ multiplicatively in the variance term. 
& $C_\F, \CPI, L_\F$ all multiplicatively . Linear dependence with respect to $C_\F$ can be sharpened to $C_\F^{1-\frac{2}{d}}$ (cf.~\Cref{rem:C_sigma_claculation}). \\
\hline
$\gamma \in [0,1)$
& Squared rate $(\frac{C_\F}{N})^{1/2}$ (\Cref{thm: slow rate}) .
& Squared rate $(\frac{C_\F}{N})^{2/(2+\gamma)}$ (\Cref{thm: fast rate}). \\
\hline
$\gamma \in [1, 2)$
& Squared rate stays $(\frac{C_\F}{N})^{1/2}$. The proof of~\Cref{thm: slow rate} extends but with extra $(2-\gamma)^{-1}$ constant blowing up as $\gamma\to 2$. 
&
Squared rate $\lesssim (\frac{C_\F}{N})^{2/(2+\gamma)} \vee (\frac{C_\F}{N})^{1/\gamma} = (\frac{C_\F}{N})^{2/(2+\gamma)}$ for $C_\F\ll N$. Replace~\Cref{thm: generalization for fast rate} with the truncated form of \cite[Prop. A.2]{divol2025optimal-sm} ($\varepsilon = (\frac{C_\F}{N})^{1/\gamma}$, $\tilde\varepsilon = 0$). \\
\hline
$\gamma \ge 2$
& Squared rate $\lesssim (\frac{C_\F}{N})^{1/\gamma}$. The proof of~\Cref{thm: slow rate} extends with truncated chaining at $\delta\sim (\frac{C_\F}{N})^{1/\gamma}$; 
& Squared rate $\lesssim (\frac{C_\F}{N})^{1/\gamma}$ for $C_\F\ll N$; same sample complexity as slow rate. Localization no longer helps. Replace~\Cref{thm: generalization for fast rate} with fully-truncated \cite[Prop. A.2]{divol2025optimal-sm} ($\varepsilon=\tilde\varepsilon = (\frac{C_\F}{N})^{1/\gamma}$).\\
\hline
\end{tabular}
\caption{{\color{black}Summary of slow- and fast-rate behaviour for conditional Brenier map estimation. Rates bound the variance term of $\mathbb{E}\|\nabla_u\wh\phi - \nabla_u\phi^\dagger\|^2_{L^2_\eta}$ as a function of $N$.
}}
\label{tab:cond-ot-rates}
\end{table}

\editstart
\subsection{
Extending the proof of \Cref{thm: fast rate} (fast rate) to unbounded potentials}
\label{proof: thm: fast rate unbounded extension}
Next, we extend the proof of the fast rate for the case of unbounded potentials under~\Cref{assump: slow rate} and~\Cref{assump: fast rate}. To do so, we rely on the same preliminary lemmas used for the extension of the slow rate in~\Cref{proof: thm: slow rate unbounded extension} plus the following additional results.

We start by showing that the $L^2$-metric entropy of an hypothesis class can be controlled by the metric entropy with respect to the weighted uniform norm.
\begin{lemma}\label{lem: L2 to weighted uniform}
For $q_\Y,q_\U>0$ and weight function $w(y,u) = \langle y\rangle^{-q_\Y}\langle u\rangle^{-q_\U}$, let $\mu\in\mathcal{P}(\Y\times \U)$ satisfy the finite moment condition
\begin{align}\label{eq: finite moment condition}
\mu(w^{-2}) =\int \langle y\rangle^{2q_\Y}\langle u\rangle^{2q_\U}\d\mu(y,u)<\infty\, .
\end{align}
Then, it holds that
\begin{align}\label{eq: covering L2 to weighted}
\log \mathcal{N}\left(\delta,\F,\|\cdot\|_{L^2_\mu}\right)\leq \log \mathcal{N}\left(\mu(w^{-2})^{-1/2}\delta,\F,\|\cdot\|_{L^\infty(w)}\right)\, ,
\end{align}
\end{lemma}
\begin{proof}
For any two elements $\phi,\psi\in\F$,
\begin{align*}
\|\phi - \psi \|_{L^2_\mu}^2 &= \int \langle y\rangle^{2 q_\Y}\langle u\rangle^{2 q_\U}\langle y\rangle^{-2 q_\Y}\langle u\rangle^{-2 q_\U}\left|\phi(y,u) - \psi(y,u) \right|^2\d\mu(y,u)\\
&\leq \mu(w^{-2})\|\phi-\psi\|_{L^\infty(w)}^2\, .
\end{align*}
Thus, any $\mu(w^{-2})^{-1/2}\delta$-net of $\F$ in $\|\cdot\|_{L^\infty(w)}$ is also an $\delta$-net of $\F$ in $\|\cdot\|_{L^2_\mu}$, yielding the inequality in~\eqref{eq: covering L2 to weighted}. 
\end{proof}
\begin{remark}\label{rem: finite moments due poincare}
We note that in the case where $\mu$ has marginal $\mu_\Y$ and all conditionals $\mu(\cdot\mid y)$ satisfying the Poincaré inequality with constants $\CPI^{\mu_\Y}$ and $\CPI^{\mu(\cdot\mid \Y)}$ uniform in $y$, the finite moment condition of~\eqref{eq: finite moment condition} is automatically satisfied. Indeed, upon  disintegration
\begin{align*}
\mu(w^{-2})\leq \int\langle y\rangle^{2q_\Y}\d \mu_\Y(y)\sup_{y\in\Y}\int\langle u\rangle^{2q_\U}\d \mu(u\mid y)<\infty\, ,
\end{align*}
due to the fact that distributions satisfying the Poincaré inequality are subexponential 
\cite{bobkov1997poincare}
and therefore have finite moments of any order.
\end{remark}

Next, we provide a lemma that establishes the relation between the complexity of the class $\F$ and the one of a conditionally marginalized class
\begin{align}\label{eq: conditionally marginalized class}
\mathbb{E}_{\mu(\cdot\mid \Y)}(\F)\coloneqq \left \{y\to \mathbb{E}_{\mu(\cdot\mid \Y)}(\phi)(y)\}\mid \phi\in\F\right\}\, ,\qquad \mathbb{E}_{\mu(\cdot\mid \Y)}(\phi)(y)\coloneqq\mu(\cdot\mid y)(\phi(y,\cdot)) \, .
\end{align}
where the conditional distributions $\mu(\cdot\mid y)\in\mathcal{P}(\U)$ for any $y\in \Y$.
\begin{lemma}\label{lem: conditionally marginalized class complexity}
For $q_\U>0$ and weight function $w_\U(u) = \langle u\rangle^{-q_\U}$, let $\mu\in\mathcal{P}(\Y\times \U)$ satisfy the uniform in $\Y$ finite conditional moment condition
\begin{align}\label{eq: finite moment condition 2}
\mu(\cdot\mid \Y)(w_\U^{-1})\coloneqq\sup_{y\in\Y}\mu(\cdot\mid y)(w_\U^{-1}) =\sup_{y\in\Y}\int \langle u\rangle^{q_\U}\d\mu(u\mid y)<\infty\, .
\end{align}
Then, it holds that
\begin{align}\label{eq: covering weighted in y to weighted}
\log \mathcal{N}\left(\delta,\mathbb{E}_{\mu(\cdot\mid \Y)}(\F),\|\cdot\|_{L^\infty(w_\Y)}\right)\leq \log \mathcal{N}\left(\mu(\cdot\mid \Y)(w_\U^{-1})^{-1}\delta,\F,\|\cdot\|_{L^\infty(w)}\right)\, ,
\end{align}
where $w_\Y(y)=\langle y\rangle^{-q_\Y}$ and $w(y,u) = \langle y\rangle^{-q_\Y}\langle u\rangle^{-q_\U}$.
\end{lemma}
\begin{proof}
For any two elements $\phi,\psi\in\F$ and any fixed $y\in\Y$,
\begin{align*}
\left|\mathbb{E}_{\mu(\cdot\mid \Y)}(\phi)(y) - \mathbb{E}_{\mu(\cdot\mid \Y)}(\psi)(y)\right| &\leq \|\phi-\psi\|_{L^\infty(w)}\langle y\rangle^{q_\Y}\int \langle u \rangle^{q_\U}\d\mu(u\mid y)\\
&=\mu(\cdot\mid \Y)(w_\U^{-1})\|\phi-\psi\|_{L^\infty(w)}\langle y\rangle^{q_\Y}\, .
\end{align*}
Multiplying by $\langle y\rangle^{-q_\Y}$ and taking the supremum over $y$ yields
\begin{align*}
\|\mathbb{E}_{\mu(\cdot\mid \Y)}(\phi) - \mathbb{E}_{\mu(\cdot\mid \Y)}(\psi)\|_{L^\infty(\omega_\Y)}\leq \mu(\cdot\mid \Y)(w_\U^{-1})\|\phi-\psi\|_{L^\infty(w)}\, .
\end{align*}
Thus, we have that for any $\mu(\cdot\mid \Y)(w_\U^{-1})^{-1}\delta$-net $\{\phi_1,...,\phi_n\}$ of $\F$ in $\|\cdot\|_{L^\infty(w)}$, the corresponding net $\{\mathbb{E}_{U\sim\mu(\cdot\mid \Y)}(\phi_1),...,\mathbb{E}_{U\sim\mu(\cdot\mid \Y)}(\phi_n)\}$ is a $\delta$-net of $\mathbb{E}_{\mu(\cdot\mid \Y)}(\F)$ in $\|\cdot\|_{L^\infty(w_\Y)}$, yielding the inequality in~\eqref{eq: covering weighted in y to weighted}. 
\end{proof}

\begin{remark}
The uniform finite conditional moment condition of~\eqref{eq: finite moment condition 2} is automatically satisfied whenever $\mu(\cdot\mid y)$ satisfy the Poincaré inequality with a uniform in $y$ constant $\CPI^{\mu(\cdot\mid \Y)}$ by the same argument of~\Cref{rem: finite moments due poincare}.
\end{remark}

Having established the needed preliminary results, we are now ready to extend the proof for the fast rate in \Cref{proof: thm: fast rate} to the unbounded 
potential setting. 

{\it Fast rate adaptation for unbounded potentials:} 
\begin{itemize}
\item[(1)] \textbf{The general tool: a weighted chaining bound.} A similar adaptation of the chaining bound used for the slow rate applies to the generic chaining bound of~\Cref{thm: generalization for fast rate} used in the fast rate proof. Namely, the $L^\infty$-metric entropy of the second Dudley integral in~\eqref{eq: J_n bound quantity} can be replaced by a weighted $L^\infty(\langle\cdot\rangle^{-q'_\Y}\langle\cdot\rangle^{-q'_\U})$-metric entropy for any $q_\Y,q_\U>0$ when the measure $\mu$ is sub-exponential. 
Explicitly, this is the following weighted analogue of~\Cref{thm: generalization for fast rate}:
\begin{proposition}
\label{thm: generalization for fast rate weighted}    
Let $\mu$ be a sub-exponential probability measure supported on $\Y\times\U$, and let $\{x_i\}_{i=1}^{N} \overset{\mathrm{i.i.d}}{\sim}\mu$. Suppose $\F$ is a class of real-valued functions such that $\|\phi\|_{L^2_\mu(\Y\times\U)} \leq r'$ for all $\phi \in \F$. Let $q'_\Y,q'_\U>0$, set $w'(y,u):=\langle y\rangle^{-q'_\Y}\langle u\rangle^{-q'_\U}$, and further suppose $\|\phi\|_{L^\infty(w')} \leq R'$ for all $\phi \in \F$.
Define
\begin{align*}
J_N^{w'}(\F):=
\frac{1}{\sqrt{N}}\int_{0}^{r'} &\sqrt{\log \mathcal{N}(\delta,\F, \|\cdot\|_{L^2_\mu(\Y\times\U)})}\d\delta
\\
&+ \frac{1}{N}\int_{0}^{R'}\log \mathcal{N}(\delta,\F, \|\cdot\|_{L^\infty(w')})\d\delta\, .
\end{align*}
Then, it holds that:
\begin{itemize}
\item There exists a universal constant $C_{\textnormal{exp}} > 0$ such that
\begin{align*}
\mathbb{E}\left[\underset{\phi\in\F}{\sup} \frac{1}{N} \sum_{i=1}^{N}\phi(x_i) - \mu(\phi) \right]\leq
C_{\textnormal{exp}}J_N^{w'}(\F)\, .
\end{align*}
\item There exists a universal constant $C_{\textnormal{prob}}>0$ such that for any $t\geq 0$,
\begin{align*}
\mathbb{P}\left[ \underset{\phi\in\F}{\sup} \frac{1}{N} \sum_{i=1}^{N}\phi(x_i)
- \mu(\phi) \geq C_{\textnormal{prob}}\left(J_N^{w'}(\F)+r' \sqrt{\frac{t}{N}}+R'\frac{t}{N}\right)\right] \leq \exp(-t).
\end{align*}
\end{itemize}
\end{proposition}

This result corresponds to the unweighted chaining result of~\cite[Prop~.A.2]{divol2025optimal-sm} setting $\epsilon = \widetilde{\epsilon} = 0$.

We now apply this tool three times, once for each term I, I, III of the decomposition ~\eqref{eq: empirical process bound decomposition}, with a different choice of weights each time. 
\item [(2)] \textbf{Bounding the $\eta$-term in \eqref{eq: empirical process bound decomposition}:} 
We proceed in four steps.
\begin{itemize}
\item [(2a)] Apply \Cref{thm: generalization for fast rate weighted} on the class $(\textnormal{Id}-\mathbb{E}_{\eta_\U})(\F)$ where $\textnormal{Id}$ is the identity operator and $\mathbb{E}_{\eta_\U}(\cdot)$ is the marginalized class operator with respect to $\eta_\U$ as defined in~\eqref{eq: conditionally marginalized class}. In this case, we then set  $q'_\Y = q_\Y$, $q'_\U = q_\U$, $r'=\sqrt{C_{PI}^{\eta_\U}} \epsilon$ and $R' =\sup_{\phi\in\F}\|\phi-\mathbb{E}_{\eta_\U}(\phi)\|_{L^\infty(w)} \leq (1+\eta_\U(w_\U^{-1}))R$. where $w_\U(u)\coloneqq\langle u \rangle^{-q_\U}$. Thus,
\begin{align*}
\text{I}\lesssim  J_N^w((\textnormal{Id}-\mathbb{E}_{\eta_\U})(\F)) + r' \sqrt{\frac{t}{N}} + R'\frac{t}{N}\, ,
\end{align*}
where we further bound
\begin{adjustwidth}{-\dimexpr\textwidth-\linewidth\relax}{0pt}
\begin{equation}\label{xyz}
\begin{aligned}
&J_N^w((\textnormal{Id}-\mathbb{E}_{\eta_\U})(\F))\\
&\leq \frac{2}{\sqrt{N}}\int_{0}^{r'} \sqrt{\log \mathcal{N}\left(\frac{\delta}{2},\F, \|\cdot\|_{L^2_\eta}\right)}+\sqrt{\log \mathcal{N}\left(\frac{\delta}{2},\mathbb{E}_{\eta_\U}(\F), \|\cdot\|_{L^2_{\nu_\Y}}\right)}\d\delta
\\
&+\frac{1}{N}\int_{0}^{R'}\log \mathcal{N}\left(\frac{\delta}{2},\F, \|\cdot\|_{L^\infty(w)}\right)+\log \mathcal{N}\left(\frac{\delta}{2},\mathbb{E}_{\eta_\U}(\F), \|\cdot\|_{L^\infty(w_\Y)}\right)\d\delta\, . 
\end{aligned}
\end{equation}
\end{adjustwidth}
where we recall that $w_\Y(y)\coloneqq \langle y\rangle^{-q_\Y}$.
We now bound the integrands in each integral in the above display with respect to an $L^\infty(w)$
metric entropy.

\item [(2b)] \textbf{($\mathbf{L^2\mapsto L^\infty}$)} By~\Cref{lem: L2 to weighted uniform} and its marginal analogue (identical proof),
we bound the first two terms in the right hand side of \eqref{xyz}:
\begin{adjustwidth}{-\dimexpr\textwidth-\linewidth\relax}{0pt}
\begin{equation}\label{yxz}
\begin{aligned}
\log \mathcal{N}\left(\frac{\delta}{2},\F, \|\cdot\|_{L^2_\eta}\right)&\leq \log \mathcal{N}\left(\frac{\delta}{2\eta(w^{-2})^{1/2}},\F, \|\cdot\|_{L^\infty(w)}\right)\, , \\
\log \mathcal{N}\left(\frac{\delta}{2},\mathbb{E}_{\eta_\U}(\F), \|\cdot\|_{L^2_{\nu_\Y}}\right)&\leq \log \mathcal{N}\left(\frac{\delta}{2\nu_{\Y}(w_{\Y}^{-2})^{1/2}},\mathbb{E}_{\eta_\U}(\F), \|\cdot\|_{L^\infty(w_\Y)}\right)\, .
\end{aligned}
\end{equation}
\end{adjustwidth}

\item [(2c)] \textbf{(Marginal $\mathbf{\mapsto}$ Joint)} Then, by~\Cref{lem: conditionally marginalized class complexity} we can bound the 
fourth term in \eqref{xyz} as
\begin{adjustwidth}{-\dimexpr\textwidth-\linewidth\relax}{0pt}
\[
\begin{aligned}
\log \mathcal{N}\left(\frac{\delta}{2},\mathbb{E}_{\eta_\U}(\F), \|\cdot\|_{L^\infty(w_\Y)}\right) \leq \log \mathcal{N}\left(\frac{\eta_\U(w_\U^{-1})^{-1}\delta}{2},\F, \|\cdot\|_{L^\infty(w)}\right)\, .
\end{aligned}
\]
\end{adjustwidth}
We further note that the second row of \eqref{yxz} can also be bounded 
in the same way with  $\delta$ replaced with $\frac{\delta}{2 \nu_\Y (w_\Y^{-2})^{1/2}}$ to obtain 
\begin{adjustwidth}{-\dimexpr\textwidth-\linewidth\relax}{0pt}
\[
\begin{aligned}
\log \mathcal{N}\left(\frac{\delta}{2\nu_{\Y}(w_{\Y}^{-2})^{1/2}},\mathbb{E}_{\eta_\U}(\F), \|\cdot\|_{L^\infty(w_\Y)}\right) \leq \log \mathcal{N}\left(\frac{\eta_\U(w_\U^{-1})^{-1}\delta}{2\nu_{\Y}(w_{\Y}^{-2})^{1/2}},\F, \|\cdot\|_{L^\infty(w)}\right)\, .
\end{aligned}
\]
\end{adjustwidth}

\item [(2d)] Combining items (2a)--(2c), the integrands in the upper bound in 
\eqref{xyz} can be replaced by metric entropy  of $\F$ with respect to $L^\infty(w)$ only. Thus, term $\mathrm{I}$ can be bounded following verbatim the calculation of~\Cref{lem: bound-JF}. Note that the constants $\frac{1}{2}$, $\frac{1}{2\eta(w^{-2})^{1/2}}$, and $\frac{\eta_\U(w_\U^{-1})^{-1}}{2\nu_{\Y}(w_{\Y}^{-2})^{1/2}}$ do not propagate in the final bound as they cancel by the change of variable akin to~\eqref{eq: change var Dudley int} followed by bounding above the Dudley integrals proportionally to their upper limit.   
\end{itemize}
\item [(3)] \textbf{Bounding the $\nu$-term in \eqref{eq: empirical process bound decomposition}:} We proceed in five steps.
\begin{itemize}
\item [(3a)] Apply \Cref{thm: generalization for fast rate weighted} on the class $(\textnormal{Id}-\mathbb{E}_{\nu(\cdot\mid \Y)})(\F^*)$ where $\mathbb{E}_{\nu(\cdot\mid \Y)}(\F^*)$ is the conditionally marginalized conjugate class with respect to the conditionals $\nu(\cdot\mid y)$ as defined in~\eqref{eq: conditionally marginalized class}. In this case, we then set  $q'_\Y = q_\Y+(1+\widetilde{q})q_\U$, $q'_\U = q_\U$, $r'=\sqrt{\frac{\betamax C_{PI}^{\nu(\cdot\mid \Y)}}{\alphamin}}\epsilon$ and $R' =\sup_{\phi^*\in\F^*}\|\phi^* -\mathbb{E}_{\nu(\cdot\mid \Y)}(\phi^*)\|_{L^\infty(w_*)} \leq (1+\nu(w_{\U}^{-1}\mid\Y))C^* R$ where $w_{\U}(u)\coloneqq \langle u \rangle^{-q_\U}$ and $\nu(\cdot\mid \Y)\coloneqq \sup_{y\in\Y} \nu(\cdot\mid y)$ as in~\eqref{eq: finite moment condition 2}. Thus,
\begin{align*}
\text{II}\lesssim  J_N^{w_*}(\F^*-\mathbb{E}_{\nu(\cdot\mid \Y)}(\F^*)) + r' \sqrt{\frac{t}{N}} + R'\frac{t}{N}\, ,
\end{align*}
where we further bound
\begin{adjustwidth}{-\dimexpr\textwidth-\linewidth\relax}{0pt}
\begin{equation}\label{xyz 2}
\begin{aligned}
&J_N^{w_*}((\textnormal{Id}-\mathbb{E}_{\nu(\cdot\mid \Y)})(\F^*))\\
&\leq \frac{2}{\sqrt{N}}\int_{0}^{r'} \sqrt{\log \mathcal{N}\left(\frac{\delta}{2},\F^*, \|\cdot\|_{L^2_\nu}\right)}+\sqrt{\log \mathcal{N}\left(\frac{\delta}{2},\mathbb{E}_{\nu(\cdot\mid \Y)}(\F^*), \|\cdot\|_{L^2_{\nu_\Y}}\right)}\d\delta
\\
&+\frac{1}{N}\int_{0}^{R'}\log \mathcal{N}\left(\frac{\delta}{2},\F^*, \|\cdot\|_{L^\infty(w_{*})}\right)+\log \mathcal{N}\left(\frac{\delta}{2},\mathbb{E}_{\nu(\cdot\mid \Y)}(\F^*), \|\cdot\|_{L^\infty(w_{\Y*})}\right)\d\delta\, , 
\end{aligned}
\end{equation}
\end{adjustwidth}
 and where we set $w_{\Y*}(y) \coloneqq \langle y\rangle^{-(q_\Y+(1+\widetilde{q})q_\U)}$.

\item [(3b)] \textbf{($\mathbf{L^2\mapsto L^\infty}$)} By~\Cref{lem: L2 to weighted uniform} and its marginal analogue, we bound the first two terms in the right-hand-side of~\eqref{xyz 2}
\begin{adjustwidth}{-\dimexpr\textwidth-\linewidth\relax}{0pt}
\begin{equation}\label{yxz 2}
\begin{aligned}
\log \mathcal{N}\left(\frac{\delta}{2},\F^*, \|\cdot\|_{L^2_\nu}\right)&\leq \log \mathcal{N}\left(\frac{\delta}{2\nu(w_*^{-2})^{1/2}},\F^*, \|\cdot\|_{L^\infty(w_*)}\right)\, , \\
\log \mathcal{N}\left(\frac{\delta}{2},\mathbb{E}_{\nu(\cdot\mid\Y)}(\F^*), \|\cdot\|_{L^2_{\nu_\Y}}\right)&\leq \log \mathcal{N}\left(\frac{\delta}{2\nu_{\Y}(w_{\Y*}^{-2})^{1/2}},\mathbb{E}_{\nu(\cdot\mid\Y)}(\F^*), \|\cdot\|_{L^\infty(w_{\Y*})}\right)\, .
\end{aligned}
\end{equation}
\end{adjustwidth}

\item [(3c)] \textbf{(Marginal $\mathbf{\mapsto}$ Joint)} Then, by~\Cref{lem: conditionally marginalized class complexity}, we bound the second row of~\eqref{yxz 2} 
\begin{adjustwidth}{-\dimexpr\textwidth-\linewidth\relax}{0pt}
\[
\begin{aligned}
\log \mathcal{N}\left(\frac{\delta}{2\nu_{\Y}(w_{\Y*}^{-2})^{1/2}},\mathbb{E}_{\nu(\cdot\mid\Y)}(\F^*), \|\cdot\|_{L^\infty(w_{\Y*})}\right) \leq \log \mathcal{N}\left(\frac{\eta_\U(w_\U^{-1})^{-1}\delta}{2\nu_{\Y}(w_{\Y*}^{-2})^{1/2}},\F^*, \|\cdot\|_{L^\infty(w_*)}\right)\, .
\end{aligned}
\]
\end{adjustwidth}

\item [(3d)] \textbf{(Conjugate $\mathbf{\mapsto}$ Non-Conjugate)} Ultimately, by~\Cref{lem: conjugate to non conjugate unbounded}, we can further bound the previous entropy term
\begin{adjustwidth}{-\dimexpr\textwidth-\linewidth\relax}{0pt}
\[
\begin{aligned}
\log \mathcal{N}\left(\frac{\delta}{2\nu(w_*^{-2})^{1/2}},\F^*, \|\cdot\|_{L^\infty(w_*)}\right)&\leq \log \mathcal{N}\left(\frac{C^{*-1}\delta}{2\nu(w_*^{-2})^{1/2}},\F, \|\cdot\|_{L^\infty(w)}\right)\, ,
\end{aligned}
\]
\end{adjustwidth}

Similarly the fourth term in~\eqref{yxz 2} is also bounded
\begin{adjustwidth}{-\dimexpr\textwidth-\linewidth\relax}{0pt}
\[
\begin{aligned}
\log \mathcal{N}\left(\frac{\eta_\U(w_\U^{-1})^{-1}\delta}{2\nu_{\Y}(w_{\Y}^{-2})^{1/2}},\F^*, \|\cdot\|_{L^\infty(w_*)}\right)&\leq \log \mathcal{N}\left(\frac{C^{*-1}\eta_\U(w_\U^{-1})^{-1}\delta}{2\nu_{\Y}(w_{\Y*}^{-2})^{1/2}},\F, \|\cdot\|_{L^\infty(w)}\right)\, .
\end{aligned}
\]
\end{adjustwidth}

\item [(3e)] Combining items (3a)--(3d), the integrand of the upper bound of~\eqref{xyz 2} can be defined in terms of metric entropy terms of $\F$ with respect to $L^\infty(w)$ only. Thus, $\mathrm{II}$ can be bounded following verbatim the calculation of~\Cref{lem: bound-JF}. Note that the constants $\frac{1}{2}$, $\frac{1}{2\nu(w_*^{-2})^{1/2}}$, $\frac{C^{*-1}}{2\nu(w_*^{-2})^{1/2}}$, and $\frac{C^{*-1}\eta_\U(w_\U^{-1})^{-1}}{2\nu_{\Y}(w_{\Y}^{-2})^{1/2}}$ do not propagate in the final bound as they cancel by the change of variable akin to~\eqref{eq: change var Dudley int} followed by bounding above the Dudley integrals proportionally to their upper limit.  
\end{itemize}

\item [(4)] \textbf{Bounding the $\nu_{\Y}$-term in \eqref{eq: empirical process bound decomposition}:} We proceed in five steps.
\begin{itemize}
\item [(4a)] Apply \Cref{thm: generalization for fast rate weighted} on the class $(\mathbb{E}_{\eta_\U}+\mathbb{E}_{\nu(\cdot\mid \Y)}^*)(\F)$ where $\mathbb{E}_{\nu(\cdot\mid \Y)}^*(\cdot)$ is the composition of the fiber-wise conjugation operator and $\mathbb{E}_{\nu(\cdot\mid \Y)}(\cdot)$. In this case, picking the smallest norm, we then set  $q'_\Y = q_\Y + (1+\widetilde{q})q_\U$, $q'_\U = q_\U$, $r'=\sqrt{\frac{C_{PI}^{\nu_\Y,\widetilde{q}}\, \betamax}{\alphamin}}\epsilon$ and $R' =\sup_{\phi\in\F}\|(\mathbb{E}_{\eta_\U}+ \mathbb{E}_{\nu(\cdot\mid \Y)}^*)(\phi)\|_{L^\infty(w_*)}\leq (\eta_\U(w_\U^{-1}) + \nu(w_\U^{-1}\mid \Y)C^*)R$. Thus,
\begin{align*}
\text{III}\lesssim  J_N^{w_*}((\mathbb{E}_{\eta_\U}+\mathbb{E}_{\nu(\cdot\mid \Y)}^*)(\F)) + r' \sqrt{\frac{t}{N}} + R'\frac{t}{N}\, ,
\end{align*}
where we further bound
\begin{adjustwidth}{-\dimexpr\textwidth-\linewidth\relax}{0pt}
\begin{equation}\label{xyz 3}
\begin{aligned}
&J_N^{w_*}((\mathbb{E}_{\eta_\U}+\mathbb{E}_{\nu(\cdot\mid \Y)}^*)(\F))\\
&\leq \frac{2}{\sqrt{N}}\int_{0}^{r'} \sqrt{\log \mathcal{N}\left(\frac{\delta}{2},\mathbb{E}_{\eta_\U}(\F), \|\cdot\|_{L^2_{\nu_\Y}}\right)}+\sqrt{\log \mathcal{N}\left(\frac{\delta}{2},\mathbb{E}_{\nu(\cdot\mid \Y)}(\F^*), \|\cdot\|_{L^2_{\nu_\Y}}\right)}\d\delta
\\
&+\frac{1}{N}\int_{0}^{R'}\log \mathcal{N}\left(\frac{\delta}{2},\mathbb{E}_{\eta_\U}(\F), \|\cdot\|_{L^\infty(w_{\Y})}\right)+\log \mathcal{N}\left(\frac{\delta}{2},\mathbb{E}_{\nu(\cdot\mid \Y)}(\F^*), \|\cdot\|_{L^\infty(w_{\Y*})}\right)\d\delta\, . 
\end{aligned}
\end{equation}
\end{adjustwidth}
and where we recall $w_{\Y*}(y) \coloneqq \langle y\rangle^{-(q_\Y+(1+\widetilde{q})q_\U)}$.

\item [(4b)] \textbf{($\mathbf{L^2\mapsto L^\infty}$)} By the marginal version of~\Cref{lem: L2 to weighted uniform}, we bound the first two terms in the right-hand-side of~\eqref{xyz 3}
\begin{adjustwidth}{-\dimexpr\textwidth-\linewidth\relax}{0pt}
\begin{equation}\label{yxz 3}
\begin{aligned}
\log \mathcal{N}\left(\frac{\delta}{2},\mathbb{E}_{\eta_\U}(\F), \|\cdot\|_{L^2_{\nu_\Y}}\right)&\leq \log \mathcal{N}\left(\frac{\delta}{2\nu_\Y(w_{\Y}^{-2})^{1/2}},\mathbb{E}_{\eta_\U}(\F), \|\cdot\|_{L^\infty(w_\Y)}\right)\, , \\
\log \mathcal{N}\left(\frac{\delta}{2},\mathbb{E}_{\nu(\cdot\mid\Y)}(\F^*), \|\cdot\|_{L^2_{\nu_\Y}}\right)&\leq \log \mathcal{N}\left(\frac{\delta}{2\nu_\Y(w_{\Y*}^{-2})^{1/2}},\mathbb{E}_{\nu(\cdot\mid\Y)}(\F^*), \|\cdot\|_{L^\infty(w_{\Y*})}\right)\, .
\end{aligned}
\end{equation}
\end{adjustwidth}

\item [(4c)] \textbf{(Marginal $\mathbf{\mapsto}$ Joint)} Then, by~\Cref{lem: conditionally marginalized class complexity}, we cna bound the first row of~\eqref{yxz 3}
\begin{adjustwidth}{-\dimexpr\textwidth-\linewidth\relax}{0pt}
\[
\begin{aligned}
\log \mathcal{N}\left(\frac{\delta}{2\nu_\Y(w_{\Y}^{-2})^{1/2}},\mathbb{E}_{\eta_\U}(\F), \|\cdot\|_{L^\infty(w_\Y)}\right) \leq \log \mathcal{N}\left(\frac{\eta_\U(w_\U^{-1})^{-1}\delta}{2\nu_{\Y}(w_{\Y}^{-2})^{1/2}},\F, \|\cdot\|_{L^\infty(w)}\right)\, .
\end{aligned}
\]
\end{adjustwidth}

Similarly, we bound the second row of~\eqref{yxz 3} and the last two terms of~\eqref{xyz 3} as
\begin{adjustwidth}{-\dimexpr\textwidth-\linewidth\relax}{0pt}
\[
\begin{aligned}
\log \mathcal{N}\left(\frac{\delta}{2\nu_\Y(w_{\Y*}^{-2})^{1/2}},\mathbb{E}_{\nu(\cdot\mid\Y)}(\F^*), \|\cdot\|_{L^\infty(w_{\Y*})}\right) &\leq \log \mathcal{N}\left(\frac{\eta_\U(w_\U^{-1})^{-1}\delta}{2\nu_{\Y}(w_{\Y*}^{-2})^{1/2}},\F^*, \|\cdot\|_{L^\infty(w_*)}\right)\, ,\\
\log \mathcal{N}\left(\frac{\delta}{2},\mathbb{E}_{\eta_\U}(\F), \|\cdot\|_{L^\infty(w_{\Y})}\right)&\leq \log \mathcal{N}\left(\frac{\eta_\U(w_\U^{-1})^{-1}\delta}{2},\F, \|\cdot\|_{L^\infty(w)}\right)\, ,\\
\log \mathcal{N}\left(\frac{\delta}{2},\mathbb{E}_{\nu(\cdot\mid \Y)}(\F^*), \|\cdot\|_{L^\infty(w_{\Y*})}\right)&\leq \log \mathcal{N}\left(\frac{\eta_\U(w_\U^{-1})^{-1}\delta}{2},\F^*, \|\cdot\|_{L^\infty(w_{*})}\right)\, .\\
\end{aligned}
\]
\end{adjustwidth}

\item [(4d)] \textbf{(Conjugate $\mathbf{\mapsto}$ Non-Conjugate)} Ultimately, by~\Cref{lem: conjugate to non conjugate unbounded}, we bound the first and last row of the last step by
\begin{adjustwidth}{-\dimexpr\textwidth-\linewidth\relax}{0pt}
\[
\begin{aligned}
\log \mathcal{N}\left(\frac{\eta_\U(w_\U^{-1})^{-1}\delta}{2\nu_{\Y}(w_{\Y*}^{-2})^{1/2}},\F^*, \|\cdot\|_{L^\infty(w_*)}\right)&\leq \log \mathcal{N}\left(\frac{C^{*-1}\eta_\U(w_\U^{-1})^{-1}\delta}{2\nu_{\Y}(w_{\Y*}^{-2})^{1/2}},\F, \|\cdot\|_{L^\infty(w)}\right)\, ,\\
\log \mathcal{N}\left(\frac{\eta_\U(w_\U^{-1})^{-1}\delta}{2},\F^*, \|\cdot\|_{L^\infty(w_{*})}\right)&\leq \log \mathcal{N}\left(\frac{C^{*-1}\eta_\U(w_\U^{-1})^{-1}\delta}{2},\F, \|\cdot\|_{L^\infty(w)}\right) \, .
\end{aligned}
\]
\end{adjustwidth}

\item [(4e)] Combining items (4a)--(4d), the integrand of the upper bound of~\eqref{xyz 3} can be defined in terms of metric entropy terms of $\F$ with respect to the $L^\infty(w)$-metric entropy only. Thus, $\mathrm{III}$ can be bounded following verbatim the calculation of~\Cref{lem: bound-JF}. Also in this case, all the constants sitting in front of $\delta$ in the entropy terms do not propagate in the final bound as they cancel by the change of variable akin to~\eqref{eq: change var Dudley int} followed by bounding above the Dudley integrals proportionally to their upper limit.  
\end{itemize}

\item[(5)] \textbf{Combining the terms.} Putting items (2)--(4) together, the proof of~\Cref{thm: fast rate} in the bounded case goes through also in the unbounded case by replacing the final $R \mapsto \left((1+\eta_\U(w_\U^{-1}))\vee (1+\nu(w_{\U}^{-1}\mid\Y))C^* \vee (\eta_\U(w_\U^{-1}) + \nu(w_\U^{-1}\mid \Y)C^*)\right)R 
= (\eta_\U(w_\U^{-1}) + \nu(w_\U^{-1}\mid \Y))(C^*\vee 1)R
= (\eta_\U(w_\U^{-1}) + \nu(w_\U^{-1}\mid \Y))\frac{4}{\alphamin}\widetilde{K}R$.
\end{itemize}

\begin{remark}
 The weighted chaining bound used above also 
admits truncated versions to handle the divergent Dudley integrals when $\gamma\geq 2$. In that case, 
additional polylogarithmic terms of $\log(N/C_\F)$  appear in front of the Dudley integral. To see where this extra term arises, look at the term $\langle L\rangle^\eta$ of the truncated chaining result of~\cite[Prop. A.2]{divol2025optimal-sm}. Thus, in the unbounded case, the extensions reported in the left column of~\Cref{tab:cond-ot-rates} are expected to hold modulo this extra polylogarithmic term.
\end{remark}

\editstart 
\subsection{Fiber-wise affine transformations of potentials in the conditional OT setting}\label{sec: verifying anchor point cond wolog}
In this section, we show that~\Cref{anchor point condition} in~\Cref{assump: slow rate} can be completely dropped in the proposed estimation error bounds. To do so, we state and prove the analogues of~\cite[Prop. F.1, Lem. H.1]{divol2025optimal-sm} in the conditional OT setting, analyzing fiber-wise affine transformations of potentials. To this end, we introduce the modified class
\begin{align}\label{eq: fiberwise affine transformations}
\F^{r_1,r_2}_{\textnormal{aff}, \Y}\coloneqq \left\{ (y,u)\to b\phi(y,u)+\langle t(y), u\rangle \mid \phi\in\F\, , 0\leq b\leq r_1\, , \|t\|_{L^2_{\nu_\Y,N}}\leq r_2\right\}\, ,
\end{align}
defined for any $r_1,r_2\in [0,\infty]$ and data dependent norm $\|t\|_{L^2_{\nu_\Y,N}} \coloneqq \left(\frac{1}{N}\sum_{i=1}^N\|t(y_i)\|^2\right)^{1/2}$.
We then show the conditional analogue of~\cite[Prop. F.1]{divol2025optimal-sm}.

\begin{proposition}\label{prop: affine transformation wolog}
Let $\eta_\U
$ and $\nu_\U$ be sub-exponential. Let $\F$ be a class of potentials such that $\phi(y,\cdot)$ is $\betamax$-smooth for any $y\in \Y$ and $\nabla_u\phi(y,0) = 0$ for every $\phi\in\F$, that contains at least one potential $\overline{\phi}$ such that $\overline{\phi}(y,\cdot)$ is $\alphamin$-strongly convex for any $y\in\Y$ for some $\alphamin>0$. Let $\widetilde{\F}\subseteq \F^{r_1,+\infty}_{\textnormal{aff},\Y}$ for some $1\leq r_1<+\infty$. 
Then, there exists an $N$-independent $r>0$ and $\ell>0$ such that 
\begin{align*}
\mathbb{E}^{\textnormal{train}}\left[\|\nabla_u \wh\phi_{\widetilde{\F}} - \nabla_u \phi^\dag\|_{L^2_\eta}^2\mathbbm{1}\left\{\wh \phi_{\widetilde{\F}} \neq \wh \phi_{\widetilde{\F}^r}\right\}\right]\leq c\exp\left(-CN^\ell\right)\, ,
\end{align*}
with constants $r,\ell,C$ depending on the different parameters involved, $\eta$ and $\nu$, and where $\widetilde{\F}^r\coloneqq \widetilde{\F}\cap\F_{\textnormal{aff},\Y}^{r,r}$ and $\wh \phi_{\widetilde\F}$ (similarly $\wh \phi_{\widetilde{\F}^r}$) denotes the minimizer of the marginalized empirical semidual problem in~\eqref{eq:empirical-min} over $\widetilde{\F}$ (similarly $\widetilde{\F}^r$).
\end{proposition}

\begin{proof}
Denoting the empirical samples of the marginalized empirical semidual problem of~\eqref{eq:empirical-min}, let $\{y_i, u_i\}_{i=1}^N\overset{\textnormal{i.i.d.}}{\sim} \nu \perp \{v_i\}_{i=1}^N\overset{\textnormal{i.i.d.}}{\sim} \eta_\U$, where $\perp$ is used to denote independence. 
Moreover, we define the empirical second moments $m_{N,2}^v = \frac{1}{N}\sum_{i=1}^{N} \|v_i\|^2$ and $m_{N,2}^u = \frac{1}{N}\sum_{i=1}^{N} \|u_i\|^2$.

Since, $\widetilde{\F}\subseteq \F^{r_1,+\infty}_{\textnormal{aff},\Y}$, we can write $\wh \phi_{\widetilde{\F}}(y,u)= b\phi(y,u)+\langle t(y),u\rangle$ 
for $\phi\in\F$ and parameters $b$ and $t(\cdot)$. Consequently, $\widehat{\phi}_{\widetilde{\F}}^*(y,v) = b\phi^*\left(y, \frac{v-t(y)}{b}\right)$ and
\begin{align}\label{eq: transf proof 1}
\wh\cS(\wh \phi_{\widetilde{\F}}) =  \frac{1}{N}\sum_{i=1}^N\left( b\phi(y_i,v_i) + \langle t(y_i), v_i\rangle + b\phi^* \left(y_i, \frac{u_i-t(y_i)}{b} \right)\right)\, ,
\end{align}
By~\cite[Lem. D.4]{divol2025optimal-sm}, we can bound below the conjugate terms
\begin{align}\label{eq: transf proof 2}
\phi^*\left(y_i, \frac{u_i-t(y_i)}{b}\right)\geq -\phi(y_i, x_i)+\left\langle \frac{u_i-t(y_i)}{b},x_i\right\rangle + \frac{1}{2\betamax}\left\|\frac{u_i-t(y_i)}{b}-\nabla_u\phi(y_i,x_i)\right\|^2\, .
\end{align}
for any choice of points $\{x_i\}_{i=1}^N$ in $\mathbb{R}^n$.
By substituting~\eqref{eq: transf proof 2} into~\eqref{eq: transf proof 1} we can then further bound
\begin{align}\label{eq: transf proof 3}
\wh\cS(\wh \phi_{\widetilde{\F}})\geq \left(\mathrm{I}+\mathrm{II}+\mathrm{III}+
\mathrm{IV}\right)\left(\{x_i\}_{i=1}^N\right)\, ,
\end{align}
where we define compactly the functions
\begin{align*}
\mathrm{I}\left(\{x_i\}_{i=1}^N\right)&=b\frac{1}{N}\sum_{i=1}^N\left(\phi(y_i,v_i)- \phi(y_i,x_i)\right)\, ,\quad
\mathrm{II}\left(\{x_i\}_{i=1}^N\right)= \frac{1}{N}\sum_{i=1}^N\langle t(y_i), v_i -x_i\rangle\, ,\\
\mathrm{III}\left(\{x_i\}_{i=1}^N\right)&= \frac{1}{2\betamax b N}\sum_{i=1}^N \left\|t(y_i)-\left(u_i-b\nabla_u \phi(y_i,x_i)\right)\right\|^2\, ,\quad
\mathrm{IV}\left(\{x_i\}_{i=1}^N\right)=\frac{1}{N}\sum_{i = 1}^N\langle x_i,u_i\rangle\, .
\end{align*}
Also note that, by optimality of $\wh \phi_{\widetilde{\F}}$, $\wh \cS (\wh \phi_{\widetilde{\F}})-\wh\cS(\overline\phi) \leq 0$. The goal is then to pick suitable choices of $\{x_i\}_{i=1}^N$ to show that, with overwhelming probability, this condition forces $\|t\|_{L^2_{\nu_\Y,N}}$ to be smaller than some $r$, ultimately implying $\wh \phi_{\widetilde{\F}}\in \widetilde{\F}^r$ and therefore $\wh \phi_{\widetilde{\F}} = \wh \phi_{\widetilde{\F}^r}$.

\paragraph{Bound on $\|t\|_{L^2_{\nu_\Y}}$}
For this bound, we pick points $x_i = v_i$ for any $1\leq i \leq N$ in~\eqref{eq: transf proof 2}. This substitution immediately sets $\mathrm{I}\left(\{v_i\}_{i=1}^N\right) = \mathrm{II}\left(\{v_i\}_{i=1}^N\right) = 0$ in~\eqref{eq: transf proof 3}, and then
\begin{align}\label{eq: transf proof 4}
\mathrm{III}\left(\{v_i\}_{i=1}^N\right)\leq \wh\cS(\overline{\phi})-\mathrm{IV}\left(\{v_i\}_{i=1}^N\right)\, .
\end{align}
Next, note that, by triangle inequality and the bound $\|\nabla_u \phi(y_i,v_i)\|\leq \beta\|v_i\|$ (following by~[Lem. D.1]\cite{divol2025optimal-sm} and the hypothesis $\nabla_u \phi(y_i,0) = 0$),
\begin{align}\label{eq: transf proof 5}
\|t(y_i)\|&\leq \left\|t(y_i)-\left(u_i-b\nabla_u \phi(y_i,v_i)\right)\right\| +\|u_i\| +b\|\nabla_u \phi(y_i,v_i)\|\\
&\leq \left\|t(y_i)-\left(u_i-b\nabla_u \phi(y_i,v_i)\right)\right\| +\|u_i\| +b\betamax\|v_i\|\, .\nonumber
\end{align}
Then, we can upper bound by~\eqref{eq: transf proof 4} and~\eqref{eq: transf proof 5},
\begin{align}\label{eq: transf proof 6}
\|t\|_{L^2_{\nu_\Y,N}}^2&\leq 6\betamax b\mathrm{III}\left(\{v_i\}_{i=1}^N\right) + \frac{3}{N}\sum_{i=1}^N\left(\|u_i\|^2 + b^2\betamax\|v_i\|^2\right)\\
& \leq 2\betamax b_{\max}\left(\cS(\overline{\phi})-\mathrm{IV}\left(\{v_i\}_{i=1}^N\right)\right) + 3\left(m_{N,2}^u+ _{\max}^2\betamax m_{N,2}^v\right)\, .\nonumber
\end{align}
Next, we can further bound by Cauchy-Schwartz
\begin{align}\label{eq: transf proof 7}
|\mathrm{IV}\left(\{v_i\}_{i=1}^N\right)|\leq \sqrt{m_{N,2}^u m_{N,2}^v}\leq \frac{1}{2}\left( m_{N,2}^u+ m_{N,2}^v\right)\, .
\end{align}
Next, note that $\overline{\phi}$ being $\alphamin$-strongly convex fiber-wise implies that $\overline{\phi^*}$ is $\frac{1}{\alphamin}$-smooth fiber-wise. Moreover, by the definition of smoothness and the hypothesis $\nabla_u\overline{\phi}(y,0) = 0$, we have
\begin{align*}
\overline \phi(y,u)\leq \overline{\phi}(y,0) + \frac{\betamax}{2}\|u\|^2\, ,\qquad \overline \phi^*(y,u)\leq \overline{\phi}^*(y,0) + \frac{1}{2\alphamin}\|u\|^2\, .
\end{align*}
Also notice that as a consequence of Fenchel inequality $\overline{\phi}(y,0) + \overline{\phi}^*(y,0) = 0$.
Thus, we can also control
\begin{align}\label{eq: transf proof 8}
\wh\cS(\overline{\phi})\leq \frac{1}{\betamax}m_{2,N}^v + \frac{1}{2\alphamin}m_{2,N}^u\, .
\end{align}
Therefore, combining~\eqref{eq: transf proof 6}, ~\eqref{eq: transf proof 7}, and~\eqref{eq: transf proof 8}  we can conclude that
\begin{align}
\|t\|_{L^2_{\nu_\Y, N}}^2\leq \Theta_{\max}^2\, .
\end{align}\label{eq: transf proof 9}
where $\Theta_{\max}$ is a polynomial in 
the empirical moments  $m_{2,N}^v$ and the 
random variable $b_{\max}$ and hence 
is itself a random variable.

\paragraph{Conclusion}
Let $r\geq 1$. The key observation is that $\wh\phi_{\widetilde{\F}}\neq \wh\phi_{\widetilde{\F}_r}$ implies $\Theta_{\max}>r$ or $b_{\max}>r$. Then,
\begin{align*}
\mathbb{E}^{\textnormal{train}}\left[\|\nabla_u \wh\phi_{\widetilde{\F}} - \nabla_u \phi^\dag\|_{L^2_\eta}^2\mathbbm{1}\left\{\wh \phi_{\widetilde{\F}} \neq \wh \phi_{\widetilde{\F}^r}\right\}\right]&\leq 2\mathbb{E}^{\textnormal{train}}\left[\|\nabla_u \wh\phi_{\widetilde{\F}}\|_{L^2_\eta}^2 +2\|\nabla_u\phi^\dag\|_{L^2_{\eta}}^2\mathbbm{1}\left\{\wh \phi_{\widetilde{\F}} \neq \wh \phi_{\widetilde{\F}^r}\right\}\right] \\
&+ 2\|\nabla_u\phi^\dag\|_{L^2_{\eta}}^2\mathbb{P}^{\textnormal{train}}\left( \left\{\wh \phi_{\widetilde{\F}} \neq \wh \phi_{\widetilde{\F}^r}\right\}\right) \, ,
\end{align*}
Moreover, recalling the representation $\wh \phi_{\widetilde{\F}}(y,u)= b\phi(y,u)+\langle t(y),u\rangle$, \cite[Lem. D.1]{divol2025optimal-sm}, and the hypothesis $\nabla_{u}\phi(y,0) = 0$,
\begin{align*}
\|\nabla_u \wh\phi_{\widetilde{\F}}(y,u)\|\leq b_{\max}\|\nabla_u\phi(y,u)\|+\Theta_{\max}\|u\| \leq \left(\betamax b_{\max} + \Theta_{\max}\right)\|u\|\, .
\end{align*}
Therefore, there exists a $c>0$ such that
\begin{align*}
\mathbb{E}^{\textnormal{train}}\left[\|\nabla_u \wh\phi_{\widetilde{\F}} - \nabla_u \phi^\dag\|_{L^2_\eta}^2\mathbbm{1}\left\{\wh \phi_{\widetilde{\F}} \neq \wh \phi_{\widetilde{\F}^r}\right\}\right]&\leq c \mathbb{E}^{\textnormal{train}}\left[\left(b_{\max}^2+1\right)\mathbbm{1}\left\{b_{\max}>r\right\} \right] \\
&+c \mathbb{E}^{\textnormal{train}}\left[\left(T_{\max}^2+1\right)\mathbbm{1}\left\{\Theta_{\max}>r\right\} \right]\\
& = c \mathbb{E}^{\textnormal{train}}\left[\left(T_{\max}^2+1\right)\mathbbm{1}\left\{\Theta_{\max}>r\right\} \right]\, ,
\end{align*}
where the last step follows for $r>r_1=b_{\max}$, since we assume the simplification $r_1<\infty$ by hypothesis. Note that $\Theta_{\max}$ concentrates around its expectation with high-probability since $b_{\max}<+\infty$ and $m_{2,N}^v$ and $m_{2,N}^u$ concentrate by~\cite[Lem. F.2]{divol2025optimal-sm}, applicable since $\eta_\U$ and $\nu_\U$ are assumed to be sub-exponential. Thus, there exists a sufficiently big $N$-independent $r$ and constants $c,C,\ell>0$ such that the result of the proposition holds.
\end{proof}

\begin{remark}
In~\Cref{prop: affine transformation wolog}, we assume that $\nu_\U$ is sub-exponential for the empirical moment $m_{2,N}^u$ to concentrate. However, we remark that whenever $\phi^\dag$ is assumed to be fiber-wise $\betamax$-smooth and to satisfy condition 1 of~\Cref{assump: fast rate} this condition is readily satisfied. Indeed, by applying~[Lem. D.1]\cite{divol2025optimal-sm},
\begin{align*}
\|u_i\| \leq \nabla_u\phi^\dag(y_i, 0) + \beta\|v_i\|\leq L_\F\langle y\rangle^{\widetilde{q}} + \beta\|v_i\|\, .
\end{align*}
Thus, in the case where $\widetilde{q} = 0$, $\eta_\U$ being sub-exponential immediately yields each of the $u_i$ having sub-exponential distribution and $m_{2,N}^u$ concentrates. In the case where $\widetilde{q}>0$, further requiring that $\nu_\Y$ is sub-exponential, each of the $u_i$ is distributed according to a sub-Weibull distribution (sub-exponential for $\widetilde{q}\leq 1$) and the concentration of $m_{2,N}^u$ still holds.
\end{remark}

\begin{remark}
To simplify the proof, we set the inflation parameter $r_1 < +\infty$ in~\Cref{prop: affine transformation wolog}. However, we remark that the unconditional analogue of~\Cref{prop: affine transformation wolog} has been shown for arbitrarily big $r_1 = +\infty$, important requirement to extend estimation rates to not pre-compact classes that satisfy condition 5 of~\Cref{assump: slow rate} only if additionally constrained with respect to some norm (e.g.~\cite[Eq.~44]{divol2025optimal-sm}). Due to the lack of asymmetry arising from fiber-wise transformations proving the result for $r_1=+\infty$ in the conditional case seems to be more challenging compared to the unconditional case. Nevertheless, we expect it be achievable under the additional assumption that, for any $\phi\in\F$, $\phi(y,\cdot)$ is $\alphamin$-strongly convex for any $y\in\Y$ and $\eta_\U$ has non-zero variance. As a motivation, we suggest that, when $r_1 =\infty$, by the additional fiber-wise $\alphamin$-strong convexity assumption
of all potentials in $\F$, we can bound below
\begin{align*}
\mathrm{I}\left(\{x_i\}_{i=1}^N\right)\geq b\frac{1}{N}\sum_{i=1}^N \left(\langle v_i-x_i, \nabla_u\phi(y_i,x_i)\rangle + \frac{\alphamin}{2}\left\|v_i-x_i\right\|^2\right)\, .
\end{align*}
Moreover, discarding $\mathrm{III}\left(\{x_i\}_{i=1}^N\right)\geq 0$ and recomposing $\wh \phi_{\widetilde{\F}}(y,u)= b\phi(y,u)+\langle t(y),u\rangle$,
\begin{align*}
\wh \cS(\overline{\phi}) &\geq \frac{1}{N}\sum_{i=1}^N \left(\langle v_i-x_i, b\nabla_u\phi(y_i,x_i)+t(y_i)\rangle +\langle x_i,u_i\rangle + \frac{b\alphamin}{2}\left\|v_i-x_i\right\|^2\right)\\
&= \frac{1}{N}\sum_{i=1}^N \left(\langle v_i-x_i, \nabla_u \wh \phi_{\widetilde{\F}}(y_i,x_i)\rangle +\langle x_i,u_i\rangle + \frac{b\alphamin}{2}\left\|v_i-x_i\right\|^2\right)\, .
\end{align*}
Then we can pick $x_i = \nabla_u\wh \phi^*_{\widetilde{\F}}(y_i,u_i)$ and, recalling the fiber-wise bijection relation $\nabla_u\phi^*(y,u) = \left(\nabla_u\phi(y,\cdot)\right)^{-1}(u)$ for fiber-wise strongly convex functions, we obtain 
\begin{align*}
\wh \cS(\overline{\phi})\geq \frac{b\alphamin}{2} \xi_N + \mathrm{IV}\left(\{v_i\}_{i=1}^N\right)\, ,\qquad \xi_N = \frac{1}{ N}\sum_{i=1}^N \|v_i - \nabla_u\wh \phi^*_{\widetilde{\F}}(y_i,u_i)\|^2\, .
\end{align*}
Then, rearranging, we can bound
\begin{align}
b\leq b_{\max} \coloneqq 1\vee\frac{2\left(\wh\cS(\overline{\phi})-\mathrm{IV}\left(\{v_i\}_{i=1}^N\right)\right)}{\alphamin \xi_N}\, .
\end{align} 
To conclude the proof for the case $r_1=\infty$, we then have to show that
\begin{align*}
\xi_N = \frac{1}{ N}\sum_{i=1}^N \left\|v_i - \nabla_u \phi^*\left(y_i,\frac{v_i-t(y_i)}{b}\right)\right\|^2\geq \inf_{\phi_{\textnormal{ind}}\in \F}\left\|v_i - \nabla_u \phi_{\textnormal{ind}}^*\left(y_i,\frac{v_i-t(y_i)}{b}\right)\right\|^2\, ,
\end{align*}
concentrates away from $0$ with high probability. We hypothesize two viable approaches. 

To see heuristically how the concentration could hold, note that by the crucial independence of the point clouds $\{v_i\}_{i=1}^N$ and $\{z_i \coloneqq\nabla_u\phi_{\textnormal{ind}}^*(y_i,u_i)\}_{i=1}^N$,
\begin{align}\label{eq: ind case suggestion}
\mathbb{E}^{\textnormal{train}}\|v_i - z_i\|^2 &= \mathbb{E}_{(y_i,u_i)\sim\nu}\mathbb{E}_{v_i\sim\eta_\U}\left[\|v_i-  z_i\|^2\mid z_i\right] \\
&=
\mathbb{E}_{v_i\sim\eta_\U}\left[\|v_i-\eta_\U(v_i)\|^2\right]+\mathbb{E}_{(y_i,u_i)\sim\nu}\left[\|\eta_\U(v_i) - z_i\|^2\right]\nonumber\\
&\geq \textnormal{Var}_{\eta_\U}(v_i)> 0\, .\nonumber
\end{align}
Thus, the missing part to prove is to show, perhaps additionally using the base class assumptions of~\Cref{assump: slow rate}(1)(2)(4)(5) and that the ratio $b/\|t\|_{L^2_{\nu_\Y,N}}$ is bounded in high-probability by~\eqref{eq: transf proof 6}, a uniform law of large numbers to show that the fluctuations of the lower bound of $\xi_N$ over all of $\F$ have order $O(1/\sqrt{N})$.

Alternatively, one could try to directly bound the initial expression of $\xi_N$ with a ``leave-one-out" approach that decouples the then correlated point clouds $\{v_i\}_{i=1}^N$ and $\{z_i \coloneqq\nabla_u\wh \phi_{\widetilde{\F}}^*(y_i,u_i)\}_{i=1}^N$. This would then require to show that the objective $\wh \cS$ to be stable in the sense of~\cite{bousquet2002stability, celisse2016stability}.

\end{remark}

The case $r_1 = 1<\infty$ of~\Cref{prop: affine transformation wolog} is sufficient to derive~\Cref{lem: bounded grad wolog}, which ultimately allows us to drop condition 6 in~\Cref{assump: slow rate} in the proof of the rates. This is the analogue of~\cite[Lem. H.1]{divol2025optimal-sm} transposed to the conditional OT setting.

\begin{lemma}\label{lem: bounded grad wolog}
Let $\F^{K'} = \left\{\phi\in\F\mid \|\nabla_u\phi(\cdot,0)\|_{L^2_{\nu_\Y,N}}\leq K'\right\}$. Then, for $N$-independent $K'$ large enough,
\begin{align*}
\mathbb{E}^{\textnormal{train}}\left[\|\nabla_u\wh \phi_{\F} - \nabla_u \phi^\dag\|_{L^2_\eta}\mathbbm{1}\left\{\wh \phi_\F \neq \wh \phi_{\F^{K'}}\right\}\right]\leq \frac{1}{N}\, .
\end{align*}
\end{lemma}

\begin{proof}
Introduce the fiber-wise shifted class \begin{align*}
\F_0 = \left\{(y,u)\to \phi(y,u)-\langle \nabla_{u}\phi(y,0), u\rangle \mid \phi\in\F\right\}\, .
\end{align*}
Then, $\F\subseteq (\F_0)_{\textnormal{aff},\Y}^{1,+\infty}$. We can then apply~\Cref{prop: affine transformation wolog} (with $r_1 = 1<+\infty$), which implies that there exists an $N$-independent $K'$ such that the statement of the lemma holds. 
\end{proof}

\paragraph{Obtaining estimation rates without ~\Cref{anchor point condition} in~\Cref{assump: slow rate}}
By~\Cref{lem: bounded grad wolog}, we can then decompose the estimation error by 
\begin{align*}
\mathbb{E}^{\textnormal{train}}\|\nabla_u \wh\phi_{\widetilde{\F}} - \nabla_u \phi^\dag\|_{L^2_\eta}^2\leq \mathbb{E}^{\textnormal{train}}\left[\|\nabla_u \wh\phi_{\widetilde{\F}} - \nabla_u \phi^\dag\|_{L^2_\eta}^2\mathbbm{1}\left\{\wh \phi_{\F} = \wh \phi_{\F^{K'}}\right\}\right]+\frac{1}{N}
\end{align*}
Thus, conditioning on the high-probability event, $\left\{\wh \phi_{\F} = \wh \phi_{\F^{K'}}\right\}$, the estimation rate can be bounded restricting $\F$ to $\F^{K'}$ without loss of generality. Note that on this event we have a slightly different condition compared to~\Cref{anchor point condition} in~\Cref{assump: slow rate}, i.e. $\|\nabla_u\phi(\cdot, 0)\|_{L^2_{\nu_\Y, N}}\leq K'$ for any $N$ instead of $\|\nabla_u\phi(\cdot, 0)\|_{L^1_{\nu_\Y}}\leq K$. However, this is still enough to bound the rate upon the modification of~\eqref{eq: adaptable inequality} to
\begin{align*}
\|g(0)\|\leq \frac{1}{N}\sum_{i=1}^N \left(\|g(y)\|+L_\F\langle y\rangle^{\widetilde{q}+1}\right)&\leq \|g(y)\|_{L^2_{\nu_\Y,N}}+L_\F\frac{1}{N}\sum_{i=1}^N\langle y\rangle^{\widetilde{q}+1}\, ,
\end{align*}
where $\|g(y)\|_{L^2_{\nu_\Y,N}}\leq K'$ for any $N$ and $\frac{1}{N}\sum_{i=1}^N\langle y\rangle^{\widetilde{q}+1}$ concentrates around $\nu_\Y\left(\langle \cdot\rangle^{\widetilde{q}+1}\right)$ since $\nu_\Y$ is sub-exponential. As a result, without assuming~\Cref{anchor point condition} in~\Cref{assump: slow rate}, both the rates of~\Cref{thm: slow rate} and~\Cref{thm: fast rate} hold replacing $K$ with $K'$ in the bounds.
\editend

\section{Proofs for \Cref{Sec:Error analysis for OTF}}
\label{sec: proof otf analysis}
In this section we present the theoretical details and proofs behind the OTF error analysis of \Cref{Sec:Error analysis for OTF}. In~\cref{subsec: proof triangle inequality filtering decomposition}, we establish a preliminary error decomposition applicable to both the exact and approximate filtering error of \eqref{eq: exact filter error} and \eqref{eq: approximate filter error}, reminiscent of the error analysis of~\cite{al-jarrah2023optimal}. We then combine the error decomposition with the estimation rates for conditional Brenier maps derived in~\Cref{Sec:Error analysis for conditional OT} to derive complete sample complexity bounds for the approximate and exact filtering errors in~\Cref{subsec: proof of approximate error bound},~\Cref{subsec: proof of exact error bound 2}, and~\Cref{subsec: proof of exact error bound}. In~\Cref{subsec: calculation for Y stability}, we present a small calculation showing that condition 2 in~\Cref{assump: exact filter error extra assump} can be achieved under mild conditions on the filtering system of \eqref{eq: simulater model}.

\subsection{Proof of \Cref{lem: div between true and estimated posterior}}\label{subsec: proof triangle inequality filtering decomposition}
The  bound follows from the uniform geometric stability of the filter and the stability of the divergence $D$. We follow the same argument proposed in \cite[Prop. 2]{al-jarrah2023optimal}.
\begin{proof}
Using $\widehat \pi^t = \T_{Y_{t,t}}(\widehat \pi^t)$, $\pi^t = \T_{Y_{t,0}}(\pi^0) =  \T_{Y_{t,0}}(\widehat \pi^0) $, and 
the repeated application of the triangle inequality, the divergence between $\widehat{\pi}^{t}$ and $\pi^{t}$ is decomposed  according to
\begin{align*}
D(\widehat{\pi}^{t},\pi^{t}) = D(\T_{Y_{t,t}}(\widehat \pi^t),\T_{Y_{t,0}}(\widehat \pi^0)) &\leq 
\sum_{\tau = 1}^{t}D(\T_{Y_{t,\tau}}(\widehat \pi^\tau),\T_{Y_{t,\tau-1}}(\widehat \pi^{\tau-1}))\\
&= \sum_{\tau=1}^{t} D(\T_{Y_{t,\tau}} \circ {\widehat \T}_{Y_\tau}(\widehat \pi^{\tau-1}) , \T_{Y_{t,\tau}}\circ {\T}_{Y_\tau}(\widehat \pi^{\tau-1}))\, ,
\end{align*}
Application of the uniformly geometrically stable property of the filter~\cref{eq:filter-stability}, and the stability property of the divergence according to~\cref{eq: metric D stability} yields
\begin{align}\label{eq: filter decomposition}
\begin{split}
 D(\widehat{\pi}^{t},\pi^{t}) 
&\leq \Cfs\sum_{\tau=1}^{t} (1-\lambda)^{t-\tau}   D(\widehat{\T}_{Y_\tau}(\widehat \pi^{\tau-1}), \T_{Y_\tau} (\widehat \pi^{\tau-1}))
    \\
    &=\Cfs\sum_{\tau=1}^{t} (1-\lambda)^{t-\tau} D(\nabla_{u} \widehat \phi_\tau( Y_\tau,\cdot) \# \widehat\eta_\U^\tau, \nabla_{u}  \phi^\dagger_\tau( Y_\tau,\cdot) \# \widehat\eta_\U^\tau)\\&\leq C_D \Cfs \sum_{\tau=1}^{t} (1-\lambda)^{t-\tau}
    \|\nabla_u \widehat \phi_\tau( Y_{\tau},\cdot) - \nabla_u \phi^{\dagger}_{\tau}(Y_{\tau},\cdot)\|_{L_{\widehat\eta_\U^\tau}^2}.
    \end{split}
\end{align}
In order to obtain the bound for the exact mean filtering error, we take  the expectation of both sides with respect to the true observation variables $\mathscr{Y}_t$, where each variable $Y_\tau$ is distributed according to $\nu^\tau_\Y$, and apply the Jensen's inequality 
\begin{align*}
 \mathbb{E}_{ {\mathscr{Y}}_{\tau}}\, \|\nabla_u \widehat \phi_\tau( Y_{\tau},\cdot) - \nabla_u \phi^{\dagger}_{\tau}(Y_{\tau},\cdot)\|_{L_{\widehat\eta_\U^\tau}^2} &=\mathbb{E}_{ {\mathscr{Y}}_{\tau-1}}\, \Expect_{Y_\tau}\,\|\nabla_u \widehat \phi_\tau( Y_{\tau},\cdot) - \nabla_u \phi^{\dagger}_{\tau}(Y_{\tau},\cdot)\|_{L_{\widehat\eta_\U^\tau}^2} \\&\leq \mathbb{E}_{ {\mathscr{Y}}_{\tau-1}}\|\nabla_{u}\widehat{\phi}_{\tau} - \nabla_{u}\phi^{\dagger}_{\tau}\|_{L_{\wh{\eta}^\tau_\U \otimes \nu^\tau_\Y}^2}\, .
\end{align*}
In order to obtain the bound for the approximate filtering error, we follow the same procedure and take the expectation with respect to artificial observations $\wh{\mathscr{Y}}_t$, where each variable $\wh{Y}_\tau$ is distributed according to $\wh{\nu}^\tau_\Y$. 
\end{proof}

\subsection{Proof of \Cref{thm: approx filter error bound}}\label{subsec: proof of approximate error bound}
We now move to the bound for the approximate filtering error, which follows directly from the results of~\cref{Sec:Error analysis for conditional OT} after the decomposition of~\Cref{lem: div between true and estimated posterior}.
\begin{proof}
Starting with the inequality~\cref{eq:first-bound-approx-filter-error},
    \begin{align*}
       \Expect_{1:t}^{\textnormal{train}}\,  \mathbb{E}_{ \wh{\mathscr{Y}}_{t}}D( \wh \pi^t, 
    \pi^t)\,  &\leq  C_D \Cfs \sum_{\tau=1}^t (1-\lambda)^{t-\tau}\mathbb{E}_{ \widehat{\mathscr{Y}}_{\tau-1}} \Expect_{1:\tau}^{\textnormal{train}} 
    \|\nabla_u \widehat \phi_\tau - \nabla_u \phi^{\dagger}_{\tau}\|_{L_{\widehat\eta^\tau}^2}\\
    &\leq  C_D \Cfs \sum_{\tau=1}^t (1-\lambda)^{t-\tau} e_\tau  \leq  C_D \Cfs \frac{1}{\lambda} e_t\, , 
    \end{align*}
    where we used \[\Expect_{1:\tau}^{\textnormal{train}} 
    \|\nabla_u \widehat \phi_\tau - \nabla_u \phi^{\dagger}_{\tau}\|_{L_{\widehat\eta^\tau}^2}\leq \sqrt{\mathbb{E}_{1:\tau}^{\textnormal{train}}\|\nabla_{u}\widehat{\phi}_{\tau} - \nabla_{u}\phi^{\dagger}_{\tau}\|_{L_{\widehat\eta^\tau}^2}^2}\leq e_\tau \leq e_t\, ,\] for all $\tau\leq t$, according to \cref{thm: slow rate} (for slow rate) or \cref{thm: fast rate} (for fast rate), and the fact that $\lambda \in (0,1)$ to bound the geometric sum. 
\end{proof}

\editstart
\subsection{Proof of~\Cref{multiplicative} in \Cref{thm: multiplicative+additive bound exact filtering error}}\label{subsec: proof of exact error bound 2}  In this section, we first extend the approximate mean filtering error bound to the exact mean filtering error via a multiplicative bound assuming that the minorization condition of~\eqref{eq: minorization condition} holds.

\begin{proof}
First notice that the minorization condition in~\eqref{eq: minorization condition} allows for uniform control for any $\tau\leq t$ of the density ratio
\begin{align*}
\left\|\frac{\d\nu_{\Y}^\tau(y)}{\d\widehat\nu_{\Y}^\tau(y)}\right\|_{L_{\wh\nu_\Y^\tau}^\infty} &= \left\|\frac{\int h(y\mid u)\d\mathcal{A}\pi^{\tau-1}(u)}{\int h(y\mid u)\d\mathcal{A}\wh \pi^{\tau-1}(u)}\right\|_{L_{\wh\nu_\Y^\tau}^\infty} = \left\|\frac{\d\mathcal{A}\pi^{\tau - 1}(u)}{\d\mathcal{A}\wh\pi^{\tau - 1}(u)}\right\|_{L_{\mathcal{A}\wh\pi^{\tau - 1}}^\infty}\\
& = \left\|\frac{\int a( u\mid u')\d\pi^{\tau - 1}(u')}{\int a( u\mid u')\d\wh\pi^{\tau - 1}(u')}\right\|_{L_{\mathcal{A}\wh\pi^{\tau - 1}}^\infty}\leq \frac{\sqrt{\Cfs}}{1/\sqrt{\Cfs}}  = \Cfs\, .
\end{align*}
Then, the following multiplicative bound holds
\begin{align}
\|\nabla_u \widehat \phi_\tau - \nabla_u \phi^{\dagger}_{\tau}\|_{L_{\widehat\eta^\tau}^2}&\leq \left\|\frac{\d\nu_{\Y}^\tau (y)}{\d\widehat\nu_{\Y}^\tau (y)}\right\|_{L_{\wh\nu_\Y^\tau}^\infty}\|\nabla_u \widehat \phi_\tau - \nabla_u \phi^{\dagger}_{\tau}\|_{L_{\widetilde\eta^\tau}^2}\label{eq: multiplicative bound}\leq \Cfs \|\nabla_u \widehat \phi_\tau - \nabla_u \phi^{\dagger}_{\tau}\|_{L_{\widetilde\eta^\tau}^2}\, .
\end{align}
The proof can then be concluded following the same steps employed in the proof of~\Cref{thm: approx filter error bound}.
\end{proof}

\begin{remark}\label{rem: comments on mult bound extension}
We remark that the proposed multiplicative bound approach could be further extended beyond the setting in which the minorization condition of~\eqref{eq: minorization condition} does not necessarily hold, e.g., on unbounded domains. Under a suitable set of assumptions allowing the uniform control of $\left\|\d\nu_{\Y}^\tau (y)/\d\widehat\nu_{\Y}^\tau (y)\right\|_{L_{\wh\nu_\Y^\tau}^q}$ for $q\in[1,\infty)$ and a bound on the mixed norm 
\begin{align*}
\left(\int\|\nabla_u \widehat \phi_\tau(y,\cdot) - \nabla_u \phi^{\dagger}_{\tau}(y,\cdot)\|_{L_{\widetilde\eta_\U^\tau}^{2}}^{2p} \d \widehat\nu_\Y^\tau\right)^{1/p}\lesssim \|\nabla_u \widehat \phi_\tau - \nabla_u \phi^{\dagger}_{\tau}\|_{L_{\widetilde\eta^\tau}^2}^{(p_0-p)/(p(p_0-1))}\, ,
\end{align*}
with $p = 1-1/q, p_0>p$ via Jensen's inequality and norm interpolation, one might relax the bound in~\eqref{eq: multiplicative bound} using H\"older's inequality towards obtaining slower rates proportional to up to $e_t^{1/p}$ in more general settings on unbounded domains. Moreover, by revisiting the conditional OT rates in $L^2_\eta$ of~\Cref{Sec:Error analysis for conditional OT} with respect to the stronger $L^{2,2p}_\eta$ mixed norm, improved filtering bounds avoiding the suboptimal norm interpolation argument could be obtained following this multiplicative approach. We leave this as an 
exciting avenue of future research.
\end{remark}
\editend

\editstart
\subsection{Proof of~\Cref{additive} in \Cref{thm: multiplicative+additive bound exact filtering error}}\label{subsec: proof of exact error bound}
The proof for the bound of the exact mean filtering error via the additive decomposition of $\|\nabla_u \widehat \phi_\tau - \nabla_u \phi^{\dagger}_{\tau}\|_{L_{\widehat\eta^\tau}^2}$
\editend
is more involved. In light of~\Cref{rem: general remark on mult+add bounds}, we first introduce the needed preliminary lemmas in the setup where the Lipschitz condition of~\eqref{eq:Lipshitzness in y} in~\Cref{assump: exact filter error extra assump}(1) is relaxed 
\editstart
to the H\"older one condition
\begin{align}\label{eq:Holder in y}
\begin{split}
                \| \nabla_u \phi^{\dagger}_{\tau}(y,\cdot) - \nabla_u \phi^\dagger_\tau(y',\cdot)\|            &\le L_\Y \| y- y'\|_\Y^\zeta, \quad
                \| \nabla_u \widehat\phi_{\tau}(y,\cdot) - \nabla_u \widehat\phi_\tau(y',\cdot)\|       \le L_\Y \| y- y'\|_\Y^\zeta\, ,
\end{split}
\end{align}
for $\zeta\in[0,1]$.
\editend

\begin{lemma}\label{lem:bound-phi-Y-Yhat}
    Under~\cref{assump: divergence} and~\cref{assump: exact filter error extra assump} relaxing~\eqref{eq:Lipshitzness in y} to~\eqref{eq:Holder in y} for $\zeta\in[0,1]$, we have
    \begin{align} \label{eq:bound-phi-Y-Yhat}
        \|\nabla_u \widehat \phi_\tau - \nabla_u \phi^{\dagger}_{\tau}\|_{L_{\widetilde \eta^\tau}^2} \leq  \|\nabla_u \widehat \phi_\tau - \nabla_u \phi^{\dagger}_{\tau}\|_{L_{\wh{ \eta}^\tau}^2} + 2C_{D}'
        L_\Y D(\pi^{\tau-1},\widehat \pi^{\tau-1})^{\zeta}.
    \end{align}
\end{lemma}
\begin{proof}
   By the application of the triangle inequality 
\begin{align*}
    \|\nabla_u \widehat \phi_\tau( Y_{\tau},\cdot) &- \nabla_u \phi^{\dagger}_{\tau}( Y_{\tau},\cdot)\|_{L_{\widehat\eta_\U^\tau}^2} \leq \text{I}(\tau) + \text{II}(\tau) + \text{III}(\tau),
\end{align*}
where 
\begin{align*}
    \text{I}(\tau) &\coloneqq \|\nabla_u \widehat \phi_\tau( \widehat Y_{\tau},\cdot) - \nabla_u \phi^{\dagger}_{\tau}( \widehat Y_{\tau},\cdot)\|_{L_{\widehat\eta_\U^\tau}^2}, \\
    \text{II}(\tau) &\coloneqq  \|\nabla_u \widehat \phi_\tau( Y_{\tau},\cdot) - \nabla_u \widehat \phi_{\tau}( \widehat Y_{\tau},\cdot)\|_{L_{\widehat\eta_\U^\tau}^2},\\
    \text{III}(\tau) &\coloneqq 
 \|\nabla_u  \phi^\dagger_\tau( \widehat Y_{\tau},\cdot) - \nabla_u \phi^{\dagger}_{\tau}( Y_{\tau},\cdot)\|_{L_{\widehat\eta_\U^\tau}^2}.
\end{align*}
Here,
$\widehat Y_\tau$ is distributed according to $\widehat \nu^\tau_\Y$, and coupled with $Y_\tau$ according the coupling $\Gamma_\tau$ suggested by~\cref{assump: exact filter error extra assump}. Upon taking the expectation of $\textnormal{I}(\tau)$ over $\widehat Y_\tau$, and the application of the Jensen's inequality,
\begin{align*}
    \mathbb E_{\widehat Y_\tau}\text{I}(\tau) \leq \|\nabla_u \widehat \phi_\tau - \nabla_u \phi^{\dagger}_{\tau}\|_{L_{\widehat\eta^\tau}^2}.
\end{align*}
Using the Hölder property of $\nabla_u \widehat \phi(\cdot,y)$ from~\cref{assump: exact filter error extra assump}, $\textnormal{II}(\tau)$ is bounded by
\begin{align*}
    \text{II}(\tau) \leq L_\Y \|Y_\tau  - \widehat Y_\tau\|_\Y^{\zeta}. 
\end{align*}
After taking the expectation over $Y_\tau$ and $\widehat Y_\tau$, the joint application of~\cref{eq:coupling-assump} and Jensen's inequality yields 
\begin{align*}
    \mathbb E_{Y_\tau,\widehat Y_\tau} \text{II}(\tau) &\leq L_\Y\mathbb E_{(Y_\tau,\widehat Y_\tau) \sim \Gamma_\tau} \|Y_\tau  - \widehat Y_\tau\|_\Y^{\zeta}  \leq C_D' L_\Y D(\pi^{\tau-1},\widehat \pi^{\tau-1})^{\zeta}.
\end{align*}
Following the same argument, expectation of $\text{III}(\tau)$ is also bounded according to
\begin{align*}
    \mathbb E_{Y_\tau,\widehat Y_\tau} \text{III}(\tau) &\leq  C_D' L_\Y D(\pi^{\tau-1},\widehat \pi^{\tau-1})^{\zeta}.
\end{align*}
Combining the bounds on the expectation of $\text{I}(\tau)$, $\text{II}(\tau)$, and $\text{III}(\tau)$, we arrive at the result~\cref{eq:bound-phi-Y-Yhat}.
\end{proof}

The combination of \cref{lem: div between true and estimated posterior} and \cref{lem:bound-phi-Y-Yhat} concludes the following bound for the exact mean filtering error. 
\begin{lemma}
\label{lem: exact filtering error}
 Suppose the exact filter satisfies  \cref{def: uniformly geometrically stable filter}, the divergence $D$ satisfies~\cref{assump: divergence}, and  \Cref{assump: exact filter error extra assump} relaxing~\eqref{eq:Lipshitzness in y} to~\eqref{eq:Holder in y} for $\zeta\in[0,1]$ holds. Then, the  exact mean filtering error~\cref{eq: exact filter error} satisfies the bound
 \begin{align}\label{eq:exact filtering error bound}
     \mathbb{E}&_{ \mathscr{Y}_{t}}D( \wh \pi_t, 
    \pi_t)\leq  C \overline e_t
   +C^{\zeta}\widetilde C(1-\lambda)^{t-1}\sum_{\tau=0}^{t-1}(1-\lambda)^{-\tau}\overline e_\tau^\zeta\prod_{s=\tau+1}^{t-1}\left(\frac{\zeta\widetilde C}{(1-\lambda)(C\overline e_s)^{1-\zeta}}+1\right)\, ,
\end{align}
where $\lambda\in (0,1)$, the constants $C=C_DC_{\textnormal{filter}}$ and $\widetilde C=2L_{\Y}C_DC_D'C_{\textnormal{filter}}$ are uniform in time, and $\overline e_\tau \coloneqq \sum_{k=1}^{\tau} (1-\lambda)^{\tau-k}    \mathbb{E}_{ \mathscr{Y}_{k-1}} 
    \|\nabla_u \widehat \phi_k - \nabla_u \phi^{\dagger}_{k}\|_{L_{\widehat\eta^k}^2}$. 
\end{lemma}
\begin{remark} \label{rem: zeta=1 regime}
\label{rem:general zeta bound issue}
    As introduced in~\Cref{rem: general remark on mult+add bounds}, notice that the bound of~\eqref{eq:exact filtering error bound} blows up when $\zeta<1$ and the errors $\bar{e}_{s}$ go to 0. This makes the bound meaningful only when $\zeta=1$.
\end{remark}
\begin{proof}
 Application of \cref{lem: div between true and estimated posterior} and \cref{lem:bound-phi-Y-Yhat} concludes 
\begin{align*}
\mathbb{E}_{ \mathscr{Y}_{t}}D(\widehat{\pi}_{t},\pi_{t}) &\leq C_D\Cfs \sum_{\tau=1}^{t} (1-\lambda)^{t-\tau}    \mathbb{E}_{ \mathscr{Y}_{\tau-1}} \left(
   \|\nabla_u \widehat \phi_\tau - \nabla_u \phi^{\dagger}_{\tau}\|_{L_{\widehat\eta^\tau}^2} 
   + 2L_\Y C_D' D(\pi_{\tau-1},\widehat \pi_{\tau-1})^{\zeta}
   \right)
   \\
    & \leq  C \overline e_t + \widetilde C(1-\lambda)^{t-1}\sum_{\tau=0}^{t-1} (1-\lambda)^{-\tau} \mathbb{E}_{ \mathscr{Y}_{\tau}}D(\widehat{\pi}_{\tau},\pi_{\tau})^{\zeta},
\end{align*}
where $C \coloneqq C_{D}\Cfs$, $\widetilde C \coloneqq 2L_{\Y}C_D'C_D \Cfs$, and $\overline e_t$ are defined in the statement of the Lemma. 
Then, applying a rescaled version of the discrete Gronwall's Lemma for L-type kernels (see~\cite[Lem. 100~-~107, eq. 2.66]{sever2003some}, referencing the original results of \cite{dragomir1992discrete,pachpatte1977discrete}), we obtain 
\begin{align*}
\mathbb{E}_{ \mathscr{Y}_{t}}D(\widehat{\pi}_{t},\pi_{t})
&\leq C \overline e_t
   +C^{\zeta}\widetilde C(1-\lambda)^{t-1}\sum_{\tau=0}^{t-1}(1-\lambda)^{-\tau}\overline e_\tau^\zeta\prod_{s=\tau+1}^{t-1}\left(\frac{\zeta\widetilde C}{(1-\lambda)(C\overline e_s)^{1-\zeta}}+1\right)\, .
\end{align*}
\end{proof}

Now, we are ready to apply \cref{thm: slow rate} and \cref{thm: fast rate} to the bound~\cref{eq:exact filtering error bound} with $\zeta=1$ and prove the result of{\color{black}~\cref{additive} in~\Cref{thm: multiplicative+additive bound exact filtering error} }for the exact mean filtering error. 
\begin{proof}[Proof of {\color{black}~\cref{additive} in~\Cref{thm: multiplicative+additive bound exact filtering error}}]
    Starting with the inequality~\cref{eq:exact filtering error bound} setting $\zeta =1$, the term in $\overline e_\tau$ is bounded by
    \begin{align*}
       \Expect_{1:\tau}^{\textnormal{train}}\, \overline e_\tau = \sum_{k=1}^\tau (1-\lambda)^{\tau-k}\mathbb{E}_{ \mathscr{Y}_{k-1}} \Expect_{1:k}^{\textnormal{train}} 
    \|\nabla_u \widehat \phi_k - \nabla_u \phi^{\dagger}_{k}\|_{L_{\widehat\eta^k}^2} \leq \sum_{k=1}^\tau (1-\lambda)^{\tau-k} e_\tau  \leq \frac{1}{\lambda} e_t,
    \end{align*}
    where we used $\Expect_{1:k}^{\textnormal{train}} 
    \|\nabla_u \widehat \phi_k - \nabla_u \phi^{\dagger}_{k}\|_{L_{\widehat\eta^k}^2}\leq e_\tau \leq e_t$ for all $k\leq \tau\leq t$, according to \cref{thm: slow rate} (for slow rate) or \cref{thm: fast rate} (for fast rate) , and the fact that $\lambda \in (0,1)$ to bound the geometric sum. Applying this bound to~\cref{eq:exact filtering error bound} with $\zeta=1$, we arrive at the inequality
    \begin{align*}     \Expect_{\textnormal{train}}\mathbb{E}_{ \mathscr{Y}_{t}}D( \wh \pi_t, 
    \pi_t)&\leq  \frac{C}{\lambda} e_t +  C\widetilde C\sum_{\tau=0}^{t-1}\left(\widetilde C+1-\lambda \right)^{t-\tau-1}\frac{1}{\lambda} e_t\\
    &\leq \frac{C}{\lambda}\left(1+\widetilde C\sum_{\tau=0}^{t-1}\left(\widetilde C+1-\lambda \right)^{t-\tau-1} \right) e_t\\&\leq  \frac{C}{\lambda}\left(1+ \frac{\widetilde C}{|\lambda-\widetilde C|}\max(1,\varrho^t) \right)  e_t\\&\leq \frac{C}{\lambda}\left(1+ \frac{\widetilde C}{|\lambda-\widetilde C|}\right) \max(1,\varrho^t) e_t,
    \end{align*}
    concluding the result. 
\end{proof}

\subsection{Calculation for condition 2 of \Cref{assump: exact filter error extra assump}}\label{subsec: calculation for Y stability}
Here, we show that the filtering model of \eqref{eq: simple filtering system} with the same Lipschitzness and noise assumptions of~\Cref{prop: Poincare filtering} satisfies condition 2 of~\Cref{assump: exact filter error extra assump} for $D=W_1$.

For any two distributions $\pi$ and $\wh{\pi}$, let $U\sim \pi$ and $\widehat U\sim\widehat \pi$. Also, let $V$ and $W$ be independent random variables with the same distributions as $V_t$ and $W_t$.   Then, $Y =  h( a(U)+V)+W$ and  $\wh{Y} =  h( a(\wh{U})+V)+W$ defines a coupling between $Y$ and $\wh{Y}$, with marginal distributions $\nu(y,u)=h(y\mid u)\mathcal A \pi(u)$ and $\wh{\nu}=h(y\mid u)\mathcal A \wh{\pi}(u)$, respectively. Hence,
\begin{align*}
\mathbb{E}\|Y - \wh{Y}\|_{\Y}
\le \mathbb{E}\|h( a(U)+V)- h( a(\wh{U})+V)\|_{\Y}\leq
L_{h}\mathbb{E}\| a(U)- a(\wh{U})\|_{\U}\leq 
L_{h}L_{a}\mathbb{E}\|U-\wh{U}\|_{\U}\, .
\end{align*}
Therefore, choosing the optimal $W_1$-coupling between $U$ and $\wh{U}$ yields
\begin{align*}
\mathbb{E}\|Y - \wh{Y}\|_{\Y}\leq L_{h}L_{a}D(\pi,\wh{\pi})\, ,
\end{align*}
proving condition 2 in \Cref{assump: exact filter error extra assump} for 
$D=W_1$ and the filtering model~\eqref{eq: simple filtering system}.

A particular instance of the just considered filtering system in \eqref{eq: simple filtering system} is given by the Euler-Maruyama discretized time stepping scheme characterized by
\begin{align*}
a(U_{t-1})=U_{t-1}+\Delta t \ b(U_{t-1})\, , 
\end{align*}
where $b$ is a deterministic map which is $L_b$-Lipschitz. In this case, $a$ is consequently Lipschitz and, for $D=W_1$, condition 2 in \Cref{assump: exact filter error extra assump} is satisfied with
\begin{align*}
C_D' = L_{h}L_{a} =  L_{h}(1+\Delta t \ L_b)\, ,
\end{align*}
showing how the time discretization resolution impacts the theoretical error bounds obtained in {\color{black}~\cref{additive} of~\Cref{thm: multiplicative+additive bound exact filtering error} }. Time schemes with finer resolutions (i.e., smaller $\Delta t$) require fewer training samples to achieve accurate filter predictions.

\subsection{Proof of~\cref{prop: log concavity smoothness result}}
\label{proof: thm: general conditions on FEs}
Before proceeding with the proof, we recall the following well-known Brascamp-Lieb and Cramer-Rao inequalities.
\begin{theorem}[Brascamp-Lieb inequality \cite{brascamp1976extensions}]\label{thm: BL inequality}
Let $\mu(u)=\exp(-\mathscr{W}_{\mu}(u))$ be a probability measure on $\mathbb{R}^n$, where $\mathscr{W}_\mu\in C^2(\mathbb{R}^n)$ is strictly convex. Then, for all $f\in H_{\mu}^2(\mathbb{R}^n)$, it holds that
\begin{align*}
\int|f(u)|^2\mu(u)\d u-\left(f(u)\mu(u)\d u\right)^2\leq\int\langle \nabla f(u), [\nabla^2 \mathscr{W}_\mu(u)]^{-1}\nabla f(u)\rangle\mu(u)\d u\, .
\end{align*}
\end{theorem}
\begin{theorem}[Cramer-Rao inequality \cite{saumard2014log}]\label{thm: CR inequality}
Let $\mu(u)=\exp(-\mathscr{W}_\mu(u))$ be a probability measure on $\mathbb{R}^n$, where $\mathscr{W}_\mu\in C^2(\mathbb{R}^n)$ is strictly convex. Then, for all $f\in H_{\mu}^2(\mathbb{R}^n)$, it holds that
\begin{align*}
\int|f(u)|^2\mu(u)\d u-&\left(\int  f(u)\mu(u)\d u\right)^2\geq 
\\
&\left\langle \int \nabla f(u) \mu(u)\d u, \left[\int \nabla^2 \mathscr{W}_\mu(u) \mu(u) \d u \right]^{-1} \int \nabla f(u) \mu(u) \d u \right\rangle\, .
\end{align*}
\end{theorem}
As a useful corollary of these theorems, we derive the following Hessian bounds, which will later be used to prove \Cref{prop: log concavity smoothness result}.
\begin{theorem}[Hessian bounds]\label{thm: hessian bounds}
Let $\mu(x)=\mu(u,u')=\exp(-\mathscr{W}_\mu(u,u'))$ with $u\in\mathbb{R}^n$, $u'\in\mathbb{R}^{n'}$, $x\in \mathbb{R}^{d}$ and $n+n'=d$. Suppose $\mathscr{W}_\mu\in C^2(\mathbb{R}^d)$ is strictly convex. Define the marginal density $\rho(u)=\exp(-\mathscr{W}_\rho(u))$, where
$\mathscr{W}_\rho(u)=\int\exp(-\mathscr{W}_\mu(u,u'))\d u'$. Then, it holds that:
\begin{align}\label{eq: BL hessian bound}
\nabla_u^2\mathscr{W}_\rho(u)\succeq\frac{\int(\nabla_u^2 \mathscr{W}_\mu(u,u')-\nabla_{uu'}^2 \mathscr{W}_\mu(u,u')(\nabla_{u'}^2 \mathscr{W}_\mu(u,u'))^{-1}\nabla_{u'u}^2 \mathscr{W}_\mu(u,u'))\mu(u,u')\d u'}{\int\mu(u',u)\d u'}\, ,
\end{align}
and
\begin{align}\label{eq: CR hessian bound}
\begin{split}
\nabla_u^2 &\mathscr{W}_\rho(u)\preceq\frac{\int \nabla^2_{u}\mathscr{W}_\mu(u,u') \mu(u,u')\d u'}{\int \mu(u,u')\d u'}-\\
&\frac{\int \nabla^{2}_{u' u}\mathscr{W}_\mu(u,u') \mu(u,u')\d u' \cdot (\int \nabla^{2}_{u'}\mathscr{W}_\mu(u,u') \mu(u,u')\d u)^{-1} \cdot \int \nabla^{2}_{u u'}\mathscr{W}_\mu(u,u') \mu(u,u')\d u' }{\int \mu(u,u')\d u'}\, .
\end{split}
\end{align}
\end{theorem}

\begin{proof}
The lower bound \eqref{eq: BL hessian bound} is a direct consequence of \cite[Thm. A.2]{pathiraja2021mckean}, which references the original result of \cite[Thm. 4.2]{brascamp1976extensions}. By differentiating $\mathscr{W}_\rho(u)=\int\exp(-\mathscr{W}_\mu(u,u'))\d u'$ twice with respect to $u$, we obtain the following:
\begin{align}\label{eq: differentiation BL CR corollary}
\begin{split}
\nabla^2_{u}\mathscr{W}_\rho(u)&= \frac{\int \nabla^2_{u}\mathscr{W}_\mu(u,u')\mu(u,u')\d u'}{\int \mu(u,u')\d u'} \\
&-\left( \frac{\int (\nabla_{u}\mathscr{W}_\mu(u,u'))^2 \mu(u,u')\d u'}{\int \mu(u,u')\d u'}- \frac{(\int \nabla_{u}\mathscr{W}_\mu(u,u') \mu(u',u)\d u')^2}{(\int \mu(u',u) \d u')^2}\right)
\end{split}
\end{align}
The lower bound follows by applying the Brascamp-Lieb inequality \Cref{thm: BL inequality} to the variance term in \eqref{eq: differentiation BL CR corollary}, considering the probability measure 
\begin{equation*}
    \widetilde \mu(u,u')=\frac{\mu(u,u')}{\int \mu(u,u') \d u'}, \quad f(u,u')=\nabla_u \mathscr{W}_\mu(u,u').
\end{equation*} 
Similarly, the upper bound \eqref{eq: CR hessian bound} is obtained by applying the Cramér-Rao inequality \Cref{thm: CR inequality} to the same variance term, yielding the desired result.
\end{proof}

Next, we recall the following elementary result: the log-concavity and smoothness of a measure pushed forward by an affine map is preserved.
\begin{lemma}\label{lem: affine map log concavity}
Let $\mu(u)\propto\exp(-\mathscr{W}_{\mu}(u))$ with $0\prec \sigma_{\min}(\mathscr{W}_{\mu})I\preceq\nabla^2 \mathscr{W}_{\mu}\preceq \sigma_{\max}(\mathscr{W}_{\mu})I$ and $T(u)=Q u+b$ with $0\prec \sigma_{\min}(Q) I\preceq Q\preceq \sigma_{\max}(Q) I$ . Then, $(T\# \mu)(u)\propto\exp\left(-\mathscr{W}_{T\# \mu}(u)\right)$ with $\frac{\sigma_{\min}(\mathscr{W}_{\mu})}{\sigma_{\max}^2(Q)}I\preceq \nabla^2 \mathscr{W}_{T\#\mu} \preceq \frac{\sigma_{\max}(\mathscr{W}_{\mu})}{\sigma_{\min}^2(Q)}I$.
\end{lemma}
\begin{proof}
Since $Q$ is positive definite, it is also invertible. By change of variable formula, we get
\begin{align*}
(T\#\mu)(u)= \frac{\mu(Q^{-1}(u-b))}{|\textnormal{det}(Q)|}\propto\exp\big(- \mathscr{W}_{T\# \mu}(u)\big),
\end{align*}
with potential
\begin{align*}
\mathscr{W}_{T\# \mu}(u) = \mathscr{W}_{\mu}(Q^{-1}(u-b))-\log|\det (Q)|.
\end{align*}
Differentiating twice,
\begin{align*}
\nabla^2 \mathscr{W}_{T\# \mu}(u)
= (Q^{-1})^{T}\Big[\nabla^2 \mathscr{W}_{ \mu}(Q^{-1}(u-b)\Big]Q^{-1}\, ,
\end{align*}
and consequently
\begin{align*}
\frac{\sigma_{\min}(\mathscr{W}_{\mu})}{\sigma_{\max}^2(Q)}I\preceq \nabla^2 \mathscr{W}_{T\#\mu} \preceq \frac{\sigma_{\max}(\mathscr{W}_{\mu})}{\sigma_{\min}^2(Q)}\, .
\end{align*}
\end{proof}
Having presented the necessary preliminary results, we can now dive into the calculation behind the statement of~\Cref{prop: log concavity smoothness result}.
\begin{proof}[Proof of \Cref{prop: log concavity smoothness result}]
From the filtering update equations in \eqref{eq: Filtering eqs}, we have:
    \begin{align*}
    \widehat \eta_\U^1(u)=\mathcal{A}\widehat \pi_{0}=\mathcal{A} \pi_{0}&\propto \int_{\U} \exp(- a (u|u')-\mathscr{W}_{\pi_0}(u'))\d u' \coloneqq \exp(-\mathscr{W}_{\widehat \eta_\U^1}(u))
    \end{align*}
Recalling the notation introduced in~\eqref{eq: update functions}-\eqref{eq: recursive relations gamma}-\eqref{eq: Hessian defined for measures}, by \Cref{thm: hessian bounds} (on the joint pair $x=(u,u')$) and the Hessian boundedness assumptions of (i)-(iii), the lower and upper bounds on the Hessian of $\mathscr{W}_{\widehat \eta_\U^1}$ follow:
\begin{align*}
    \gamma_1 I=m(\sigma_{\min}(\mathscr{W}_{\pi_0}))I\preceq \nabla_u^2 \mathscr{W}_{\widehat \eta_\U^1}\preceq
    M(\sigma_{\max}(\mathscr{W}_{\pi_0}))I=\Gamma_1 I.
\end{align*}
Next, we recall that
\begin{align*}
\widehat \pi_1 = \nabla_u\widehat \phi_1(Y_1,\cdot)\#\eta_\U^1\, ,
\end{align*}
and, by~\Cref{lem: affine map log concavity}, we have that $\widehat \pi_1(u) =\exp(-\mathscr{W}_{\wh \pi_1}(u))$ with the Hessian bound
\begin{align*}
\frac{\gamma_1}{\sigma_{\max}^2(Q)} I\preceq \nabla_u^2 \mathscr{W}_{\widehat \pi_1}\preceq \frac{\Gamma_1}{\sigma_{\min}^2(Q)} I\, .
\end{align*}
Consequently, again by \Cref{thm: hessian bounds},
\begin{align*}
\gamma_2 I=m\left(\frac{\gamma_1}{\sigma_{\max}^2(Q)}\right)I\preceq \nabla_u^2 \mathscr{W}_{\widehat \eta_\U^2}\preceq M\left(\frac{\Gamma_1}{\sigma_{\min}^2(Q)}\right) I=\Gamma_2 I.
\end{align*}
Repeating this procedure in time yields the first string of inequalities in~\eqref{eq: Hessian bounds for measures of interest}:
\begin{align*}
\gamma_t I=m\left(\frac{\gamma_{t-1}}{\sigma_{\max}^2(Q)}\right)I\preceq \nabla_u^2 \mathscr{W}_{\widehat \eta_\U^{\tau}}\preceq M\left(\frac{\Gamma_{t-1}}{\sigma_{\min}^2(Q)}\right) I=\Gamma_t I.
\end{align*}
Next, we have that
\begin{align*}
\widehat \nu^t(y,u)&= \exp(-h(y\mid u))\widehat \eta_{\U}^t (u)
= \int_{\U} \exp(- h(y \mid u)- a (u|u')-\mathscr{W}_{\wh \pi_{t-1}}(u'))\d u' 
\\
&\coloneqq \exp(-\mathscr{W}_{\wh \nu^t}(y,u))\, .
\end{align*}
By a further application of~\Cref{thm: hessian bounds} (this time on the joint block $x=((y,u),u')$) and the Hessian boundedness assumptions of (i)-(iii) one also obtains the second recursive relation of \eqref{eq: Hessian bounds for measures of interest}:
\begin{align*}
\left(\theta_{\min}+\gamma_t\right) I \preceq \nabla^2 \mathscr{W}_{\widehat \nu^t}(y,u)\preceq \left(\theta_{\max}+\Gamma_\tau\right) I \, .
\end{align*}
In order to get the just derived Hessian bounds to remain bounded uniformly in time, the additional condition~\eqref{eq: conditions for fixed point} ensures that $|m'(\cdot/\sigma_{\max}^2(Q))|, |M'(\cdot/\sigma_{\min}^2(Q))|\leq 1$. As a consequence, one can calculate the fixed points of the contractions $m(\cdot/\sigma_{\max}^2(Q)),M(\cdot/\sigma_{\min}^2(Q))$ to bound
\begin{align}\label{eq: bounds uniform in time}
\begin{split}
\min\{\gamma^*,\sigma_{\min}(\mathscr{W}_{\pi_0})\} I&\preceq \mathscr{W}_{\widehat \eta_\U^t}\preceq \max\{\Gamma^*, \sigma_{\max}(\mathscr{W}_{\pi_0})\} I\, ,\qquad \forall \tau>0\, ,\\
\left(\theta_{\min}+\min\{\gamma^*,\sigma_{\min}(\mathscr{W}_{\pi_0})\} \right)I&\preceq \mathscr{W}_{\widehat \nu^t}\preceq \left(\theta_{\max}+\max\{\Gamma^*, \sigma_{\max}(\mathscr{W}_{\pi_0})\}\right)I\, ,
\end{split}
\end{align}
for any $\tau>0$, where
\begin{align*}
    \gamma^* &= \frac{\sigma_{\min}( a_u) - \sigma_{\max}^2(Q)\sigma_{\min}( a_{u'})}{2\sigma_{\max}^2(Q)}\\ 
    &+\frac{\sqrt{(\sigma_{\max}^2(Q)\sigma_{\min}( a_{u'}) - \sigma_{\min}( a_u))^2 
    + 4\sigma_{\max}^2(Q)(\sigma_{\min}( a_u)\sigma_{\min}( a_{u'}) - \sigma_{\max}^2( a_{uu'}))}}{2\sigma_{\max}^2(Q)},
\end{align*}
and
\begin{align*}
    \Gamma^* &= \frac{\sigma_{\max}( a_u) - \sigma_{\min}^2(Q)\sigma_{\max}( a_{u'})}{2\sigma_{\min}^2(Q)}\\ 
    &+\frac{\sqrt{(\sigma_{\min}^2(Q)\sigma_{\max}(a_{u'}) - \sigma_{\max}(a_u))^2 
    + 4\sigma_{\min}^2(Q)(\sigma_{\max}(a_u)\sigma_{\max}(a_{u'}) - \sigma_{\min}^2(a_{uu'}))}}{2\sigma_{\min}^2(Q)}\, ,
\end{align*}
are the solutions of $m(\gamma^*)=\gamma^*$ and $M(\Gamma^*)=\Gamma^*$.
\end{proof}

\section{Regularity of optimal transport maps}
In this section, we briefly recall a few important results on the regularity of optimal transport maps used in this paper. We first recall a result for log-concave OT transport.
\begin{theorem}[\cite{caffarelli2000monotonicity}]
\label{thm: caffarelli contraction thrm}
    If the reference and target measure are in the forms $\eta = \exp(-F) \d u$ and $\nu = \exp(-G) \d u$ with $\nabla^{2} F \preceq \beta_{F} I$, $\nabla^2 G \succeq \alpha_{G}I \succ 0$. Then the Brenier map $T^\dagger$ pushing forward $\eta$ to $\nu$
    has the form $T^\dagger  = \nabla \phi^{\dagger}$ where 
     $\nabla^2 \phi^\dagger \preceq \sqrt{\beta_{F}/\alpha_{G}} I$.
\end{theorem}

Thanks to the particular relation between the inverse of an optimal transport map and its associated potential's convex conjugate, the theorem just presented can be extended to the following corollary.
\begin{corollary}\label{thm: caffarelli contraction thrm plus convexity}
   Consider the above setting with the stronger 
   assumptions that $ \alpha_F I \preceq \nabla^2 F \preceq \beta_F I$
 and $\alpha_G I \preceq \nabla^2 G \preceq \beta_G I$.  
 Then the Brenier potential $\phi^\dagger$ satisfies 
 $\sqrt{\alpha_F / \beta_G} \preceq \nabla^2 \phi^\dagger \preceq \sqrt{\beta_F / \alpha_G} I$. 
\end{corollary}
\begin{proof}
The Hessian upper bound is already given by~\Cref{thm: caffarelli contraction thrm}, so we just need to prove the lower bound. Recall that the inverse of the optimal transport map can be expressed as
\begin{align*}
(T^{\dagger})^{-1} = \left(\nabla \phi^\dagger\right)^{-1} = \nabla \left(\phi^{\dagger*}\right)\, ,
\end{align*}
 and it is the optimal transport map achieving the reverse transport from $\nu$ to $\eta$. 
 Then, applying~\Cref{thm: caffarelli contraction thrm} one gets
 \begin{align*}
 \nabla^2 \phi^{\dagger*} \preceq \sqrt{\beta_{G}/\alpha_{F}} I\, .
 \end{align*}
 Moreover, since the convex conjugate of an $\alpha$-strongly convex map is $\beta$-smooth with $\beta=1/\alpha$, we also get that 
 \begin{align*}
\sqrt{\alpha_F/\beta_G} I \preceq \nabla^2 \phi^{\dagger}\, ,
 \end{align*}
 concluding the lower bound.
\end{proof}
\cref{thm: caffarelli contraction thrm} can also be extended to the conditional setting considered in this paper. A direct application of the Caffarelli regularity result in~\cite{caffarelli1992regularity} on the conditional measures implies the following result.
\begin{proposition}
\label{prop:BT-regularity}
Suppose that $\eta$ and $\nu$ are supported on a bounded, convex set
$\Omega\subseteq\Y\times\U$ and admit strictly positive densities in
$C^{k}(\Omega)$.
Define
\begin{align*}
&\Omega_\Y:=\left\{y \mid \exists u\in\U \ : \ (y,u)\in\Omega\right\}\, ,\\
&\Omega_{\U\mid y}:=\left\{u \mid (y,u)\in\Omega\right\}\, .
\end{align*}
Assume further that there exists a version of the conditional
densities $\eta(\cdot\mid y)$ and $\nu(\cdot\mid y)$ that lies in
$C^{k}(\Omega_{\U\mid y})$ for every $y\in\Omega_{Y}$.  
Then the conditional Brenier map $T_{\lowerrighttriangle}^\dag$ exists with component
\[
  T_{\U}(y,\cdot)\in C^{k+1}(\Omega_{\U\mid y})
  \quad\text{for all }y\in\Omega_{Y}.
\]
\end{proposition}

\section{Summary of OTF algorithms}\label{sm:otf_alg}

\editstart

We summarized the steps of the OTF algorithm with and without the EnKF-based reference as described in \Cref{Sec:Numerical experiments} in~\Cref{alg:otf}.

\begin{algorithm}[htp]
\small
\caption{Optimal transport filter (OTF)}
\begin{algorithmic}%
\STATE \textbf{Input:}
\editstart
Initial particles $\{{x}_{i}\}_{i=1}^N\sim \pi^0$, observation signal $\{Y_{t}\}_{t=1}^{t_f}$, noise generators for $\widetilde V_t, W_t$,
dynamic and observation models $a,h$, learning rates for $\psi,T$ or $\widetilde{T}$, regularization parameters $\lambda_\psi ,\lambda_T$. 
\STATE \textbf{Initialize:} initialize weights for neural nets $\psi$ and $T$ or $\widetilde{T}$. Create  random permutations $\{\sigma_i^t\}_{t,i=1,1}^{t_f,N}$. And set $u^{1}_i = a(x_i,\widetilde V_0^i), \quad y_i^1= h(u_i^1,W_1^i),\quad v^{1}_i = u^{1}_{\sigma_i^1}, \quad \forall i=1,\dots ,N$
\editend

\begin{minipage}[t]{0.42\linewidth}
\STATE \textbf{Without EnKF-based reference}
\FOR{$t = 1$ to $t_f$}
\STATE \textbf{Optimization step:}\\
Update $\theta_T,\theta_\psi$ according to~\eqref{eq:empirical-optimization}

\STATE \textbf{Particle update:}  \\
$u^{t+1}_i = a(\cdot,\widetilde V_{t+1}^i)\circ T(Y_{t},u_{i}^{t}),$\\
$\quad y_i^{t+1}= h(u_i^{t+1},W_{t+1}^i),\\v^{t+1}_i = u^{t+1}_{\sigma_i^{t+1}},\quad \forall i=1,\dots ,N$

\ENDFOR
\end{minipage}\hfill
\begin{minipage}[t]{0.55\linewidth}
\STATE \textbf{With EnKF-based reference}
\FOR{$t = 1$ to $t_f$}

\STATE \textbf{EnKF update:}
\STATE Compute EnKF gain $\widehat K^t$ using $\{u^{t}_i,{y}^{t}_{i}\}_{i=1}^N$
\STATE $v^{t}_i = u^{t}_{\sigma_i^t},\quad \overline y^t_i = h(v^t_i,W_{t}^i),\\  
w_i^t = v_i^t +\widehat K^t (y^t_i - \overline y^t_i),\quad
\forall i=1,\dots ,N$

\STATE \textbf{Optimization step:}\\
Update $\theta_\psi,\theta_{\widetilde{T}},$ according to~\eqref{eq:empirical-optimization} and~\eqref{eq:OT_EnKF_layer_loss}
\STATE \textbf{Particle update:}  \\
$w_i^{t,Y_{t}} = u_{i}^{t} + \widehat K^{t} (Y_{t} - y^{t}_i),$\\
$T(Y_{t},u_{i}^{t}) = w_i^{t,Y_{t}}  + \widetilde{T}(Y_{t},w_i^{t,Y_{t}}),$\\
$u^{t+1}_i = a(\cdot,\widetilde V_{t+1}^i)\circ T(Y_{t},u_{i}^{t}),\\
y_i^{t+1}= h(u_i^{t+1},W_{t+1}^i),\quad
\forall i=1,\dots ,N$

\ENDFOR
\end{minipage}

\end{algorithmic}
\label{alg:otf}
\end{algorithm}

\editend

\section{ of Posteriors in the bimodal static example}\label{sm: static bimodal example}
\editstart
In figure \Cref{sm:fig:marginal_densities} we compare the 
marginal densities produced by OTF with the ground truth 
posterior produced by SIR with a large number of particles. 
Going from 20 to 30 dimensions we see a noticable shift 
in the ability of OTF to capture the multi-modal structure 
of the poster leading to the drop in accuracy observed in 
\Cref{fig:bimodal_vs_dim_time_particles}. We note that 
this deterioration happens for a fixed network and training data
size and additional training data and larger size of the model can overcome this however; highlighting that the lack of 
error is due to the accuracy of the model and the fact that 
the ground truth posterior is highly irregular.

\begin{figure}[htp]
         \centering
     \begin{subfigure}[b]{0.65\textwidth}
         \centering
         \includegraphics[width=1\linewidth]{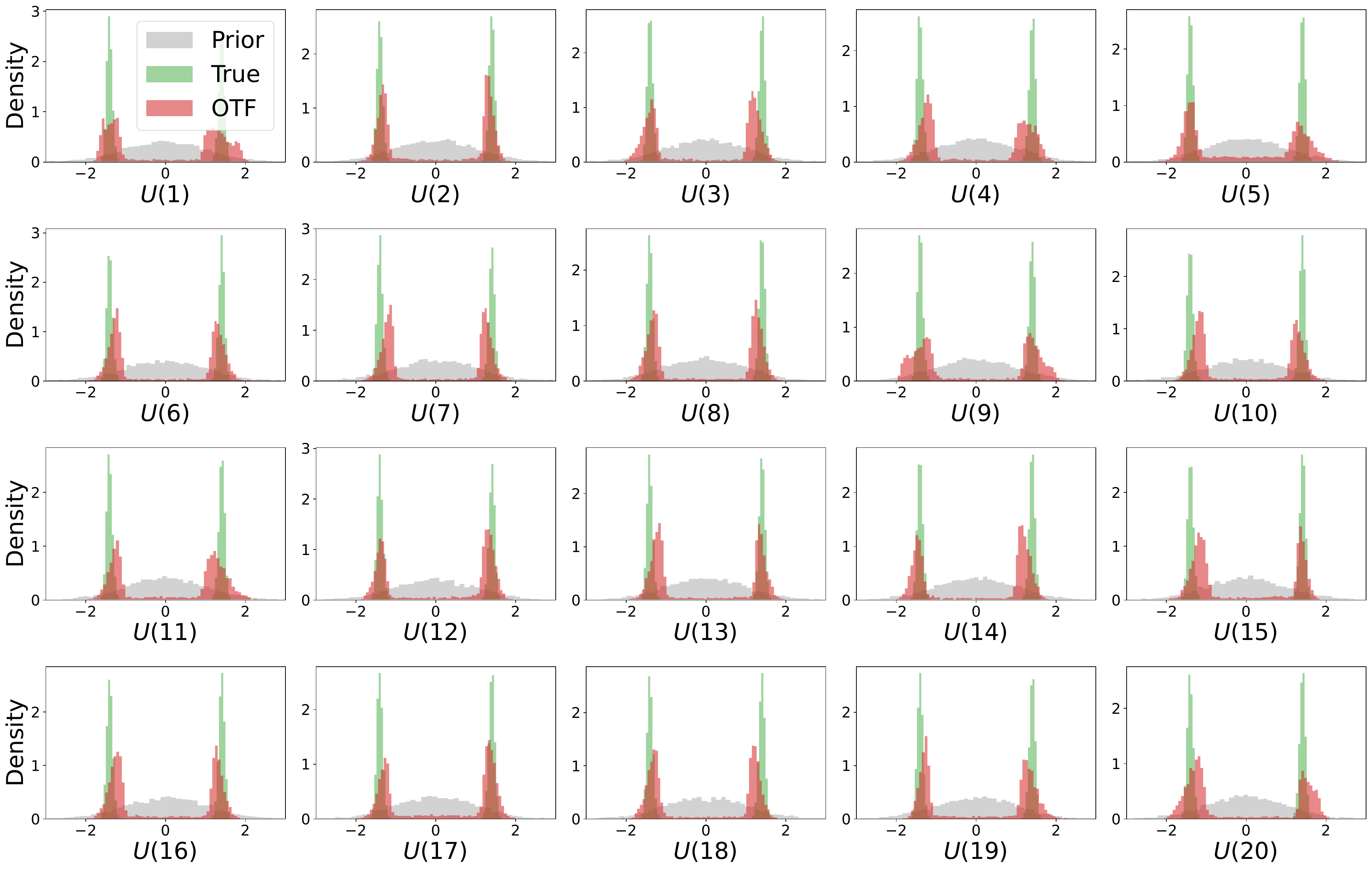}
         \caption{ $\dim = 20$.}
     \end{subfigure}
     \hfill
     \begin{subfigure}[b]{0.65\textwidth}
         \centering
         \includegraphics[width=1\linewidth]{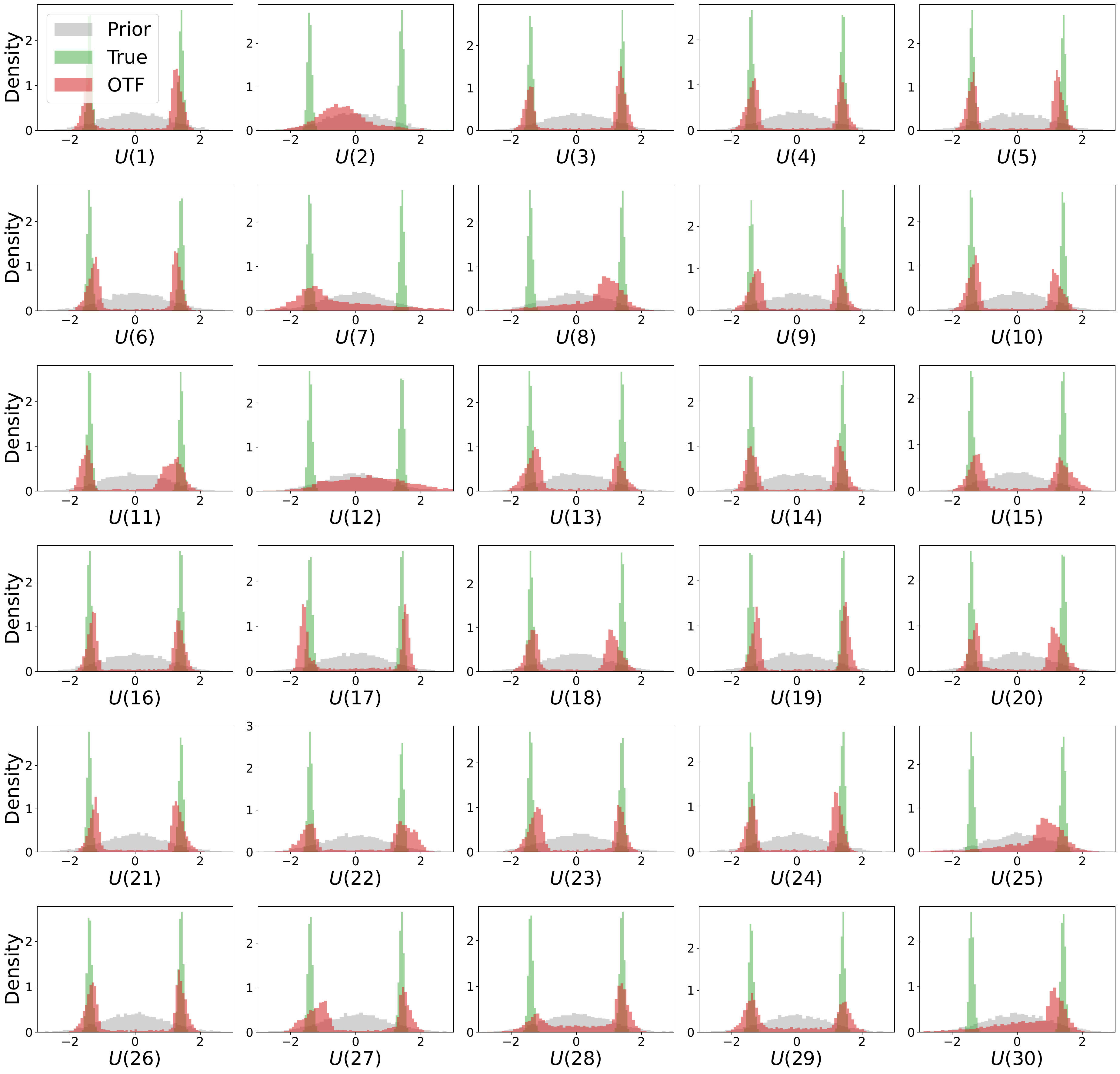}
         \caption{ $\dim = 30$.}
     \end{subfigure}
    \caption{
    \editstart
    Marginal densities of the OTF method for the bimodal static example of \Cref{sec: static bimodal example} for different problem dimensions. The OTF results capture the posterior modes accurately up to 20 dimensions and begin to break down around 30 dimensions.
    \editend
    }
    \label{sm:fig:marginal_densities}
\end{figure}
\editend

\section{Lorenz models}\label{sm:Lorenz-models}
\editstart

We briefly summarize the Lorenz 63 and 96 models used in \Cref{sec: L63,sec:L96}.
For Lorenz 63 we considered $U_t \in \Re^3$ satisfying the ODE \cite[Ex.~2.6]{law2015data}  
\begin{align*}
    \dot U_t (1) & =   \sigma ( U_t(2) - U_t(1) ) \\ 
    \dot U_t(2) & =  U_t(1) (\rho - U_t(3)) - U_t(2) \\ 
    \dot U_t(3) & =  U_t(1) U_t(2) - \beta U_t(3),
\end{align*}
with given initial conditions. We then consider the discrete time dynamical system 
\begin{equation*}
    U_{t+1} = a_{\textnormal{L63}}(U_t)
\end{equation*}
obtained by a forward Euler discretization of the ODE with step size $\Delta t$. 
In our experiments we used the typical parameter values $\sigma = 10, \rho= 28$, and $\beta = 8/3$. 

For Lorenz 96 we considered $U_t \in \Re^n$ satisfying the ODE system  \cite[Ex.~2.7]{law2015data}
\begin{equation*}
   \dot U_t(k) =   U_{t}(k-1) \big( U_t(k+1) - U_t(k-2)\big) - U_t(k) + F 
\end{equation*}
for $k= 1, \dots, n$ and subject to the periodic conditions 
$U_t(0) = U_t(n)$, $U_t(n+1) = U_1, U_{t}(-1) = U_{t}(n-1)$ at all times. 
We then obtain the discrete time dynamical system 
\begin{equation*}
    U_{t+1} = a_{\textnormal{L96}}(U_t)
\end{equation*}
by approximating the ODE with a Runge-Kutta scheme (RK4).
We used the typical parameter value $F = 10$.

\editend

\end{document}